\documentclass[11pt]{amsart}
\usepackage{amsthm,amsmath,amssymb}
\usepackage[a4paper, margin=2.40cm]{geometry}
\usepackage{bm}
\usepackage{graphicx}
\usepackage{caption}
\usepackage{subcaption}
\usepackage{bibentry}
\usepackage{mathtools}
\usepackage{enumerate}
\usepackage{mathrsfs} 
\usepackage[titletoc]{appendix}

\usepackage{tikz}

\newtheorem{theorem}{Theorem}[section]
\newtheorem{remark}{Remark}[section]
\newtheorem{definition}{Definition}[section]
\newtheorem{proposition}[definition]{Proposition}
\newtheorem{lemma}[definition]{Lemma}

\newtheorem{problem}[definition]{Problem}
\newtheorem{corollary}[definition]{Corollary}
\numberwithin{equation}{section}
\numberwithin{figure}{section}
\def\qed{\mbox{}\hfill\raisebox{-2pt}{\rule{5.6pt}{8pt}\rule{4pt}{0pt}}\medskip\par}
\newcommand{\co}{\mathcal{O}}
\newcommand{\cl}[1]{\overline{#1}}
\newcommand{\defeq}{\coloneqq}
\newcommand{\abs}[1]{|#1|}
\newcommand{\eqdef}{\eqqcolon}
\newcommand{\btheta}{{\bm{\theta}}}
\newcommand{\qiter}{\mathcal{Q}^{\rm iter}}
\newcommand{\qint}{\mathcal{Q}^{\rm int}}

\newcommand{\qtheta}{Q^{\btheta}}
\newcommand{\gsh}{\mathfrak{g}_{\rm sh}}
\newcommand{\gfive}{\mathfrak{g}_5}
\newcommand{\gsix}{\mathfrak{g}_6}
\newcommand{\gtheta}{\mathcal{G}_1^{\btheta}}
\newcommand{\vphithetastar}{\varphi_{\btheta}^\ast}
\newcommand{\Futheta}{\mathfrak{F}_{(u,\btheta)}}
\newcommand{\dfive}{\mathcal{D}^5}
\newcommand{\dsix}{\mathcal{D}^6}
\newcommand{\rcal}{\mathcal{R}}
\newcommand{\norm}[1]{\Vert #1 \Vert}
\DeclareMathOperator{\supp}{{\rm supp}}

\newcommand{\vmin}{v_{\min}}
\newcommand{\tabs}[1]{\Big|#1\Big|}

\newcommand{\bfx}{\mathbf{x}}
\newcommand{\mCase}[1]{{\rm \textbf{Case {#1}.}}}

\newcommand{\mytbf}[1]{\textbf{#1}}

\DeclareMathOperator{\sgn}{sgn}
\newcommand{\qeqref}[1]{(C-\ref{#1})}
\renewcommand{\qed}{\hfill\ensuremath{\square}}

\newcommand{\blank}{\mathrel{\,\cdot\,}}
\newcommand{\idot}{\mathcal{I}(\cdot,\btheta)}

\begin{document}
\title[Two-Dimensional Riemann Problem with Four-Shock Interactions]{Global Solutions of the Two-Dimensional \\
Riemann Problem with Four-Shock Interactions\\ for the Euler Equations for Potential Flow}

\author{Gui-Qiang G.~Chen}
\address[Gui-Qiang G.~Chen]{Mathematical Institute, University of Oxford, Oxford, OX2 6GG, UK}
\email{\tt chengq@maths.ox.ac.uk}

\author{Alex Cliffe}
\address[Alex Cliffe]{Mathematical Institute, University of Oxford, Oxford, OX2 6GG, UK}
\email{\tt alex.cliffe@maths.ox.ac.uk}

\author{Feimin Huang}
\address[Feimin Huang]{Academy of Mathematics and Systems Science, Chinese Academy of Sciences, Beijing 100190, China}
\email{\tt fhuang@amt.ac.cn}

\author{Song Liu}
\address[Song Liu]{Department of Mathematics, City University of Hong Kong, Hong Kong, China;
Academy of Mathematics and Systems Science, Chinese Academy of Sciences, Beijing 100190, China
}
\email{\tt liusong@amss.ac.cn}

\author{Qin Wang}
\address[Qin Wang]{Department of Mathematics, Yunnan University, Kunming, 650091, China}
\email{\tt mathwq@ynu.edu.cn}

\date{\today}

\begin{abstract}
We present a rigorous approach and related techniques to construct global solutions
of the two-dimensional (2-D) Riemann problem with four-shock interactions for
the Euler equations for potential flow.
With the introduction of three critical angles: the vacuum critical angle from the compatibility conditions,
the detachment angle, and the sonic angle, we clarify all
configurations of the Riemann solutions for the interactions of two-forward and two-backward shocks,
including the subsonic-subsonic reflection configuration that has not
emerged in previous results.
To achieve this, we first identify the three critical angles that determine the configurations,
whose existence and uniqueness follow from our
rigorous proof of the strict monotonicity of the steady detachment
and sonic angles for 2-D steady potential flow
with respect to the Mach number of the upstream state.
Then we reformulate the 2-D Riemann problem into the shock reflection-diffraction problem
with respect to a symmetric line, along with
two independent incident angles
and two sonic boundaries varying with the choice of incident angles.
With these, the problem can be further reformulated as a free boundary problem
for a second-order
quasilinear equation
of mixed elliptic-hyperbolic type.
The difficulties arise from the degenerate ellipticity of the
nonlinear equation
near the sonic boundaries, the nonlinearity of the free boundary condition,
the singularity of the solution near the corners of the domain,
and the geometric properties of the free boundary.
To solve the problem, we need to analyze
the solutions for a quasilinear degenerate elliptic equation
by the maximum principle of the mixed-boundary value problem,
the theory of the oblique derivative boundary value problem, the uniform
{\it a priori} estimates, and the iteration method.
To the best of our knowledge, this is the first rigorous result
for the 2-D Riemann problem with four-shock interactions
for the Euler equations.
The approach and techniques developed for the Riemann problem for four-wave interactions
should be useful for solving other 2-D Riemann problems
for more general Euler equations and
related nonlinear hyperbolic systems of conservation laws.

\end{abstract}

\maketitle
\smallskip
\tableofcontents

\smallskip
\section{Introduction}
We are concerned with the two-dimensional (2-D) Riemann problem for the Euler equations for potential flow,
consisting of the conservation law of mass
and the Bernoulli law:
\begin{align}\label{System:2DPotentialFlowEqs}
\begin{cases}
\, \partial_{t}\rho+{\rm div}_{\mathbf{x}}(\rho \nabla_{\mathbf{x}}\Phi) = 0\,,\\[1mm]
\, \partial_{t}\Phi+\frac{1}{2}\abs{\nabla_{\mathbf{x}}\Phi}^2+h(\rho) = B\,,
\end{cases}
\end{align}
where $(t,\mathbf{x})\in(0,\infty)\times\mathbb{R}^2$, $\rho = \rho(t,\mathbf{x})$ is the density, $\Phi=\Phi(t,\mathbf{x})$ is the velocity potential
({\it i.e.}, $\nabla_{\mathbf{x}}\Phi = (u,v)$
is the velocity),
$h(\rho) \defeq \int_{1}^{\rho}\frac{p'(s)}{s} \, {\rm d} s$ is the enthalpy with the polytropic pressure law $p(\rho)=A\rho^{\gamma}$ for $\gamma\geq1$ and $A > 0$, and
the Bernoulli constant \(B\) is determined by the far field flow.
Without loss of generality, we may fix $A = \gamma^{-1}$ by scaling and
denote the sonic speed by $c(\rho)\defeq \sqrt{p'(\rho)} = \rho^{\frac{\gamma-1}{2}}$.

In this paper, we focus on the 2-D Riemann problem with four-shock interactions and corresponding Riemann initial
given by
\begin{equation}
\label{Eq:RiemannInitialData}
(\rho,\Phi)(0 , \mathbf{x}) = (\rho_i, (u_i, v_i) \cdot \mathbf{x} ) \qquad\,\, \mbox{for $\mathbf{x} \in \Lambda_{i}$}\,,
\end{equation}
for suitable constant states $(i)$
with values \(U_i \defeq (\rho_i, u_i, v_i),\)
and corresponding scale-invariant domains $\Lambda_i \subseteq \mathbb{R}^2$ for $i=1,2,3,4$.
We choose these domains to be four sectors of \(\mathbb{R}^2,\) defined by
\begin{align}\label{Eq:InitialDomain}
\begin{aligned}
&\quad \Lambda_1 \defeq \big\{\mathbf{x} \in \mathbb{R}^2 \,:\, -{\theta}_{14} < \theta < {\theta}_{12} \big\} \,,
&\,\,
\Lambda_2 &\defeq \big\{\mathbf{x} \in \mathbb{R}^2 \,:\, {\theta}_{12} < \theta < \pi-{\theta}_{32} \big\}\,,\\
&\quad \Lambda_3 \defeq \big\{\mathbf{x} \in \mathbb{R}^2 \,:\, \pi-{\theta}_{32} < \theta < \pi+{\theta}_{34} \big\}\,,
&\,\,
\Lambda_4 &\defeq \big\{\mathbf{x} \in \mathbb{R}^2 \,:\, \pi+{\theta}_{34} < \theta < 2\pi-{\theta}_{14} \big\}\,,
\end{aligned}
\end{align}
where \(\theta\) is the polar angle of point \(\mathbf{x}\in\mathbb{R}^2,\)
and the four parameters \({\theta}_{12}, {\theta}_{32}, {\theta}_{34}, {\theta}_{14} \in (0,\frac{\pi}{2})\).
The initial data and shock discontinuities are depicted in Fig.~\ref{fig1:InitialFourShocks}.

\smallskip
The Euler equations for potential flow~\eqref{System:2DPotentialFlowEqs} are
the oldest, yet still prominent, paradigm of fluid dynamic equations for inviscid fluids,
which have been widely used in aerodynamics.
System~\eqref{System:2DPotentialFlowEqs}
has further been extended to the full Euler equations
for more general inviscid compressible fluids when the effect of vortex sheets and the derivation of
vorticity become significant.
One of the main features of the Euler equations is the formation of shocks
in finite time,
no matter how smooth the initial data are.
The shock is one of the most fundamental nonlinear waves
in compressible fluid flow,
and the analysis of shocks dates back to the 19th century;
see {\it e.g.}~\cite{Earnshaw1860, Hugoniot1889, Rankine1870, Riemann1860, Stokes1848}.
In~1860, Riemann~\cite{Riemann1860} first considered a special initial value problem
for the one-dimensional \mbox{(1-D)} Euler system, with two constant states
separated at the origin---now known as the Riemann problem.
For general 1-D strictly hyperbolic systems of conservation laws,
Lax~\cite{PDLax57} showed that the Riemann solutions are the combination of
three fundamental waves: the rarefaction waves, the shocks, and the contact discontinuities.
In 1965, Glimm~\cite{Glimm-65} employed the Riemann solutions as building blocks
of the Glimm scheme to establish the global existence of solutions in BV (bounded variation)
of 1-D strictly hyperbolic systems of conservation laws
for the initial data of small BV, for which the Riemann solutions play a fundamental role.
For the uniqueness and stability of BV solutions,
we refer to~\cite{Bressan2000, BressanLeFloch97, BressanLY99, LiuTP-81, LiuYang-99} and the references cited therein.
For a systematic theory of 1-D hyperbolic systems of conservation laws,
we refer to~\cite{Bressan2000, Dafermos2016, GlimmMajda91, PDLax73, MajdaRM-84}
and the references cited therein.
\begin{figure}
\centering
\includegraphics[width=13cm]{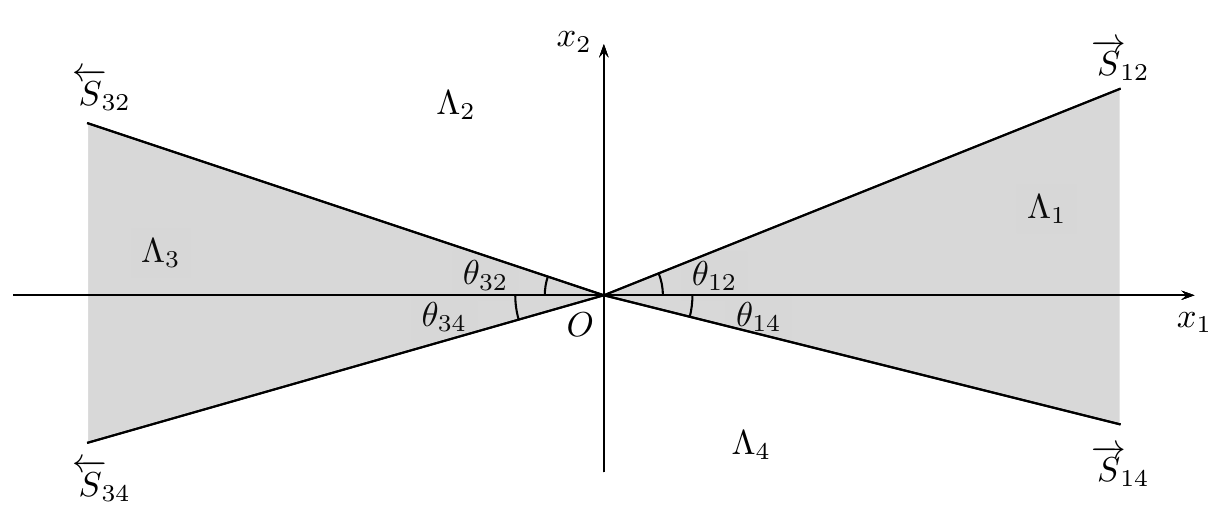}
\caption{The four sectors of the Riemann initial data.} \label{fig1:InitialFourShocks}
\end{figure}

Multi-dimensional (M-D)
Riemann problems for the Euler system are vastly more
complicated; see~\cite{ChangHsiao-89, Dafermos2016, GlimmMajda91, PDLax73, PDLax06}
and the references cited therein.
For the 2-D four-quadrant Riemann problem with four initial constant states
given in the four quadrants
so that each jump between two neighbouring quadrants projects exactly one planar fundamental wave,
Zhang-Zheng~\cite{ZhangZheng1990SIAM} predicted that there are a total of 16 genuinely different configurations
of the Riemann solutions for polytropic gases.
In Chang-Chen-Yang~\cite{ZhangChenYang1995DCDS, ZhangChenYang2000DCDS},
it is first observed that, when two initially parallel slip lines are present,
it makes a distinguished difference whether the vorticity waves generated have the same or opposite sign,
which, along with Lax-Liu~\cite{LaxLiu1998}, leads to the classification with a total of 19 genuinely
different configurations of the Riemann solutions for the compressible Euler equations
for polytropic gases, via characteristic analysis;
also see~\cite{Glimm-08,KTa,LiZhangYang98,SchulzR1993} and the references cited therein.
However, there have been few rigorous global results
for the 2-D four-quadrant Riemann problem for the Euler equations,
except for the result obtained by Li-Zheng in~\cite{JLiZheng10}, in which the four-rarefaction wave case was considered.
On the other hand, the 2-D Riemann problem has been
solved for simplified models such as the pressure gradient system~\cite{ChenWZ-21},
and the local structure of wave interactions in the Riemann problem has been analyzed in~\cite{JLiZhangZheng06, JLiZheng09, LiZheng11}.
For the 2-D Riemann problem for Chaplygin gases, we refer to~\cite{SXChen2012, Serre-09} and the references cited therein.

The purpose of this paper is to solve rigorously the 2-D Riemann problem with four-shock interactions
globally for the Euler equations for potential flow~\eqref{System:2DPotentialFlowEqs}.
The 2-D Riemann initial data \eqref{Eq:RiemannInitialData} under consideration in this paper satisfy that
\begin{equation*}
\begin{cases}
\,\text{one forward shock}\;
\overrightarrow{S}\!_{1j}\;\text{is generated between states} \; (1) \;\text{and}\; (j),\\[1mm]
\,\text{one backward shock}\; \overleftarrow{S}\!_{3j} \;\text{is generated between states} \; (3)\;\text{and}\; (j),
\end{cases}
\end{equation*}
for $j=2,4$,
and that the angular bisectors at the vertices of $\Lambda_{1}$ and $\Lambda_{3}$ in Fig.~\ref{fig1:InitialFourShocks} coincide.
Without loss of generality, we choose $\theta_{12}=\theta_{14}$ and $\theta_{32}=\theta_{34}$,
and then the $x_1$--axis may be considered as a rigid wall.
The above Riemann problem corresponds to one of the conjectures for
the Euler equations proposed in Zhang-Zheng~\cite{ZhangZheng1990SIAM}.

The problem considered has a close relation to the shock reflection-diffraction problem, whose research dates back to 1878, in which Ernst Mach revealed the complexity of
the shock reflection-diffraction configurations that include two fundamental modes:
the regular reflection-diffraction and
the Mach reflection-diffraction.
The shock reflection phenomena occur when a moving planar shock wave impinges on a rigid wall;
see~Courant-Friedrichs~\cite{Courant48} and Chen-Feldman~\cite{ChenFeldman-RM2018}.
In the recent three decades, there have been many theoretical results
on shock reflection/diffraction problems.
Indeed, Elling-Liu~\cite{EllingTPLiu2008CPAM}
considered the self-similar solutions of the Prandtl-Meyer problem for supersonic potential flow
onto a solid wedge for a class of physical parameters.
Chen-Feldman~\cite{ChenFeldman10} proved the first global existence
of solutions of the shock reflection-diffraction problem for potential flow
when the incident angles are small; later, they succeeded to establish the global existence of solutions
for the incident angles all the way up to the detachment angle in~\cite{ChenFeldman-RM2018}.
Furthermore, Chen-Feldman-Xiang~\cite{ChenFeldmanXiang2020ARMA}
proved the convexity of transonic shocks, which leads to the uniqueness
and stability of the admissible solution~\cite{CFX-20AIMS}.
More recently, Bae-Chen-Feldman~\cite{BCF-2019} obtained the existence and the optimal regularity of admissible
solutions of the Prandtl-Meyer reflection problem for the general case up to the detachment angle.
For the supersonic flow passing through a wing or curved wedge, see also~\cite{SXChen92, SXChen98, ChenFeldman22}.

For the 2-D Riemann problem with four-shock interactions under consideration,
we can reformulate the problem into a shock reflection-diffraction problem with respect to a symmetric line.
Compared with the previous results, such as the shock reflection-diffraction problem~\cite{ChenFeldman-RM2018}
and the Prandtl-Meyer reflection problem~\cite{BCF-2019},
we now have two independent incident angles in the four-shock reflection problem.
Correspondingly, we have two sonic boundaries varying with the choice of incident angles,
which leads to more complicated configurations.
After introducing the self-similar coordinates, the original problem can
be reformulated as a free boundary problem for a second-order quasilinear
partial differential equation (PDE)
of mixed elliptic-hyperbolic type.
The difficulties come from the degenerate ellipticity of the nonlinear PDE
near the sonic boundaries, the nonlinearity of the free boundary conditions,
the singularity of the solution near the corners of the domain,
and the geometric properties of the free boundary.
To solve our problem, we need to
analyze
the solutions
for a quasilinear degenerate elliptic equation
by the maximum principle of mixed-boundary value problems
and the theory of the oblique derivative boundary value
problems~\cite{Gilbarg-Trudinger83, Lieberman-86, Lieberman2013},
as well as the uniform estimates and the iteration method
from~\cite{BCF-2019, ChenFeldman-RM2018, EllingTPLiu2005}.

An important part of our analysis includes the determination of the
three critical angles: the vacuum critical angle, the detachment angle, and the sonic angle.
Their existence and uniqueness follow
from Lemma~\ref{lem:MonotonicityOfSteadyAngles},
whereby we prove the strict monotonicity of the steady detachment
and sonic angles for 2-D steady potential flow with respect to
the Mach number of the upstream state.
Using these critical angles,
we classify all configurations of the solutions of the Riemann problem.
In particular, the subsonic-subsonic reflection configuration
has never emerged in the previous works.
Lemma~\ref{lem:MonotonicityOfSteadyAngles} should be
useful in the analysis
of other shock reflection-diffraction problems.

To the best of our knowledge, this is the first result on the shock reflection-diffraction
of the Riemann problem for potential flow with piecewise constant initial data
in four sectors.
Furthermore, the mathematical methods developed in solving the free boundary problem
for nonlinear PDEs such as the uniform {\it a priori} estimate on the ellipticity of
the equation,
the anisotropic scaling method, and the sophisticated construction of the iteration set
and iteration map will provide insights for the analysis of more general
nonlinear problems.
In particular, the Euler equations for potential flow are the core part of
the full Euler equations, coupled with the incompressible-type Euler equations,
in a local domain, the techniques and approaches developed for the Riemann problem for four-wave interactions
will be useful to guide the analysis of the structure of Riemann solutions, shock reflection/diffraction problems,
and numerical simulations of the full Euler equations.

The organisation of this paper is as follows:
In Section~\ref{Sec:FormulateProblemsDefAdmisSolu},
we present the mathematical formulation of the Riemann problem with four-shock interactions
and the definition of admissible solutions of the free boundary problem,
and state three main results.
In Section~\ref{Sec:UniformEstiAdmiSolu},
we obtain the directional monotonicity, the strict ellipticity, and the uniform weighted
H\"{o}lder estimates of admissible solutions.
In Section~\ref{Sec:IterationMethod-Existence},
we prove the existence of admissible solutions by constructing a suitable iteration set and iteration map via
applying the Leray-Schauder degree theory.
Finally, in Section~\ref{Sec:OptimalRegularity-Convexity},
we give the proof of the optimal regularity of admissible solutions,
as well as the proof of the convexity of free boundaries and transonic shocks.
In Appendix~\ref{Sec:Appendix-A}, we give the details for some calculations
involving the potential flow equation
and, in Appendix~\ref{Sec:Appendix-B},
we present some known results that are needed throughout the proofs.

\section{The Riemann Problem and Main Theorems for Four-Shock Interactions}
\label{Sec:FormulateProblemsDefAdmisSolu}

In this section, we first present the mathematical formulation of the Riemann problem
with four-shock interactions
as a free boundary problem, along with the definition of admissible solutions of the free boundary problem,
and then state three main theorems.

\subsection{Mathematical formulation}\label{SubSec:MathFormulation}
The Riemann problem~\eqref{System:2DPotentialFlowEqs}--\eqref{Eq:RiemannInitialData} is invariant
under the self-similar scaling:
\begin{align*}
    (t,\mathbf{x}) \mapsto (\lambda t, \lambda \mathbf{x})\,,
    \quad (\rho , \Phi)(t,\mathbf{x}) \mapsto (\rho,
    \frac{1}{\lambda} \Phi)(\lambda t, \lambda \mathbf{x}) \qquad\;\; \mbox{for any $\lambda>0$}\,.
\end{align*}
This leads us to consider the self-similar solution $(\rho,\phi)(\bm{\xi})\defeq(\rho,\frac{\Phi}{t})(t, \mathbf{x})$
in the self-similar coordinates $\bm{\xi} = (\xi,\eta) \defeq
\frac{\mathbf{x}}{t} \in\mathbb{R}^2$.
Moreover, by introducing the pseudo-potential function $\varphi(\bm{\xi})\defeq\phi(\bm{\xi})-\frac12|\bm{\xi}|^2$,
we obtain the pseudo-steady potential flow equation for $\varphi(\bm{\xi})$ in the form:
\begin{eqnarray}
\label{Eq4PseudoPoten}
\mbox{div}(\rho(|D\varphi|,\varphi)D\varphi)+2\rho(|D\varphi|,\varphi) = 0\,,
\end{eqnarray}
where \({\rm div}\) and \(D\) represent the divergence and gradient operators with respect to the self-similar
coordinates $\bm{\xi}\in\mathbb{R}^2$, and the density is determined from the Bernoulli law as
\begin{eqnarray}\label{Relat4Rho}
\rho(|D\varphi|,\varphi) =
\begin{cases}
\big(1+(\gamma-1)(B-\frac12|D\varphi|^2-\varphi)\big)^{\frac{1}{\gamma-1}}
&\quad \text{if} \; \gamma>1\,,\\
\,\text{exp}(B-\frac12|D\varphi|^2-\varphi) &\quad \text{if} \; \gamma=1\,.
\end{cases} \end{eqnarray}
Equation~\eqref{Eq4PseudoPoten} with~\eqref{Relat4Rho} is a second-order nonlinear equation of mixed elliptic-hyperbolic type.
It is elliptic if and only if $|D\varphi|<c(|D\varphi|,\varphi)$ (pseudo-subsonic),
while it is hyperbolic if and only if $|D\varphi|>c(|D\varphi|,\varphi)$ (pseudo-supersonic).

We seek a weak solution for the pseudo-steady potential flow equation~\eqref{Eq4PseudoPoten}--\eqref{Relat4Rho}.

\begin{definition}\label{Def:WeakSoluForPotentialFlowEq}
A function  $\varphi\in W^{1,\infty}_{{\rm loc}}(\mathbb{R}^2)$ is called a weak solution
of equation~\eqref{Eq4PseudoPoten}--\eqref{Relat4Rho} in $\mathbb{R}^2$
if $\varphi$ satisfies the following properties{\rm :}
\begin{enumerate}[{\rm (i)}]
\item \label{Def-2-1:item1}
    $B-\frac12|D\varphi|^2-\varphi \geq  h(0^+) \quad \text{a.e.~in } \mathbb{R}^2 ;$

\item \label{Def-2-1:item3}
$\int_{\mathbb{R}^2}\big(\rho(|D\varphi|,\varphi)D\varphi\cdot D\zeta-2\rho(|D\varphi|,\varphi)\zeta\big)
\, {\rm d}\bm{\xi}=0\,$ for every $\zeta\in C^{\infty}_{\rm {c}}(\mathbb{R}^2)$.
\end{enumerate}
\end{definition}

Let $S$ be a $C^1$--curve across which $|D\varphi|$ is discontinuous.
Then $\varphi$ is a weak solution if and only if $\varphi$ satisfies the Rankine--Hugoniot conditions on $S$:
\begin{eqnarray}\label{RHCondi}
[\varphi]_S=\left[\rho(|D\varphi|,\varphi)D\varphi\cdot\bm{\nu}_{S}\right]_S=0\,,
\end{eqnarray}
where $[F]_{S}$ represents the jump of a quantity $F$ across $S$,
and $\bm{\nu}_{S}$ is any unit normal vector on $S$.
By Definition~\ref{Def:WeakSoluForPotentialFlowEq}\eqref{Def-2-1:item1}
and~\eqref{RHCondi},
\(D\varphi \cdot \bm{\nu}_{S}\) is either positive or negative across the discontinuity \(S\).
We say that \(\bm{\nu}_{S}\) points from upstream to downstream
if \(D\varphi \cdot \bm{\nu}_{S} > 0\),
and from downstream to upstream if \(D\varphi \cdot \bm{\nu}_{S} < 0\).
The discontinuity $S$ is called a \textit{shock} if it further satisfies
the physical entropy condition: {\it the corresponding density $\rho(|D\varphi|,\varphi)$
increases across $S$ from upstream to downstream,
or equivalently, \(|D\varphi \cdot \bm{\nu}_S|\) decreases across \(S\)
from upstream to downstream.}

In the self-similar coordinates $\bm{\xi}=(\xi,\eta)$, the initial condition~\eqref{Eq:RiemannInitialData}
becomes the following asymptotic boundary condition at infinity:
For any \(\theta \in [0,2\pi)\) and \(i = 1,2,3,4\),
\begin{align*}
    \lim_{r \to \infty} \norm{\varphi - \varphi_i}_{C^{0,1}(R_\theta \setminus B_r(\bm{0}))} = 0\,,
\end{align*}
whenever ray \(R_\theta \defeq \{ r(\cos{\theta},\sin{\theta})\,:\,r > 0 \}\) is contained
in sector \(\Lambda_i\) as defined in \eqref{Eq:InitialDomain},
where the pseudo-potential function \(\varphi_i\) corresponding to
the uniform state \( (i)\) is given by
\begin{align} \label{4States}
    \varphi_i (\bm{\xi}) \defeq -\frac12 \abs{\bm{\xi}}^2 + (u_i,v_i) \cdot \bm{\xi} + k_i \,,
\end{align}
and constant \(k_i\) is uniquely determined to satisfy~\eqref{Relat4Rho}
with \((\rho, \varphi) = (\rho_i,\varphi_i)\) for $i=1,2,3,4$.
We may simply fix \(k_2 = 0\) in~\eqref{4States} since the Bernoulli constant can be shifted to \(B - k_2\).

Since system~\eqref{System:2DPotentialFlowEqs} is of Galilean
invariance,
it is convenient to normalize the initial data to satisfy \(u_2 = v_1 = 0\).
From~\eqref{Relat4Rho} and~\eqref{4States} with \(k_2 = 0\), the Bernoulli constant \(B\) is then fixed as
\begin{align} \label{eq:bernoulli-constant-def}
    B = \frac12 v_2^2 + h(\rho_2)\,.
\end{align}

In this paper, we assume that the initial data~\eqref{Eq:RiemannInitialData} are
chosen such that,
in the far field region, any two neighboring states are connected by exactly one planar shock discontinuity.
Furthermore, we restrict our attention to the case:
\begin{align} \label{eq:initial-data-entropy-condition}
    \max\{ \rho_1,\rho_3 \} < \min\{ \rho_2, \rho_4\}\,,
\end{align}
so that the initial data consist
precisely of two forward shocks
$\{\overrightarrow{S}\!_{12}, \overrightarrow{S}\!_{14}\}$
and two backward shocks $\{\overleftarrow{S}\!_{32}, \overleftarrow{S}\!_{34}\}$,
which travel away from the origin in pairs; see
Case 5.2 in~\cite{ZhangChenYang1995DCDS}.
Moreover, to facilitate our analysis, we focus
on the symmetric case:
\begin{align} \label{eq:symmetry-assumption}
    \theta_{12}=\theta_{14}\eqdef \theta_1, \quad\,
    \theta_{32}=\theta_{34}\eqdef \theta_2\, \qquad\, \text{with $\theta_1,\theta_2 \in (0,\tfrac{\pi}{2})$}\,.
\end{align}
We refer to parameters \(\btheta \defeq (\theta_1,\theta_2)\) as the incident angles of the shock discontinuities;
see Fig.~\ref{fig2-IncidShock}.
\begin{figure}
\centering
\includegraphics[width=15cm]{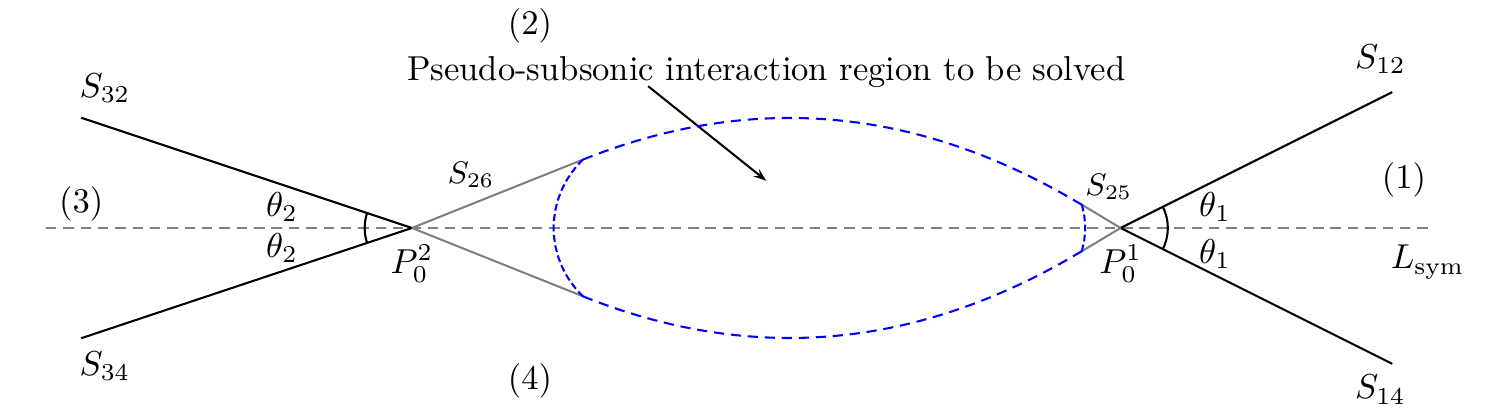}
\caption{Four-shock interaction with symmetric incident shocks.} \label{fig2-IncidShock}
\end{figure}

\subsubsection{Compatibility conditions}
\label{SubSec-201InitialData}
It is clear that the initial data and parameters \(\btheta\) cannot be specified arbitrarily.
In the following lemma, we give the necessary and sufficient conditions for
the initial data \(U_i = (\rho_i,u_i,v_i)\) and parameters \(\btheta\) to generate two forward
shocks $\{\overrightarrow{S}\!_{12}, \overrightarrow{S}\!_{14}\}$
and two backward shocks $\{\overleftarrow{S}\!_{32}, \overleftarrow{S}\!_{34}\}$.

\begin{lemma}\label{Lem:Existence4FourShocks}
Fix $\gamma\geq1$ and $\rho_2>0$.
There exist both a constant \(\vmin \in [-\infty,0)\) depending on $(\gamma,\rho_2)$
and a critical angle \(\theta^{\rm cr} \in (0,\frac{\pi}{2}]\) depending on $(\gamma,\rho_2,v_2)$
such that, whenever $v_2\in (\vmin,0)$ and $\btheta\in (0,\theta^{\rm cr})^2,$
there exists a unique choice of constant states $U_i=(\rho_i,u_i,v_i),$ for $i=1,3,4,$
satisfying the Rankine--Hugoniot conditions~\eqref{RHCondi} between any two neighbouring states
and the entropy condition~\eqref{eq:initial-data-entropy-condition}
so that two forward shocks $\{\overrightarrow{S}\!_{12}, \overrightarrow{S}\!_{14}\}$
and two backward shocks $\{\overleftarrow{S}\!_{32}, \overleftarrow{S}\!_{34}\}$
are generated.
\end{lemma}

\begin{proof}
Throughout the proof, we take \(i =1,3\), and \(j = 2,4\).
In the self-similar coordinates, we denote the shock discontinuity line
between state $(i)$ and state $(j)$ by
\begin{equation} \label{eq:incident-shock-sij}
S_{ij}\defeq \{\bm{\xi} \in \mathbb{R}^2 \,:\, \varphi_i(\bm{\xi})=\varphi_j(\bm{\xi})\}
= \big\{\bm{\xi}\in\mathbb{R}^2 \,:\, \eta=(-1)^{\frac{i+j+1}{2}}\xi\tan\theta_{ij}+a_{ij}
\big\}\,,
\end{equation}
where the expressions of $\varphi_{i}$ and $\varphi_{j}$ are given in~\eqref{4States},
$\theta_{ij}\in(0,\frac{\pi}{2})$ is the angle between the $\xi$--axis and the shock discontinuity
line $S_{ij}$, and $a_{ij}$ is a constant determined below.
We fix the unit normal vector on \(S_{ij}\) to be
\begin{equation*}
\bm{\nu}_{ij}=(
(-1)^{\frac{i+1}{2}} \sin\theta_{ij},(-1)^{\frac{j+2}{2}}\cos\theta_{ij}
)\,,
\end{equation*}
which points towards the downstream state according to the entropy
condition~\eqref{eq:initial-data-entropy-condition}, so that
\begin{align} \label{eq:incident-shock-entropy-condition}
    D\varphi_i(\bm{\xi}) \cdot \bm{\nu}_{ij} > D\varphi_j(\bm{\xi}) \cdot \bm{\nu}_{ij} > 0\,
    \qquad\; \text{for $\bm{\xi} \in S_{ij}$}\,.
\end{align}

A straightforward calculation by using~\eqref{Relat4Rho}--\eqref{4States}  and~\eqref{eq:incident-shock-sij}--\eqref{eq:incident-shock-entropy-condition} gives
\begin{equation}\label{Eq:RelationForUij}
(u_i,v_i)-(u_j,v_j)=\ell_{ij}\bm{\nu}_{ij}\,, \quad
a_{ij}=(-1)^{\frac{i+j+1}{2}}
\Big(v_i -u_i\tan\theta_{ij} +  \frac{(-1)^{\frac{i-1}{2}}\rho_j\ell_{ij}}{(\rho_i-\rho_j)\cos\theta_{ij}} \Big)\,,
\end{equation}
where  $\ell_{ij}\defeq\ell(\rho_i,\rho_j)$ is defined as
\begin{equation} \label{eq:sonic-centre-relations}
\ell(\rho_i,\rho_j) \defeq  \sqrt{ \frac{2(\rho_i-\rho_j)(h(\rho_i) - h(\rho_j))}{\rho_i+\rho_j}}\,.
\end{equation}
It is direct to check that \(\ell\) is symmetric and satisfies the strict monotonicity: For any \(\bar{\rho} > 0\),
\begin{align} \label{eq:monotonicity-of-ell}
    \rho \mapsto \ell(\bar{\rho}, \rho) \;\; \text{is
    strictly increasing on $\rho \in (\bar{\rho},\infty)$ and
    strictly decreasing on $\rho \in (0, \bar{\rho})$}\,.
\end{align}
From~\eqref{eq:symmetry-assumption} and~\eqref{Eq:RelationForUij}, we deduce that
\begin{equation*}
(\ell_{12}-\ell_{14})\sin\theta_{1}
=(\ell_{34}-\ell_{32})\sin\theta_{2}\,,
\end{equation*}
from which $\rho_4 = \rho_2$ must hold by virtue of the entropy condition~\eqref{eq:initial-data-entropy-condition} and~\eqref{eq:monotonicity-of-ell}.
Thus, we obtain the following relations:
\begin{equation*}\label{CompatiCondi}
\ell_{12}\cos{\theta}_{1} = \ell_{32}\cos{\theta}_{2} \,,\qquad
\left\{
\begin{alignedat}{3}
\vphantom{f^2}{}
u_2 &= u_4 &&= u_1+\ell_{12}\sin{\theta}_{1} &&= u_3-\ell_{32}\sin{\theta}_{2}\,, \\
v_1 &= v_3 &&= v_2+\ell_{12}\cos{\theta}_{1} &&= v_4-\ell_{32}\cos{\theta}_{2}\,.
\end{alignedat}
\right.
\end{equation*}
Recall that $u_2 = v_1 = 0$ are fixed, so that
\begin{align} \label{eq:states-1-4}
    \quad U_1 = (\rho_1,-\ell_{12}\sin\theta_1,0)\,, \,\,\,
    U_2 = (\rho_2, 0 , v_2)\,, \,\,\,
    U_3 = (\rho_3,\ell_{32}\sin\theta_2,0)\,,
    \,\,\,
    U_4 = (\rho_2,0,-v_2)\,,
\end{align}
whenever $\rho_1,\rho_3\in(0,\rho_2)$ satisfy
\begin{equation}\label{CompatiCondiGTI}
\ell_{12}\cos{\theta}_{1} = \ell_{32} \cos{\theta}_{2} = -v_2\,.
\end{equation}
We refer to~\eqref{CompatiCondiGTI} as the compatibility conditions for the initial data.
We define
\begin{align} \label{Eq:DefThetaVac}
    \vmin \defeq -\ell(0^+,\rho_2) \in [-\infty,0)\,, \qquad
    \theta^{\rm cr} \defeq \arccos{\big(\frac{-v_2}{\ell(0^+,\rho_2)} \big)} \in (0, \tfrac{\pi}{2}]\,.
\end{align}
Then, by~\eqref{eq:monotonicity-of-ell}, the necessary and sufficient conditions for the existence
of  $\rho_1$ and $\rho_3$ uniquely solving~\eqref{CompatiCondiGTI} subject to~\eqref{eq:initial-data-entropy-condition} are
\begin{align*}
    v_2 \in (\vmin,0)\,, \qquad
    \btheta \in (0,\theta^{\rm cr})^2\,.
\end{align*}
The constant states $U_i$, $i=1,3,4$, are then uniquely determined
by~\eqref{eq:states-1-4}--\eqref{CompatiCondiGTI}
depending only on \((\gamma, \rho_2, v_2, \btheta)\).
\end{proof}

We call $\theta^{\rm cr}$ the vacuum critical angle, because \(\rho_1 \to 0^+\) as \(\theta_1 \to \theta^{{\rm cr}-}\)
and \(\rho_3 \to 0^+\) as \(\theta_2 \to \theta^{{\rm cr}-}\), according to~\eqref{CompatiCondiGTI}.
Note that, in potential flow, the incident shock with vacuum upstream state
is mathematically consistent, which is in contrast to the non-potential Euler flow for which the maximal density ratio across a shock is bounded.

It follows from Lemma~\ref{Lem:Existence4FourShocks} that $\rho_2 = \rho_4$.
Without loss of generality, we fix $\rho_2 = \rho_4 = 1$ via the scale-invariance of
equation~\eqref{Eq4PseudoPoten} as follows
\begin{equation*}
\bm{\xi} \mapsto c_{2} \bm{\xi}\,, \qquad
(\rho, \varphi, \rho_2, v_2) \mapsto \big(\frac{\rho}{\rho_{2}}, \frac{\varphi}{c_{2}^{2}}, 1, \frac{v_2}{c_2} \big)\,.
\end{equation*}
Then the entropy condition~\eqref{eq:initial-data-entropy-condition}
becomes $\max\{\rho_1,\rho_3\} < 1$, and  \(v_{\min}= -\ell(0^+,1)\) depends only on \(\gamma \geq 1\).
For simplicity, we occasionally use the abbreviation $\ell(\cdot) \defeq \ell(\cdot,1)$.

\subsubsection{The asymptotic boundary value problem in the upper half-plane}
Equation~\eqref{System:2DPotentialFlowEqs} and the initial data in~\eqref{Eq:RiemannInitialData}
given by~\eqref{eq:states-1-4} are invariant under reflection with respect to the \(x_1\)--axis.
Thus, we look for solutions of~\eqref{System:2DPotentialFlowEqs}--\eqref{Eq:RiemannInitialData}
in the symmetric form:
\begin{equation*}
    \Phi(t,x_1,x_2) = \Phi(t,x_1,-x_2)
     \qquad\; \text{for all $(t,x_1,x_2) \in (0,\infty) \times \mathbb{R}^2$}\,,
\end{equation*}
which is equivalent to
\begin{align*}
    \varphi(\xi,\eta) = \varphi(\xi, -\eta)\qquad \text{for all $(\xi,\eta) \in \mathbb{R}^2$}\,.
\end{align*}
For this reason, it suffices to consider the restriction of solutions to the upper half-plane
\begin{align*}
    \mathbb{R}^2_+ \defeq \big\{ \bm{\xi} \in \mathbb{R}^2 \,:\, \eta > 0\big\}\,,
\end{align*}
so long as we impose the slip boundary condition
\begin{equation*} \label{eq:slip-boundary-condition-varphi-2}
    D\varphi \cdot (0,1) = 0 \qquad\; \text{on $L_{\rm sym} \defeq \big\{\bm{\xi}\in\mathbb{R}^2 \,:\, \eta=0\big\}$}\,.
\end{equation*}
This condition means that \(L_{\rm sym}\) can be regarded as a solid wall.

Hereafter, we use symbols $S_{12}$ and $S_{32}$ to represent only the closure of the rays
of the incident shocks ${S}_{12}$ and ${S}_{32}$ lying in the upper half-plane.
We denote by \(P_0^1= (\xi^{P_0^1}, \eta^{P_0^1})\) the point of intersection
between \(S_{12}\) and \(L_{\rm sym}\),
and by \(P_0^2 = (\xi^{P_0^2},\eta^{P_0^2})\) the point of intersection
between \(S_{32}\) and \(L_{\rm sym}\).
We expect the incident shocks \(S_{12}\) and \(S_{32}\) to undergo shock reflection-diffraction
at points \(P_0^1\) and \(P_0^2\) respectively; see Fig.~\ref{fig2-IncidShock}.
From~\eqref{eq:incident-shock-sij}, \eqref{Eq:RelationForUij}, and~\eqref{eq:states-1-4}, we have
\begin{equation}\label{P0-Pos}
\begin{aligned}
P_0^1 = ( {-}\ell(\rho_1) \sin\theta_{1}+\frac{1}{1-\rho_1}\frac{\ell(\rho_1)}{\sin{\theta}_{1}},\,0),\qquad
P_0^2 = (\ell(\rho_3) \sin\theta_{2}-\frac{1}{1-\rho_3}\frac{\ell(\rho_3)}{\sin{\theta}_{2}},\,0)\,.
\end{aligned}
\end{equation}
Using~\eqref{4States}, \eqref{eq:states-1-4}--\eqref{CompatiCondiGTI}, and~\eqref{P0-Pos},
the pseudo-potentials of states \((1)\)--\((3)\) are given by
\begin{equation}\label{PseudoVelocity123}
\,\,\varphi_1 = -\frac{1}{2}|\bm{\xi}|^2 + v_2 (\xi-\xi^{P_0^1}) \tan{\theta_1}, \;
\varphi_2 = -\frac{1}{2}|\bm{\xi}|^2 + v_2 \eta, \;
\varphi_3 = -\frac{1}{2}|\bm{\xi}|^2 - v_2 (\xi-\xi^{P_0^2}) \tan{\theta_2}.
\end{equation}
In particular, \(\varphi_1\) is independent of \(\theta_2\), while \(\varphi_3\) is independent of \(\theta_1\).
The following lemma concerning the pseudo-Mach number of state \((2)\)
at the reflection points \(P_0^1\) and \(P_0^2\) is fundamental to our later analysis
of the critical angles in~\S\ref{SubSec-203DetachAngle}.

\begin{lemma} \label{lem:MonotonicityOfMach2}
Fix \(\gamma \geq 1\) and  \(v_2\in(v_{\min},0)\).
For any incident angle \(\theta_1 \in (0, \theta^{\rm cr}),\)
the pseudo-Mach number of state \((2)\) at the reflection point \(P_0^1\) is given by
\begin{align}
\label{eq:appendixMachNumberState12}
    M_2^{(\theta_1)} \defeq
    \abs{D\varphi_2(P_0^1)} =
    \frac{\ell(\rho_1)}{1-\rho_1} \big( \rho_1^2 + \cot^2{\theta_1} \big)^{\frac12} \,,
\end{align}
with \(\rho_1 = \rho_1(\theta_1;v_2) \in (0,1)\) given by~\eqref{CompatiCondiGTI}.
Moreover, the pseudo-Mach number \(M_2^{(\theta_1)}\) is a strictly decreasing function of the incident angle \(\theta_1\)
and satisfies
that \(M_2^{(\theta_1)} \to \infty\) as \(\theta_1 \to 0^+\).
\end{lemma}

\begin{proof}
Formula~\eqref{eq:appendixMachNumberState12} follows directly from~\eqref{CompatiCondiGTI} and~\eqref{P0-Pos}--\eqref{PseudoVelocity123}.

For the second property,
the limit is clear since~\eqref{CompatiCondiGTI} implies that \(\ell(\rho_1) \to -v_2 > 0\) as \(\theta_1 \to 0^+\),
whilst \(\cot^2{\theta_1} \to \infty\) as \(\theta_1 \to 0^+\).
To determine the strict monotonicity of \(M_2^{(\theta_1)}\), we differentiate~\eqref{CompatiCondiGTI} and~\eqref{eq:appendixMachNumberState12}
with respect to \(\theta_1\) and re-arrange it to obtain
\begin{align}
\label{eq:appendixDerivativeExpressionMachNumberState2}
    \frac{ (1-\rho_1)^2 \cot{\theta_1} }{2 ( \rho_1 + \cot^2{\theta_1} ) (\ell(\rho_1))^2} \frac{{\rm d} (M_2^{(\theta_1)})^2}{{\rm d} \theta_1} =
    \rho_1 + \frac{1}{1-\rho_1}\frac{\ell(\rho_1)}{\ell'(\rho_1)} - \cot^2{\theta_1}\,.
\end{align}
Furthermore, from the definition of \(\ell(\rho)\) in~\eqref{eq:sonic-centre-relations}, we directly compute
\begin{align*}
    (1 - \rho) \frac{\ell'(\rho)}{\ell(\rho)} = - \frac{(\ell(\rho))^2 + (1-\rho)^2 h'(\rho)}{(\ell(\rho))^2(1 + \rho)}\,.
\end{align*}
Substituting the above expression into the right-hand side
of~\eqref{eq:appendixDerivativeExpressionMachNumberState2} gives
\begin{align*}
    \text{RHS of~\eqref{eq:appendixDerivativeExpressionMachNumberState2}} =
    \frac{ \rho_1 (1-\rho_1)^2 h'(\rho_1) - (\ell(\rho_1))^2}{(\ell(\rho_1))^2 + (1-\rho_1)^2 h'(\rho_1) } - \cot^2{\theta_1}\,.
\end{align*}
Denote the numerator of the first term above by
\begin{align*}
    f(\rho) \defeq \rho (1-\rho)^2 h'(\rho) - (\ell(\rho))^2 \equiv \frac{1-\rho}{1+\rho} \, \tilde{f}(\rho)\,,
\end{align*}
where $\tilde{f}(\rho) \defeq (\gamma+1) h(\rho) - (\gamma-1) \rho^2 h(\rho) + (1 - \rho^2)$.
Observe that $\tilde{f}(1) = 0$ and
\begin{align*}
\tilde{f}'(\rho) = (\gamma + 1) (1-\rho^2) h'(\rho) > 0\, \qquad \mbox{for any $\rho\in(0,1)$}\,.
\end{align*}
Thus, \(\tilde{f}(\rho) < 0\) for all \(\rho\in(0,1)\).
It directly follows that the RHS of~\eqref{eq:appendixDerivativeExpressionMachNumberState2} is negative.
We conclude that \(M_2^{(\theta_1)}\) is a strictly decreasing function of \(\theta_1 \in (0,\theta^{\rm cr})\).
\end{proof}

It follows immediately from Lemma~\ref{lem:MonotonicityOfMach2} that \(\xi^{P_0^1} > 0\) is
a strictly decreasing function of \(\theta_1 \in (0,\theta^{\rm cr})\)
with \(\xi^{P_0^1} \to \infty\) as \(\theta_1 \to 0^+\),
since
\(M_2^{(\theta_1)} = \abs{( -\xi^{P_0^1}, v_2)}\).
Similarly, \(\xi^{P_0^2} < 0\) is a strictly increasing function of \(\theta_2 \in (0,\theta^{\rm cr})\)
with \(\xi^{P_0^2} \to -\infty\) as \(\theta_2 \to 0^+\).

We define the background solution $\bar{\varphi} \in C^{0,1}_{\rm loc}(\mathbb{R}^2_+)$ as
\begin{align}
\label{BackGroundSolu}
\bar{\varphi}(\bm{\xi}) \defeq \min\big\{ \varphi_1(\bm{\xi}), \varphi_2(\bm{\xi}), \varphi_3(\bm{\xi})\big\} \,,
\end{align}
and define the three open domains
\begin{equation}
\label{eq:domains-omega-123}
\Omega_i \defeq {\rm int}\,\{ \bm{\xi} \in \mathbb{R}^2_+ \,:\, \varphi_i(\bm{\xi}) = \bar{\varphi}(\bm{\xi})\}\,
\qquad \mbox{for $i = 1,2,3$}\,,
\end{equation}
where \({\rm int}\,G\) denotes the interior of a set \(G \subseteq \mathbb{R}^2\).
Note that \(\Omega_1\) and \(\Omega_3\) are the wedge-shaped domains with wedge
angles \(\theta_1\) and \(\theta_2\) respectively, and they become empty sets
in the limits: \(\theta_1 \to 0^+\) and \(\theta_2 \to 0^+\), respectively.

In light of the above discussion, we seek solutions of the following asymptotic boundary value problem
in the self-similar coordinates $\bm{\xi}=(\xi,\eta)$ in the upper half-plane $\mathbb{R}^2_+$.

\begin{problem}[Asymptotic boundary value problem]\label{BVP}
Fix $\gamma\geq1$ and $v_2\in(v_{\min},0)$.
For any $\btheta \in (0,\theta^{\rm cr})^2,$ determine the existence of a weak
solution $\varphi\in W^{1,\infty}_{\rm loc}(\mathbb{R}_+^2)$ of
equations~\eqref{Eq4PseudoPoten}--\eqref{Relat4Rho} in $\mathbb{R}_+^2$
satisfying the following conditions{\rm :}
\begin{enumerate}[{\rm (i)}]
\item\label{BVP-item1}
The asymptotic boundary condition at infinity{\rm :}
\begin{equation*}
\lim_{r\to\infty} \norm{\varphi-\bar{\varphi}}_{C^{0,1}(R_{\theta} \setminus B_r(\bm{0}))} = 0\,
\qquad\, \text{for any $\theta \in (0,\pi)$}\,;
\end{equation*}

\item\label{BVP-item2}
The slip boundary condition on the symmetric boundary{\rm :}
\begin{equation*}
D\varphi\cdot(0,1)=0 \qquad \text{on $L_{\rm sym} = \big\{\bm{\xi} \in \mathbb{R}^2 \,:\, \eta=0\big\}$}\,.
\end{equation*}
\end{enumerate}
\end{problem}

\subsubsection{Normal reflection configurations}
\label{SubSec-202NormalShock}
It is meaningful to extend the range of parameters in Problem~\ref{BVP} to allow \(\btheta \in [0,\theta^{\rm cr})^2\).
We study the reflection for the case: \(\btheta = \bm{0}\), which we call the normal reflection configuration.
It is clear that setting \(\varphi= \bar{\varphi} \equiv \varphi_2\) would satisfy Problem~\ref{BVP}\eqref{BVP-item1}
at infinity.
However, since \(\varphi_2\) does not satisfy Problem~\ref{BVP}\eqref{BVP-item2},
a boundary layer must be present.
For this reason, we introduce a uniform downstream state \((0)\), determined by a pseudo-potential \(\varphi_0\)
satisfying~Problem~\ref{BVP}\eqref{BVP-item2},
such that
a straight reflected shock \(S_0\) is formed between states \((2)\) and \((0)\).
In fact, the only possible straight reflected shock that satisfies~\ref{BVP}\eqref{BVP-item1}
is a normal shock \(S_0 \defeq \{\bm{\xi} \in \mathbb{R}^2_+ \,:\, \eta = \eta_0 \}\) for some \(\eta_0 > 0\), which is parallel to \(L_{\rm sym}\).
The constant density of state \((0)\) given by \(\rho_0 \defeq \rho(\abs{D\varphi_0},\varphi_0)\) should satisfy the entropy condition \(\rho_0 > 1\).
We demonstrate that state \((0)\) described above is uniquely determined by \((\gamma, v_2)\).

The pseudo-potential \(\varphi_0\) has form~\eqref{4States} with $i=0$,
for suitable constants \((u_0,v_0,k_0)\) to be determined.
It follows from Problem~\ref{BVP}\eqref{BVP-item2} that
\(v_0 = 0\), whilst \(u_0 = 0\) follows from the Rankine--Hugoniot conditions~\eqref{RHCondi}
between states \((2)\) and \((0)\) because \(S_0\) is parallel to \(L_{\rm sym}\).
The value of \(k_0 = v_2 \eta_0\) is then obtained by using the continuity
of \(\varphi = \varphi_0 = \varphi_2\) on \(S_0\).

It remains to determine constant \(\eta_0 > 0\), which fixes the location of \(S_{0}\).
Combining the Bernoulli law~\eqref{Relat4Rho} with the Rankine--Hugoniot conditions~\eqref{RHCondi}, we find the necessary condition
\begin{align} \label{eq:normal-reflection-density-relation}
    &\ell(\rho_0) = -v_2\,.
\end{align}
It is clear that~\eqref{eq:normal-reflection-density-relation} admits a unique solution \(\rho_0 \in (1 ,\infty)\)
depending only on \((\gamma, v_2)\).
Indeed, \(\ell(\cdot)\) is strictly increasing on \((1,\infty)\) and \(\ell(1) = 0\).
Constant \(\eta_0\) is then uniquely determined by the Rankine--Hugoniot conditions~\eqref{RHCondi} to be
\begin{equation}\label{Sol-NormalReflec0-RH}
{\eta}_0 \defeq -\frac{v_2}{\rho_0 - 1 } > 0\, \qquad \text{where \(\rho_0 \in (1,\infty)\)
satisfies~\eqref{eq:normal-reflection-density-relation}.}
\end{equation}
Thus, the pseudo-potential of state \((0)\) has been uniquely determined above by \((\gamma,v_2)\) as
\begin{align*}
    \varphi_0 (\bm{\xi}) \defeq -\frac12 \abs{\bm{\xi}}^2 + v_2 \eta_0 \,.
\end{align*}
It is direct to verify that
\begin{align}
\label{eq:normal-reflection-solution}
    \varphi_{\rm norm}(\bm{\xi}) \defeq \min\{\varphi_0(\bm{\xi}), \varphi_2(\bm{\xi})\} =
    \begin{cases}
    \varphi_2(\bm{\xi}) &\quad \text{for}\; \eta > \eta_0\,, \\
    \varphi_0(\bm{\xi}) \, \quad &\quad \text{for}\; 0 < \eta < \eta_0\,,
    \end{cases}
\end{align}
is a solution to {\rm Problem~\ref{BVP}} when \(\btheta = \bm{0}\); see {\rm Fig.~\ref{fig:Normal}(a)}.

Briefly, we mention the case:
\(\btheta \in \big( \{0\} \times (0,\theta^{\rm cr}) \big) \cup \big( (0,\theta^{\rm cr}) \times \{0\}\big)\),
which is called the unilateral normal reflection; a normal reflection on one side with respect to
the symmetric boundary $L_{\rm sym}$ occurs by the same argument as the normal reflection
case $\btheta=\mathbf{0}$ above, while the other side undergoes a regular reflection at the reflection point,
of which details will be discussed in \S\ref{subsec:regular-reflection-configurations} below;
see the case: $\btheta \in (0,\theta^{\rm cr})\times \{0\}$ in {\rm Fig.~\ref{fig:Normal}(b)}.
Moreover, the unilateral normal reflections are essentially the same as the Prandtl-Meyer reflection
configurations considered in~\cite{BCF-2019}.

\begin{figure}
    \centering
    \begin{subfigure}{0.39\textwidth}
        \centering
        \includegraphics[height=3.7cm]{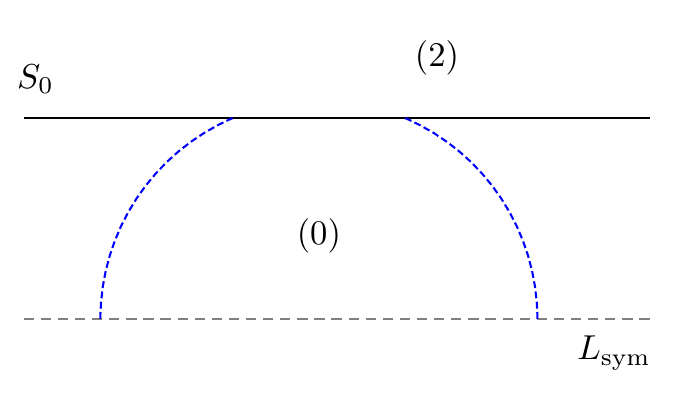}
        \caption*{(a) Normal reflection when $\theta_1=\theta_2={0}$.}
        \label{fig:N-a}
    \end{subfigure}
    \hfill
    \begin{subfigure}{0.59\textwidth}
        \centering
        \includegraphics[height=3.7cm]{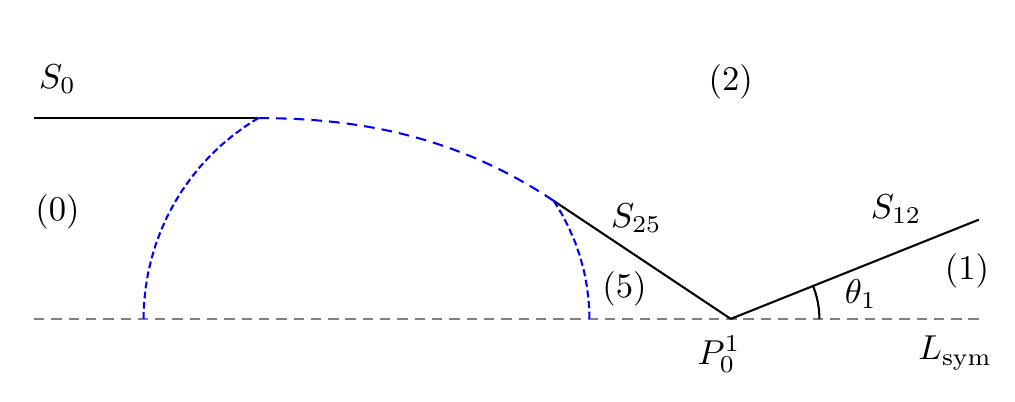}
        \caption*{(b) Unilateral normal reflection when $\theta_1>\theta_2=0$.}
        \label{fig:N-b}
    \end{subfigure}
    \caption{Structure of solutions of Problem~\ref{BVP} involving normal shocks.}
    \label{fig:Normal}
\end{figure}

\subsection{Regular reflection configurations}
\label{subsec:regular-reflection-configurations}
In this section, we introduce the detachment and sonic angles and
describe the structure of three genuinely different configurations of the four-shock interactions.

\subsubsection{Detachment and sonic angles}\label{SubSec-203DetachAngle}
For \(\btheta \in (0,\theta^{\rm cr})^2\), the incident shocks \(S_{12}\) and \(S_{32}\) intersect
the symmetric line \(L_{\rm sym}\) at points \(P_0^1\) and \(P_0^2\) respectively,
with coordinates given by~\eqref{P0-Pos}.
As discussed above, the symmetric line can be regarded as a rigid wall and, similarly to
the previous works~\cite{ChenFeldman-RM2018}, we expect that, if the incident angle $\theta_{1}$ is
less than a detachment angle $\theta^{\rm d}$,
there exists a regular reflection of \(S_{12}\) at point \(P_0^1\), {\it i.e.},
there exist both a uniform state~\((5)\) determined by a pseudo-potential \(\varphi_5\)
and a straight reflected shock \(S_{25}\) passing through \(P_0^1\)
such that \(S_{25}\) separates the uniform upstream state~\((2)\) from the uniform downstream state~\((5)\).
Similarly, we expect that there exists a regular reflection of \(S_{32}\) at point \(P_0^2\)
if $\theta_2$ is less than $\theta^{\rm d}$;
that is, there exist both a uniform state~\((6)\) determined by a pseudo-potential \(\varphi_6\)
and a straight reflected shock \(S_{26}\) passing through \(P_0^2\).
The expected structure of shock regular reflection on the upper half-plane
can be seen in Fig.~\ref{fig2-IncidShock}.
We now determine the detachment angle $\theta^{\rm d}$ and the sonic angle $\theta^{\rm s}$.

The regular reflection of \(S_{12}\) at point \(P_0^1\) can be reduced to the study of an algebraic system.
By the argument in~\cite[\S7.1]{ChenFeldman-RM2018},
to determine the uniform state \((5)\), it suffices to apply the Rankine--Hugoniot conditions
at the reflection point~\(P_0^1\) only.
The pseudo-potential \(\varphi_5\) has form~\eqref{4States} with $i=5$, and the straight reflected shock \(S_{25}\)
must have the form:
\begin{align*}
S_{25} \defeq \big\{ \bm{\xi} \in \cl{\mathbb{R}^2_+} \,:\, \varphi_5(\bm{\xi})
  = \varphi_2(\bm{\xi})\big\} = \big\{ \bm{\xi} \in \cl{\mathbb{R}^2_+} \,:\,
  \eta = \xi \tan{\theta_{25}} + a_{25}
  \big\}
\end{align*}
for some constants \((u_5,v_5,k_5, \theta_{25}, a_{25})\) depending on \((\gamma, v_2, \theta_1)\) to be determined,
where \(\theta_{25}\) is the angle between the reflected shock \(S_{25}\) and the positive \(\xi\)--axis, and \(a_{25}\)
is the \(\eta\)--intercept of the reflected shock \(S_{25}\).
Applying the Rankine--Hugoniot conditions~\eqref{RHCondi} between states \((2)\) and \((5)\) at \(P_0^1\), we have
\begin{align} \label{eq:state5RHcondition1}
    \begin{split}
       &\varphi_5(P_0^1) = \varphi_2(P_0^1)\,, \;\;
       D\varphi_5(P_0^1) \cdot \bm{\tau}_{25} = D \varphi_2(P_0^1) \cdot \bm{\tau}_{25}\,,
       \;\; \rho_5 D\varphi_5(P_0^1) \cdot \bm{\nu}_{25}  = D \varphi_2(P_0^1) \cdot \bm{\nu}_{25}\,,
    \end{split}
\end{align}
where \(\bm{\nu}_{25} \defeq \frac{D(\varphi_2 - \varphi_5)(P_0^1)}{|D(\varphi_2 - \varphi_5)(P_0^1)|}\), \(\bm{\tau}_{25}
\defeq \bm{\nu}^\perp\), and \(\rho_5 \defeq \rho( |D\varphi_5|,\varphi_5 )\) is given by~\eqref{Relat4Rho}.
Furthermore, the pseudo-potential \(\varphi_5\) should satisfy the entropy condition:
\begin{align} \label{eq:state5entropycondition1}
    \rho_5 > 1\,,
    \quad \text{or equivalently}\,, \quad
    D \varphi_2 (P_0^1) \cdot \bm{\nu}_{25} > D \varphi_5(P_0^1) \cdot \bm{\nu}_{25} > 0\,;
\end{align}
and the slip boundary condition:
\begin{equation} \label{eq:state5slipBC}
    D\varphi_5 \cdot (0,1) = 0 \qquad \text{on $L_{\rm sym}$}\,.
\end{equation}

A necessary condition for the regular reflection of \(S_{12}\) at \(P_0^1\) is clearly
that state \((2)\) must be pseudo-supersonic at \(P_0^1\).
From Lemma~\ref{lem:MonotonicityOfMach2},
there exists a unique \(\theta^+ \in (0,\theta^{\rm cr}]\) such that \(M_2^{(\theta_1)} > 1\) if and only if \(\theta_1 \in (0,\theta^+)\).
However, note that \(M_2^{(\theta_1)}|_{\theta_1 = \theta^+}\) is not necessarily $1$ if \(\theta^+=\theta^{\rm cr}\).

Introduce the steady detachment angle \(\theta_{\rm stdy}^{\rm d}(\rho_\infty,u_\infty)\)
for the steady potential flow with constant supersonic upstream state \(U_\infty = (\rho_\infty,u_\infty,0)\),
as defined in~\cite[§7.1]{ChenFeldman-RM2018}.
Also, for \(\theta_1 \in (0,\theta^{\rm cr})\),
denote by \(\hat{\theta}_{25}(\theta_1)\) the acute angle between the pseudo-velocity \(D\varphi_2(P_0^1) = (-\xi^{P_0^1}, v_2)\)
and the symmetric line \(L_{\rm sym}\):
\begin{align} \label{eq:state5SelfSimilarTurningAngle}
    \hat{\theta}_{25}(\theta_1) \defeq \arccos{(
    \frac{\xi^{P_0^1}}
    {|D\varphi_2(P_0^1)|})}
    = \arcsin{( \frac{- v_2}{|D\varphi_2(P_0^1)|}) } \in ( 0, \tfrac{\pi}{2} )\,.
\end{align}
Similarly to~\cite[\S7.4]{ChenFeldman-RM2018}, whenever \(\theta_1 \in (0 ,\theta^+)\),
the existence and multiplicity of solutions of the algebraic system~\eqref{eq:state5RHcondition1}--\eqref{eq:state5slipBC}
are determined as follows:
\begin{enumerate}[\quad(a)]
    \item If \(\hat{\theta}_{25}(\theta_1) < \theta_{\rm stdy}^{\rm d}(1,|D\varphi_2(P_0^1)|)\), there are two solutions;
    \item If \(\hat{\theta}_{25}(\theta_1) = \theta_{\rm stdy}^{\rm d}(1,|D\varphi_2(P_0^1)|)\), there is one solution;
    \item If \(\hat{\theta}_{25}(\theta_1) > \theta_{\rm stdy}^{\rm d}(1,|D\varphi_2(P_0^1)|)\), there are no solutions.
\end{enumerate}

We demonstrate that there exists at most one value \(\theta^{\rm d} \in (0,\theta^+)\)
such that \(\theta_1 = \theta^{\rm d}\) satisfies the equality in~(b) above.
Indeed, using \eqref{eq:state5SelfSimilarTurningAngle} and Lemma~\ref{lem:MonotonicityOfMach2},
it is clear that \(\hat{\theta}_{25}(\theta_1)\) is a strictly increasing function
of \(\theta_1\in (0 ,\theta^{\rm cr})\), whilst,
for \(\theta_{\rm stdy}^{\rm d}(1,|D\varphi_2(P_0^1)|)\), we state below a general fact
about 2-D steady potential flow, from which it is clear that \(\theta_{\rm stdy}^{\rm d}(1,|D\varphi_2(P_0^1)|)\)
is strictly decreasing with \(\theta_1 \in (0,\theta^+)\).
The proof of this statement is given in Appendix~\ref{SubSec:A1}.

\begin{lemma} \label{lem:MonotonicityOfSteadyAngles}
Let \(\theta_{\rm stdy}^{\rm d}\) and \(\theta_{\rm stdy}^{\rm s}\) denote the steady detachment and sonic angles
for the steady potential flow with constant supersonic upstream state \(U_\infty = (\rho_\infty,u_\infty,0)\).
Then \(\theta_{\rm stdy}^{\rm d}\) and \(\theta_{\rm stdy}^{\rm s}\) are smooth, strictly
increasing functions of the upstream Mach number \(M_\infty \defeq \frac{u_\infty}{c(\rho_\infty)} > 1\).
    Moreover, the following limits hold{\rm :}
    \begin{align*}
        \lim_{M_\infty \to 1^+} (\theta^{\rm d}_{\rm stdy}, \theta^{\rm s}_{\rm stdy}) = (0 , 0) \,, \qquad\,\,
        \lim_{M_\infty \to \infty} (\theta^{\rm d}_{\rm stdy}, \theta^{\rm s}_{\rm stdy})
        = (\tfrac{\pi}{2}, \arctan{\sqrt{\tfrac{2}{\gamma-1}}})\,.
    \end{align*}
\end{lemma}

With the above discussion, the following proposition gives the necessary and sufficient criterion
on the incident angle $\theta_1$ for the existence of a regular reflection of \(S_{12}\) at point \(P_0^1\).

\begin{proposition}[Local reflection theory]
\label{Prop:localtheorystate5}
Fix \(\gamma\geq 1\) and \(v_2 \in (v_{\min},0)\).
There exists a unique \(\theta^{\rm d} = \theta^{\rm d}(\gamma, v_2) \in (0, \theta^{\rm cr} ],\)
called the detachment angle, such that,
for \(\theta_1 \in (\theta^{\rm d},\theta^{\rm cr}),\) there are no states \((5)\) of form~\eqref{4States}
satisfying~\eqref{eq:state5RHcondition1}--\eqref{eq:state5slipBC} and,
for each \(\theta_1 \in (0,\theta^{\rm d}),\) there are exactly two states \((5)\) of form~\eqref{4States}
satisfying~\eqref{eq:state5RHcondition1}--\eqref{eq:state5slipBC}{\rm :}
the weak reflection state \(\varphi_5^{\rm wk}\) and the strong reflection state \(\varphi_5^{\rm sg},\)
distinguished by
 \(1 < \rho_5^{\rm wk} < \rho_5^{\rm sg}\,,\)
where \(\rho_5^{\rm wk} \defeq \rho(\abs{D\varphi_5^{\rm wk}},\varphi_5^{\rm wk})\)
and \(\rho_5^{\rm sg} \defeq \rho(\abs{D\varphi_5^{\rm sg}},\varphi_5^{\rm sg})\)
are the constant densities of the weak and strong reflection states \((5)\) respectively.
Furthermore, denoting by
\(O_5^{\rm wk} \defeq (u_5^{\rm wk}, v_5^{\rm wk})\) and \(O_5^{\rm sg} \defeq (u_5^{\rm sg}, v_5^{\rm sg})\) the sonic centres,
and by \(c_5^{\rm wk}\) and \(c_5^{\rm sg}\) the sonic speeds of the weak and strong reflection states~\((5)\) respectively,
then
\begin{enumerate}[{\rm (i)}]
    \item \label{Prop:localtheorystate5:continuity-properties}
    \((\rho_5^{\rm wk}, O_5^{\rm wk})\in C([0,\theta^{\rm d}])\cap C^{\infty}([0,\theta^{\rm d}))\) and
    \((\rho_5^{\rm sg}, O_5^{\rm sg}) \in C((0 ,\theta^{\rm d}])\cap C^{\infty}((0,\theta^{\rm d}))\).
    In particular, the following limits exist{\rm :}
    \begin{align*}
        \lim_{\theta_1 \to 0^+} (\rho_5^{\rm wk}, O_5^{\rm wk}, u_5^{\rm wk} \xi^{P_0^1}) = (\rho_0, \mathbf{0}, -v_2 \eta_0)\,,
        \qquad
        \lim_{\theta_1 \to \theta^{\rm d-}} (\rho_5^{\rm wk}, \rho_5^{\rm sg}, O_5^{\rm wk}, O_5^{\rm sg})\,,
    \end{align*}
    where \(\rho_0\) and \(\eta_0\) are the density and shock location of the normal reflection state~$(0)$
    which are given in~{\rm\S\ref{SubSec-202NormalShock}}.

    \item \label{Prop:localtheorystate5:3n4}
    For any \(\theta_1 \in (0,\theta^{\rm d}),\)
    \begin{align}
    &| D\varphi_2 (P_0^1) | > | D\varphi_2 (P_0^1|_{\theta_1 = \theta^{\rm d}}) | > 1\,,\label{eq:state2SupersonicCondition}\\
    &0 < u_5^{\rm wk} < u_5^{\rm sg} < \xi^{P_0^1}\,, \qquad  v_5^{\rm wk } = v_5^{\rm sg} = 0\,. \label{eq:localtheorystate5-sonic-centres}
    \end{align}

    \item \label{Prop:localtheorystate5:6}
    There exists a unique \(\theta^{\rm s} = \theta^{\rm s}(\gamma, v_2) \in (0, \theta^{\rm d}],\) called the sonic angle, such that
    \begin{alignat*}{2}
        &| D \varphi_5^{\rm wk}(P_0^1) | > c_5^{\rm wk}\, \qquad &&\text{for all} \;\; \theta_1 \in (0, \theta^{\rm s})\,, \\
        &| D \varphi_5^{\rm wk}(P_0^1) | < c_5^{\rm wk}\, \qquad &&\text{for all} \;\; \theta_1 \in (\theta^{\rm s}, \theta^{\rm d}]\,.
    \end{alignat*}
    Furthermore, if \(\theta^{\rm s} < \theta^{\rm cr},\) then \(\theta^{\rm s} < \theta^{\rm d}\)
    and $| D \varphi_5^{\rm wk}(P_0^1) | = c_5^{\rm wk}$ when $\theta_1 = \theta^{\rm s}$.
\end{enumerate}
\end{proposition}

 \noindent
\textit{Proof.}
The proof
follows similarly to~\cite[Theorem~7.1.1]{ChenFeldman-RM2018}, once $\theta^{\rm d}$ and $\theta^{\rm s}$ are determined.
Thus, it remains to show the existence and uniqueness of $\theta^{\rm d}$ and $\theta^{\rm s}$.

As discussed above, \(\hat{\theta}_{25}(\theta_1)\) is strictly increasing with \(\theta_1\in (0 ,\theta^{\rm cr})\),
whilst \(\theta_{\rm stdy}^{\rm d}(1,|D\varphi_2(P_0^1)|)\) is strictly decreasing with \(\theta_1 \in (0,\theta^+)\).
Furthermore, Lemmas~\ref{lem:MonotonicityOfMach2} and~\ref{lem:MonotonicityOfSteadyAngles} show that
\begin{align*}
    \lim_{\theta_1\to 0^+} \hat{\theta}_{25}(\theta_1)
    = 0 \,\,
    < \,\, \tfrac\pi2
    = \lim_{\theta_1\to 0^+} \theta_{\rm stdy}^{\rm d}(1,|D\varphi_2(P_0^1)|)\,.
\end{align*}
Then it follows that there exists at most one solution \(\vartheta^{\rm d} \in (0, \theta^+)\) of the equation:
\begin{align} \label{eq:uniquenessOfThetaD}
    \hat{\theta}_{25}(\vartheta^{\rm d}) = \theta_{\rm stdy}^{\rm d}(1,|D\varphi_2(P_0^1|_{\theta_1 = \vartheta^{\rm d}})|)\,.
\end{align}
We define the detachment angle \(\theta^{\rm d} \in (0 , \theta^+]\) by
\begin{align*}
\theta^{\rm d} \defeq
\min\big\{\theta^+, \,\inf\{ \vartheta^{\rm d} \in (0,\theta^+) \,:\,
    \text{\(\vartheta^{\rm d}\) satisfies~\eqref{eq:uniquenessOfThetaD}}\} \big\}\,.
\end{align*}
In this sense, \(\theta^{\rm d}\) may refer to either the conventional detachment angle \(\vartheta^{\rm d}\)
that satisfies equation~\eqref{eq:uniquenessOfThetaD} when such a solution exists, or angle \(\theta^+\) otherwise.

The existence and properties of the sonic angle \(\theta^{\rm s}\) contained in statement~\eqref{Prop:localtheorystate5:6}
of Proposition~\ref{Prop:localtheorystate5} follow from a similar argument as above.
In particular, symbol \(\theta^{\rm s}\) may refer to either the conventional
sonic angle \(\vartheta^{\rm s} \in (0,\theta^{\rm d}]\) which should satisfy
\begin{align} \label{eq:uniqueness-of-theta-s}
    \hat{\theta}_{25}(\vartheta^{\rm s}) = \theta^{\rm s}_{\rm stdy}(1, \abs{D\varphi(P_0^1|_{\theta_1 = \vartheta^{\rm s}})})\,;
\end{align}
or we simply define \(\theta^{\rm s} = \theta^{\rm d}\) if no such solution exists.
That is, we define
\begin{flalign*}
&&
\theta^{\rm s} \defeq
\min\big\{\theta^{\rm d},\, \inf\{\vartheta^{\rm s} \in (0,\theta^{\rm d}) \,:\,  \text{\(\vartheta^{\rm s}\) satisfies~\eqref{eq:uniqueness-of-theta-s}} \}\big\}\,.
&& \mathllap{\qed}
\end{flalign*}

\begin{remark}
\label{remark:local-theory}
We have the following remarks about {\rm Proposition~\ref{Prop:localtheorystate5}:}
\begin{enumerate}[{\rm (i)}]
\item \label{remark:local-theory-item-i}
Fix \(\gamma \geq 1\).
For any $v_2\in(\vmin,0),$ let \(\theta^{\rm cr}\) be the critical angle given by~\eqref{Eq:DefThetaVac}{\rm,}
and let \(v_2^{\rm s},v_2^{\rm d}\in(\vmin,0)\) be the constants given by {\rm Lemma~\ref{Lem:Defv2d-v2s}}.
Then
\begin{align*}
    \theta^{\rm s} = \theta^{\rm d} = \theta^{\rm cr} \quad &\Longleftrightarrow \quad v_2 \in (v_{\min}, v_2^{\rm s}]\,,\\
    \theta^{\rm s} < \theta^{\rm d} = \theta^{\rm cr} \quad &\Longleftrightarrow \quad v_2 \in (v_2^{\rm s}, v_2^{\rm d}]\,,\\
    \theta^{\rm s} < \theta^{\rm d} < \theta^{\rm cr} \quad &\Longleftrightarrow \quad v_2 \in (v_2^{\rm d}, 0)\,.
\end{align*}

\item
\label{remark:local-theory-item-ii}
By the symmetry of the problem, the conclusions of {\rm Proposition~\ref{Prop:localtheorystate5}} immediately
apply to the regular reflection of \(S_{32}\) at \(P_0^2\) after exchanging \((\theta_1,P_0^1,(5))\) with \((\theta_2,P_0^2,(6)),\)
and reversing the inequalities in~\eqref{eq:localtheorystate5-sonic-centres}.
In particular, the same detachment angle $\theta^{\rm{d}}(\gamma, v_2)$ and
sonic angle $\theta^{\rm{s}}(\gamma, v_2)$ apply to the incident angle $\theta_2$ and state \((6)\).

\item
\label{remark:local-theory-item-iii}
The main difference between {\rm Proposition~\ref{Prop:localtheorystate5}} and~\cite[Theorem~7.1.1]{ChenFeldman-RM2018}
is that we assert the uniqueness of solutions to~\eqref{eq:uniquenessOfThetaD} and~\eqref{eq:uniqueness-of-theta-s}{\rm,}
although we cannot in general
assert the existence of such solutions due to the restriction \(\theta_1 \in (0,\theta^{\rm cr})\){\rm ;}
whereas~\cite[Theorem~7.1.1]{ChenFeldman-RM2018} asserts the existence,
but not the uniqueness of such solutions.
\end{enumerate}
\end{remark}

\begin{definition}[Admissible parameters and states \((5)\) and \((6)\)] \label{def:state5andstate6}
Fix \(\gamma \geq 1\) and \(v_2 \in (v_{\min},0)\).
Let ${\theta}^{\rm{s}}\defeq{\theta}^{\rm{s}}(\gamma,v_2)$ and ${\theta}^{\rm{d}} \defeq {\theta}^{\rm{d}}(\gamma,v_2)$
be given by {\rm Proposition~\ref{Prop:localtheorystate5}}.
Denote the set of admissible parameters
\begin{equation*}
\Theta \defeq [0,\theta^{\rm d})^2 \setminus \{\bm{0}\}
\subseteq [0,\theta^{\rm cr}]^2\,.
\end{equation*}
Hereafter, for \(\btheta \in \cl{\Theta},\)  states \((5)\) and \((6)\) refer to the weak reflection states \((5)\) and \((6)\)
from {\rm Proposition~\ref{Prop:localtheorystate5}}.
For \(j = 5, 6,\) we write \((\rho_j,c_j,u_j,v_j) \defeq (\rho_j^{\rm wk},c_j^{\rm wk},u_j^{\rm wk},v_j^{\rm wk}),\) \(O_j \defeq (u_j,v_j),\)
and \(\varphi_{j} \defeq \varphi_j^{\rm wk}\).
The pseudo-potential of state \((j)\) is given by
\begin{equation} \label{eq:def-weak-state-5-6}
    \varphi_j (\bm{\xi}) = -\frac12 \abs{\bm{\xi}}^2 + (u_j,v_j) \cdot \bm{\xi} + v_2 a_{2j}\,,
\end{equation}
where $a_{2j} \defeq - \frac{1}{v_2} \xi^{P_0^{j-4}} u_j$ for $\theta_{j-4} > 0$,
and $a_{2j} \defeq \eta_0$ when $\theta_{j-4} = 0$ $($cf. {\rm Proposition 2.6(i)}$)$.

Define the reflected shocks \(S_{25}\) and \(S_{26}\) by
\begin{equation*}
S_{2j} \defeq
\big\{\bm{\xi}\in\cl{\mathbb{R}^2_+} \,:\, \varphi_j(\bm{\xi}) = \varphi_2(\bm{\xi})\big\}
= \big\{ \bm{\xi} \in \cl{\mathbb{R}^2_+} \,:\,
\eta = \xi \tan \theta_{2j}   +   a_{2j}
\big\} \qquad\text{for}\;\; j=5,6\,,
\end{equation*}
where \(\theta_{25}\) and \(\theta_{26}\) are the angles between the reflected shocks \(S_{25}\) and \(S_{26}\)
and the positive \(\xi\)--axis, given respectively by
\begin{align} \label{eq:def-reflected-shock-angle-25-26}
    \theta_{25} \defeq \arctan{\big(\frac{u_5}{v_2}\big)} \in (\tfrac{\pi}{2},\pi]\,,
    \qquad \theta_{26} \defeq \arctan{\big( \frac{u_6}{v_2} \big)} \in [0 ,\tfrac{\pi}{2} )\,.
\end{align}
The unit tangent vectors to \(S_{25}\) and \(S_{26}\) are given respectively by
\begin{align} \label{eq:tangent-vectors-e25-e26}
    \bm{e}_{S_{25}}  \defeq (\cos\theta_{25}, \sin\theta_{25})\,, \qquad
\bm{e}_{S_{26}} \defeq (\cos\theta_{26}, \sin\theta_{26})\,.
\end{align}
\end{definition}

It follows from
{\rm Proposition~\ref{Prop:localtheorystate5}\eqref{Prop:localtheorystate5:continuity-properties}}
and {\rm Definition~\ref{def:state5andstate6}} that
\(\varphi_j = \varphi_0\), and $\theta_{2j}=(6-j)\pi$ when \(\theta_{j-4} = 0\) for \(j = 5,6\).

\begin{lemma}[Further properties of states \((5)\) and \((6)\)] \label{lem:properties-of-state-5-and-6}
Fix \(\gamma \geq 1\) and \(v_2 \in (v_{\min}, 0 )\).
Let \(\rho_5\) and \(\rho_6\) be the constant densities of states \((5)\) and \((6)\) given
by {\rm Proposition~\ref{Prop:localtheorystate5}}{\rm ,} and let \(\theta_{25}\) and \(\theta_{26}\)
be the reflected shock angles given by~\eqref{eq:def-reflected-shock-angle-25-26}.
Then, for \(j = 5,6,\) there exists a constant \(\delta_{5,6} > 0\) depending only on \((\gamma, v_2)\)
such that, for any \(\btheta \in \cl{\Theta},\)
\begin{align}
 &1 + \delta_{5,6} \leq \rho_j \leq \delta_{5,6}^{-1}\,,\label{eq:density-state-5-6-uniform-bounds} \\
 &\delta_{5,6} \leq \abs{\theta_{2j} - \tfrac\pi2}\,.\label{eq:reflected-shock-angle-upper-bound}
\end{align}
Moreover,  for any \(\bar{\theta} \in (0,\theta^{\rm d})\), there exists
\(\delta_{5,6}^{(\bar{\theta})} > 0\)
depending only on \((\gamma, v_2, \bar{\theta})\) such that
\begin{align}
\label{eq:reflected-shock-angle-lower-bound}
    \delta_{5,6}^{(\bar{\theta})} &< \abs{\theta_{2j} - (6-j)\pi}
    \qquad \mbox{whenever $\btheta \in \cl{\Theta} \cap \{\theta_{j-4} > \bar{\theta}\}$}\,.
\end{align}
\end{lemma}

\begin{proof}
We prove the results only for state \((5)\), since the proof for state \((6)\) is similar.

Fix \(\gamma \geq 1\) and \(v_2 \in (v_{\min},0)\).
By
Proposition~\ref{Prop:localtheorystate5}(\ref{Prop:localtheorystate5:continuity-properties}),
we see that \(\rho_5\in C(\cl{\Theta})\), so that the uniform upper bound
in~\eqref{eq:density-state-5-6-uniform-bounds} is clear.
For the uniform lower bound, the entropy condition~\eqref{eq:state5entropycondition1} gives \(\rho_5 > 1\).
Furthermore, the Rankine--Hugoniot conditions between states~\((2)\) and \((5)\) imply that
\begin{equation}\label{Eq:LengthO2O5CompatiCondi}
\ell(\rho_5) = \big|(u_2,v_2) - (u_5,v_5)\big|
= \sqrt{v_2^2 + u_5^2} \geq |v_2| > 0\, \qquad \text{for any $\btheta \in \cl{\Theta}$}\,,
\end{equation}
with  \(\ell(\cdot)\) given by~\eqref{eq:sonic-centre-relations}.
In particular, \(\ell(\cdot)\) is an increasing function on \([1,\infty)\) with \(\ell(1) = 0\),
and hence~\eqref{Eq:LengthO2O5CompatiCondi} implies the existence of a small constant
\(\delta^{(1)}_{5,6} > 0\), depending only on \((\gamma, v_2)\),
such that \(\rho_5 > 1 + \delta^{(1)}_{5,6}\) for any \(\btheta \in \cl{\Theta}\).

Next, we prove~\eqref{eq:reflected-shock-angle-upper-bound}.
From~\eqref{eq:def-reflected-shock-angle-25-26} and the geometric properties of the shock polar curve for 2-D
steady potential flow, we have
\begin{align} \label{eq:simple-reflection-angle-estimate}
    \tfrac{\pi}{2} + \hat{\theta}_{25} < \theta_{25} \leq \pi \,
    \qquad\, \text{for any $\btheta \in \cl{\Theta}$}\,,
\end{align}
with \(\hat{\theta}_{25}\) given by~\eqref{eq:state5SelfSimilarTurningAngle}.
From Proposition~\ref{Prop:localtheorystate5}(i) and~\eqref{eq:def-reflected-shock-angle-25-26},
we see that \(\theta_{25} \to \pi^- \) as \(\theta_1 \to 0^+\).
By continuity, there exists a small constant \(\delta_{5,6}^{(2)} > 0\) depending only on \((\gamma, v_2)\)
such that \(\theta_{25} > \frac{3\pi}{4}\) whenever \(\theta_1 \in [ 0 , \delta_{5,6}^{(2)} ) \).
On the other hand, by \eqref{eq:state5SelfSimilarTurningAngle} and Lemma~\ref{lem:MonotonicityOfMach2},
there exists a constant \(\delta_{5,6}^{(3)} > 0 \) depending only on \((\gamma, v_2)\)
such that \(\hat{\theta}_{25} > \delta_{5,6}^{(3)}\) whenever \( \theta_1 \in [\delta_{5,6}^{(2)}, \theta^{\rm d}]\).
Together with~\eqref{eq:simple-reflection-angle-estimate}, we conclude
that \(\theta_{25} > \min\{\frac{3\pi}{4}, \frac{\pi}{2} + \delta_{5,6}^{(3)}\}\) for all \(\theta_1 \in [0, \theta^{\rm d}]\),
which leads to~\eqref{eq:reflected-shock-angle-upper-bound}.

Finally, bound~\eqref{eq:reflected-shock-angle-lower-bound} follows directly
from~\eqref{eq:simple-reflection-angle-estimate} and Lemma~\ref{lem:monotonicityOfTheta5},
which states that \(\theta_{25}\) is a strictly decreasing function of \(\theta_1 \in [0, \theta^{\rm d}]\).
\end{proof}

\subsubsection{Configurations for four-shock interactions}
\label{SubSec-204Configs}
The four-shock interaction configurations are split into three genuinely different cases,
depending on parameters \(\btheta \in \Theta\) and the critical angles \(\theta^{\rm s}\) and \(\theta^{\rm d}\).
We briefly describe the expected structure of each configuration and introduce some important notation.\\

\noindent{\mCase{I} {\it Supersonic-supersonic reflection configuration.}} \label{SubSubSec20401-case1}
When \(\btheta \in (0,{\theta}^{\rm s})^2\), the supersonic-supersonic reflection configuration is expected; see Fig.~\ref{fig4-SupSupRefConfig}.
In this case, both states \((5)\) and \((6)\) are pseudo-supersonic at the reflection points $P_0^1$ and $P_0^2$,
respectively.

\begin{figure}
\centering
\includegraphics[width=0.8\textwidth,trim={0 25 0 8},clip]{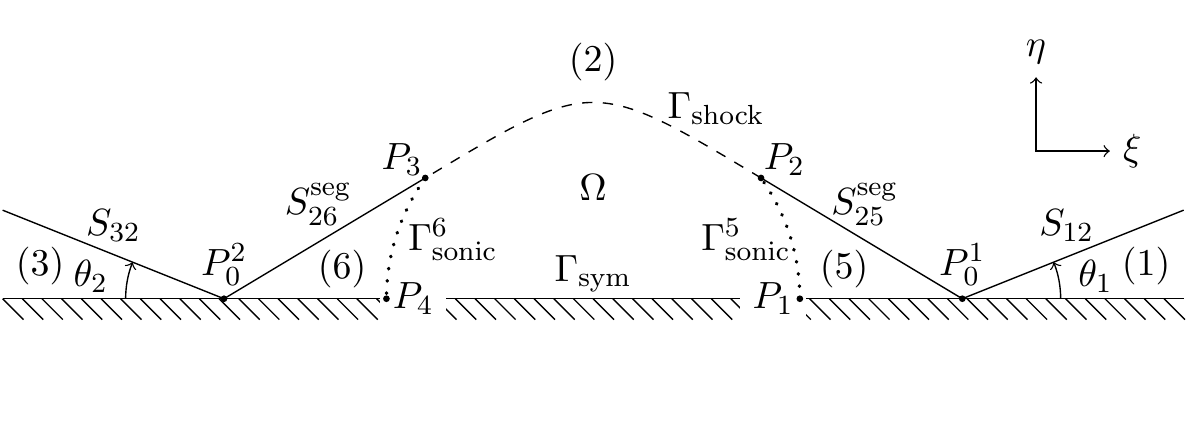}
\caption{\mCase{I} The supersonic-supersonic reflection configuration.} \label{fig4-SupSupRefConfig}
\end{figure}

\smallskip
\noindent{\mCase{II} {\it Supersonic-subsonic reflection configuration.}}\label{SubSubSec20402-case2}
When \(\btheta \in [\theta^{\rm{s}},\theta^{\rm{d}}) \times (0,\theta^{\rm{s}})\),
or \(\btheta \in (0,\theta^{\rm{s}}) \times [\theta^{\rm{s}},\theta^{\rm{d}})\), the supersonic-subsonic reflection configuration is expected;
see Fig.~\ref{fig5-SupSubRefConfig} for the case: \(\btheta \in [\theta^{\rm{s}},\theta^{\rm{d}}) \times (0,\theta^{\rm{s}})\).
In this case, either state \((5)\) is pseudo-sonic/pseudo-subsonic at $P_0^1$ and state \((6)\) is pseudo-supersonic at $P_0^2$, or vice versa.

\begin{figure}
\centering
\includegraphics[width=0.8\textwidth,trim={0 25 0 8},clip]{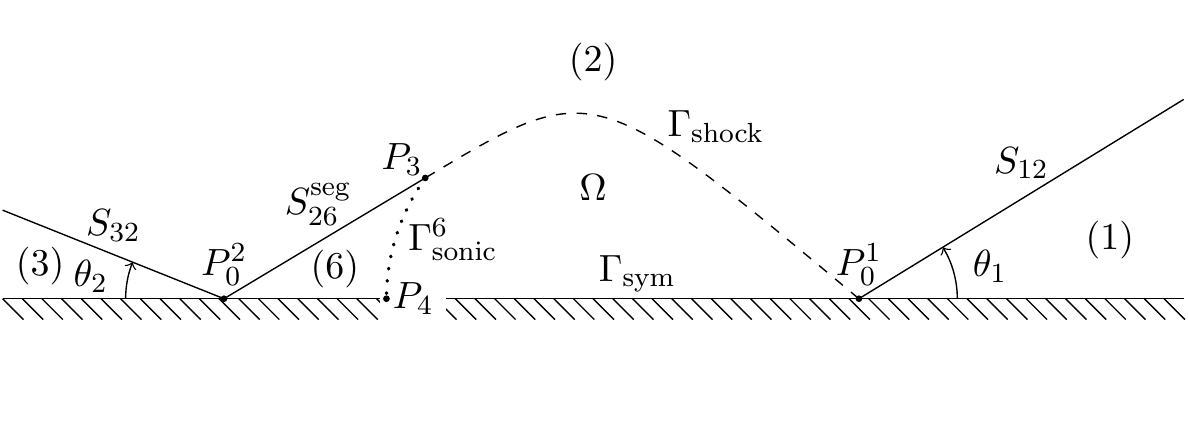}
\caption{\mCase{II} The supersonic-subsonic reflection configuration in the case:
\(\btheta \in [\theta^{\rm{s}},\theta^{\rm{d}}) \times (0,\theta^{\rm{s}})\), {\it i.e.},~state \((5)\) is pseudo-sonic/pseudo-subsonic at \(P_0^1\).}
\label{fig5-SupSubRefConfig}
\end{figure}

\smallskip
\noindent{\mCase{III} {\it Subsonic-subsonic reflection configuration.}}
When \(\btheta \in [\theta^{\rm s},\theta^{\rm d})^2\), the subsonic-subsonic reflection configuration is expected;
see Fig.~\ref{fig7-SubSubRefConfig}. In this case, both states $(5)$ and $(6)$ are pseudo-subsonic at the reflection points $P_0^1$ and $P_0^2$, respectively.

\smallskip
It follows from Remark~\ref{remark:local-theory}\eqref{remark:local-theory-item-i} that the configurations of \textbf{Case~II} and \textbf{Case~III}
are not possible unless \(v_2 \in (v_2^{\rm s},0)\), for constant \(v_2^{\rm s}\) from Lemma~\ref{Lem:Defv2d-v2s}.

\begin{figure}
\centering
\includegraphics[width=0.8\textwidth,trim={0 25 0 8},clip]{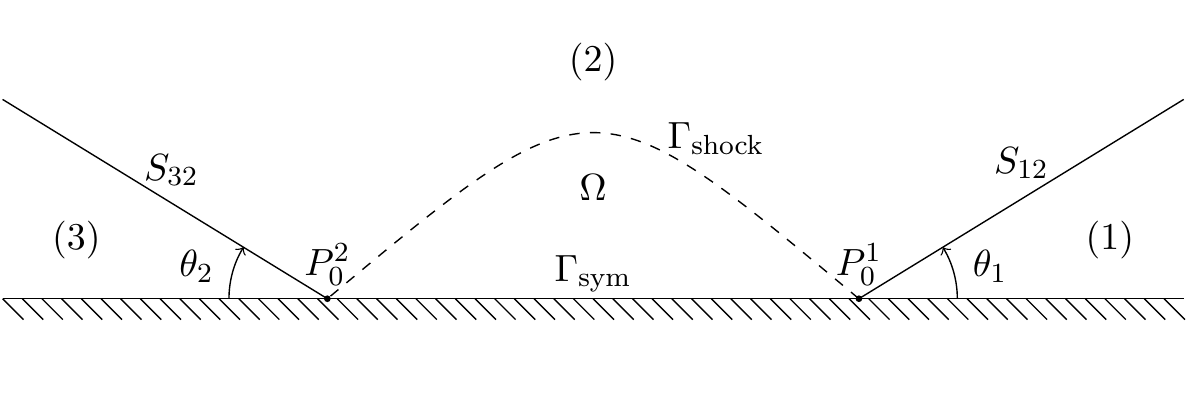}
\caption{\mCase{III} The subsonic-subsonic reflection configuration.} \label{fig7-SubSubRefConfig}
\end{figure}

\begin{definition}[The sonic boundaries and intersection points]
\label{Def:GammaSonicAndPts}
Fix $\gamma\geq1$ and $v_2\in(v_{\min},0)$.
For any $\btheta \in\cl{\Theta}${\rm ,} let $Q_5$ and $Q_6$ be the orthogonal projections of \(O_5\) and \(O_6\) onto \(S_{25}\) and \(S_{26},\) respectively.
When \(\theta_1 \in (0,\theta^{\rm s}),\) the right sonic boundary and intersection points are given by
\begin{equation*}
\begin{aligned}
&\Gamma_{\rm sonic}^5\defeq
\big\{\bm{\xi}\in\partial B_{c_5}(O_5) \cap \overline{\mathbb{R}^2_+} \,:\, \xi^{Q_5} \leq \xi \leq \xi^{P_0^1},\, \varphi_5(\bm{\xi}) \leq \varphi_2(\bm{\xi})\big\}\,,\\
&\{P_1\} \defeq \Gamma_{\rm sonic}^5 \cap L_{\rm sym}\,, \quad
\{P_2\} \defeq \Gamma_{\rm sonic}^5 \cap S_{25}\,,
\end{aligned}
\end{equation*}
whilst, for \(\theta_1 \in [\theta^{\rm s},\theta^{\rm d}),\)
\begin{align*}
    \Gamma_{\rm sonic}^5 \equiv \{P_1\} \equiv \{P_2\} \defeq \{P_0^1\}\,.
\end{align*}
Similarly, when \(\theta_2 \in (0,\theta^{\rm s}),\) the left sonic boundary and intersection points are given by
\begin{equation*}
\begin{aligned}
&\Gamma_{\rm sonic}^6\defeq \big\{\bm{\xi}\in\partial B_{c_6}(O_6) \cap \overline{\mathbb{R}^2_+} \,:\, \xi^{P_0^2}
\leq \xi \leq \xi^{Q_6}, \, \varphi_6(\bm{\xi}) \leq \varphi_2(\bm{\xi}) \big\}\,,\\
&\{P_3\} \defeq \Gamma_{\rm sonic}^6 \cap S_{26}\,,
\quad \{P_4\} \defeq \Gamma_{\rm sonic}^6 \cap L_{\rm sym}\,,
\end{aligned}
\end{equation*}
whilst, for \(\theta_2 \in [\theta^{\rm s},\theta^{\rm d}),\)
\begin{align*}
    \Gamma_{\rm sonic}^6 \equiv \{P_3\} \equiv \{P_4\} \defeq \{P_0^2\} \,.
\end{align*}
The open regions are given by
\begin{equation*}
\begin{aligned}
\Omega_5 &\defeq \big\{ \bm{\xi} \in \mathbb{R}^2_+ \setminus \cl{B_{c_5}(O_5)} \,:\, \varphi_5(\bm{\xi}) < \varphi_2(\bm{\xi}),
     \, \xi > \xi^{P_2} \big\}\,,\\
\Omega_6 &\defeq \big\{ \bm{\xi} \in \mathbb{R}^2_+ \setminus \cl{B_{c_6}(O_6)} \,:\, \varphi_6(\bm{\xi}) < \varphi_2(\bm{\xi}), \,
\xi < \xi^{P_3} \big\}\,.
\end{aligned}
\end{equation*}
Finally, define the symmetric boundary
\begin{align*}
\Gamma_{\rm sym} \defeq \big\{ (\xi, 0)\in L_{\rm sym} \,:\, \xi^{P_4} < \xi < \xi^{P_1}\big\}\,,
\end{align*}
and the line segments
\begin{equation*}
S^{\rm seg}_{25} \defeq S_{25} \cap \big\{ \xi^{P_2} \leq \xi \leq \xi^{P_0^1}\big\}\,,
\qquad
S^{\rm seg}_{26} \defeq S_{26} \cap \big\{ \xi^{P_0^2} \leq \xi \leq \xi^{P_3}\big\}\,.
\end{equation*}
\end{definition}

Note that curves $\Gamma_{\rm{sonic}}^5$, $\Gamma_{\rm{sonic}}^6$, and $\cl{\Gamma_{\rm{sym}}}$
as defined in Definition~\ref{Def:GammaSonicAndPts} above are fixed by \((\gamma,v_2,\btheta)\)
and do not share any common points except their endpoints $\{P_1,P_4\}$.
Also, note that $\Omega_j=\varnothing$ if $\theta_{j-4}\in[\theta^{\rm s},\theta^{\rm d})$ for $j=5,6$.
For a suitable simple open curve segment \(\Gamma_{\rm shock}\) (yet to be determined) with
endpoints \(P_2\) and \(P_3\),
we use $\Omega$ to denote the open bounded domain enclosed by the simple closed curve
$\Gamma_{\rm shock}\cup\Gamma_{\rm sonic}^5\cup\Gamma_{\rm sym}\cup\Gamma_{\rm sonic}^6$;
see Figs.~\ref{fig4-SupSupRefConfig}--\ref{fig7-SubSubRefConfig}.

\subsection{Admissible solutions and main theorems}
\label{SubSec-205AdmisSolu}
We reformulate the asymptotic boundary value problem, Problem~\ref{BVP},
into a free boundary problem.
Our goal is to find a suitable curved free boundary \(\Gamma_{\rm shock}\) between points \(P_2\) and \(P_3\),
and a solution \(\varphi\) of Problem~\ref{BVP} which satisfies~\eqref{Eq4PseudoPoten}--\eqref{Relat4Rho}
in the open region \(\Omega\) (as depicted in Figs.~\ref{fig4-SupSupRefConfig}--\ref{fig7-SubSubRefConfig})
and the corresponding boundary conditions
on $\partial\Omega\defeq\Gamma_{\rm shock}\cup\Gamma^5_{\rm sonic}\cup\Gamma_{\rm sym}\cup\Gamma^6_{\rm sonic}$.
Furthermore, equation~\eqref{Eq4PseudoPoten} should be elliptic in \(\Omega\),
and the free boundary \(\Gamma_{\rm shock}\) should be a transonic shock that separates the hyperbolic
and elliptic regions of the solution.

\begin{problem}[Free boundary problem]\label{FBP}
Fix $\gamma\geq1$ and $v_2\in(v_{\min},0)$.
For any $\btheta \in\overline{\Theta},$ find a free boundary $($curved reflected-diffracted shock$)$
$\Gamma_{\rm{shock}}$ in region $\mathbb{R}^2_+\cap\{\xi^{P_3}<\xi<\xi^{P_2}\}$ and a function $\varphi$
defined in region $\Omega,$ as shown
in {\rm Figs.~\ref{fig4-SupSupRefConfig}--\ref{fig7-SubSubRefConfig},} such that $\varphi$ satisfies{\rm :}
\begin{enumerate}[{\rm (i)}]
\item \label{item1-FBP} Equation~\eqref{Eq4PseudoPoten} in \(\Omega\)
so that \eqref{Eq4PseudoPoten} is elliptic inside \(\Omega\){\rm ;}
\item \label{item2-FBP} $\varphi=\varphi_2$ and $\rho D\varphi\cdot\bm{\nu}=\rho_2 D\varphi_2\cdot\bm{\nu}$
on the free boundary $\Gamma_{\rm{shock}}${\rm ;}
\item \label{item3n4-FBP} $\varphi=\varphi_j$ and $D\varphi=D\varphi_j$ on $\Gamma_{\rm{sonic}}^j$ for $j=5,6${\rm ;}
\item \label{item5-FBP}
$D\varphi\cdot(0,1)=0$ on $\Gamma_{\rm{sym}},$
\end{enumerate}
where $\bm{\nu}$ is the unit normal vector on $\Gamma_{\rm{shock}}$ pointing into $\Omega$.
\end{problem}

\begin{definition}[Admissible solutions] \label{Def:AdmisSolus}
Fix $\gamma\geq1$ and $v_2\in(\vmin,0)$. For any $\btheta \in\overline{\Theta},$ a function $\varphi\in W^{1,\infty}_{\rm loc}(\mathbb{R}^2_+)$
is called an admissible solution to the four-shock interaction problem if $\varphi$
is a solution of {\rm Problem~\ref{FBP}} and satisfies the following properties{\rm :}
\begin{enumerate}[{\rm (i)}]
\item
\label{item1-Def:AdmisSolus} There exists a relatively open curve segment $\Gamma_{\rm shock}$ $($without self-intersection$)$
with endpoints $P_2=(\xi^{P_2},\eta^{P_2})$ and $P_3=(\xi^{P_3},\eta^{P_3})$
given by {\rm Definition~\ref{Def:GammaSonicAndPts}}
such that $\Gamma_{\rm shock}$ is $C^2$ in its relative interior{\rm,}
$\Gamma_{\rm shock}^{\rm ext}\defeq S^{\rm seg}_{25} \cup \Gamma_{\rm{shock}} \cup S^{\rm seg}_{26}$
is a \(C^1\)--curve including at \(P_2\) and \(P_3,\) with $S^{\rm seg}_{25}$ and $S^{\rm seg}_{26}$
given in {\rm Definition~\ref{Def:GammaSonicAndPts},}
\begin{equation*}
\Gamma_{\rm shock} \subseteq
\big\{\bm{\xi}\in \mathbb{R}^2_+ \setminus \cl{B_{c_2}(O_2)} \,:\, \xi^{P_3} < \xi < \xi^{P_2} \big\} \,,
\end{equation*}
with $\partial B_{c_2}(O_2)$ as the sonic circle of state \((2)\){\rm,}
and
curves $\cl{\Gamma_{\rm shock}}$ and $\Gamma_{\rm sonic}^5 \cup \Gamma_{\rm sym} \cup \Gamma_{\rm sonic}^6$
share no common points except their endpoints $\{P_2,P_3\}$ so that
$\Gamma_{\rm shock} \cup \Gamma_{\rm sonic}^5 \cup \Gamma_{\rm sym} \cup \Gamma_{\rm sonic}^6$ is a closed curve without self-intersection.

\item \label{item2-Def:AdmisSolus}
For the open bounded domain $\Omega\subseteq\mathbb{R}_+^2$
enclosed
by $\Gamma_{\rm shock} \cup \Gamma_{\rm sonic}^5 \cup \Gamma_{\rm sym} \cup \Gamma_{\rm sonic}^6,$
solution $\varphi$ is in
$C^3(\Omega)\cap C^2\big(\overline{\Omega}\setminus(\Gamma_{\rm{sonic}}^5 \cup \Gamma_{\rm{sonic}}^6)\big)\cap C^1(\overline{\Omega})$
and satisfies $\varphi \in C^{0,1}_{\rm loc}(\mathbb{R}^2_+)
\cap C_{\rm loc}^1(\overline{\mathbb{R}^2_+}\setminus\overline{\Gamma_{\rm shock}^{\rm ext} \cup S_{12} \cup S_{32} })$
and
\begin{align}\label{eq:extension-of-varphi}
\varphi =
\begin{cases}
\max\{\varphi_5,\varphi_6\} &\,\,\, \text{in $\, \Omega_5\cup\Omega_6$}\,,\\
\bar{\varphi}  &\,\,\, \text{in $\, \mathbb{R}^2_+ \setminus (\Omega \cup \Omega_5 \cup \Omega_6)$}\,,
\end{cases}
\end{align}
where $\Omega_5$ and $\Omega_6$ are given in {\rm Definition~\ref{Def:GammaSonicAndPts},} and $\bar{\varphi}$ is given
by \eqref{BackGroundSolu}.

\smallskip
\item \label{item3-Def:AdmisSolus}
Equation~\eqref{Eq4PseudoPoten} is strictly elliptic in
$\cl{\Omega} \setminus ( \Gamma_{\rm{sonic}}^5 \cup \Gamma_{\rm{sonic}}^6)${\rm ,} i.e.,
solution $\varphi$ satisfies
\begin{equation*}
|D\varphi(\bm{\xi})| < c(|D\varphi(\bm{\xi})|,\varphi(\bm{\xi})) \qquad
\text{for all $\bm{\xi} \in  \cl{\Omega}\setminus(\Gamma_{\rm{sonic}}^5 \cup \Gamma_{\rm{sonic}}^6)$}\,.
\end{equation*}

\item \label{item4-Def:AdmisSolus}
In $\Omega,$ solution $\varphi$ satisfies
\begin{align}
&\max\{\varphi_5,\varphi_6\}\leq\varphi\leq\varphi_2\,,\label{EntropyIneqMutant}\\
&D (\varphi_2 - \varphi) \cdot (\cos\theta,\sin\theta) \leq 0
  \qquad\, \mbox{for all $\theta \in [\theta_{26},\theta_{25}]$}\,.
\label{DirecDerivMonotone-S25-S26}
\end{align}
\end{enumerate}
\end{definition}

\begin{remark} \label{Rmk4AdmissibleSolu}
From the definition of admissible solutions, we have
\begin{enumerate}[{\rm (i)}]
\item \label{item0-Rmk4AdmissibleSolu}
The requirement that \(\varphi\) is a solution of {\rm Problem~\ref{FBP}}
includes the condition that \(\varphi = \varphi_j\) and \(D\varphi = D\varphi_j\) on \(\Gamma_{\rm sonic}^j\)
for \(j = 5, 6,\) which becomes a one-point boundary condition in the case
\(\theta_{j-4} \in (\theta^{\rm s}, \theta^{\rm d}],\) according to the definition of \(\Gamma_{\rm sonic}^{j}\)
given by {\rm Definition~\ref{Def:GammaSonicAndPts}}.

\item \label{item1-Rmk4AdmissibleSolu}
Let \(\varphi\) be an admissible solution in the sense of {\rm Definition~\ref{Def:AdmisSolus},}
and let $\bm{\nu}$ be the unit normal vector on $\Gamma_{\rm shock}$ interior to $\Omega$.
Similarly to~\cite[Lemma~2.26]{BCF-2019}{\rm ,} it can be shown that \(\Gamma_{\rm shock}\) is a transonic shock,
since \(\varphi\) satisfies
\begin{equation*}
D\varphi_2 \cdot \bm{\nu} > D\varphi \cdot \bm{\nu} > 0\,,
\quad
0 < \frac{D \varphi \cdot \bm{\nu}}{c(\abs{D\varphi},\varphi)} < 1 < D \varphi_2 \cdot \bm{\nu}
        \qquad\,\, \text{on $\Gamma_{\rm shock}$}\,.
    \end{equation*}

\item \label{item2-Rmk4AdmissibleSolu}
It follows from
{\rm Proposition~\ref{Prop:localtheorystate5}\eqref{Prop:localtheorystate5:continuity-properties}}
that $\varphi_j$ depends continuously on \(\theta_{j-4} \in [0,\theta^{\rm d}]\) for \(j = 5,6\).
Moreover, for any \(r > 0,\)
\begin{equation*}
 \lim\limits_{\theta_{j-4} \to 0^+}\norm{\varphi_j - \varphi_0}_{C^{0,1}(B_r(\bm{0}) \cap \mathbb{R}^2_+)} = 0\,
 \qquad\,\text{for $j=5,6$}\,,
\end{equation*}
{\it i.e.},~the weak reflection state \((j)\) coincides with the normal reflection state $(0)$
when \(\theta_{j-4} = 0\). One can verify directly that the normal reflection solution
$\varphi_{\rm{norm}}\in C^{0,1}_{\rm loc}(\mathbb{R}^2_{+})$ as defined in~\eqref{eq:normal-reflection-solution}
is an admissible solution corresponding to
$\btheta=\bm{0}$ in the sense of {\rm Definition
~\ref{Def:AdmisSolus}}.
\end{enumerate}
\end{remark}

The first main theorem of this paper is to provide
the global existence of such admissible solutions in the sense of {\rm Definition~\ref{Def:AdmisSolus}}.
The proof is given in~\S\ref{subsec:proof-existence}.

\begin{theorem}[Existence of admissible solutions]\label{Thm:Existence-AdmisSolu}
Fix \(\gamma\geq1\) and  \(v_2 \in (v_{\min},0)\).
For any \(\btheta \in \Theta,\)
there exists an admissible solution \(\varphi\) of the four-shock interaction problem
corresponding to parameters \(\btheta\) in the sense of {\rm Definition~\ref{Def:AdmisSolus}}.
Moreover, the global solution $\varphi$
converges to the normal reflection solution
$\varphi_{\rm{norm}}$ in $W_{\rm loc}^{1,1}(\mathbb{R}_+^2)$ as \(\btheta \in \Theta\) converges to \(\bm{0},\)
where \(\varphi_{\rm norm}\) is given by~\eqref{eq:normal-reflection-solution}.
\end{theorem}

The admissible solutions have additional regularity properties, given in Theorem~\ref{RegularityThm} below.
Moreover, the transonic shocks, as free boundaries of admissible solutions,
satisfy additional convexity properties, given in Theorem~\ref{Thm3:ConvexTransonicShock} below.
The proofs of Theorems~\ref{RegularityThm}--\ref{Thm3:ConvexTransonicShock} are given
in~\S\ref{subsec:optimal-reg}--\S\ref{subsec:convexity}, respectively.

\begin{theorem}[Regularity of admissible solutions]\label{RegularityThm}
Fix $\gamma\geq1,$ $v_2 \in (v_{\min},0),$ and $\btheta \in \Theta$.
Let $\varphi$ be an admissible solution corresponding to parameters \(\btheta\) in the sense of {\rm Definition~\ref{Def:AdmisSolus},} with
transonic shock $\Gamma_{\rm{shock}}$.
Then the following properties hold{\rm:}
\begin{enumerate}[{\rm (i)}]
\item \label{ReguThm01} $\Gamma_{\rm{shock}}$ is $C^{\infty}$ in its relative interior and
\begin{enumerate}[]

\smallskip
\item \mCase{I}  When $\btheta \in [0,\theta^{\rm{s}})^2,$
$\varphi\in C^{\infty}(\overline{\Omega} \setminus ({\Gamma^5_{\rm{sonic}}}\cup{\Gamma^6_{\rm{sonic}}}) )\cap C^{1,1}(\overline{\Omega}),$ and
$\overline{S^{\rm seg}_{25}\cup\Gamma_{\rm{shock}}\cup S^{\rm seg}_{26}}$ is a $C^{2,\alpha}$--curve
for any $\alpha\in(0,1),$ including at points $P_2$ and $P_3$.

\smallskip
\item \mCase{II} When $\btheta \in [\theta^{\rm{s}},\theta^{\rm{d}}) \times [0,\theta^{\rm{s}}),$
$\varphi\in C^{\infty}(\overline{\Omega}\setminus(\Gamma^6_{\rm{sonic}} \cup\{P_1\}) )
\cap C^{1,1}(\overline{\Omega}\setminus\{P_1\})\cap C^{1,\bar{\alpha}}(\overline{\Omega}),$ and
$\overline{S^{\rm seg}_{26}\cup\Gamma_{\rm{shock}}}$ is a $C^{1,\bar{\alpha}}$--curve for some $\bar{\alpha}\in(0,1)$.
Furthermore, $\overline{S^{\rm seg}_{26} \cup \Gamma_{\rm{shock}} } \setminus {\{P_{1}\}}$ is a $C^{2,\alpha}$--curve for any $\alpha\in(0,1),$ including at $P_3$.\\
When $\btheta \in [0,\theta^{\rm{s}}) \times [\theta^{\rm{s}},\theta^{\rm{d}}),$
$\varphi\in C^{\infty}(\overline{\Omega}\setminus (\Gamma^5_{\rm{sonic}}\cup\{P_4\}) )
\cap C^{1,1}(\overline{\Omega}\setminus\{P_4\})\cap C^{1,\bar{\alpha}}(\overline{\Omega}),$ and
$\overline{S^{\rm seg}_{25}\cup\Gamma_{\rm{shock}}}$ is a $C^{1,\bar{\alpha}}$--curve for some $\bar{\alpha}\in(0,1)$.
Furthermore, $\overline{S^{\rm seg}_{25}\cup\Gamma_{\rm{shock}}}\setminus{\{P_{4}\}}$
is a $C^{2,\alpha}$--curve for any $\alpha\in(0,1),$ including at $P_2$.

\smallskip
\item \mCase{III} When $\btheta \in [\theta^{\rm{s}},\theta^{\rm{d}})^2,$
$\varphi\in C^{\infty}(\overline{\Omega}\setminus \{P_1,P_4\} )\cap C^{1,\bar{\alpha}}(\overline{\Omega}),$
and $\overline{\Gamma_{\rm{shock}}}$ is a $C^{1,\bar{\alpha}}$--curve for some $\bar{\alpha}\in(0,1)$.
Furthermore, $\overline{\Gamma_{\rm{shock}}}\setminus{\{P_{1},P_{4}\}}$ is a $C^{2,\alpha}$--curve for any $\alpha\in(0,1)$.
\end{enumerate}

\smallskip
\item \label{ReguThm02}
Let
$\mathcal{U}\defeq \{\bm{\xi}\in\mathbb{R}^2_+ \,:\, \max\{\varphi_5(\bm{\xi}),\varphi_6(\bm{\xi})\}<\varphi_2(\bm{\xi})\}$.
For any constant $\sigma>0$ and \(j=5,6,\) define domain $\mathcal{U}^j_{\sigma}$ by
\begin{equation*}
\mathcal{U}^j_{\sigma}\defeq
\begin{cases}
\mathcal{U}\cap\big\{\bm{\xi}\in\mathbb{R}^2_+ \,:\, {\rm{dist}}(\bm{\xi},\Gamma_{\rm{sonic}}^{j}) < \sigma\big\}
  \cap B_{c_j}(O_j)
&\quad \text{if $\theta_{j-4}\in [0,\theta^{\rm{s}})$}\,,\\
\varnothing\, &\quad \text{if $\theta_{j-4}\in[\theta^{\rm{s}},\theta^{\rm{d}})$}\,.
\end{cases}
\end{equation*}
Fix any point $\bm{\xi}_0\in({\Gamma^5_{\rm{sonic}}}\cup{\Gamma^6_{\rm{sonic}}})\setminus\{P_2,P_3\}$ and
denote $d\defeq \rm{dist}(\bm{\xi}_0,\Gamma_{\rm{shock}})$.
Then, for a sufficiently small constant $\varepsilon_0>0$ and for any $\alpha\in(0,1),$
there exists a constant $K<\infty$ depending on $(\gamma,v_2,\varepsilon_0,\alpha,d)$ and
$\|\varphi\|_{C^{1,1}(\Omega\cap(\mathcal{U}^5_{\varepsilon_0}\cup\,\mathcal{U}^6_{\varepsilon_0}))}$ such that
\begin{equation*}
\|\varphi\|_{2,\alpha,\overline{\Omega\cap B_{d/2}(\bm{\xi}_0)\cap(\mathcal{U}^5_{\varepsilon_0/2}\cup\, \mathcal{U}^6_{\varepsilon_0/2})}}\leq K\,;
\end{equation*}

\item \label{ReguThm03} For any $\bm{\xi}_0\in({\Gamma^5_{\rm{sonic}}}\cup{\Gamma^6_{\rm{sonic}}})\setminus\{P_2,P_3\},$
\begin{equation*}
\lim\limits_{ \bm{\xi}\to\bm{\xi}_0,\,\bm{\xi}\in \Omega}(D_{rr}\varphi-D_{rr}\max\{\varphi_5,\varphi_6\})(\bm{\xi})=\frac{1}{\gamma+1}\,,
\end{equation*}
where $r=|\bm{\xi}-O_j|$ for $\bm{\xi}$ near $\Gamma^j_{\rm{sonic}}$ for \(j=5,6\){\rm;}

\smallskip
\item \label{ReguThm04} For \(i = 1, 2,\) whenever $\theta_{i}\in [0,\theta^{\rm{s}}),$
$\lim\limits_{ \bm{\xi}\to P_{i+1},\, \bm{\xi} \in \Omega }D^2\varphi (\bm{\xi})$ does not exist.
\end{enumerate}
\end{theorem}

\begin{theorem}[Convexity of transonic shocks]
\label{Thm3:ConvexTransonicShock}
Fix \(\btheta \in \Theta\).
For any admissible solution
of the four-shock interaction problem in the sense of {\rm Definition~\ref{Def:AdmisSolus},} the transonic shock $\Gamma_{\rm{shock}}$ is uniformly convex
on closed subsets of its relative interior
in the sense described in {\rm Theorems~\ref{ThC1}--\ref{ThC2}}.
Furthermore, for any weak solution of~\eqref{Eq4PseudoPoten}--\eqref{Relat4Rho}
in the sense of {\rm Definition~\ref{Def:WeakSoluForPotentialFlowEq},}
satisfying also all the properties of {\rm Definition~\ref{Def:AdmisSolus}} except~\eqref{DirecDerivMonotone-S25-S26}{\rm,}
the transonic shock $\Gamma_{\rm{shock}}$ is a strictly convex graph if and only if condition~\eqref{DirecDerivMonotone-S25-S26} holds.
\end{theorem}

\section{Uniform Estimates for Admissible Solutions}
\label{Sec:UniformEstiAdmiSolu}

To establish the first main theorem, Theorem~\ref{Thm:Existence-AdmisSolu},
we apply the Leray-Schauder degree
theorem to a suitable iteration map, similar to~\cite{BCF-2019, ChenFeldman-RM2018}.
To achieve this, in this section, we make several essential uniform estimates of
the admissible solutions.
In \S 4,  we construct a suitable iteration map between two function spaces with weighted H{\"o}lder norms,
where any fixed point of the iteration map is an admissible solution
in the sense of Definition~\ref{Def:AdmisSolus}.
The main results we establish in this section are the following:
\begin{enumerate}[\;(a)]
\item The strict directional monotonicity of \(\varphi_2-\varphi\), and the directional monotonicity of \(\varphi-\varphi_5\) and \(\varphi-\varphi_6\);
\item The uniform positive lower bound of ${\rm dist}(\Gamma_{\rm{shock}},\partial B_1(O_2))$;
\item The uniform estimate of the ellipticity of equation~\eqref{Eq4PseudoPoten} in \(\Omega\);
\item The uniform estimates of admissible solutions in some suitable weighted H{\"o}lder norms.
\end{enumerate}

Our problem shares many similarities with the Prandtl-Meyer shock reflection problem~\cite{BCF-2019},
and the von Neumann shock reflection-diffraction problem~\cite{ChenFeldman-RM2018}.
Many of the following proofs are similar to those within~\cite{BCF-2019,ChenFeldman-RM2018},
but have been adapted to the current setting with careful estimates.

\subsection{Monotonicity and ellipticity}

\subsubsection{Directional monotonicity and uniform bounds}
\label{Sec3-1a:DirectMonotPropts4AdmisSolu}
Fix $\gamma\geq1$ and $v_2\in(v_{\min},0)$.
Let \(\varphi\) be an admissible solution corresponding to parameters \(\btheta \in \Theta\) in the sense of Definition~\ref{Def:AdmisSolus},
and let the constant states $\varphi_2$  and $\varphi_j$, $j=5,6$, be given by~\eqref{PseudoVelocity123} and~\eqref{eq:def-weak-state-5-6} respectively.
We aim to establish the directional monotonicity properties for $\phi$ that satisfies
\begin{equation}\label{Eq4phi}
(c^2-\varphi_{\xi}^2)\phi_{\xi\xi}-2\varphi_{\xi}\varphi_{\eta}\phi_{\xi\eta}+(c^2-\varphi_{\eta}^2)\phi_{\eta\eta}=0
\qquad\, \text{in $\Omega$}\,,
\end{equation}
where $\phi$ can be taken as either $\varphi_2 - \varphi$ or $\varphi - \varphi_j$, $j=5,6$, and
\begin{equation}\label{Def4SonicSpeedInSec3}
c^2(|D\varphi|,\varphi)=1+(\gamma-1)\big(B-\frac12|D\varphi|^2-\varphi\big)\,,
\end{equation}
where $B$ is the Bernoulli constant given by~\eqref{eq:bernoulli-constant-def}.

\begin{lemma}
\label{lem:monotonicityForPhi2-Phi}
Fix~\(\gamma\geq 1\) and \(v_2 \in (v_{\min},0)\).
For \(\btheta \in \Theta,\)
let \(\varphi\) be an admissible solution in the sense of {\rm Definition~\ref{Def:AdmisSolus}} corresponding to parameters \(\btheta\).
Then \(\varphi\) satisfies that, for \(j = 5,6\),
\begin{enumerate}[{\rm (i)}]
\item \label{item1-Lem3-2}
\(\partial_{\bm{e}} (\varphi - \varphi_j)\) is not constant in \(\Omega\)
for any unit vector \(\bm{e} \in \mathbb{R}^2\){\rm ;}

\item \label{item2-Lem3-2}
For vector \(\bm{e}_{S_{2j}}\) given by~\eqref{eq:tangent-vectors-e25-e26}{\rm,}
    \begin{equation*}
    \partial_{\bm{e}_{S_{2j}}}(\varphi_2 - \varphi) < 0 \qquad \text{in $\cl{\Omega} \setminus \Gamma_{\rm{sonic}}^j$}\,;
    \end{equation*}

\item \label{item3-Lem3-6}
$\partial_{\bm{e}_{S_{2j}}} (\varphi - \varphi_j)\,, \, {-\partial_{\eta}}(\varphi - \varphi_j) \geq 0\;\; $ in  $\cl{\Omega}$\,.
\end{enumerate}
\end{lemma}
\begin{proof}
The proof of~\eqref{item1-Lem3-2} is given by a slight modification of~\cite[Lemma~3.1]{BCF-2019}.
Indeed, for \(\phi \defeq \varphi - \varphi_5\), if \(\partial_{\bm{e}} \phi\) is constant in \(\Omega\),
then  \(\partial_{\bm{e}} \phi = 0 \) in \(\Omega\) since \(D\phi = \bm{0}\) at \(P_1\).
Furthermore, \(\partial_{\bm{e}} \phi = \bm{e}\cdot(u_6-u_5,0) \) at \(P_4\)
so that \(\bm{e}\) must be parallel to \((0,1)\), since \(u_5-u_6 > 0\) when \(\btheta\neq \bm{0}\).
Thus, \(\partial_\eta \phi \equiv 0\) in \(\Omega\) and hence
\(\partial_{\xi\eta} \phi = \partial_{\eta\eta} \phi \equiv 0\) in \(\Omega\).
By the strict ellipticity of equation~\eqref{Eq4phi} in \(\Omega\),
\(\partial_{\xi\xi}\phi \equiv 0\) in \(\Omega\), which implies that
there exist constants \((u,v,k)\) such that \(\phi = u\xi + v\eta + k\) in \(\Omega\).
We note that \(v = 0\) since \(\partial_\eta \phi \equiv 0\).
Applying the boundary conditions: \(\phi = 0\) and \(D\phi =\bm{0}\) at \(P_1\),
we find that \(u = k = 0\), so that \(\phi \equiv 0 \) in \(\Omega\).
However, \(D\phi(P_4) = D\varphi_6(P_4) - D\varphi_5(P_4) = (u_6-u_5, 0) \neq \bm{0}\), which is a contradiction.
The proof is similar for \(\phi = \varphi - \varphi_6\).
This concludes the proof of~\eqref{item1-Lem3-2}.

The proofs of~\eqref{item2-Lem3-2} and~\eqref{item3-Lem3-6} are similar to~\cite[Lemma~3.2]{BCF-2019}
and~\cite[Lemma~3.6]{BCF-2019}, respectively.
\end{proof}

For any \(\btheta \in \cl{\Theta}\), define the following fan-shaped set by
\begin{equation*}
{\rm Cone}(\bm{e}_{S_{25}},\bm{e}_{S_{26}})\coloneqq
\big\{ r (\cos{\theta},\sin{\theta}) \,:\,  \theta_{26} \leq \theta \leq \theta_{25}, \, r \geq 0\big\}\,.
\end{equation*}
Let ${\rm Cone}^0(\bm{e}_{S_{25}},\bm{e}_{S_{26}})$ be the interior of ${\rm Cone}(\bm{e}_{S_{25}},\bm{e}_{S_{26}})$.
It follows from Lemma~\ref{lem:monotonicityForPhi2-Phi} that,
if~\(\varphi\) is an admissible solution corresponding to \(\btheta\in\Theta\), then \(\varphi\) satisfies
\begin{align}
\partial_{\bm{e}} (\varphi_2 - \varphi ) < 0
    \qquad
    \text{in $\cl{\Omega}\, $ for any $\bm{e} \in {\rm Cone}^0 (\bm{e}_{S_{25}},\bm{e}_{S_{26}})$}\,,
     \label{nStrictDirectMonotPropty}
\end{align}
and ${-}\bm{\nu}(P) \in \big\{ (\cos\theta, \sin\theta) \,:\,
    \theta_{25} - \tfrac{\pi}{2} < \theta < \theta_{26} + \tfrac{\pi}{2} \big\}$
    for all $P \in \Gamma_{\rm{shock}}$\,,
where $\bm{\nu}$ is the unit normal vector on \(\Gamma_{\rm{shock}}\) pointing into the interior of~\(\Omega\).

\begin{lemma}
\label{Prop:GammaShockExpres}
Fix $\gamma\geq1$ and $v_2\in(v_{\min},0)$.
For \(\btheta \in \Theta,\) let $\varphi$ be an admissible solution in the sense of {\rm Definition~\ref{Def:AdmisSolus}} with parameters $\btheta$.
Then there exists a $C^1$--function
\(f_{\rm sh}\) such that
\begin{equation*}
    \Gamma_{\rm{shock}}=\{ \bm{\xi}\in\mathbb{R}^2_+ \,:\,  \eta = f_{\rm{sh}}(\xi), \, \xi^{P_3} < \xi < \xi^{P_2} \}\,,
\end{equation*}
and $f'_{\rm sh}(\xi)\in(\tan\theta_{25},\tan\theta_{26})$ for $\xi\in(\xi^{P_3},\xi^{P_2})$
with $f'_{\rm{sh}}(\xi^{P_{j-3}})=\tan\theta_{2j},$ for $j=5,6$.
\end{lemma}

\begin{proof}
By Lemma~\ref{lem:properties-of-state-5-and-6}, $\bm{e}_{\eta} = (0,1) \in {\rm Cone}^0(\bm{e}_{S_{25}},\bm{e}_{S_{26}})$.
From the strict directional monotonicity~\eqref{nStrictDirectMonotPropty}, we see that
$\partial_{\bm{e}_{\eta}}(\varphi_2-\varphi)<0$
on $\overline{\Gamma_{\rm{shock}}}$,
which implies the existence of a $C^1$--function $f_{\rm{sh}}$ that solves the implicit relation $\varphi_{2}(\xi,f_{\rm sh}(\xi))=\varphi(\xi,f_{\rm sh}(\xi))$.
Taking the derivative, we have
\begin{equation*}
f'_{\rm{sh}}(\xi)=-\frac{\partial_{\bm{e}_{\xi}}(\varphi_2-\varphi)(\xi,f_{\rm{sh}}(\xi))}{\partial_{\bm{e}_{\eta}}(\varphi_2-\varphi)(\xi,f_{\rm{sh}}(\xi))}\,.
\end{equation*}
From Definition~\ref{Def:AdmisSolus}\eqref{item1-Def:AdmisSolus}, we obtain that
$f'_{\rm{sh}}(\xi^{P_2})=\tan\theta_{25}$
and $f'_{\rm{sh}}(\xi^{P_3})=\tan\theta_{26}$.
Denoting by $\bm{\nu}(P)$ the unit normal vector on \(\Gamma_{\rm shock}\) at a point $P \in \Gamma_{\rm shock}$,
interior to \(\Omega\), we have
\begin{equation*}
 \bm{\nu}(P) =\frac{D(\varphi_2-\varphi)(P)}{|D(\varphi_2-\varphi)(P)|}=\frac{(f'_{\rm{sh}}(\xi),-1)}{\sqrt{1+\big(f'_{\rm{sh}}(\xi)\big)^2}}\,.
\end{equation*}
Combining the above with property~\eqref{item2-Lem3-2} in Lemma~\ref{lem:monotonicityForPhi2-Phi}, for any \((a_1,a_2) \in [0,\infty)^2\), \((a_1,a_2)\neq (0,0)\),
\begin{equation*}\begin{aligned}
&a_1\big(f'_{\rm{sh}}(\xi)\cos\theta_{25}-\sin\theta_{25}\big)+a_2\big(f'_{\rm{sh}}(\xi)\cos\theta_{26}-\sin\theta_{26}\big)  \\
&\;\;= \sqrt{1+\big(f'_{\rm{sh}}(\xi)\big)^2} \; \bm{\nu}(P)  \cdot(a_1\bm{e}_{S_{25}}+a_2\bm{e}_{S_{26}})<0
\qquad\,\,\mbox{for any $\xi\in(\xi^{P_3},\xi^{P_2})$}\,.
\end{aligned}
\end{equation*}
Choosing $(a_1,a_2)$ as $(1,0)$ and $(0,1)$ in turn, we obtain that
$f'_{\rm{sh}}(\xi)>\tan\theta_{25}$ and $f'_{\rm{sh}}(\xi)<\tan\theta_{26}$, respectively, where we have used $\theta_{25}\in(\tfrac{\pi}{2},\pi]$ and $\theta_{26}\in [0,\frac{\pi}{2})$. \end{proof}

\begin{lemma}
\label{UniformBound-Lem}
Fix $\gamma\geq1$ and $v_2\in(v_{\min},0)$.
For any $\btheta \in \Theta,$ let $\varphi$ be an admissible solution corresponding to
parameters $\btheta \in \Theta$ in the sense of {\rm Definition~\ref{Def:AdmisSolus}},
and let $\Omega$ be the pseudo-subsonic region of \(\varphi\).
Then there exists a constant $C_{\rm ub}>0$ depending only on $(\gamma,v_2)$ such that the following properties hold{\rm:}
\begin{align}
\label{UB-01}
&\cl{\Omega} \subseteq B_{C_{\rm ub}}(\bm{0})\,, \\
\label{UB-02}
&\|\varphi\|_{C^{0,1}(\overline{\Omega})}\leq C_{\rm ub}\,, \\
\label{UB-03}
&\rho^*(\gamma)\leq\rho\leq C_{\rm ub} \quad \text{in $\Omega$}\,, \qquad 1<\rho\leq C_{\rm ub} \quad \text{on $\Gamma_{\rm{shock}}$}\,,
\end{align}
where $\rho^*(\gamma) \defeq (\frac{2}{\gamma+1})^{\frac{1}{\gamma-1}}$ when $\gamma>1,$ and $\rho^*(1) \defeq e^{-\tfrac{1}{2}}=\lim\limits_{\gamma\to1+}\rho^*(\gamma)$.
\end{lemma}

\noindent
\textit{Proof.}
We divide the proof into three steps.

\smallskip
\textbf{1}. To prove~\eqref{UB-01}, we use the expression of $\Gamma_{\rm{shock}}$ from Lemma~\ref{Prop:GammaShockExpres}.
In particular,
it follows from
from Definition~\ref{Def:AdmisSolus}\eqref{item1-Def:AdmisSolus} and Lemma~\ref{Prop:GammaShockExpres}
that
$0 < f_{\rm{sh}}(\xi) \leq \min\{\xi \tan \theta_{25} + a_{25}, \xi \tan \theta_{26} + a_{26}\}$
whenever \(\xi^{P_3} < \xi < \xi^{P_2}\), which implies that
\begin{equation*}
\Omega \subseteq \big\{\bm{\xi} \,:\, u_6-c_6 < \xi < u_5+c_5 , \,
0 < \eta \leq \max\{a_{25}, a_{26}\}\big\}\,.
\end{equation*}
From~\eqref{P0-Pos} and Proposition~\ref{Prop:localtheorystate5}\eqref{Prop:localtheorystate5:continuity-properties}
and Definition~\ref{def:state5andstate6},
we know that $(\rho_5,u_5, a_{25})$
depend continuously on $\theta_1\in[0,\theta^{\rm{d}}]$,
while $(\rho_6,u_6,a_{26})$
depend continuously on $\theta_2\in[0,\theta^{\rm{d}}]$,
from which~\eqref{UB-01} follows directly.

\smallskip
\textbf{2}. From Definition~\ref{Def:AdmisSolus}\eqref{item4-Def:AdmisSolus}, we see that
$\inf_{\overline{\Omega}} \max \{\varphi_5,\varphi_6\} \leq \varphi \leq \sup_{\overline{\Omega}} \varphi_2$.
By~\eqref{UB-01} and the expressions of $(\varphi_2,\varphi_5,\varphi_6)$ given in~\eqref{PseudoVelocity123}
and~\eqref{eq:def-weak-state-5-6}, there exists a constant $C_1>0$ depending only on $(\gamma,v_2)$ such that
$-C_1 \leq \min_{\cl{\Omega}}\max\{\varphi_5,\varphi_6\} \leq \max_{\cl{\Omega}} \varphi_2 \leq C_1$,
which implies that $\max\limits_{\cl{\Omega}} |\varphi| \leq C_1$.
Then the uniform bound \eqref{UB-02} follows from
\begin{equation*}
|D\varphi|^2<c^2(|D\varphi|,\varphi)=\frac{2}{\gamma+1}\big(1+(\gamma-1)(B-\varphi)\big)\,,
\end{equation*}
where we have used Definition~\ref{Def:AdmisSolus}\eqref{item3-Def:AdmisSolus} and \eqref{Def4SonicSpeedInSec3}.

\smallskip
\textbf{3}. The upper bounds in~\eqref{UB-03} follow directly from the expression of density~\eqref{Relat4Rho} along with~\eqref{UB-02}.
For the lower bound in \(\Omega\), we combine the Bernoulli law~\eqref{Relat4Rho} with Definition~\ref{Def:AdmisSolus}\eqref{item4-Def:AdmisSolus} to obtain
\begin{equation*}\begin{aligned}
\frac12|D\varphi|^2 + h(\rho)
= \frac12|D\varphi_2|^2 + (\varphi_2-\varphi) + h(\rho_2)\geq \frac12 |D\varphi_2|^2 \geq 0\,.
\end{aligned} \end{equation*}
Then, from Definition~\ref{Def:AdmisSolus}\eqref{item3-Def:AdmisSolus}, we have
\begin{equation*}
\frac{1}{\gamma-1} \big( \frac12 (\gamma+1) \rho^{\gamma-1} - 1 \big) = \frac12c^2 + h(\rho) \geq \frac12|D\varphi|^2+h(\rho)\geq0\,,
\end{equation*}
so that the lower bound of $\rho$ in $\Omega$ is obtained.
For the lower bound of $\rho$ on \(\Gamma_{\rm shock}\),
we combine Problem~\ref{FBP}\eqref{item2-FBP} and Remark~\ref{Rmk4AdmissibleSolu}\eqref{item2-Rmk4AdmissibleSolu}
to obtain
\begin{flalign*}
&& \rho = \frac{D\varphi_2\cdot\bm{\nu}}{D\varphi\cdot\bm{\nu}} > 1 \qquad \mbox{on $\Gamma_{\rm{shock}}$}\,.
&& \qed
\end{flalign*}

\subsubsection{Uniform positive lower bound of \({\rm dist}(\Gamma_{\rm{shock}}, \partial B_1 (O_2))\)}
We now obtain a uniform estimate for the positive lower bound of
\({\rm{dist}}(\Gamma_{\rm{shock}}, \partial B_1(O_2))\) for any admissible solution.
This allows us to make the uniform estimate for the ellipticity of equation~\eqref{Eq4PseudoPoten} within the pseudo-subsonic domain of admissible solutions.

\begin{proposition}
\label{prop:lowerBoundBetweenShockAndSonicCircle}
Fix \(\gamma \geq 1\) and \(v_2 \in (\vmin,0)\).
There exists a constant \(C_{\rm sh}>0\) depending only on \((\gamma, v_2)\) such that any admissible solution corresponding to parameters \(\btheta \in \Theta \) satisfies
\begin{equation}\label{Eq:Result-Prop-3-4}
    {\rm{dist}} \big( \Gamma_{\rm{shock}}, \, \partial B_1(O_2) \big) \geq C_{\rm sh}^{-1}\,.
\end{equation}
\end{proposition}

\begin{proof} The proof of~\eqref{Eq:Result-Prop-3-4} follows the same argument
as for~\cite[Proposition 3.7]{BCF-2019} except the case that \( v_2 \leq -1\).
In the following, we focus on the case that \( v_2 \leq -1\) and give the proof
in three steps.

\smallskip
\textbf{1}. We first rewrite equation~\eqref{Eq4PseudoPoten} as
\begin{equation*}\label{Newnotation4Eq}
\mbox{div}\mathcal{A}(D\varphi,\varphi)+\mathcal{B}(D\varphi,\varphi) = 0\,,
\end{equation*}
where
$\mathcal{A}(\bm{p},z)\defeq \rho(|\bm{p}|,z)\bm{p}$ and
$\mathcal{B}(|\bm{p}|,z)\defeq 2\rho(|\bm{p}|,z)$
with $\bm{p}=(p_1,p_2)\in\mathbb{R}^2$, $z\in\mathbb{R}$, and
\begin{equation}\label{Eq:DefRhopzSec302}
\rho(|\bm{p}|,z)=
\begin{cases}
\big(1+(\gamma-1)(\frac{1}{2}v_2^2-\frac{1}{2}|\bm{p}|^2-z)\big)^{\frac{1}{\gamma-1}}
&\quad\text{if} \; \gamma>1\,,\\
\,\exp(\frac{1}{2}v_2^2-\frac{1}{2}|\bm{p}|^2-z)  &\quad \text{if} \;\gamma = 1\,.\\
\end{cases}
\end{equation}
We denote the sonic speed as $c(|\bm{p}|,z) \defeq \rho^{(\gamma-1)/2}$ for $\gamma\geq1$.
For a constant $R>1$, define
\begin{equation}\label{KRSet}
\mathcal{K}_R \defeq
\Big\{(\bm{p},z) \in \mathbb{R}^2 \times \mathbb{R} \,:\, |\bm{p}| + |z| \leq R, \,
 \rho(|\bm{p}|,z) \geq R^{-1}, \,\frac{|\bm{p}|^2}{c^2(|\bm{p}|,z)} \leq 1-R^{-1}
\Big\}\,.
\end{equation}
For each $R>1$, there exists a constant $\lambda_R>0$ depending only on $(\gamma,v_2,R)$ such that,
for any $\bm{\kappa}=(\kappa_1,\kappa_2)\in\mathbb{R}^2$,
\begin{equation*}
\sum_{i,j=1}^2\partial_{p_j}\mathcal{A}_i(\bm{p},z)\kappa_i\kappa_j\geq\lambda_R|\bm{\kappa}|^2
\qquad\, \text{for any $(\bm{p},z)\in\mathcal{K}_R$}\,.
\end{equation*}
The next two lemmas follow directly from~\cite[Lemmas~3.8--3.9]{BCF-2019}.

\begin{lemma}\label{CoeffExtdLem}
For $R>2,$ let $\mathcal{K}_R$ be given by~\eqref{KRSet}.
Then there exist functions $(\tilde{\mathcal{A}},\tilde{\mathcal{B}})(\bm{p},z)$ in $\mathbb{R}^2\times\mathbb{R}$ satisfying the following properties{\rm :}
\begin{enumerate}[{\rm (i)}]
\item \label{CoeffExtdLem-item1}
For any $(\bm{p}_0,z_0)\in\mathcal{K}_R,$
$\,(\tilde{\mathcal{A}},\tilde{\mathcal{B}})(\bm{p},z)
= (\mathcal{A},\mathcal{B})(\bm{p},z)$
for all $(\bm{p},z)\in B_{\varepsilon}((\bm{p}_0,z_0))${\rm ;}

\item \label{CoeffExtdLem-item2}
For any $(\bm{p},z) \in\mathbb{R}^2\times\mathbb{R}$ and $\bm{\kappa}=(\kappa_1,\kappa_2)\in\mathbb{R}^2,$
$\,\,\sum_{i,j=1}^2\partial_{p_j}\tilde{\mathcal{A}}_i(\bm{p},z)\kappa_i\kappa_j\geq\lambda_R|\bm{\kappa}|^2${\rm ;}

\item \label{CoeffExtdLem-item3}
$|\tilde{\mathcal{B}}(\bm{p},z)|\leq C_0$
and
$|D^m_{(\bm{p},z)}(\tilde{\mathcal{A}},\tilde{\mathcal{B}})(\bm{p},z)|\leq C_m$
in $\mathbb{R}^2\times\mathbb{R}$ for each $m=1,2,\cdots,$
\end{enumerate}
where the positive constants
$\varepsilon$ and $\lambda_{R}$ depend only on $(\gamma,v_2,R)${\rm ,}
and $C_m$ depends on $(\gamma,v_2,R,m)$.
\end{lemma}

\begin{lemma}\label{LocBddLem}
Fix $\gamma\geq1$ and $v_2\in(v_{\min},0)$.
For any given constants $\alpha\in(0,1),$ $m\in\mathbb{N},$ and $r>0,$ there exist constants $C,C_m>0$
depending only on $(\gamma,v_2,\alpha,r),$ with $C_m$ depending additionally on $m,$
such that any admissible solution $\varphi$ corresponding to parameters
$\btheta\in\Theta$ satisfies the following estimates{\rm :}
\begin{enumerate}[{\rm (i)}]
\item \label{LocBddLem-item1}
For any $B_{4r}(P)\subseteq\Omega,$
\begin{align*}
\|\varphi\|_{2,\alpha,\overline{B_{2r}(P)}} \leq C\,, \qquad
\|\varphi\|_{m,\overline{B_{r}(P)}} \leq C_m\,;
\end{align*}

\item  \label{LocBddLem-item2}
If $P\in\Gamma_{\rm{sym}},$ and $B_{4r}(P)\cap\Omega$ is the half-ball $B^+_{4r}(P) \defeq B_{4r}(P)\cap\{\eta>0\},$ then
\begin{flalign*}
&&
\|\varphi\|_{2,\alpha,\overline{B_{2r}(P)}\cap\Omega} \leq C\,, \qquad
\|\varphi\|_{m,\overline{B_{r}(P)}\cap\Omega} \leq C_m\,.
&&
\end{flalign*}
\end{enumerate}
\end{lemma}

\textbf{2.}
Fix $\varepsilon>0$ small enough such that $2\xi^{P_3}<\xi^{P_0}$ for all $\theta_2\in[0,\varepsilon)$,
where $\xi^{P_0} < 0$ is the $\xi$--coordinate of $P_0\defeq\lim\limits_{\theta_2\to0^+}P_3$, which depends only on $(\gamma,v_2)$.
We define
\begin{equation}\label{Def4dsep}
d_{\rm sep}\defeq \min{\big\{ {-}\frac{\xi^{P_0}}{2}, \, \inf\limits_{\theta_2\in[\varepsilon,\theta^{\rm d}]}|u_6|\big\}}\,.
\end{equation}
It is direct to verify that \(d_{\rm sep} > 0\) by using \(u_6 = v_2\tan\theta_{26}\)
and property~\eqref{eq:reflected-shock-angle-lower-bound}, and $d_{\rm sep}$ depends only on $(\gamma,v_2)$.
Then, for any $\btheta\in\Theta$,  ${\rm dist}(P_2,P_3)\geq 2 d_{\rm sep}$.
We also define
\begin{equation}\label{eq:defOfr1}
    r_1 \defeq \frac{1}{2} \inf_{\btheta \in \Theta}
    \big\{ | D\varphi_5 (P_0^1 ) | , | D\varphi_6 (P_0^2) | \big\}\,.
\end{equation}
By~\eqref{eq:state5RHcondition1} and Lemma~\ref{lem:properties-of-state-5-and-6},
it is direct to verify that $r_1>0$ depends only on $(\gamma,v_2)$.

For any $r\in(0,r_1)$ and for $i=1,2$, we define two sets:
\begin{equation*}
    \begin{aligned}
        &   \mathscr{A}\defeq \big\{\theta_i\in(0,\theta^{\rm d}] \,:\, |P_0^i-P_{i+1}|
        \geq \frac{r}{20}\big\}\cup\{0\}\,,\quad  \mathscr{B}\defeq \big\{\theta_i\in[0,\theta^{\rm d}] \,:\,
         |P_{3i-2}-P_{i+1}| \geq \frac{r}{20}\big\}\,.
    \end{aligned}
\end{equation*}
From the continuous dependence of $P_0^i,P_{3i-2}$, and $P_{i+1}$ on $\theta_i\in(0,\theta^{\rm d}]$,
we know that $\mathscr{A}$ and $\mathscr{B}$ both are closed sets.
Then there exists a constant $C_1>0$ such that
\begin{equation*}
    {\rm dist}(P_{i+1},\Gamma_{\rm{sym}})\geq\frac{2}{C_1}
    \qquad \text{for all $\theta_i\in \mathscr{A}\cup\mathscr{B}$}\,.
\end{equation*}
If ${\rm dist}(P_{i+1},\Gamma_{\rm{sym}})\leq\frac{1}{C_1}$,
then $\theta_i\notin \mathscr{A}\cup \mathscr{B}$, which implies that
\begin{equation*}
    \max\{|P_0^i-P_{i+1}|, |P_{3i-2}-P_{i+1}|\} < \frac{r}{20}\,.
\end{equation*}

We now seek a constant $C_r > 0$ depending only on $(\gamma,v_2,r)$ such that,
for any admissible solution of {\rm Problem~\ref{FBP}} corresponding to parameters $\btheta\in\Theta$,
\begin{equation}\label{Eq:Lem942a}
\sup\limits_{P\in\Gamma_{\rm{shock}}\cap B_{r/2}(P_{3i-2})}{\rm{dist}}(P,\Gamma_{\rm{sym}})
>C_r^{-1}\qquad  \text{if $|P_{3i-2}-P_{i+1}|\leq\frac{r}{10}$}\,.
\end{equation}
To obtain the above result, we follow the proof as in~\cite[Lemma 9.4.2]{ChenFeldman-RM2018},
in which we need to take a limit of parameters $\btheta^{(n)}$.

In order to take the limit, we use a
compactness lemma below.
For concreteness, in the following, we use the notation defined in {\rm Definition~\ref{Def:GammaSonicAndPts}},
and append superscripts $(n)$ and $*$ to indicate that the corresponding quantities are related
to parameters $\btheta^{(n)}$ and $\btheta^{*}$, respectively.

\begin{lemma}
\label{Lem:SeqResult}
Fix $\gamma\geq1$ and $v_2\in(v_{\min},0)$.
Let $\big\{ \btheta^{(n)}\big\}_{n \in \mathbb{N}} \subseteq \Theta$ be any sequence of parameters satisfying
$\, \lim_{n\to\infty}\btheta^{(n)} = \btheta^{*}$
for some $\btheta^{*} \in \cl{\Theta}$.
For each \(n\in\mathbb{N},\) let
$\varphi^{(n)}$ be an admissible solution corresponding to parameters $\btheta^{(n)}$.
Then there exists a subsequence $\{\varphi^{(n_j)}\}_{j\in\mathbb{N}}$ converging uniformly on
any compact subset of $\mathbb{R}^2_+$ to a function $\varphi^*\in C^{0,1}_{\rm loc}(\overline{\mathbb{R}^2_+})$.
Moreover, $\varphi^*$ is a weak solution of equation~\eqref{Eq4PseudoPoten}
in $\mathbb{R}^2_+$ and satisfies the following properties{\rm :}
\begin{enumerate}[{\rm (i)}]
\item \label{Lem:SeqResult-item1}
For the corner points $P_m,$ $m=1,2,3,4,$ and the reflection points $P_0^i,$ $i=1,2,$
\begin{eqnarray*}
\lim\limits_{j\to\infty}P_m^{(n_j)}=P_m^*\,, \qquad
\lim\limits_{j\to\infty}P_0^{i,(n_j)}=P_0^{i,*}\,.
\end{eqnarray*}
Note that $\xi^{P_0^{1,*}} = \infty$ in the case that $\theta^*_1=0,$ and $\xi^{P_0^{2,*}}=-\infty$
in the case that $\theta^*_2=0$.

\item \label{Lem:SeqResult-item2}
Let $\eta=f_{2l}^{(n_j)}(\xi),$ $l=5,6,$ and $\eta=f^{(n_j)}_{\rm{sh}}(\xi)$ be the expressions
of $S^{(n_j)}_{2l}$ and $\Gamma^{(n_j)}_{\rm shock}$
from {\rm Definition~\ref{def:state5andstate6}} and {\rm Lemma~\ref{Prop:GammaShockExpres},} respectively.
Extend $f_{\rm{sh}}^{(n_j)}$ by
\begin{equation*}\label{Express4ShSeq}
f^{(n_j)}_{\rm{sh}}(\xi) =
\begin{cases}
\, f_{25}^{(n_j)}(\xi) &\quad \text{for $\xi^{P^{(n_j)}_2}< \xi < \xi^{P_0^{1,(n_j)}}$}\,,\\
\, f_{26}^{(n_j)}(\xi)  &\quad \text{for $\xi^{P_0^{2,(n_j)}} < \xi < \xi^{P_3^{(n_j)}}$}\,.
\end{cases}
\end{equation*}
Then
$\{f^{(n_j)}_{\rm{sh}}\}_{j\in\mathbb{N}}$ converges uniformly
to a function $f^{*}_{\rm{sh}}\in C^{0,1}([\xi^{P_0^{2,*}},\xi^{P_0^{1,*}}])$.

\item \label{Lem:SeqResult-item3}
Let smooth functions $\eta = g^{(n_j)}_{l,\rm{so}}(\xi)$
be the expressions of $\Gamma^{l,(n_j)}_{\rm{sonic}}$ for $l=5,6$.
Then $\{g^{(n_j)}_{5,\rm{so}}\}_{j\in\mathbb{N}}$ converges pointwise
to $g^{*}_{5,\rm{so}}$ on $(\xi^{P^{*}_2},\xi^{P^{*}_1}),$ and $\{g^{(n_j)}_{6,\rm{so}}\}_{j\in\mathbb{N}}$
converges pointwise to $g^{*}_{6,\rm{so}}$ on $(\xi^{P^{*}_4},\xi^{P^{*}_3})$.

\item \label{Lem:SeqResult-item4}
Denote
$\widehat{\Omega^*}\defeq \big\{\bm{\xi}\in[\xi^{P^*_4},\xi^{P^*_1}]\times\overline{\mathbb{R}_+}\,:\,
0\leq\eta\leq f^*_{\rm{bd}}(\xi)\big\},$
where $f^*_{\rm{bd}}(\cdot)$ is given by
\begin{equation*}
f^*_{\rm{bd}}(\xi) =
\begin{cases}
\, g^{*}_{6,\rm{so}}(\xi)  &\quad \text{for $\xi^{P^*_4}\leq\xi\leq\xi^{P^*_3}$}\,,\\
\, f^*_{\rm{sh}}(\xi) \qquad  &\quad \text{for $\xi^{P^*_3}<\xi\leq\xi^{P_2}$}\,,\\
\, g^*_{5,\rm{so}}(\xi) \qquad  &\quad \text{for $\xi^{P^*_2}<\xi\leq\xi^{P^*_1}$}\,.
\end{cases}
\end{equation*}
Denote by $\Omega^*$ the interior of $\widehat{\Omega^*},$
$\Gamma^*_{\rm{shock}} \defeq \{(\xi,f^*_{\rm{sh}}(\xi)) \,:\, \xi\in(\xi^{P^*_3},\xi^{P^*_2})\},$
and $\Gamma^*_{\rm{sym}}\defeq \{(\xi,0) \,:\, \xi\in(\xi^{P^*_4},\xi^{P^*_1})\},$
and denote by $\Gamma^{*,0}_{\rm{sym}}$ the relative interior of $\Gamma^*_{\rm{sym}}\setminus\Gamma^*_{\rm{shock}}$.
Then $\varphi^* \in C^\infty(\Omega^*\cup\Gamma^{*,0}_{\rm{sym}}),$ and
\begin{enumerate}[{\rm (\ref{Lem:SeqResult-item4}.1)}]
\item \label{Lem:SeqResult-item4a}
$\lim\limits_{j\to\infty}\|\varphi^{(n_j)}-\varphi^*\|_{C^2(K)}=0\,$ for any $K \Subset \Omega^* \cup \Gamma^{*,0}_{\rm{sym}}${\rm,}

\item \label{Lem:SeqResult-item4b}
$D (\varphi_2 - \varphi^*) \cdot (\cos\theta,\sin\theta) \leq 0\,$ for all $\theta \in (\theta^*_{26},\theta^*_{25})${\rm,}

\item \label{Lem:SeqResult-item4c}
Equation~\eqref{Eq4PseudoPoten} is strictly elliptic in $\Omega^*\cup\Gamma^{*,0}_{\rm{sym}}${\rm,}

\item \label{Lem:SeqResult-item4d}
In $\mathbb{R}^2_+\setminus\Omega^*,$ $\varphi^*$ is equal to the constant states $\bar{\varphi}^*,$ $\varphi^*_5,$
and $\varphi^*_6$ in their respective domains as given
in {\rm Definition~\ref{Def:AdmisSolus}\eqref{item2-Def:AdmisSolus},}
where $\bar{\varphi}^*$ is given by~\eqref{BackGroundSolu} with $\btheta = \btheta^*$.
\end{enumerate}

\item \label{Lem:SeqResult-item5}
$\Gamma_{\rm{sym}}^{*,0}=\{\varphi_2>\varphi^*\}\cap\Gamma_{\rm{sym}}^{*}$ is dense in $\Gamma_{\rm{sym}}^{*}$.
\end{enumerate}
\end{lemma}

Using Lemmas~\ref{UniformBound-Lem} and~\ref{LocBddLem},
properties~\eqref{Lem:SeqResult-item1}--\eqref{Lem:SeqResult-item4} in Lemma~\ref{Lem:SeqResult}
can be proved by
similar arguments
as for~\cite[Corollary~3.10]{BCF-2019}.

For property~\eqref{Lem:SeqResult-item5}, if it is not true,
we can choose $\xi_1,\xi_2\in(\xi^{P_3^*},\xi^{P^*_2})$ such that
$\{(\xi,0) \,:\, \xi_1<\xi<\xi_2\}\subseteq\{\varphi_2\leq\varphi^*\}\cap\Gamma_{\rm{sym}}^{*}$.
By Definition~\ref{Def:AdmisSolus}\eqref{item4-Def:AdmisSolus},
$\varphi^{(n_j)}\leq\varphi_2$ in $\Omega^{(n_j)}$ so that
$\varphi^{*}\leq\varphi_2$ in $\overline{\Omega^{*}}$.
Combining with property (iv.4), we obtain that $\varphi^*(\bm{\xi})=\varphi_2(\bm{\xi})$ on $\{\bm{\xi}\in\overline{\mathbb{R}_+^2}: \xi_1<\xi_2\}$.
We take any non-negative function $\zeta\in C^{\infty}_{\rm c}(\overline{\mathbb{R}^2_+})$
satisfying that ${\rm supp} \, \zeta\subseteq \{\bm{\xi}\in\overline{\mathbb{R}^2_+} \,:\, \xi_1 < \xi < \xi_2 \}$
and $\zeta > 0$ on some subset of $\{(\xi,0) \,:\, \xi_1 < \xi < \xi_2 \}$ with positive measure.
Then
\begin{align*}
&\int_{\mathbb{R}^2_+} (\rho(|D\varphi^*|,\varphi^*) D\varphi^* \cdot D\zeta(\bm{\xi}) - 2\rho(|D\varphi^*|,\varphi^*) \zeta(\bm{\xi})) \, {\rm d}\bm{\xi}\\
&\;\;= \int_{\mathbb{R}^2_+} ( \rho_2 D\varphi_2 \cdot D\zeta(\bm{\xi}) - 2 \rho_2 \zeta(\bm{\xi})) \, {\rm d} \bm{\xi}
= -v_2 \int_{\xi_1}^{\xi_2} \zeta(\xi,0) \, {\rm d} \xi > 0\,,
\end{align*}
which contradicts the fact that $\varphi^*$ is a weak solution of
equation~\eqref{Eq4PseudoPoten} in $\mathbb{R}^2_+$.
Then Lemma~\ref{Lem:SeqResult}\eqref{Lem:SeqResult-item5} is verified.

\smallskip
We prove~\eqref{Eq:Lem942a} by
contradiction. If it is not true, we can choose a sequence of parameters $\{\btheta^{(n)}\}_{n\in\mathbb{N}}$ such that
\begin{equation*}
    \sup\limits_{P\in\Gamma^{(n)}_{\rm{shock}}\cap B_{r/2}(P^{(n)}_{3i-2})}{\rm{dist}}(P,\Gamma^{(n)}_{\rm{sym}})\leq\frac{1}{n} \qquad
    \text{if $|P^{(n)}_{3i-2}-P^{(n)}_{i+1}|\leq\frac{r}{10}$}\,.
\end{equation*}
Then, by Lemma~\ref{Lem:SeqResult}, we can further choose
a subsequence $\{\btheta^{(n_j)}\}_{j\in\mathbb{N}}$ and take a limit
to some $\btheta^{*}\in\overline{\Theta}$, which provides a contradiction
to Lemma~\ref{Lem:SeqResult}\eqref{Lem:SeqResult-item5}.

From~\eqref{Eq:Lem942a}, we can choose $\hat{Q}_i\in\overline{\Gamma_{\rm{shock}}\cap B_{r/2}(P_{3i-2})}$
with ${\rm dist}(\hat{Q}_i,\Gamma_{\rm{sym}})\geq C_r^{-1}$,
so that
$\hat{Q}_i\in \overline{\Gamma_{\rm shock} \cap B_{3r/4}(P_0^i)}$.
For $i = 1,2$, we take $Q_i=P_{i+1}\in\overline{\Gamma_{\rm shock}}$
if ${\rm dist}(P_{i+1},\Gamma_{\rm sym})> C_1^{-1}$;
otherwise, we take $Q_i=\hat{Q}_{i}\in\overline{\Gamma_{\rm shock}\cap B_{3r/4}(P_0^i)}$ as chosen above.
Then it follows that \begin{equation*}
    \min\{{\rm dist}(Q_1,Q_2), \, {\rm dist}(Q_1,\Gamma_{\rm sym}),\, {\rm dist}(Q_2,\Gamma_{\rm sym})\}
    \geq \min\{d_{\rm sep}, \, C_1^{-1}, \, C_{r}^{-1}\} \eqdef a\,.
\end{equation*}
Combining with Lemma~\ref{Lem:SeqResult} and following a similar proof
as for~\cite[Lemmas~9.4.1 and~15.4.1]{ChenFeldman-RM2018}, we obtain that,
for
$a>0$ defined above, there exists a constant $C_2>0$ depending only on $(\gamma,v_2,a)$
such that, for any admissible solution
corresponding to parameters $\btheta\in\Theta$,
\begin{equation*}
    {\rm dist}(\Gamma_{\rm shock}[Q_1,Q_2], \,\Gamma_{\rm sym})\geq C_2^{-1}\,,
\end{equation*}
where $\Gamma_{\rm shock}[Q_1,Q_2]$ represents the segment on $\Gamma_{\rm shock}$ between points $Q_1$ and $Q_2$.
Therefore,
\begin{equation}\label{Eq:ResultInStep3}
{\rm dist}(\Gamma_{\rm shock} \setminus ( B_r (P_0^1) \cup B_r (P_0^2) ) , \, \Gamma_{\rm{sym}} ) > C_2^{-1}\,.
\end{equation}

\textbf{3.} Since \( v_2 \leq -1\), \(B_1(O_2) \subseteq \mathbb{R}\times(-\infty,0)\).
By Lemma~\ref{lem:MonotonicityOfMach2} and property~\eqref{Prop:localtheorystate5:3n4} of Proposition~\ref{Prop:localtheorystate5},
\(d_{\rm ref}\defeq |D\varphi_2(P_0^1|_{\theta_1 = \theta^{\rm d}})| - 1 > 0\) depending only on \((\gamma, v_2)\) satisfies
\begin{align*}
    {\rm dist}(P_0^1, \,B_1(O_2))
    = |D\varphi_2(P_0^1)| - 1 \geq d_{\rm ref}\,.
\end{align*}
By symmetry, we also see that ${\rm dist}(P_0^2, \, B_1(O_2)) \geq d_{\rm ref}$.
Denote \(\bar{r} \defeq \frac12 \min \{r_1, d_{\rm ref}\} > 0\)
with \(r_1\) defined by~\eqref{eq:defOfr1}. Then
\begin{equation*}
    {\rm dist}(\Gamma_{\rm shock} \cap (B_{\bar{r}}(P_0^1) \cup B_{\bar{r}}(P_0^2)) , \, B_1(O_2) ) \geq \bar{r}\,.
\end{equation*}
Meanwhile, from the argument in Step~\textbf{2}, there exists a constant \(C_2>0\)
depending only on \((\gamma, v_2, \bar{r})\) such that~\eqref{Eq:ResultInStep3} holds
with $r$ replaced by $\bar{r}$.
Thus, for any admissible solution corresponding to parameters \(\btheta \in \Theta\), we have
\begin{align*}
    {\rm dist}(\Gamma_{\rm shock}, \, B_1(O_2) ) \geq \min\{ \bar{r}, \, C_2^{-1}\} \eqdef C_{\rm sh}^{-1}\,.
\end{align*}
This completes the proof of~\eqref{Eq:Result-Prop-3-4}.
\end{proof}

\subsubsection{Uniform estimate for the ellipticity of equation~\eqref{Eq4phi}}
We now show a uniform estimate for the ellipticity of equation~\eqref{Eq4phi},
which may be degenerate near $\Gamma_{\rm{sonic}}^{5}\cup\Gamma_{\rm{sonic}}^{6}$.
Fix a function \(h^\ast \in C^\infty(\mathbb{R})\) such that
\begin{equation*}
    h^\ast (s) =
    \begin{cases}
        \, s &\quad \text{for $s\leq \frac12$}\,, \\[1mm]
        \, 1 &\quad \text{for $s\geq1$}\,,
    \end{cases}
\end{equation*}
satisfying $0 \leq (h^\ast)'(\cdot)\leq 2$ on $\mathbb{R}$.
For each \( \btheta \in \cl{\Theta}\), define
\begin{equation}
\label{Eq:Sec3-Def4hatcj}
    \hat{c}_j \defeq \min{\{ c_{j}, \, \abs{D\varphi_{j} (P_0^{j-4})} \}}\,,
    \quad
    q_j  \defeq {\rm dist}(O_j, \, S_{2j}) \qquad\, \mbox{for $j=5,6$}\,.
\end{equation}
By Lemma~\ref{lem:properties-of-state-5-and-6} and~\cite[Lemma~6.1.2]{ChenFeldman-RM2018},
we see that \(\hat{c}_j > q_j\) for every \(\btheta \in \overline{\Theta}\), so that
there exists a constant \(\hat{\delta}>0\) depending only on \((\gamma,v_2)\) such that
\begin{equation}\label{Eq:DistO5P01minusO5S25}
\inf_{\btheta \in \cl{\Theta}}\min\{\hat{c}_5 - q_5, \, \hat{c}_6 - q_6\} \geq \hat{\delta} \,.
\end{equation}
For $j=5,6$, define \(g^\ast_j:\mathbb{R}^2 \to \mathbb{R}_+\) by
\begin{equation*}
g^\ast_j(\bm{\xi}) \defeq \frac{1}{2} (\hat{c}_j - q_j)
    h^\ast (\frac{{\rm{dist}}(\bm{\xi}, \, \partial B_{\hat{c}_j}(O_j))}{\hat{c}_j - q_j})
+\big(1-\frac{\hat{c}_j^2}{c^2_j}\big)\,.
\end{equation*}
We choose a function \(\chi\in C^{\infty}(\mathbb{R})\) such that
\begin{equation*}
    \chi(\xi) =
    \begin{cases}
       \,1  & \quad \text{for} \;\xi \leq -\tfrac{d_{\rm sep}}{4}\,, \\
       \,0  &  \quad \text{for} \;\xi \geq \tfrac{ d_{\rm sep}}{4}\,,
    \end{cases}
\end{equation*}
satisfying  $-\frac{4}{d_{\rm sep}} \leq\chi'(\cdot)\leq 0$ on $\mathbb{R}$,
where $d_{\rm sep}>0$ is given by~\eqref{Def4dsep} depending only on \((\gamma,v_2)\).
Finally, define a function \(g:\mathbb{R}^2 \to \mathbb{R}_+\) by
\begin{equation}  \label{eq:g_theta1theta2}
g(\bm{\xi}) \defeq \chi(\xi) g_6^\ast(\bm{\xi})  + (1 - \chi(\xi))g_5^\ast(\bm{\xi})\,.
\end{equation}

By {\rm Definition~\ref{Def:AdmisSolus}} and {\rm Lemma~\ref{UniformBound-Lem}},
there exist constants \(d>0\) and \(C>1\) depending only on \((\gamma, v_2)\) such that,
if \(\varphi\) is an admissible solution corresponding to \(\btheta\in\Theta\)
with pseudo-subsonic region \(\Omega\),
then \(g\) given by~\eqref{eq:g_theta1theta2} satisfies the following properties:
\begin{enumerate}[\quad(a)]
\item
\label{g-property-a}
For \(\bm{\xi}\in\Omega\) satisfying \({\rm{dist}}(\bm{\xi},\Gamma^j_{\rm{sonic}}) < d\),
     $C^{-1} {\rm{dist}}_{j}(\bm{\xi}, \,\Gamma_{\rm{sonic}}^j) \leq
        g(\bm{\xi}) \leq
        C\,{\rm{dist}}_{j}(\bm{\xi}, \,\Gamma_{\rm{sonic}}^j)$,
    where ${\rm{dist}}_{j}(\bm{\xi}, \,\Gamma_{\rm{sonic}}^j)
    \defeq {\rm{dist}}(\bm{\xi}, \,\Gamma_{\rm{sonic}}^j) + ( c_j - \hat{c}_j
    )$ for $j=5,6$.

\item
\label{g-property-b}
For each \(\varepsilon>0\), there exists a constant \(C_{\varepsilon} > 1\) depending only
on \((\gamma, v_2,\varepsilon)\) such that, if a point \(\bm{\xi}\in\cl{\Omega}\)
satisfies \({\rm{dist}}(\bm{\xi}, \, \Gamma_{\rm{sonic}}^5 \cup \Gamma_{\rm{sonic}}^6) > \varepsilon\),
then \(C_{\varepsilon}^{-1}\leq g(\bm{\xi})\leq C_{\varepsilon}\).
\end{enumerate}

In fact, for property~\eqref{g-property-a} above, it suffices
to choose $d\defeq\min\{\tfrac{d_{\rm sep}}{2}, \hat{\delta}\}$,
with constants $d_{\rm sep}$ and $\hat{\delta}$ given by~\eqref{Def4dsep} and~\eqref{Eq:DistO5P01minusO5S25}
respectively.
Then, choosing \(\varepsilon \defeq \tfrac{d}{2}\) in property~\eqref{g-property-b},
we deduce that there exists a constant \({C}_{\flat}>1\) depending only on \((\gamma, v_2)\) such that,
for any \(\btheta\in\Theta\),
\begin{equation}\label{Eq:FuncgEquivDistb}
    {C}_{\flat}^{-1} {\rm{dist}}^\flat(\bm{\xi}, \, \Gamma^5_{\rm{sonic}}\cup \Gamma^6_{\rm{sonic}} ) \leq
    g(\bm{\xi}) \leq
    {C}_{\flat} \, {\rm{dist}}^\flat(\bm{\xi}, \, \Gamma^5_{\rm{sonic}}\cup \Gamma^6_{\rm{sonic}} )
    \qquad
    \text{for all $\bm{\xi}\in\cl{\Omega}$}\,,
\end{equation}
where
\begin{equation} \label{eq:dist-flat}
    {\rm{dist}}^\flat(\bm{\xi}, \, \Gamma^5_{\rm{sonic}}\cup \Gamma^6_{\rm{sonic}} ) \defeq
    \min \big\{ \frac{d}{2}, \, {\rm{dist}}_{5}(\bm{\xi}, \,\Gamma_{\rm{sonic}}^5), \,{\rm{dist}}_{6}(\bm{\xi}, \, \Gamma_{\rm{sonic}}^6) \big\}\,.
\end{equation}
\begin{proposition}
\label{prop:ellipticDegeneracyNearSonicBoundary}
There exists a constant \(\mu>0\) such that, if \(\varphi\) is an admissible solution corresponding
to \(\btheta\in\Theta\) and \(\Omega\) is its pseudo-subsonic region, then the pseudo-Mach number
\begin{equation*}
    M(\bm{\xi})\defeq \frac{|D\varphi(\bm{\xi})|}{c(|D\varphi(\bm{\xi})|,\varphi(\bm{\xi})) }
\end{equation*}
satisfies
\begin{equation} \label{eq:MachNumberNearSonicBoundary}
    M^2(\bm{\xi}) \leq 1 - \mu g(\bm{\xi}) \qquad \text{in $\cl{\Omega}$}\,,
\end{equation}
and there exists a constant \(C>1\) such that
\begin{equation}\label{Eq:Ellipticity4PsedoPF}
C^{-1} {\rm{dist}}^\flat(\bm{\xi}, \, \Gamma^5_{\rm{sonic}}\cup \Gamma^6_{\rm{sonic}}) |\bm{\kappa}|^2
\leq \sum_{i,j=1}^2 \mathcal{A}^i_{p_j}(D\varphi(\bm{\xi}),\varphi(\bm{\xi}))\kappa_i\kappa_j \leq C|\bm{\kappa}|^2
\end{equation}
for all \(\bm{\xi}\in\overline{\Omega}\) and \(\bm{\kappa}=(\kappa_1,\kappa_2)\in\mathbb{R}^2\),
where \(\mathcal{A}(\bm{p},z) = \rho(\abs{\bm{p}},z)\bm{p},\)
\({\rm{dist}}^\flat(\blank, \blank)\) is given by~\eqref{eq:dist-flat}{\rm,} and
constants \(\mu\) and \(C\) depend only on \((\gamma, v_2)\).
\end{proposition}

\begin{proof}
Let \(\varphi\) be any admissible solution corresponding to \(\btheta\in\Theta\)
with pseudo-subsonic region \(\Omega\) and curved transonic shock \(\Gamma_{\rm{shock}}\).

\smallskip
\textbf{1.}
From Lemma~\ref{UniformBound-Lem}, there exists a constant \(R>1\)
depending only on $(\gamma, v_2)$
such that
\begin{equation*}
    \Omega \subseteq B_{R}(\bm{0})\,,
    \qquad
    \| c(|D\varphi|,\varphi) \|_{C^0(\cl{\Omega})}
    +
    \| g \|_{C^2(\cl{\Omega})}
    \leq R\,,
\end{equation*}
where \(g(\cdot)\) is given by~\eqref{eq:g_theta1theta2}.
Since \(O_5, O_6 \in \Gamma_{\rm sym}\),
\(\partial_\eta g = 0\) on \(\Gamma_{\rm{sym}}\).
Then, by~\cite[Theorem~5.3.1]{ChenFeldman-RM2018}, we can choose
constants $C_0>1$, $\delta\in(0,\frac{3}{4C_0})$, and \(\mu_1 \in (0,1)\)
depending only on $(\gamma,v_2)$ such that, for any $\mu \in(0,\mu_1]$,
either inequality \(M^2+\mu g \leq C_0 \delta <1\) holds in \(\Omega\)
or the maximum of \(M^2 + \mu g\) over \(\cl{\Omega}\) cannot
be attained in \(\Omega \cup \Gamma_{\rm{sym}}\).

In the first case,~\eqref{eq:MachNumberNearSonicBoundary} follows immediately.
Thus, we focus on the second case that the maximum of \(M^2 + \mu g\) must
be attained on \(\partial \Omega \setminus \Gamma_{\rm{sym}}\).

\smallskip
\textbf{2.}
We claim that there exists a constant \(\mu_2 \in (0,\mu_1]\) depending only
on \((\gamma,v_2)\) such that, whenever \(\mu \in (0,\mu_2]\),
\((M^2 + \mu g) \leq 1\) holds in \(\Omega\), from which~\eqref{eq:MachNumberNearSonicBoundary} follows.

Let the maximum of \(M^2 + \mu g\) be attained
at \(P_{\rm{max}}\in \partial \Omega \setminus \Gamma_{\rm{sym}}\),
and let \((M^2 + \mu g)(P_{\rm{max}}) > 1\).
Then, using Proposition~\ref{prop:lowerBoundBetweenShockAndSonicCircle}, we may proceed as in Steps~\textbf{2}--\textbf{3}
of the proof of~\cite[Proposition~3.15]{BCF-2019}
to find \(\mu_2 \in (0,\mu_1]\) depending only on \((\gamma,v_2)\)
such that \(P_{\rm max} \notin \Gamma_{\rm shock}\) when
\(\mu \in (0,\mu_2]\).
Therefore, \(P_{\rm max} \in \Gamma^5_{\rm sonic} \cup \Gamma^6_{\rm sonic}\).
However, reducing \(\mu_2\) depending only on \((\gamma,v_2)\) if necessary, we have \begin{equation*}
    (M^2 + \mu g)(P_{\rm max}) = \sup_{\bm{\xi} \in \Gamma^5_{\rm sonic} \cup \Gamma^6_{\rm sonic}} (M^2 + \mu g)(\bm{\xi}) \leq 1\,,
\end{equation*}
which is a contradiction.
In the final inequality above,
we have used that \(g|_{\Gamma^j_{\rm sonic}} = 0\) for \(\theta_{j-4} \in [0,\theta^{\rm s})\),
whilst \(g|_{\Gamma^j_{\rm sonic}} \leq C_{\flat} (c_j - \abs{D\varphi_j(P_0^{j-4})})\)
for \(\theta_{j-4} \in [\theta^{\rm s},\theta^{\rm d})\), for \(j = 5,6\).
\end{proof}

\subsection{Interior H{\"o}lder estimates away from the sonic boundaries}
\label{SubSec-EstimatesAwaySonic5n6}
For any set $U\subseteq\mathbb{R}^2$ and any constant $\varepsilon>0$, define
the $\varepsilon$-neighbourhood of $U$ as
\begin{equation}\label{Eq:Def-N-eps}
\mathcal{N}_{\varepsilon}(U)\defeq \{\bm{\xi}\in\mathbb{R}^2 \,:\, {\rm dist}(\bm{\xi},\, U)<\varepsilon\}\,.
\end{equation}
From Proposition~\ref{prop:lowerBoundBetweenShockAndSonicCircle},
there exists a constant $C_{\rm sh}>0$ depending only on $(\gamma, v_2)$ such that
\begin{equation*}
{\rm dist}(\Gamma_{\rm shock}, \, \partial B_1(O_2))\geq C_{\rm sh}^{-1}
\end{equation*}
for any admissible solution and parameters \(\btheta\in\Theta\).
It follows that
\begin{equation*}
|D\varphi_2(\bm{\xi})|^2\geq1+d_0 \qquad
\text{for any $\bm{\xi}\in\mathcal{N}_{\tfrac{1}{2 C_{\rm sh}}}(\Gamma_{\rm{shock}})$}
\end{equation*}
for some constant $d_0>0$ depending only on $(\gamma, v_2)$.
Subsequently, we claim that there exist constants \(d_0'\in(0,d_0)\) and \(\varepsilon \in (0,\frac{1}{2 C_{\rm sh}})\)
depending only on \((\gamma,v_2)\) such that, for any admissible solution \(\varphi\),
\begin{equation}\label{PositiveDifference}
 |D\varphi_2(\bm{\xi})|-|D\varphi(\bm{\xi})|\geq d_0'
 \qquad\,\text{for any $\bm{\xi}\in\overline{\Omega}\cap\mathcal{N}_{\varepsilon}(\Gamma_{\rm shock})$}\,.
\end{equation}
Indeed, in the case: $\gamma>1$, from Definition~\ref{Def:AdmisSolus}\eqref{item3-Def:AdmisSolus}--\eqref{item4-Def:AdmisSolus}
and the Bernoulli law:
\[
\frac{1}{2}|D\varphi|^2+\varphi+\frac{\rho^{\gamma-1}-1}{\gamma-1}=B=\frac{1}{2}|D\varphi_2|^2+\varphi_2\,,
\]
we have
\[\begin{aligned} \frac{\gamma+1}{2}\left(\rho^{\gamma-1}-1\right)
\geq\, & \left(\rho^{\gamma-1}-1\right)+\frac{\gamma-1}{2}\left(|D\varphi|^2-1\right)\\
=\, & (\gamma-1)(\varphi_2-\varphi)+\frac{\gamma-1}{2}\left(|D\varphi_2|^2-1\right)\geq\frac{\gamma-1}{2}d_0\,,
\end{aligned}\]
and
\begin{equation*}
|D\varphi_2|^2-|D\varphi|^2 = \frac{2\left(\rho^{\gamma-1}-1\right)}{\gamma-1}+(\varphi-\varphi_2)\geq\frac{2d_0}{\gamma+1}+(\varphi-\varphi_2)\,.
\end{equation*}
Since $\varphi$ is Lipschitz continuous in $\overline{\Omega}$ from Lemma~\ref{UniformBound-Lem},
and $\varphi=\varphi_2$ on $\Gamma_{\rm shock}$, we obtain
\begin{equation*}
|D\varphi_2(\bm{\xi})|^2-|D\varphi(\bm{\xi})|^2 \geq\frac{d_0}{\gamma+1}
\qquad\text{for any $\bm{\xi}\in\overline{\Omega}\cap\mathcal{N}_{\varepsilon}(\Gamma_{\rm shock})$}\,,
\end{equation*}
with some $\varepsilon\in(0,\frac{1}{2C_{\rm sh}})$ sufficiently small.
Then~\eqref{PositiveDifference} follows from the above inequality and the uniform boundedness property in Lemma~\ref{UniformBound-Lem}.
In the case: $\gamma=1$, it is straightforward to obtain~\eqref{PositiveDifference}
by using Definition~\ref{Def:AdmisSolus}\eqref{item3-Def:AdmisSolus}:
for any admissible solution $\varphi$,
$|D\varphi|^2\leq c^2(|D\varphi|,\varphi)\equiv1$ in $\overline{\Omega}$, since
the sonic speed is a constant.

For \(\bm{\xi} = (\xi,\eta) \in \mathbb{R}^2 \setminus \{O_2\}\), define the polar coordinates $(r,\theta)$ centred at \(O_2=(u_2,v_2)\) by
\begin{align*}
    r(\cos\theta,\sin\theta) \defeq (\xi,\eta) - O_2\,.
\end{align*}
It follows from~\eqref{PositiveDifference} that
\begin{equation}\label{PositiveDifferenceDeriv2r}
\partial_r(\varphi_2-\varphi)\leq-(|D\varphi_2|-|D\varphi|)\leq -d_0' \qquad
\text{in $\overline{\Omega}\cap\mathcal{N}_{\varepsilon}(\Gamma_{\rm shock})$}\,.
\end{equation}
Therefore, we may apply the implicit function theorem to
express $\Gamma_{\rm shock}$ under these polar coordinates as
\begin{equation*}
\Gamma_{\rm shock}=\{\bm{\xi}(r,\theta)\,:\, r=f_{O_2,\rm{sh}}(\theta), \, \theta_{P_2} < \theta < \theta_{P_3} \}\,,
\end{equation*}
where $(f_{O_2,\rm{sh}}(\theta_{P_i}),\theta_{P_i})$ are the polar coordinates of points $P_i$, for $i=2,3$.
By Lemma~\ref{UniformBound-Lem} and~\eqref{PositiveDifferenceDeriv2r},
there exists a constant $C_1>0$ depending only on $(\gamma, v_2)$ such that
\begin{equation*}
\|f_{O_2,\rm{sh}}\|_{C^{0,1}\left([\theta_{P_2},\theta_{P_3}]\right)}\leq C_1\,.
\end{equation*}

From the above, with the help of Lemma~\ref{CoeffExtdLem},
we deduce the following properties; the proof is similar to that of~\cite[Lemmas~3.17--3.18]{BCF-2019}.

\begin{lemma}\label{Lem:LocalEstimates4Interior}
Fix $\gamma\geq1$ and $v_2\in(v_{\min},0)$.
Let $\varphi$ be an admissible solution corresponding to parameters $\btheta\in \Theta$.
\begin{enumerate}[{\rm (i)}]
\item  \label{Lem:LE4Int-1}
There exists a constant $\delta^{\prime}>0$ depending only on $(\gamma, v_2)$ such that
\begin{equation*}
\partial_{\bm{\nu}}\varphi_2-\delta^{\prime}>\partial_{\bm{\nu}}\varphi\geq\delta^{\prime}
\qquad \text{on $\overline{\Gamma_{\rm shock}}$}\,,
\end{equation*}
where $\bm{\nu}=\frac{D(\varphi_2-\varphi)}{|D(\varphi_2-\varphi)|}$
is the unit normal vector on $\Gamma_{\rm shock}$ pointing into the interior of $\Omega$.

\item \label{Lem:LE4Int-2}
For each $d>0$ and $m=2,3,\cdots,$ there exist positive constants $s$ and $C_m$
depending only on $(\gamma, v_2,d)$ such that,
whenever $P = \bm{\xi}(r_P,\theta_{P}) \in \Gamma_{\rm shock}$ satisfies
${\rm{dist}}(P,\Gamma_{\rm sonic}^{5}\cup\Gamma_{\rm sonic}^{6})\geq d,$
\begin{flalign*}
&&
|f^{(m)}_{O_2,\rm{sh}}(\theta_P)| + \sup_{B_s(P)\cap\Omega} |D^m\varphi|\leq C_m\,.
&&
\end{flalign*}
\end{enumerate}
\end{lemma}

From Lemmas~\ref{LocBddLem}, \ref{prop:ellipticDegeneracyNearSonicBoundary}, and \ref{Lem:LocalEstimates4Interior}\eqref{Lem:LE4Int-2},
we have the following {\it a priori} interior estimates.

\begin{corollary}
\label{Corol:Esti-AwaySonic5n6}
Fix $\gamma\geq1$ and $v_2\in(v_{\min},0)$.
For each $d>0$ and $m=2,3,\cdots,$
there exists a constant $C_{m,d}>0$ depending only on $(\gamma, v_2,m,d)$ such that
any admissible solution $\varphi$ corresponding to $\btheta\in\Theta$ satisfies
\begin{flalign*}
&&
\|\varphi\|_{m,\overline{\Omega\cap\{{\rm dist}(\bm{\xi},\Gamma_{\rm sonic}^{5}\cup\Gamma_{\rm{sonic}}^{6})>d\}}} \leq C_{m,d}\,.
&&
\end{flalign*}
\end{corollary}

\subsection{Weighted H{\"o}lder estimates near the sonic boundaries}
\label{SubSec-EstimatesNearSonic5}
From Definition~\ref{Def:GammaSonicAndPts}, the sonic boundary $\Gamma_{\rm sonic}^{5}$ depends continuously
on $\theta_1\in[0,\theta^{\rm d})$ and shrinks into a single point $P_0^1=P_1=P_2$
when $\theta_1\in[0,\theta^{\rm s})$ tends to $\theta^{\rm s}$ from below (if additionally $v_2 \in [v_2^{\rm s},0)$).
Moreover, from Proposition~\ref{prop:ellipticDegeneracyNearSonicBoundary},
the ellipticity of equation~\eqref{Eq4phi} degenerates on the sonic boundary
$\Gamma_{\rm sonic}^5$ for $\theta_1\in[0,\theta^{\rm s})$,
while equation~\eqref{Eq4phi} is uniformly elliptic in the case that
$\theta_1\in(\theta^{\rm s},\theta^{\rm d})$ up to $\Gamma_{\rm sonic}^5$ away from $\Gamma_{\rm sonic}^6$.

The above reasons lead us to consider the following four cases separately
for the weighted H{\"o}lder estimates of the admissible solutions near $\Gamma^5_{\rm sonic}$.
Since $\theta^{\rm s}$ can be equal to $\theta^{\rm cr}$ depending on \((\gamma,v_2)\),
some of the following cases may not
occur (\textit{cf.}~Remark~\ref{remark:local-theory}\eqref{remark:local-theory-item-i}):
\begin{enumerate}[]
\item {\rm Case~1.} \textit{ ${\theta}_{1}\in[0,{\theta}^{\rm{s}})$ away from ${\theta}^{\rm{s}}$};
\item {\rm Case~2.} \textit{ ${\theta}_{1}\in[0,{\theta}^{\rm{s}})$ close to ${\theta}^{\rm{s}}$};
\item {\rm Case~3.} \textit{ ${\theta}_{1}\in[{\theta}^{\rm{s}},{\theta}^{\rm{d}})$ close to ${\theta}^{\rm{s}}$};
\item {\rm Case~4.} \textit{ ${\theta}_{1}\in[{\theta}^{\rm{s}},{\theta}^{\rm{d}})$  away from ${\theta}^{\rm{s}}$}.
\end{enumerate}
Corresponding weighted $C^{2,\alpha}$--estimates near \(\Gamma^{6}_{\rm sonic}\) for the four cases with respect
to \(\theta_2 \in [0,\theta^{\rm d})\) are obtained immediately due to the symmetry of our problem.

For any \(\btheta \in \Theta\), define an open set $\Omega^5 \subseteq \Omega$
by
\begin{equation*}
\Omega^5 \defeq
\big(\Omega\cap\{\bm{\xi}\in\mathbb{R}^2_+ \,:\, \xi > u_5 + q_5\sin\theta_{25} \}\big)\setminus \overline{B_{c_5^*}(O_5)}\,,
\end{equation*}
where $q_5={\rm dist}(O_5,S_{25})$ is given by~\eqref{Eq:Sec3-Def4hatcj}, and
$c_5^* \defeq \frac{1}{2} \big(|O_5P_2|+{\rm dist}(O_5,S_{25})\big)$.
For \(\bm{\xi}\in \cl{\Omega^5}\), we introduce the modified polar coordinates \((x,y)\) centred at \(O_5\) defined by
\begin{equation}\label{PolarCoordiNearO5}
(x,y) \defeq \rcal (\xi,\eta)
\quad \Longleftrightarrow
\quad
(\xi,\eta)-O_5 = (c_5-x)(\cos y,\sin y) \qquad \text{with $y \in [0,\tfrac\pi2]$}\,.
\end{equation}
It is clear that $\rcal(\Omega^5) \subseteq \{(x,y) \,:\, 0<x<c_5-c_5^*\}$.
For the pseudo-potential $\varphi_5$ given by~\eqref{eq:def-weak-state-5-6}, we have
\begin{equation*}
\varphi_5 \circ \rcal^{-1} (x,y) = -\frac12 (c_5-x)^2 + \frac12 u_5^2 - u_5 \xi^{P_0^1}
\qquad\text{for $(x,y) \in
\rcal(\Omega^{5})$}\,.
\end{equation*}
Note that \(u_5\xi^{P_0^1} \to -v_2 \eta_0 \) as \(\theta_1 \to 0^+\), where \(\eta_0 >0 \)
is defined in \S\ref{SubSec-202NormalShock}.

For any admissible solution $\varphi$ corresponding to parameters $\btheta\in \Theta$, let $\psi$ be given by
\begin{equation*}
\psi(x,y) \defeq (\varphi-\varphi_5) \circ \rcal^{-1}(x,y) \qquad \text{for $(x,y) \in \rcal(\Omega^5)$}\,.
\end{equation*}
Then $\psi$ satisfies a free boundary problem for an elliptic equation:
\begin{enumerate}[{\rm (i)}]
\item \textit{Equation for $\psi$ in $\rcal(\Omega^5)$}:
Direct calculation gives
\begin{equation}\label{Eq4PsiInNewCoordi}
\mathcal{N}(\psi)=\big(2x-(\gamma+1)\psi_x+\co_1\big)\psi_{xx}+\co_2\psi_{xy}
+\big(\frac{1}{c_5}+\co_3\big)\psi_{yy}-\big(1+\co_4\big)\psi_x+\co_5\psi_y=0\,,
\end{equation}
with $\co_j=\co_j(\psi_x, \psi_y,\psi,x,c_5)$ for $j=1,\cdots,5$, defined as follows:
\begin{equation}\label{Eq:Def-co-k}
\begin{aligned}
\co_1(p_x,p_y,z,x,c) &= -\frac{x^2}{c} + \frac{\gamma+1}{2c} (2x-p_x) p_x - \frac{\gamma-1}{c} \Big( z + \frac{p_y^2}{2(c-x)^2} \Big)\,,\\
\co_2(p_x,p_y,z,x,c) &= -\frac{2(p_x+c-x)p_y}{c(c-x)^2}\,,\\
\co_3(p_x,p_y,z,x,c) &= \frac{1}{c(c-x)^2} \Big( (2c-x)x
- (\gamma-1)\big( z+(c-x)p_x - \frac{p_x^2}{2}\big) + \frac{(\gamma-1)p_y^2}{2(c-x)^2} \Big)\,,\\
\co_4(p_x,p_y,z,x,c) &= \frac{1}{c-x} \Big( x-\frac{\gamma-1}{c} \big( z+(c-x)p_x + \frac{p_x^2}{2}
+ \frac{(\gamma+1)p_y^2}{2(\gamma-1)(c-x)^2} \big) \Big)\,,\\
\co_5(p_x,p_y,z,x,c) &= -\frac{2(p_x+c-x)p_y}{c(c-x)^3}\,.
\end{aligned}
\end{equation}

\smallskip
\item \textit{Boundary condition for $\psi$ on $\rcal(\Gamma_{\rm shock} \cap \partial\Omega^5)$}:
Denote the Rankine--Hugoniot condition for \(\varphi\) on $\Gamma_{\rm shock}$ as
\begin{equation}\label{Def4FreeBdryFunc}
g^{\rm sh}(D\varphi,\varphi,\bm{\xi})\defeq \rho(|D\varphi|,\varphi)D\varphi\cdot\bm{\nu}-D\varphi_{2}\cdot\bm{\nu}
=(\mathcal{A}(D\varphi,\varphi)-D\varphi_2)\cdot\bm{\nu}=0\,,
\end{equation}
where $\mathcal{A}(\bm{p},z)=\rho(|\bm{p}|,z)\bm{p}$,
and $\bm{\nu}=\frac{D(\varphi_2-\varphi)}{|D(\varphi_2-\varphi)|}$
is the unit interior normal vector on $\Gamma_{\rm shock}$.
For constant $\delta^{\prime}>0$ given in Lemma~\ref{Lem:LocalEstimates4Interior}, we define a
function $\zeta \in C^{\infty}(\mathbb{R})$ by
\begin{equation*}
\zeta(t) = \begin{cases}

t &\quad  \text{for} \; t \geq \tfrac{3\delta^{\prime}}{4}\,,\\
\tfrac{\delta^{\prime}}{2}
& \quad \text{for} \; t < \tfrac{\delta^{\prime}}{2}\,,
\end{cases}
\end{equation*}
satisfying $\zeta'(t)\geq0$ on $\mathbb{R}$. Also, we define an extension of
$g^{\rm sh}(\bm{p},z,\bm{\xi})$ onto $\mathbb{R}^2\times\mathbb{R} \times \overline{\Omega^{5}}$:
\begin{equation}\label{FBFunc:gshmod}
g_{\rm mod}^{\rm sh}(\bm{p},z,\bm{\xi})\defeq \big(\tilde{\mathcal{A}}(\bm{p},z)-D\varphi_2(\bm{\xi})\big)
\cdot\frac{D\varphi_2(\bm{\xi})-\bm{p}}{\zeta(|D\varphi_2(\bm{\xi})-\bm{p}|)}\,,
\end{equation}
where $\tilde{\mathcal{A}}(\bm{p},z)$ is given by Lemma~\ref{CoeffExtdLem}, and
$g_{\rm mod}^{\rm sh}(\bm{p},z,\bm{\xi})=g^{\rm sh}(\bm{p},z,\bm{\xi})$ if $|D\varphi_2(\bm{\xi})-\bm{p}|\geq\frac{3\delta_1}{4}$.
Then we define
\begin{equation*}
M^{(\theta_1)}(\bm{p},z,\xi)\defeq g_{\rm mod}^{\rm sh}(\bm{p}+D\varphi_5,z+\varphi_5, \xi,
\xi \tan \theta_{25} + a_{25}
+\frac{z}{v_2})
\end{equation*}
with $(D\varphi_5,\varphi_5)$ evaluated at $\bm{\xi}=(\xi,\,\xi \tan \theta_{25} + a_{25}
+\frac{z}{v_2})$.
Note that, by Proposition~\ref{Prop:localtheorystate5}\eqref{Prop:localtheorystate5:continuity-properties} and
Definition~\ref{def:state5andstate6},
$(\theta_{25},\,a_{25})$
depend continuously on $\theta_1\in[0,\theta^{\rm d}]$
with the limit:
$\lim\limits_{\theta_1\to0^+}(\theta_{25}, a_{25})
=(\pi, \eta_0)$,
where $\eta_0>0$ is given by~\eqref{Sol-NormalReflec0-RH}.

We prescribe the free boundary condition for $\psi$ to be
\begin{equation*}
\mathcal{B}^{(\theta_1)}(\psi_x,\psi_y,\psi,x,y)=0 \qquad
\text{on $\rcal(\Gamma_{\rm shock} \cap \partial \Omega^5)$}\,,
\end{equation*}
where $\mathcal{B}^{(\theta_1)}$ is given by
\begin{equation}\label{DefBTheta1}
\mathcal{B}^{(\theta_1)}(p_x,p_y,z,x,y)\defeq \abs{D\bar{\phi}_{25}-(p_1,p_2)} M^{(\theta_1)}(p_1,p_2,z,\xi)\,,
\end{equation}
with function $\bar{\phi}_{25}(\bm{\xi}) \defeq (\varphi_2-\varphi_5)(\bm{\xi})
= v_2(\eta-\xi\tan\theta_{25}-a_{25})$,
and
\begin{equation*}
\xi = v_2\tan\theta_{25} + (c_5-x) \cos y\,, \qquad
\begin{pmatrix}p_1\\p_2\end{pmatrix}
=\begin{pmatrix}\cos y & \sin y \\
\sin y & -\cos y\end{pmatrix}
\begin{pmatrix}-p_x \\ -\frac{p_y}{c_5-x}\end{pmatrix}.
\end{equation*}

\item \textit{Other conditions for $\psi$}: From Problem~\ref{FBP}\eqref{item3n4-FBP}--\eqref{item5-FBP}
and Definition~\ref{Def:AdmisSolus}\eqref{item4-Def:AdmisSolus}, $\psi$ satisfies
\begin{equation*}
\psi \geq 0 \,\,\,\, \text{in $\rcal(\Omega^5)$}\,,
\qquad \;\;
\psi = 0 \,\,\,\,  \text{on $\rcal(\Gamma_{\rm sonic}^{5})$}\,,
\qquad \;\;
\psi_y = 0 \,\,\,\, \text{on $\rcal(\Gamma_{\rm sym}\cap\partial\Omega^5)$}\,.
\end{equation*}
\end{enumerate}

\begin{definition}[Parabolic norm]
\label{Def:ParabolicNorm01}
Fix \(\alpha \in (0,1)\).
Define the $\alpha$-th parabolic distance between points
$\bm{x}=(x,y), \; \tilde{\bm{x}}=(\tilde{x},\tilde{y})\in (0,\infty) \times \mathbb{R}$ to be
\begin{equation*}
\delta^{({\rm par})}_{\alpha}(\bm{x},\tilde{\bm{x}})\defeq \left(|x-\tilde{x}|^2+\max\{x,\tilde{x}\}|y-\tilde{y}|^2\right)^{\frac{\alpha}{2}}.
\end{equation*}
Then, given constants $\sigma>0$ and $m\in\mathbb{Z}_{+},$ the parabolic norms are defined as follows{\rm :}
\begin{enumerate}[{\rm (i)}]

\item \label{DefPNorm04}
For any open set $\mathcal{D}\subseteq (0,\infty) \times \mathbb{R}$
and any function $u\in C^2(\mathcal{D})$ in the $(x,y)$--coordinates, define
\begin{equation*} \begin{aligned}
	& \|u\|^{(\sigma),({\rm par})}_{m,0,\mathcal{D}}
\defeq \sum\limits_{0\leq k+l\leq m}\sup_{\bm{x}\in\mathcal{D}}
   \Big(x^{k+\frac{l}{2}-\sigma}|\partial^{k}_{x}\partial^{l}_y u(\bm{x})|\Big)\,,\\
	& [u]^{(\sigma),({\rm par})}_{m,\alpha,\mathcal{D}}
 \defeq \sum\limits_{k+l=m}\sup_{\substack{\bm{x},\tilde{\bm{x}}\in\mathcal{D}\\
 \bm{x}\neq\tilde{\bm{x}}}}\Big(\min\big\{x^{\alpha+k+\frac{l}{2}-\sigma},\tilde{x}^{\alpha+k+\frac{l}{2}-\sigma}\big\}
 \frac{|\partial^{k}_{x}\partial^{l}_y u(\bm{x})-\partial^{k}_{x}\partial^{l}_y
   u(\tilde{\bm{x}})|}{\delta^{({\rm par})}_{\alpha}(\bm{x},\tilde{\bm{x}})}\Big)\,,\\
    &\|u\|^{(\sigma),({\rm par})}_{m,\alpha,\mathcal{D}}
    \defeq\|u\|^{(\sigma),({\rm par})}_{m,0,\mathcal{D}}+[u]^{(\sigma),({\rm par})}_{m,\alpha,\mathcal{D}}\,.
    \end{aligned}
\end{equation*}

\item \label{DefPNorm03}
For any \(a>0,\) fix an open interval $I\defeq (0,a)$. For a function $f\in C^2(I),$ define
\begin{equation*}
    \begin{aligned}
	&\|f\|^{(\sigma),({\rm par})}_{m,0,I}\defeq \sum\limits_{k=0}^{m}\sup_{x\in I} \Big(x^{k-\sigma}|\partial^{k}_{x}f(x)|\Big)\,,\\
	&[f]^{(\sigma),({\rm par})}_{m,\alpha,I}\defeq \sup_{\substack{x,\tilde{x}\in I\\
 x\neq\tilde{x}}}\Big(\min\big\{x^{\alpha+m-\sigma},\tilde{x}^{\alpha+m-\sigma}\big\} \frac{|\partial^{m}_{x}f(x)-\partial^{m}_{x}f(\tilde{x})|}{|x-\tilde{x}|^{\alpha}}\Big)\,,\\
    &\|f\|^{(\sigma),({\rm par})}_{m,\alpha,I}\defeq \|f\|^{(\sigma),({\rm par})}_{m,0,I}+[f]^{(\sigma),({\rm par})}_{m,\alpha,I}\,.
\end{aligned} \end{equation*}

\item
In the case \(\sigma = 2,\) write $\|u\|^{({\rm par})}_{2,\alpha,\mathcal{D}} \defeq \|u\|^{(2),({\rm par})}_{2,\alpha,\mathcal{D}}$ and
$\|f\|^{({\rm par})}_{2,\alpha,I} \defeq \|f\|^{(2),({\rm par})}_{2,\alpha,I}$.
\end{enumerate}
\end{definition}

\smallskip
\subsubsection*{{\rm Case~1}. Admissible solutions for ${\theta}_{1}<\theta^{\rm s}$ away from $\theta^{\rm s}$.}

\begin{proposition}\label{Prop:Esti4theta1a}
Fix $\gamma\geq1$ and $v_2\in(v_{\min},0)$.
For some constant \(\bar{\varepsilon} > 0\) depending only on \((\gamma,v_2),\) any \(\sigma \in (0,\theta^{\rm s}),\)
and any $\alpha \in(0,1),$ there exist both $\varepsilon \in(0, \bar{\varepsilon}]$ depending only on $(\gamma,v_2,\sigma)$
and $C>0$ depending only on $(\gamma,v_2,\alpha)$
such that any admissible solution $\varphi = \psi \circ \rcal + \varphi_5$ corresponding to
$\btheta\in \Theta \cap \{0 \leq \theta_1 \leq \theta^{\rm s}-\sigma\}$ satisfies
\begin{equation}\label{Result4theta1a}
\norm{\psi}_{2, \alpha, \rcal(\Omega^5_{\varepsilon})}^{(\mathrm{par})} +
\norm{f_{5, \mathrm{sh}}-f_{5,0}}_{2, \alpha,(0, \varepsilon)}^{(\rm par)} \leq C\,,
\end{equation}
where \(\Omega^{5}_{\varepsilon} \subseteq \Omega\) is a neighbourhood of \(\Gamma^5_{\rm sonic},\)
and functions ${f}_{5, \mathrm{sh}}$ and ${f}_{5,0}$ represent $\Gamma_{\rm shock}$ and $S_{25}$ respectively
to be defined in {\rm Step~\textbf{1}} below.
\end{proposition}

The proof of this proposition is similar to that of~\cite[Propositions~3.26 and~3.30]{BCF-2019}.
We sketch the main steps of the proof, while omitting most
details.

\smallskip
\textbf{1.} \textit{The depiction and the properties of the reflected shock $S_{25}$ and
the free boundary $\Gamma_{\rm shock}$ in a small neighbourhood of the sonic boundary $\Gamma^5_{\rm sonic}$.}
It follows from the geometric properties of $\overline{\Omega^5}$ that
there exist positive constants $(\varepsilon_1,\varepsilon_0,\omega_0)$ depending only on $(\gamma,v_2)$,
with $\varepsilon_0<\varepsilon_1$ and \(\omega_0 \in (0,1)\),
and a unique function ${f}_{5,0}\in C^{\infty}([-\varepsilon_0,\varepsilon_0])$ satisfying
\begin{align*}\label{Eq:PreLemBSP5}
\rcal \big( S_{25}\cap\mathcal{N}_{\varepsilon_1}(\Gamma^5_{\rm sonic})
\big)
\cap\{|x|<\varepsilon_0\}
=\big\{(x,y) \,:\, |x|<\varepsilon_0, \, y={f}_{5,0}(x)\big\}\,,
\end{align*}
where $\mathcal{N}_{\varepsilon_1}(\Gamma^5_{\rm sonic})$ is defined by~\eqref{Eq:Def-N-eps},
and $2\omega_0\leq{f}'_{5,0}(x)\leq\omega_0^{-1}$ for $x\in(-\varepsilon_0,\varepsilon_0)$.
Then, for any $\varepsilon\in(0,\varepsilon_1]$, we define a set $\Omega^5_{\varepsilon} \subseteq \Omega$ by
\begin{equation}\label{Eq:Def4Omega5epsilon}
 \Omega^5_{\varepsilon}\defeq \Omega\cap\mathcal{N}_{\varepsilon_1}(\Gamma^{5}_{\rm sonic})\cap \rcal^{-1} \big(
 \{x<\varepsilon\}
 \big).
\end{equation}

Using~\eqref{Eq:Ellipticity4PsedoPF} and~\eqref{Eq4PsiInNewCoordi},
we can show that there exist $\bar{\varepsilon}\in(0,\varepsilon_0]$, $L\geq1$, $\delta\in(0,\frac12)$,
and $\omega\in(0,\omega_0)$ depending only on $(\gamma,v_2)$ such that,
for all $(x,y)\in \rcal(\Omega^5_{\bar{\varepsilon}})$,
\begin{equation}\label{Eq:PrepLemASP05}
0 \leq \psi(x,y) \leq Lx^2\,, \qquad 0\leq\psi_x\leq\frac{2-\delta}{1+\gamma}x\leq Lx\,,
\qquad
|\psi_y(x,y)|\leq Lx\,,
\end{equation}
and there exists a unique function ${f}_{5,{\rm sh}}\in C^1([0,\bar{\varepsilon}])$ such that
\begin{align*}
&\rcal(\Omega^5_{\bar{\varepsilon}})
=\big\{
(x,y) \,:\, 0 < x < \bar{\varepsilon}, \, 0<y<{f}_{5,{\rm sh}}(x)
\big\}\,,\\
&\rcal(\Gamma_{\rm shock}\cap \partial\Omega^5_{\bar{\varepsilon}})
=\big\{
(x,y) \,:\, 0 < x < \bar{\varepsilon}, \, y={f}_{5,{\rm sh}}(x)
\big\}\,,
\end{align*}
and $\omega\leq{f}'_{5,{\rm sh}}(x)\leq L$ for any $x\in(0,\bar{\varepsilon})$.

\smallskip
\textbf{2.} \textit{Ellipticity for the equation and the free boundary condition.}
Equation~\eqref{Eq4PsiInNewCoordi} for $\psi$
in $\rcal(\Omega^5)$ can be written as
\begin{equation}\label{QuasilinearEq4Theta1}
\sum_{i,j=1}^{2}\hat{A}_{ij}^{(\theta_1)}(\psi_x,\psi_y,\psi,x)
\partial_{i} \partial_{j} \psi+\sum_{i,j=1}^{2}\hat{A}_{i}^{(\theta_1)} \partial_i \psi=0\,,
\end{equation}
with $(\partial_1,\partial_2) =(\partial_x,\partial_y)$
and $\hat{A}^{(\theta_1)}_{12}=\hat{A}^{(\theta_1)}_{21}$.
Then there exist $\varepsilon_5\in(0,\frac{\bar{\varepsilon}}{4}]$ and $\lambda_5>0$
depending only on $(\gamma,v_2)$ such that, for all $(x,y)\in \rcal( \overline{\Omega^5_{4\varepsilon_5}})$,
\begin{equation*}\label{PreLemBCP0}
\frac{\lambda_5}{2}|\bm{\kappa}|^2
\leq \sum_{i,j=1}^{2}\hat{A}^{(\theta_1)}_{ij}(\psi_x (x,y), \psi_y (x,y), \psi(x,y),x)\frac{\kappa_i\kappa_j}{x^{2-\frac{i+j}{2}}}
\leq\frac{2}{\lambda_5}|\kappa|^2 \qquad \text{for all $\bm{\kappa}\in\mathbb{R}^2$}\,.
\end{equation*}
Moreover, $\mathcal{B}^{(\theta_1)}$ defined by~\eqref{DefBTheta1}
satisfies $\mathcal{B}^{(\theta_1)}(\mathbf{0},0,x,y)=0$ for all $(x,y)\in\mathbb{R}^2$.
Let $y_{P_2}$ be the $y$--coordinate of point $P_2$ defined in Definition~\ref{Def:GammaSonicAndPts}.
For each $m=2,3,\cdots$, there exist constants $\delta_{\rm bc}>0$ and $C>1$ depending only on $(\gamma,v_2,m)$
such that, whenever $|(p_x,p_y,z,x)|\leq\delta_{\rm bc}$ and $|y-y_{P_2}|\leq\delta_{\rm bc}$,
\begin{equation*}\label{PreLemBCP2}
 \big|D^{m}\mathcal{B}^{(\theta_1)}(p_x,p_y,z,x,y)\big|\leq C\,.
\end{equation*}
Furthermore, there exist  constants $\hat{\delta}_{\rm bc}>0$, $C>1$,
and $\varepsilon'>0$ depending only on $(\gamma,v_2)$ such that,
whenever $|(p_x,p_y,z,x)|\leq\hat{\delta}_{\rm bc}$ and $|y-y_{P_2}|\leq\hat{\delta}_{\rm bc}$,
\begin{equation*}
\begin{aligned}
& \partial_{a}\mathcal{B}^{(\theta_1)}(p_x,p_y,z,x,y) \leq -C^{-1}
  &&\quad \text{for $\partial_a = \partial_{p_x}, \partial_{p_y}$, or $\partial_z$}\,,\\
& D_{(p_x,p_y)}\mathcal{B}^{(\theta_1)}(p_x,p_y,z,x,y)\cdot\bm{\nu}^{(x,y)}_{\rm sh} \geq C^{-1}
 &&\quad \text{on $\rcal(\Gamma_{\rm shock}\cap\partial\Omega_{\varepsilon'}^{5})$}\,,
 \end{aligned}
\end{equation*}
where
the vector field $\bm{\nu}^{(x,y)}_{\rm sh}$ is the unit normal vector to $\rcal(\Gamma_{\rm shock})$
expressed in the $(x,y)$--coordinates and oriented towards the interior of $\rcal(\Omega)$.

\smallskip
\textbf{3.} \textit{Extension of the domain for the coefficients to equation~\eqref{QuasilinearEq4Theta1}.}
Let $\varepsilon_0>0$ and $L\geq1$ be the constants defined as above.
Then there exist constants $\varepsilon\in(0,\frac{\varepsilon_0}{2}]$
and $C>0$ depending only on $(\gamma,v_2)$ such that,
for any admissible solution $\varphi$
corresponding to parameters $\btheta\in \Theta$, function \(\psi = (\varphi - \varphi_5)\circ\rcal^{-1}\) satisfies
\begin{equation}\label{ModQuasilinearEq4Theta1}
\sum\limits_{i,j=1}^{2}\hat{A}_{ij}^{(\rm mod)}(\psi_x, \psi_y, \psi,x) \partial_{i} \partial_{j} \psi
+\sum\limits_{i,j=1}^2\hat{A}_{i}^{(\rm mod)}(\psi_x ,\psi_y, \psi,x) \partial_{i} \psi=0\qquad \text{in $\rcal(\Omega^5_{\varepsilon})$}\,,
\end{equation}
with $(\partial_1,\partial_2)\defeq (\partial_x,\partial_y)$ and
coefficients $(\hat{A}_{ij}^{(\rm mod)},\hat{A}_{i}^{(\rm mod)})$ satisfying
\begin{equation*}
(\hat{A}_{ij}^{(\rm mod)},\hat{A}_{i}^{(\rm mod)})=(\hat{A}_{ij}^{(\theta_1)},\hat{A}_{i}^{(\theta_1)})
\quad\,\, \text{in $\{(p_x,p_y,z,x) \,:\, |(p_x,p_y)|\leq Lx, \, |z|\leq Lx^2, \, 0 < x < \varepsilon \}$}\,.
\end{equation*}
Moreover, the modified coefficients satisfy
\begin{align*}
& \big|(\hat{A}_{11}^{(\rm mod)},\hat{A}_{12}^{(\rm mod)},\hat{A}_{2}^{(\rm mod)})(p_x,p_y,z,x)\big|
  \leq Cx \qquad\text{in $\mathbb{R}^2\times\mathbb{R}\times(0,\varepsilon)$}\,,\\
& \big\|(\hat{A}_{22}^{(\rm mod)},\hat{A}_{1}^{(\rm mod)})\big\|_{0,\mathbb{R}^2\times\mathbb{R}\times(0,\varepsilon)}
+\big\|D_{(p_x,p_y,z,x)}(\hat{A}_{ij}^{(\rm mod)},\hat{A}_{i}^{(\rm mod)})\big\|_{0,\mathbb{R}^2\times\mathbb{R}\times(0,\varepsilon)}\leq C\,.
\end{align*}

\smallskip
\textbf{4.} \textit{Rescaling method in the local domain to obtain the weighted H{\"o}lder estimate.}
Using the strictly decreasing property of $\eta^{P_2}$ with respect
to $\theta_1 \in [0,\theta^{\rm s}]$ which is shown in Lemma~\ref{lem:monotonicityOfTheta5}\eqref{item2-lem:monotonicityOfTheta5},
we obtain a constant $l_{\rm so}>0$ depending only on $(\gamma,v_2,\sigma)$ such that,
for any admissible solution $\varphi$ corresponding to
$\btheta\in \Theta \cap \{0 \leq \theta_1 \leq \theta^{\rm s}-\sigma\}$,
\begin{equation*}
y_{P}={f}_{5,{\rm sh}}(x_P)\geq l_{\rm so} \qquad \text{for any $x_P\in[0,\bar{\varepsilon}]$}\,.
\end{equation*}
Subsequently, we define $\varepsilon_*\defeq\min\{\tfrac{\bar{\varepsilon}}{2}, \, l^2_{\rm so}\}$.

For fixed $\varepsilon\in(0,\varepsilon_*]$
and any $z_0=(x_0,y_0) \in \rcal(\overline{\Omega^5_{\varepsilon}}\setminus\overline{\Gamma_{\rm sonic}^{5}})$,
define the rescaled coordinates $(S,T)$ by $(x,y) \eqdef z_0+\frac14(x_0S,\sqrt{x_0}T)$, which take values in a local domain
\begin{equation*}
Q_r^{(z_0)}\defeq\big\{
(S,T)\in(-1,1)^2 \,:\, z=z_0+\frac{r}{4}(x_0S,\sqrt{x_0}T)\in \rcal(\Omega_{2\varepsilon}^{5})
\big\}
\qquad \text{for any $r\in(0,1]$}\,.
\end{equation*}
Considering property~\eqref{Eq:PrepLemASP05}, we define a rescaled function $\psi^{(z_0)}$  by
\begin{equation*}
\psi^{(z_0)}(S,T)\defeq\frac{1}{x_0^2}\psi(x_0+\frac{x_0}{4}S,y_0+\frac{\sqrt{x_0}}{4}T)
\qquad \text{for $(S,T)\in Q_1^{(z_0)}$}\,.
\end{equation*}
With a suitable rescaling on the coefficients of the modified equation~\eqref{ModQuasilinearEq4Theta1},
we obtain a uniformly elliptic equation for $\psi^{(z_0)}$
under the $(S,T)$--coordinates.

Note that, for $z_0\in \rcal( \overline{\Omega^5_{\varepsilon}}\cap\Gamma_{\rm shock} )$,
we can express $\Gamma_{\rm shock}$ locally under the rescaled coordinates
as $\Gamma^{(z_0)}_{\rm shock}\defeq \big\{(S,T)\in(-1,1)^2 \,:\, T=F^{(z_0)}(S) \big\}\subseteq\partial Q^{(z_0)}_1$
with $F^{(z_0)}(S)$ defined by
\begin{equation*}
F^{(z_0)}(S)\defeq \frac{4}{\sqrt{x_0}}\big({f}_{5,{\rm sh}}(x_0+\frac{x_0}{4}S)-{f}_{5,{\rm sh}}(x_0)\big)
\qquad \text{for $S\in(-1,1)$}\,.
\end{equation*}
Then, in the local domain $Q_1^{(z_0)}$,
we apply~\cite[Theorem 4.2.3]{ChenFeldman-RM2018} for $z_0\in \rcal( \Omega_{\varepsilon}^{5})$,
\cite[Theorem~4.2.8]{ChenFeldman-RM2018} for
$z_0\in \rcal( \Gamma_{\rm shock} \cap \partial\Omega_{\varepsilon}^{5} )$,
and~\cite[Theorem~4.2.10]{ChenFeldman-RM2018}
for $z_0 \in \rcal(\Gamma_{\rm sym}\cap\partial\Omega_{\varepsilon}^{5}) $, respectively,
to obtain the uniform bound:
\begin{equation*}
\sup\limits_{z_0\in \rcal(\Omega_{\varepsilon}^{5})} \big\|\psi^{(z_0)}\big\|_{2,\alpha,\overline{Q^{(z_0)}_{1/4}}}\,
+\sup\limits_{z_0\in \rcal( \Gamma_{\rm shock}\cap\partial\Omega_{\varepsilon}^{5})}\frac{1}{\sqrt{x_0}}
\big\|F^{(z_0)}\big\|_{2,\alpha,[-\frac{1}{4},\frac{1}{4}]}\leq C\,.
\end{equation*}
Scaling back into the $(x,y)$--coordinates,
we obtain the weighted H{\"o}lder estimates~\eqref{Result4theta1a}.

\subsubsection*{{\rm Case~2}.  Admissible solutions for ${\theta}_{1}<\theta^{\rm s}$ close to $\theta^{\rm s}$.}

\begin{proposition}
\label{Prop:Esti4theta1b}
Fix $\gamma\geq1$ and $v_2\in [v_2^{\rm s},0)$.
Let $\bar{\varepsilon}>0$ be the constant from {\rm Proposition~\ref{Prop:Esti4theta1a}}.
For any $\alpha \in(0,1),$ there exist constants $\varepsilon \in(0, \bar{\varepsilon}]$
and $\sigma_1 \in(0,{\theta}^{\rm{s}})$ depending only on $(\gamma,v_2),$ and $C>0$ depending only on $(\gamma, v_2, \alpha)$
such that any admissible solution $\varphi = \psi \circ \rcal + \varphi_5$
corresponding
to $\btheta\in \Theta \cap \{ \theta^{\rm s}-\sigma_1 \leq \theta_1 < \theta^{\rm s}\}$ satisfies
\begin{equation*}
\norm{\psi}_{2, \alpha, \rcal(\Omega^5_{\varepsilon})}^{(\mathrm{par})}
+\norm{{f}_{5, \mathrm{sh}}-{f}_{5,0}}_{2, \alpha,(0, \varepsilon)}^{(\operatorname{par})} \leq C\,,
\end{equation*}
where \(\Omega_{\varepsilon}^{5} \subseteq \Omega,\)
and functions ${f}_{5, \mathrm{sh}}$ and ${f}_{5,0}$ are defined in {\rm Step}~{\bf 1} of {\rm Case~1} above.
\end{proposition}

The proof is similar to that of~\cite[Proposition 3.32]{BCF-2019}.
The admissible solutions for {\rm Case~2} share the same properties as
in Steps~\textbf{1}--\textbf{3} of {\rm Case~1}. We omit the details of the proof
and only draw the main differences compared with the proof
of Proposition~\ref{Prop:Esti4theta1a} here.

\smallskip
\textbf{1.} \textit{Weighted H{\"o}lder estimate near $\rcal(P_0^1)$.}
Because
\(v_2 \in [v_2^{\rm s}, 0)\),
$y_{P_2}>0$ tends to zero as $\theta_1\rightarrow\theta^{\rm s-}$.
Therefore, for the same rescaled coordinates $(S,T)$
and local domain $Q_r^{(z_0)}$ with $z_0=(x_0,y_0)\in \rcal(\Omega^{5}_{\varepsilon})$ as given in
Step~\textbf{4} of Case~1 above,
it is possible that ``$Q_r^{(z_0)}$ does not fit into $\Omega^{5}_{\varepsilon}$'',
meaning that the image of  $Q_r^{(z_0)}$ under
transform \((S,T)\mapsto(x,y)\) intersects the boundaries:
$\rcal(\Gamma_{\rm shock}\cap\partial\Omega_{\varepsilon}^{5})$
and $\rcal(\Gamma_{\rm sym}\cap\partial\Omega_{\varepsilon}^{5})$ simultaneously.

For that reason, we define $\Omega^5_{\rm fc} \defeq \Omega^{5}_{\bar{\varepsilon}} \cap \rcal^{-1}(\{x < y^2_{P_2}\})$.
It follows that \(Q_r^{(z_0)}\) fits into \(\Omega^{5}_{\rm fc}\) for any \(z_0 \in \rcal(\Omega^{5}_{\rm fc})\)
due to similar considerations as the choice of $\varepsilon_*$ above.
In this case, we follow the same idea as for Proposition~\ref{Prop:Esti4theta1a}
to obtain that there exists a constant $C>0$ depending only on $(\gamma,v_2,\alpha)$ such that,
for $\btheta\in\Theta \cap \{\theta^{\rm s}-\sigma^* \leq \theta_1 < \theta^{\rm s}\},$
$\norm{\psi}^{({\rm par})}_{2,\alpha,\rcal(\Omega^{5}_{\rm fc})}\leq C$
for some small constant $\sigma^*>0$ depending only on $(\gamma,v_2)$.

\smallskip
\textbf{2.} \textit{Weighted H{\"o}lder estimate away from $\rcal(P_0^1)$.}
Note that there exists sufficiently large $k>1$ depending only on $(\gamma,v_2)$ such that the
following geometric relations hold:
\begin{equation*}
  \{0<x<2\bar{\varepsilon} , \, 0<y<y_{P_2}+\frac{x}{k}\}
  \; \subseteq \;
  \rcal(\Omega^5_{2\bar{\varepsilon}})
  \; \subseteq \;
  \{0<x<2\bar{\varepsilon} , \, 0<y<y_{P_2}+kx\}\,.
\end{equation*}
In the rest of the domain, there exist constants $\hat{\varepsilon} \in (0, \frac{\bar{\varepsilon}}{2}]$,
$\sigma'\in(0,\sigma^*]$, and $C^*>0$ depending only on $(\gamma,v_2)$ such that,
for $\btheta\in \Theta \cap \{\theta^{\rm s}-\sigma' \leq \theta_1 < \theta^{\rm s}\}$,
\begin{equation}
\label{PropertyForPhiInCase2}
0\leq\psi(x,y)\leq C^*x^4 \qquad \text{in $\rcal(\Omega^{5}_{2\hat{\varepsilon}}) \cap \{x>\tfrac{1}{10}{y^2_{P_2}}\}$}\,.
\end{equation}
Then, for any $z_0=(x_0,y_0)\in \rcal( \overline{\Omega^5_{\hat{\varepsilon}}})\cap \{ x > \frac{1}{5} y^2_{P_2}\}$,
we introduce the rescaled coordinates $(S,T)$ as $(x,y)\eqdef z_0+\frac{\sqrt{x_0}}{10k}(x_0S,\sqrt{x_0}T)$,
which take values in a local domain $Q_r^{(z_0)}$ defined by
\begin{equation*}
Q_r^{(z_0)}\defeq \big\{(S,T)\in(-1,1)^2 \,:\,
z=z_0+\frac{r\sqrt{x_0}}{10k}(x_0S,\sqrt{x_0}T)\in \rcal(\Omega_{2\hat{\varepsilon}}^{5}) \big\}\qquad\text{for any $r\in(0,1]$}\,.
\end{equation*}
Thus, the local domain $Q_r^{(z_0)}$ fits into  $\Omega_{2\hat{\varepsilon}}^{5}$, meaning that the image of
$Q_r^{(z_0)}$ under transform \((S,T) \mapsto (x,y)\) does not intersect both $\rcal(\Gamma_{\rm shock})$
and $\rcal(\Gamma_{\rm sym})$ simultaneously.
Considering property~\eqref{PropertyForPhiInCase2}, we define a rescaled function $\psi^{(z_0)}$ by
\begin{equation*}
\psi^{(z_0)}(S,T)\defeq\frac{1}{x_0^4}\psi(x_0+\frac{x^{3/2}_0}{10k}S,y_0+\frac{x_0}{10k}T)
\qquad \text{for $(S,T)\in Q_1^{(z_0)}$}\,.
\end{equation*}
The rescaled function $\psi^{(z_0)}$ satisfies a uniformly elliptic equation
in the $(S,T)$--coordinates after a suitable rescaling on the coefficients of the modified equation~\eqref{ModQuasilinearEq4Theta1}.
The rest of the proof follows the same idea as in Step~\textbf{4} of the proof of Proposition~\ref{Prop:Esti4theta1a}.

\subsubsection*{{\rm Case~3}.  Admissible solutions
for ${\theta}_{1} \geq \theta^{\rm s}$ close to $\theta^{\rm s}$.}
\begin{proposition}
\label{Prop:Esti4theta1c}
Fix $\gamma\geq1$ and $v_2\in(v_2^{\rm s},0)$.
Let $\bar{\varepsilon}>0$ be the constant
from {\rm Proposition~\ref{Prop:Esti4theta1a}}.
For each $\alpha\in(0,1),$  there exist constants $\varepsilon \in(0,\bar{\varepsilon}]$
and $\sigma_3\in(0,\theta^{\rm d}-\theta^{\rm s})$ depending only on $(\gamma,v_2)$,
and a constant $C>0$ depending only on $(\gamma,v_2,\alpha)$ such that
any admissible solution $\varphi$
corresponding to
$\btheta \in \Theta \cap \{{\theta}^{\rm{s}} \leq \theta_1 < {\theta}^{\rm{s}}+\sigma_3\}$ satisfies
\begin{equation}\begin{aligned}
&\|\varphi - \varphi_5\|_{2,\alpha,\overline{\Omega_{\varepsilon}^{5}}}
+ \|{f}_{5, \mathrm{sh}}-{f}_{5,0}\|_{2, \alpha,(0, \varepsilon)} \leq C \,,\\
&|D^m_{\bm{\xi}}(\varphi - \varphi_5) (P_0^1)|= \frac{{\rm d}^m}{{\rm d}x^m}\left( {f}_{5, \mathrm{sh}}-{f}_{5,0} \right)(0)
= 0  \qquad \text{for $m=0, 1, 2$}\,,
\end{aligned}\end{equation}
where \(\Omega_{\varepsilon}^{5} \subseteq \Omega,\) and functions ${f}_{5, \mathrm{sh}}$ and ${f}_{5,0}$ are defined in {\rm Step}~{\bf 1} below.
\qed
\end{proposition}
The proof is similar to that of~\cite[Proposition~3.39]{BCF-2019},
so we omit most details here and only emphasize the difference compared with the previous two cases.

\smallskip
\textbf{1.} \textit{The depiction and the properties of the reflected shock $S_{25}$ and
the free boundary $\Gamma_{\rm shock}$ in a small neighbourhood of the sonic boundary $\Gamma^5_{\rm sonic}$.}
Let $\varepsilon_1, \varepsilon_0,\omega_0>0$ be chosen as
in Step~\textbf{1} of the proof of Proposition~\ref{Prop:Esti4theta1a}.
For the coordinates: \((x,y) = \rcal(\bm{\xi})\) given by~\eqref{PolarCoordiNearO5},
we define $\hat{x}\defeq x-x_{P_2}$.
Note that $x_{P_2}=0$ if $\theta_1\in[0,\theta^{\rm s}]$ and $x_{P_2}>0$ if $\theta_1\in(\theta^{\rm s},\theta^{\rm s}+\sigma_2]$.
Then there exists a unique function ${f}_{5,0}\in C^{\infty}([-\varepsilon_0,\varepsilon_0])$ satisfying
\begin{align*}
\rcal \big(
S_{25}\cap\mathcal{N}_{\varepsilon_1}(\Gamma^5_{\rm sonic})
\big)
\cap\{(x,y) \,:\, |\hat{x}|<\varepsilon_0\}
=\big\{(x,y) \,:\, |\hat{x}|<\varepsilon_0,\, y={f}_{5,0}(\hat{x})\big\}\,,
\end{align*}
and $2\omega_0\leq{f}'_{5,0}(\hat{x})\leq\omega_0^{-1}$ for $\hat{x}\in(-\varepsilon_0,\varepsilon_0)$.

Reducing $\bar{\varepsilon}\in(0,\varepsilon_0]$ further,
we can choose $\sigma_2\in(0,\theta^{\rm d} - \theta^{\rm s})$ depending only on $(\gamma,v_2,\bar{\varepsilon})$ such that,
for any admissible solution $\varphi$ corresponding to
$\btheta\in \Theta \cap \{0 \leq \theta_1 \leq \theta^{\rm s}+\sigma_2\}$,
$\Omega^5_{\bar{\varepsilon}}$ defined by~\eqref{Eq:Def4Omega5epsilon} is nonempty
due to the monotonicity of the pseudo-Mach number in Lemma~\ref{lem:MonotonicityOfMach2}, from which we can generalize the definition of $\Omega^5_{\bar{\varepsilon}}$ to
\begin{equation}\label{Def4Omega5epsilon-New}
 \Omega^5_{\bar{\varepsilon}}
 \defeq \Omega\cap\mathcal{N}_{\varepsilon_1}(\Gamma^{5}_{\rm sonic})\cap
 \rcal^{-1} \big(
 \{ (x,y) \,:\, 0 < \hat{x} < \bar{\varepsilon}\}
 \big)\,.
\end{equation}
There exist $L\geq1$, $\delta\in(0,\frac12)$,
and $\omega\in(0,\omega_0)$
depending only on $(\gamma,v_2)$ such that,
for all $(x,y)\in \rcal(\Omega^5_{\bar{\varepsilon}})$,
\begin{equation}\label{Eq:PrepLemASP05-Case3}
0 \leq \psi(x,y) \leq Lx^2\,, \qquad
0\leq\psi_x\leq\frac{2-\delta}{1+\gamma}x\leq Lx\,,
\qquad
|\psi_y(x,y)|\leq Lx\,,
\end{equation}
for \(\psi \defeq (\varphi-\varphi_5) \circ \rcal^{-1}\),
and there exists a unique function ${f}_{5,{\rm sh}}\in C^1([0,\bar{\varepsilon}])$ such that
\begin{align*}
& \rcal(\Omega^5_{\bar{\varepsilon}}) =
\{(x,y) \,:\, 0 < \hat{x} < \bar{\varepsilon}, \,
 0 < y < {f}_{5,{\rm sh}}(\hat{x})\}\,,\\
& \rcal(\Gamma_{\rm shock}\cap\partial\Omega^5_{\bar{\varepsilon}})=
\{ (x,y) \,:\, 0 < \hat{x} < \bar{\varepsilon}, \, y={f}_{5,{\rm sh}}(\hat{x})\}\,,
\end{align*}
and $\omega\leq {f}'_{5,{\rm sh}}(\hat{x})\leq L$ for any $\hat{x}\in(0,\bar{\varepsilon})$.

\smallskip
\textbf{2.} \textit{The free boundary condition and the modification on equation~\eqref{QuasilinearEq4Theta1}.}
Similar to~\cite[Lemma~3.37 and Corollary~3.38]{BCF-2019},
it follows that, for $\sigma_2\in(0,\theta^{\rm d}-\theta^{\rm s})$ defined in Step~\textbf{1},
there exists a constant $\mu_0>0$ depending only on $(\gamma,v_2,\sigma_2)$ such that,
for any $\btheta\in\Theta\cap \{\theta^{\rm s} \leq \theta_1 \leq \theta^{\rm d}-\sigma_2\}$,
$g^{\rm sh}_{\rm mod}$ defined by~\eqref{FBFunc:gshmod} satisfies the following properties:
\begin{equation*}
\partial_{a} g^{\rm sh}_{\rm mod}(D\varphi_5(P_2),\varphi_5(P_2),P_2)\leq-\mu_0\,,
\end{equation*}
where $\partial_a = \partial_{p_1}, \partial_{p_2}$, or $\partial_z$.
From these properties above, we can show that $\mathcal{B}^{(\theta_1)}$ defined by~\eqref{DefBTheta1}
satisfies the same properties as in Step~\textbf{2} of {\rm Case~1}.
The extension of the domain for the coefficients of equation~\eqref{QuasilinearEq4Theta1} is the same as Step~\textbf{3}
of \textit{\rm Case~1}.

\smallskip
\textbf{3.} \textit{Rescaling method in the local domain to obtain the weighted H{\"o}lder estimate.}
Note that, from the properties of ${f}_{5,0}$ and ${f}_{5,{\rm sh}}$ in Step~\textbf{1},
there exists $k>1$ depending only on $(\gamma,v_2)$ such that
\begin{equation*}
  \{0<\hat{x}<\bar{\varepsilon}, \, 0<y<\frac{\hat{x}}{k}\}
  \; \subseteq \;
  \rcal( \Omega^5_{\bar{\varepsilon}} )
  \; \subseteq \;
  \{0<\hat{x}<\bar{\varepsilon} , \, 0<y<k\hat{x}\}\,.
\end{equation*}
Moreover, there exist constants $\varepsilon\in(0,\frac{\bar{\varepsilon}}{2}], \sigma_3\in(0,\sigma_2]$,
and $C>1$ depending only on $(\gamma,v_2)$ such that,
for $\btheta \in \Theta \cap \{{\theta}^{\rm{s}} \leq \theta_1 \leq {\theta}^{\rm{s}}+\sigma_3\}$,
\begin{equation}\label{Claim4Esti1c}
x_{P_2}\leq \frac{\varepsilon}{10}\,, \qquad\,
0\leq\psi(x,y)\leq C (x-x_{P_2})^5 \quad \text{in $\rcal(\Omega^5_{2\varepsilon})$}\,.
\end{equation}

Then, for any $z_0=(x_0,y_0)\in \rcal( \overline{\Omega^5_{\varepsilon}}\setminus\{P_2\})$,
we introduce the rescaled coordinates $(S,T)$ as $(x,y) \eqdef z_0+\frac{x_0-x_{P_2}}{10k}(\sqrt{x_0}S,T)$,
taking values in a local domain $Q_r^{(z_0)}$ defined by
\begin{equation*}
Q_r^{(z_0)}\defeq
\big\{(S,T)\in(-1,1)^2 \,:\,  z=z_0+\frac{r(x_0-x_{P_2})}{10k}(\sqrt{x_0}S,T)\in \rcal( \Omega_{2\varepsilon}^{5})\big\}
\qquad \text{for any $r\in(0,1]$}\,.
\end{equation*}
Then the local domain $Q_r^{(z_0)}$ fits into $\Omega_{2\varepsilon}^{5}$, meaning that the image of $Q_r^{(z_0)}$
under transform $(S,T)\mapsto (x,y)$ does not intersect both \(\rcal(\Gamma_{\rm shock})\) and \(\rcal(\Gamma_{\rm sym})\) simultaneously.
Considering property~\eqref{Claim4Esti1c}, we define a rescaled function $\psi^{(z_0)}$ by
\begin{equation*}
\psi^{(z_0)}(S,T)\defeq\frac{1}{(x_0-x_{P_2})^5}\psi(x_0+\frac{x_0-x_{P_2}}{10k}\sqrt{x_0}S,\,y_0+\frac{x_0-x_{P_2}}{10k}T)
\qquad \text{for $(S,T)\in Q_1^{(z_0)}$}.
\end{equation*}
The rescaled function $\psi^{(z_0)}$ satisfies a uniformly elliptic equation in the $(S,T)$--coordinates
after a suitable rescaling of the coefficients of the modified equation~\eqref{ModQuasilinearEq4Theta1}.
The rest of the proof follows
the same idea as Step~\textbf{4} of {\rm Case~1}.

\subsubsection*{{\rm Case~4}.
Admissible solutions for ${\theta}_{1} \geq \theta^{\rm s}$ away from $\theta^{\rm s}$.} \phantom{abcs}

Let $\sigma_3>0$ be from Proposition~\ref{Prop:Esti4theta1c}.
By Proposition~\ref{prop:ellipticDegeneracyNearSonicBoundary}, there exists a constant  $\delta\in(0,1)$ depending
only on $(\gamma,v_2)$ such that any admissible solution $\varphi$ corresponding to
$\btheta\in\Theta \cap \{\theta^{\rm s}+\frac{\sigma_3}{2} \leq \theta_1 < \theta^{\rm d}\}$ satisfies
\begin{equation}\label{Case4UniformSubsonicConst}
\frac{|D\varphi|^2}{c^2(|D\varphi|,\varphi)}\leq1-\delta \qquad\, \text{in $\overline{\Omega}\cap\{\xi\geq0\}$}\,,
\end{equation}
where $c(|D\varphi|,\varphi)$ is defined by~\eqref{Def4SonicSpeedInSec3}.
By Lemma~\ref{UniformBound-Lem} and~\eqref{Case4UniformSubsonicConst},
$(D\varphi(\bm{\xi}),\varphi(\bm{\xi}))\in\mathcal{K}_{R_*}$ for some constant $R_*\geq2$
depending only on $(\gamma,v_2)$,
where $\mathcal{K}_{R}$ is defined by~\eqref{KRSet}.
In particular, there exist $\lambda_*>0$ and $r_*>0$, depending only on $(\gamma,v_2)$,
such that any admissible solution $\varphi$ corresponding to
$\btheta\in\Theta \cap \{\theta^{\rm s}+\frac{\sigma_3}{2} \leq \theta_1 < \theta^{\rm d}\}$ satisfies
\begin{equation}\label{EllipticityConstCase4}
\lambda_*|\bm{\kappa}|^2\leq \sum_{i,j=1}^2\partial_{p_j}\mathcal{A}_{i}(D\varphi(\bm{\xi}),\varphi(\bm{\xi}))\kappa_i\kappa_j
\leq \lambda^{-1}_*|\bm{\kappa}|^2
\end{equation}
for any $\bm{\xi}\in\overline{\Omega}\cap B_{2r_*}(P_0^1)$ and any $\bm{\kappa}=(\kappa_1,\kappa_2)\in\mathbb{R}^2$.

Set $\bar{\phi}(\bm{\xi})\defeq \varphi_2(\bm{\xi})-\varphi(\bm{\xi})$.
Recalling equation~\eqref{Eq4phi} for function $\phi=\varphi-\varphi_5$, the free boundary condition~\eqref{Def4FreeBdryFunc}
on $\Gamma_{\rm shock}$, and the slip boundary condition
defined in Problem~\ref{FBP}\eqref{item5-FBP} on $\Gamma_{\rm sym}$,
it follows that $\bar{\phi}$ satisfies the system:
\begin{equation}\label{MixedBrdyCondisSystem}
\begin{cases}
\, \sum\limits_{i,j=1}^2 \bar{A}_{ij}(D\bar{\phi},\bar{\phi},\bm{\xi}) \partial_{i} \partial_{j} \bar{\phi}=0
  &\quad \text{in} \; \Omega\cap B_{r_*}(P_0^1) \,,\\
\, \bar{g}^{\rm sym}(D\bar{\phi},\bar{\phi},\bm{\xi})=0 &\quad \text{on} \; \Gamma_{\rm sym} \,,\\
\, \bar{g}^{\rm sh}(D\bar{\phi},\bar{\phi},\bm{\xi})=0 &\quad \text{on} \; \Gamma_{\rm shock} \,,
\end{cases}
\end{equation}
with coefficients $\bar{A}_{ij}$ and the boundary functions $\bar{g}^{\rm sym}$ and $\bar{g}^{\rm sh}$ defined by
\begin{equation}\label{Def4CoeffiBdryFuncs}
\begin{aligned}
& \bar{A}_{ij}(\bm{p},z,\bm{\xi})\defeq c^2(|D\varphi_2-\bm{p}|,\varphi_2-z)\delta_{ij}-(\partial_i\varphi_2-p_i)(\partial_j\varphi_2-p_j) \,,\\
&  \bar{g}^{\rm sym}(\bm{p},z,\bm{\xi})\defeq  v_2-p_2-\eta \,,\\
& \bar{g}^{\rm sh}(\bm{p},z,\bm{\xi})\defeq  -g^{\rm sh}(D\varphi_2-\bm{p},\varphi_2-z,\bm{\xi})\,,
\end{aligned}
\end{equation}
for $g^{\rm sh}$ given by~\eqref{Def4FreeBdryFunc}, where $(\partial_1,\partial_2)=(\partial_{\xi},\partial_{\eta})$.

We define a vector $\bm{e}$ as
\begin{equation*}
\bm{e}\defeq\frac{D\varphi_2(P_0^1)}{|D\varphi_2(P_0^1)|}=\frac{O_2-P_0^1}{|O_2-P_0^1|}\,,
\end{equation*}
and choose $\bm{e}^{\perp}$ as the clockwise rotation of $\bm{e}$ by $\frac{\pi}{2}$.
Note that $\bm{e}$ depends only on $(\gamma,v_2,\theta_1)$ and satisfies that,
for any $\bm{\xi}\in\Gamma_{\rm shock}\cap B_{r_*}(P_0^1)$,
\begin{equation}\label{NonDegeneracyPropety}
\partial_{\bm{e}}(\varphi_2-\varphi)(\bm{\xi})=-\partial_r(\varphi_2-\varphi)
+\big(\bm{e}-\frac{D\varphi(\bm{\xi})}{|D\varphi(\bm{\xi})|}\big)\cdot D(\varphi_2-\varphi)\geq\frac{d'_0}{2}\,,
\end{equation}
where $d'_0>0$ is the constant from~\eqref{PositiveDifferenceDeriv2r}, depending only on $(\gamma,v_2)$.
Then we define the new coordinates $(S,T)$ by
\begin{equation}\label{Def:LocalCoordiForCase4}
\bm{\xi} \eqdef P_0^1+ S\bm{e}+T\bm{e}^{\perp}\,.
\end{equation}
Under these $(S,T)$--coordinates,
we can express $S_{25}, \Gamma_{\rm shock}, \Gamma_{\rm sym}$, and the domain near $P_0^1$ as
\begin{equation} \label{NewnotationsUnderST}
\begin{aligned}
&S_{25}\cap B_{r_*}(P_0^1)=\{\bm{\xi}(S,T) \,:\, S=a_{\rm ref}T, \, T>0, \, (S,T)\in B_{r_*}(\bm{0}) \}\,,\\
&\Gamma_{\rm shock}\cap B_{r_*}(P_0^1)=\{\bm{\xi}(S,T) \,:\, S=f_{\bm{e}}(T), \, T>0, \,  (S,T)\in B_{r_*}(\bm{0}) \}\,,\\
&\Gamma_{\rm sym}\cap B_{r_*}(P_0^1)=\{\bm{\xi}(S,T) \,:\, S=a_{\rm sym}T, \, T>0, \, (S,T) \in B_{r_*}(\bm{0}) \}\,,\\
&\Omega \cap B_{r_*}(P_0^1)=\{\bm{\xi}(S,T) \,:\, a_{\rm ref}T \leq f_{\bm{e}}(T) < S < a_{\rm sym}T, \, T>0, \,  (S,T) \in B_{r_*}(\bm{0}) \}\,,
\end{aligned} \end{equation}
where the positive constants $a_{\rm ref} = \tan({\theta_{25} - \hat{\theta}_{25} - \frac\pi2})$
and $a_{\rm sym} = \cot\hat{\theta}_{25}$ depend continuously
on $\theta_1\in (0,\theta^{\rm d})$,
and there exists a constant $C>0$ depending only on $(\gamma,v_2)$
such that $C^{-1} \leq a_{\rm ref} < a_{\rm sym} \leq C$ for all $\theta_1\in[\theta^{\rm s},\theta^{\rm d})$.
Note that $\hat{\theta}_{25}$ is defined by~\eqref{eq:state5SelfSimilarTurningAngle},
and we have used Proposition~\ref{Prop:localtheorystate5}, including~\eqref{eq:simple-reflection-angle-estimate}.

\begin{definition}[Weighted H{\"o}lder norm]
\label{Def:WeightedHolder-Sec3}
For any $\sigma\in\mathbb{R},$ $\alpha\in(0,1),$ and $m\in\mathbb{N},$
the weighted H{\"o}lder norms are defined as follows{\rm :}
\begin{enumerate}[{\rm (i)}]
\item \label{Def:WeightedHolder-Sec3-item1}
For any open bounded connected set $U\subseteq\mathbb{R}^2,$ let $\Gamma$ be a closed portion of $\partial U$.
Write
\begin{equation*}
\delta_{\bm{\xi}_1} \defeq {\rm dist}(\bm{\xi}_1,\Gamma)\,, \quad\, \delta_{\bm{\xi}_1,\bm{\xi}_2} \defeq \min\{\delta_{\bm{\xi}_1},\delta_{\bm{\xi}_2}\} \qquad\,
\text{for any $\bm{\xi}_1,\bm{\xi}_2\in U$}\,.
\end{equation*}
For any $u\in C^{m}(U),$ define
\begin{equation*}
\begin{aligned}
&\norm{u}^{(\sigma),\Gamma}_{m,0,U} \defeq \sum_{0\leq|\bm{\beta}|\leq m}\sup_{\bm{\xi}_1\in U}
\Big(\delta_{\bm{\xi}_1}^{\max\{|\bm{\beta}|+\sigma,0\}}\big|D^{\bm{\beta}}u(\bm{\xi}_1)\big|\Big) \,,\\
&[u]^{(\sigma),\Gamma}_{m,\alpha,U} \defeq \sum_{|\bm{\beta}|=m}
\sup_{\substack{ \bm{\xi}_1,\bm{\xi}_2\in U \\ \bm{\xi}_1\neq\bm{\xi}_2}}  \Big(\delta_{\bm{\xi}_1,\bm{\xi}_2}^{\max\{m+\alpha+\sigma,0\}}
\frac{|D^{\bm{\beta}}u(\bm{\xi}_1)-D^{\bm{\beta}}u(\bm{\xi}_2)|}{|\bm{\xi}_1-\bm{\xi}_2|^{\alpha}}\Big)\,,\\
&\norm{u}^{(\sigma),\Gamma}_{m,\alpha,U} \defeq \norm{u}^{(\sigma),\Gamma}_{m,0,U} + [u]^{(\sigma),\Gamma}_{m,\alpha,U}\,,
\end{aligned}\end{equation*}
where $D^{\bm{\beta}} \defeq \partial_{\xi}^{\beta_1}\partial_{\eta}^{\beta_2}$
for $\bm{\beta}\defeq(\beta_1,\beta_2)$ with $\beta_1,\beta_2\in\mathbb{N}$
and $|\bm{\beta}|=\beta_1+\beta_2$.

\item \label{Def:WeightedHolder-Sec3-item2}
For any open bounded interval $I\subseteq\mathbb{R},$ let $x_0\in \partial I$ be an endpoint of $I$.
For any $f\in C^{m}(I),$ define
\begin{equation*}
\begin{aligned}
& \norm{f}^{(\sigma),\{x_0\}}_{m,0,I} \defeq \sum_{k=0}^{m}\sup_{x\in I}\Big(|x-x_0|^{\max\{k+\sigma,0\}}\big|f^{(k)}(x)\big|\Big)\,,\\
& [f]^{(\sigma),\{x_0\}}_{m,\alpha,I} \defeq \sup_{\substack{ x_1,x_2\in I \\ x_1\neq x_2}}  \Big(\min\big\{|x_1-x_0|,|x_2-x_0|\big\}^{\max\{m+\alpha+\sigma,0\}} \frac{|f^{(m)}(x_1)-f^{(m)}(x_2)|}{|x_1-x_2|^{\alpha}}\Big)\,,\\
&\norm{f}^{(\sigma),\{x_0\}}_{m,\alpha,I} \defeq \norm{f}^{(\sigma),\{x_0\}}_{m,0,I} + [f]^{(\sigma),\{x_0\}}_{m,\alpha,I}\,.
\end{aligned}\end{equation*}
\end{enumerate}
\end{definition}

\begin{proposition}
\label{Prop:Esti4theta1d}
Fix $\gamma\geq1$ and $v_2 \in (v_2^{\rm s},0)$.
Let $\sigma_3>0$ be the constant from {\rm Proposition~\ref{Prop:Esti4theta1c}}.
For small constants $\sigma_{\rm{s}} \in(0,\frac{\sigma_3}{2}]$
and $\sigma_{\rm{d}}\in(0,\frac{\theta^{\rm d}}{10}),$
there exist $r \in(0,r_*),$ $\alpha \in(0,1),$ and $C>0$ depending only
on $(\gamma,v_2,\sigma_{\rm{s}},\sigma_{\rm{d}})$ such that any admissible solution $\varphi$
corresponding to
$\btheta\in\Theta \cap\{{\theta}^{\rm{s}}+\sigma_{\rm{s}} \leq \theta_1 \leq {\theta}^{\rm{d}}-\sigma_{\rm{d}}\}$ satisfies the estimates{\rm :}
\begin{flalign}\label{ResultCase4}
\big\|\varphi\big\|_{2, \alpha, \Omega\cap B_r(P_0^1)}^{(-1-\alpha),\{P_0^1\}}
+\big\|f_{\bm{e}}\big\|_{2, \alpha,(0,r)}^{(-1-\alpha),\{0\}} \leq C\,, \qquad
|D^m_{\bm{\xi}}(\varphi-\varphi_5)(P_0^1)|=0 \,\,\,\, \text{for $m=0,1$}\,.
\end{flalign}
\end{proposition}

The proof is similar to the proof of~\cite[Proposition~3.42]{BCF-2019}.
We omit the details of the proof and sketch only the main ideas here.

\smallskip
\textbf{1}. \textit{Hodograph transform.} Define \(\bar{\phi}^{(h)}(S,T) \defeq \bar{\phi}(\bm{\xi}(S,T))\)
for function $\bar{\phi} = \varphi_2 - \varphi$ and the \((S,T)\)--coordinates given by~\eqref{Def:LocalCoordiForCase4}.
By property~\eqref{NonDegeneracyPropety}, we can apply the hodograph transform
$\bm{y}=(y_1,y_2)=(\bar{\phi}^{(h)}(S,T),T)$ to $\bar{\phi}$ for the original system~\eqref{MixedBrdyCondisSystem}.
We write $S \eqdef v(y_1,y_2)$ if it satisfies $\bar{\phi}^{(h)}(S,y_2)=y_1$ for any
\begin{equation*}
\bm{y}\in\mathcal{D}^{(\theta_1)}_{r_*}\defeq \big\{ (\bar{\phi}^{(h)}(S,T),T) \,:\,  \bm{\xi}(S,T)\in \Omega\cap B_{r_*}(P_0^1)\big\}\,,
\end{equation*}
from which we obtain
\begin{equation*}
\partial_{y_1}v=\frac{1}{\partial_{S} \bar{\phi}^{(h)}}\,, \qquad
\quad \partial_{y_2}v=-\frac{\partial_{T}\bar{\phi}^{(h)}}{\partial_{S}\bar{\phi}^{(h)}}\,.
\end{equation*}
It follows from Lemma~\ref{UniformBound-Lem} and property~\eqref{NonDegeneracyPropety} that
there exists a constant $K>1$ depending only on $(\gamma,v_2)$ such that
\begin{equation}\label{Eq:UniformBound4v}
\frac{1}{K}\leq\partial_{y_1}v\leq\frac{2}{d_0'}\, , \qquad |v|+|D_{\bm{y}}v|\leq2K\,.
\end{equation}
Under the hodograph transform, system~\eqref{MixedBrdyCondisSystem} is equivalent to
\begin{equation}\label{Eqs:TransfMixedBrdyCondisSystem}
\begin{cases}
\, \sum\limits_{i,j=1}^2 a_{ij}(D_{\bm{y}}v,v,\bm{y}) \partial_{y_iy_j} v=0 &\quad \text{in}\; \mathcal{D}^{(\theta_1)}_{r_*}\,,\\
\, g_h^{\rm sym}(D_{\bm{y}}v,v,\bm{y})=0 &\quad \text{on}\; \Gamma^{(h)}_{\rm sym}\,,\\
\, g_h^{\rm sh}(D_{\bm{y}}v,v,\bm{y})=0 &\quad \text{on}\; \Gamma^{(h)}_{\rm shock}\,,
\end{cases}
\end{equation}
for suitable functions $(a_{ij},g_h^{\rm sym},g_h^{\rm sh})$, with the boundary sets given by
\begin{equation*}
    \begin{aligned}
    & \Gamma^{(h)}_{\rm sym}\defeq \big\{(\bar{\phi}^{(h)}(S,T),T) \,:\, \bm{\xi} (S,T)\in\Gamma_{\rm sym}\cap B_{r_*}(P_0^1)\big\} \,,\\
    & \Gamma^{(h)}_{\rm shock}\defeq \big\{ (0,T) \,:\, \bm{\xi}(S,T)\in\Gamma_{\rm shock}\cap B_{r_*}(P_0^1)\big\}\,.
    \end{aligned}
\end{equation*}
Due to~\eqref{Eq:UniformBound4v}, we extend the domain of functions $(a_{ij},g_h^{\rm sym},g_h^{\rm sh})$ to the set:
\begin{equation*}
U\defeq \big\{(\bm{p},z,\bm{y})\in\mathbb{R}^2\times\mathbb{R}\times\mathcal{D}^{(\theta_1)}_{r_*} \,:\,
|\bm{p}|+|z|\leq 2K\big\}\,,
\end{equation*}
and continue to use the same notation as in~\eqref{Eqs:TransfMixedBrdyCondisSystem} for this modified system.

\smallskip
\textbf{2}. \textit{H{\"o}lder continuity of ${g}_h^{\rm sym}$ at the corner.}
From Lemma~\ref{UniformBound-Lem}, there exists a constant $\tau_0>0$ depending only on $(\gamma,v_2)$ such that
\begin{equation*}
   \Gamma^{(h)}_{\rm sym}\cup \Gamma^{(h)}_{\rm shock} \subseteq \{\bm{y}\in\mathbb{R}^2 \,:\, y_2>\tau_0|y_1|\}\,.
\end{equation*}
Moreover, $\Gamma^{(h)}_{\rm shock}$ is in $C^{2}$ up to endpoint $\{\bm{0}\}$.

From the definition of admissible solutions and the hodograph transform,
we obtain that
$v\in C^1(\overline{\mathcal{D}^{(\theta_1)}_{r_*}})\cap C^2(\mathcal{D}^{(\theta_1)}_{r_*}\cup\Gamma^{(h)}_{\rm shock})\cap C^3(\mathcal{D}^{(\theta_1)}_{r_*})$
satisfies~\eqref{Eq:UniformBound4v}.
From~\eqref{Def4CoeffiBdryFuncs} and the notation defined in Step~\textbf{1},
we can obtain the boundedness of coefficients $({a}_{ij},{g}_h^{\rm sym},{g}_h^{\rm sh})$ in the modified system~\eqref{Eqs:TransfMixedBrdyCondisSystem}.
For any $\btheta\in \Theta\cap\{\theta^{\rm s}+\sigma_{\rm s}\leq \theta_1 <\theta^{\rm d}\}$
with $\sigma_{\rm s}\in(0,\frac{\sigma_3}{2}]$, we may reduce $\lambda_*>0$ depending only on $(\gamma,v_2,\sigma_{\rm s})$
 in~\eqref{EllipticityConstCase4} to obtain the uniform ellipticity of the modified system~\eqref{Eqs:TransfMixedBrdyCondisSystem} in $\mathcal{D}^{(\theta_1)}_{r_*}$
by using~\eqref{Eq:UniformBound4v} and the hodograph transform.
It follows from Lemma~\ref{UniformBound-Lem}, \eqref{EllipticityConstCase4}, and~\eqref{Eq:UniformBound4v}
that there exists a constant $\lambda_1>0$ depending only on $(\gamma,v_2,\sigma_{\rm s})$ such that,
for any $y\in \Gamma^{(h)}_{\rm shock}$,
\begin{equation*}\begin{aligned}
\bm{\nu}\cdot D_{\bm{p}} g_{h}^{\rm sh}(D_{\bm{y}}v,v,\bm{y})
=\frac{1}{v_{y_1}}D\bar{\phi} \cdot D_{\bm{q}}g^{\rm sh}(D\varphi,\varphi,\bm{\xi}) \geq \lambda_1\,,
\end{aligned}\end{equation*}
where we have used the expression of $D_{\bm{q}}g^{\rm sh}(D\varphi,\varphi,\bm{\xi})$
on $\Gamma^{(h)}_{\rm shock}$, and $\bm{\nu}=(1,0)$ is the unit normal vector to $\Gamma^{(h)}_{\rm shock}$.
From~\eqref{Eq:UniformBound4v}, $|D_{\bm{p}} {g}_{h}^{\rm sym}(D_{\bm{y}}v,v,\bm{y})| \geq\frac{d_0'}{2}>0$.

Similar to the proof of~\cite[Lemma 3.40]{BCF-2019},
one can show that, for small $\sigma_{\rm{d}}\in(0,\frac{\theta^{\rm d}}{10})$,
there exist constants $M>0$ and $\bar{r}\in(0,r_*)$ depending only on $(\gamma,v_2,\sigma_{\rm{d}})$ such that
any admissible solution $\varphi$ corresponding to
$\btheta\in\Theta\cap\{\theta^{\rm s} \leq \theta_1 < \theta^{\rm d}-\sigma_{\rm{d}}\}$ satisfies
\begin{equation*}\label{Eq:Case4PreLem3:40}
\partial_{q_1}g_{\rm mod}^{\rm sh}(D\varphi,\varphi,\bm{\xi})\leq- M^{-1}  \qquad
\text{for all $\bm{\xi}\in\Gamma_{\rm shock}\cap B_{\bar{r}}(P_0^1)$}\,.
\end{equation*}
Then, for any $\bm{y}\in\Gamma^{(h)}_{\rm shock}$,
$|{\rm det}D_{\bm{p}}( {g}_{h}^{\rm sym},{g}_{h}^{\rm sh})|>0$ holds.

It follows from~\cite[Proposition 4.3.7]{ChenFeldman-RM2018} that
there exist constants $\alpha_1\in(0,1)$, $C>0$,
and $r_0\in(0,\bar{r}]$
depending only on $(\gamma,v_2,\sigma_{\rm s},\sigma_{\rm d})$ such that,
for any $\bm{y}\in\overline{\mathcal{D}^{(\theta_1)}_{\bar{r}}\cap B_{r_0}(\bm{0})}$,
\begin{equation}\label{Eq:ResultInStep2Case4}
\big|{g}^{\rm sym}_{h}(D_{\bm{y}}v(\bm{y}),v(\bm{y}),\bm{y}) - {g}^{\rm sym}_{h}(D_{\bm{y}}v(\bm{0}),v(\bm{0}),\bm{0})\big|
\leq C |\bm{y}|^{\alpha_1}\,.
\end{equation}

\smallskip
\textbf{3}. \textit{H{\"o}lder continuity of $v(\bm{y})$ at the corner.}
Notice that both ${g}_{h}^{\rm sym}$ and ${g}_{h}^{\rm sh}$ are Lipschitz continuous on $(p,z,\bm{y})$.
Since $v\in C^1(\overline{\mathcal{D}^{(\theta_1)}_{\bar{r}}})$,
inequality~\eqref{Eq:ResultInStep2Case4} also holds for any  $\bm{y}\in\Gamma^{(h)}_{\rm shock}$ after replacing ${g}_{h}^{\rm sym}$
by ${g}_{h}^{\rm sh}$.
Moreover, functions $({g}_{h}^{\rm sym},{g}_{h}^{\rm sh})(\cdot,v(\bm{0}),\bm{0})$
are bounded in  $C^{1,\alpha_1}(\overline{B_{2K}(D_{\bm{y}}v(\bm{0}))})$,
and $\abs{{\rm det} D_{\bm{p}} ({g}_{h}^{\rm sym},{g}_{h}^{\rm sh}) (D_{\bm{y}}v(\bm{0}),v(\bm{0}),\bm{0})}>0$ follows from Step~\textbf{2}.
Then, from~\cite[Proposition~4.3.9]{ChenFeldman-RM2018}, there exist $r_1\in(0,r_0)$ and $C_1>0$
depending only on $(\gamma,v_2,\sigma_{\rm s},\sigma_{\rm d})$ such that
\begin{equation}\label{Eq:ResultInStep3Case4}
|D_{\bm{y}}v(\bm{y})-D_{\bm{y}}v(\bm{0})|\leq C_1|\bm{y}|^{\alpha_1} \qquad
\text{for any $\bm{y}\in\Gamma^{(h)}_{\rm shock}\cap B_{r_1}(\bm{0})$}\,.
\end{equation}
By the hodograph transform and the geometric properties of the domain near corner $P_0^1$,
it follows from~\eqref{Eq:ResultInStep2Case4}--\eqref{Eq:ResultInStep3Case4} that
\begin{equation}
\label{Eq:HolderPropertyOnBdryD2}
\begin{aligned}
&|\partial_\eta \varphi (\bm{\xi}) - \partial_\eta \varphi (P_0^1)|\leq C|\bm{\xi}-P_0^1|^{\alpha_1}
\quad &&\text{for any $\bm{\xi}\in \overline{\Omega\cap B_{r_1}(P_0^1)}$} \,,\\
&|D\varphi(\bm{\xi})-D\varphi(P_0^1)|\leq C|\bm{\xi}-P_0^1|^{\alpha_1}
\quad &&\text{for any $\bm{\xi}\in \Gamma_{\rm shock}\cap B_{r_1}(P_0^1)$}\,.
\end{aligned}\end{equation}

The open domain $\Omega_{r_1}\defeq \Omega\cap B_{r_1}(P_0^1)$ contains two Lipschitz boundaries
$\Gamma^1\defeq \Gamma_{\rm shock}\cap B_{r_1}(P_0^1)$
and $\Gamma^2\defeq \Gamma_{\rm sym}\cap B_{r_1}(P_0^1)$ intersecting at corner $P_0^1$.
We also take
\begin{equation*}
B^{(1)}(\bm{p},z,\bm{\xi})\defeq (1,0)\cdot(\bm{p}-D\bar{\phi}(P_0^1))\,, \qquad
B^{(2)}(\bm{p},z,\bm{\xi})\defeq \bar{g}_{\rm sym}(\bm{p},z,\bm{\xi})\,.
\end{equation*}
Now we update system~\eqref{MixedBrdyCondisSystem}
for $\bar{\phi}\in C^1(\overline{\Omega_{r_1}})\cap C^2(\Omega_{r_1}\cup\Gamma^{1})\cap C^3(\Omega_{r_1})$
into domain $\Omega_{r_1}$,  replacing the boundary conditions by
\begin{equation*}
 B^{(1)}(D\bar{\phi},\bar{\phi},\bm{\xi}) = \bar{h}(\bm{\xi}) \,\,\,\, \text{on} \; \Gamma^1\,,
 \qquad\,\,
\, B^{(2)}(D\bar{\phi},\bar{\phi},\bm{\xi})=0 \,\,\,\, \text{on} \; \Gamma^2\,,
\end{equation*}
where $\bar{h}(\bm{\xi})\defeq (1,0)\cdot(D\bar{\phi}(\bm{\xi})-D\bar{\phi}(P_0^1))$, which satisfies \begin{equation*}
|\bar{h}(\bm{\xi})-\bar{h}(P_0^1)|\leq  C|\bm{\xi}-P_0^1|^{\alpha_1}
\qquad \text{for any $\bm{\xi}\in\Gamma_{\rm shock}\cap B_{r_1}(P_0^1)$}\,.
\end{equation*}
We can also extend the domain of functions $(\bar{A}_{ij},B^{(1)},B^{(2)})$ for the updated system to the set:
\begin{equation*}
V\defeq \{(\bm{p},z,\bm{\xi})\in\mathbb{R}^2\times\mathbb{R}\times\overline{\Omega_{r_1}} \,:\, |\bm{p}|+|z|\leq 2K_1\}\,.
\end{equation*}
Using~\cite[Proposition 4.3.7]{ChenFeldman-RM2018}
again, we obtain that there exist $\alpha\in(0,\alpha_1]$, $r_2\in(0,r_1]$, and $C>0$
depending only on $(\gamma,v_2,\sigma_{\rm s},\sigma_{\rm d})$ such that,
for any $\bm{\xi}\in\overline{\Omega\cap B_{r_2}(P_0^1)}$,
\begin{equation*}
\big|B^{(1)}(D\bar{\phi}(\bm{\xi}),\bar{\phi}(\bm{\xi}),\bm{\xi})-B^{(1)}(D\bar{\phi}(P_0^1),\bar{\phi}(P_0^1),P_0^1)\big|
\leq C\big|\bm{\xi}-P_0^1\big|^{\alpha}\,,
\end{equation*}
which, combined with~\eqref{Eq:HolderPropertyOnBdryD2},
implies that
\begin{equation}\label{Eq:Result4Step3Case4}
\big|D\varphi(\bm{\xi})-D\varphi(P_0^1)\big|\leq C\big|\bm{\xi}-P_0^1\big|^{\alpha}
\qquad \text{for any $\bm{\xi}\in\overline{\Omega\cap B_{r_2}(P_0^1)}$}\,.
\end{equation}

\textbf{4}. \textit{Weighted H{\"o}lder estimate near corner $P_0^1$.}
From Lemma~\ref{Prop:GammaShockExpres},
Remark~\ref{Rmk4AdmissibleSolu}\eqref{item0-Rmk4AdmissibleSolu},
and~\eqref{Eq:Result4Step3Case4},
there exists a constant $\varepsilon_{\rm bd}>0$ depending only on $(\gamma,v_2)$ such that,
for any $\bm{\xi}_0\in \Gamma_{\rm bd}\cap B_{r_3}(P_0^1)$ for some $r_3\in(0,r_2]$
with either $\Gamma_{\rm bd}\defeq \Gamma_{\rm shock}$ or $\Gamma_{\rm bd}\defeq\Gamma_{\rm sym}$,
\begin{equation*}
   \big\{ \bm{\xi} \in\partial\Omega \,:\, |\bm{\xi}-\bm{\xi}_0|<\varepsilon_{\rm bd}|\bm{\xi}_0-P_0^1| \big\}
   \subseteq \Gamma_{\rm bd}\,.
\end{equation*}

We now consider the following problem
for $\varphi\in C^{1}(\overline{\Omega\cap B_{r_3}(P_0^1)})\cap C^3({\Omega\cap B_{r_3}(P_0^1)})$:
\begin{equation*}\begin{cases}
\, \mbox{div}\tilde{\mathcal{A}}(D\varphi,\varphi)+\tilde{\mathcal{B}}(D\varphi,\varphi) = 0
&\quad \text{in}\;{\Omega\cap B_{r_3}(P_0^1)}\,,\\
\, g_{\rm mod}^{\rm sh}(D\varphi,\varphi,\bm{\xi})=0 &\quad \text{on} \; \Gamma_{\rm shock}\cap B_{r_3}(P_0^1)\,,\\
\, g_{\rm sym}(D\varphi,\varphi,\bm{\xi})=0 &\quad \text{on} \; \Gamma_{\rm sym}\cap B_{r_3}(P_0^1)\,,
\end{cases}
\end{equation*}
where $(\tilde{\mathcal{A}},\tilde{\mathcal{B}})(\bm{q},w)$ are defined on
$\mathbb{R}^2\times\mathbb{R}$
from Lemma~\ref{CoeffExtdLem}, $g_{\rm mod}^{\rm sh}(\bm{q},w,\bm{\xi})$ is defined
on $\mathbb{R}^2\times\mathbb{R}\times\overline{\Omega\cap B_{r_3}(P_0^1)}$ by~\eqref{FBFunc:gshmod},
and $g_{\rm sym}(\bm{q},w,\bm{\xi})=(0,1)\cdot \bm{q}$.
Combining the above with property~\eqref{Eq:Result4Step3Case4} and
using~\cite[Proposition 4.3.11(i)]{ChenFeldman-RM2018},
we obtain that there exist $\tilde{\alpha}\in(0,\alpha]$
and $C>0$ depending only on $(\gamma,v_2,\sigma_{\rm s},\sigma_{\rm d})$ such that
$\varphi\in C^{1,\tilde{\alpha}}(\Omega\cap B_{r}(P_0^1))$ satisfies
$\|\varphi\|_{1,\,\tilde{\alpha},\,\Omega \cap B_{r}(P_0^1)}\leq C$
and, for $\Gamma_{\rm shock}$ given by~\eqref{NewnotationsUnderST},
$\|f_{\bm{e}}\|_{1,\tilde{\alpha},(0,r)}\leq C$.
Finally, using~\cite[Proposition 4.3.11(ii)]{ChenFeldman-RM2018}, we obtain~\eqref{ResultCase4}.

\subsection{Compactness of the set of admissible solutions}
\label{Sec3-5:Compactness4AdmisSoluSet}
Fix $\gamma\geq 1$ and $v_2 \in (v_{\min},0)$.
For any $\theta_{\ast} \in (0,\theta^{\rm d})$,
the arguments in \S\ref{SubSec-EstimatesNearSonic5} are also valid
for the weighted $C^{2,\alpha}$--estimates near $\Gamma^6_{\rm sonic}$ with respect to $\theta_2\in[0,\theta^{\rm d})$.
According to all the {\it a priori} estimates obtained in Lemma~\ref{Lem:LocalEstimates4Interior},
Corollary~\ref{Corol:Esti-AwaySonic5n6}, and  Propositions~\ref{Prop:Esti4theta1a}--\ref{Prop:Esti4theta1c}
and~\ref{Prop:Esti4theta1d},  there exists a constant $\bar{\alpha}\in(0,1)$ depending only on $(\gamma, v_2,\theta_{\ast})$ such that the set:
\begin{equation*}
\Big\{ \, \|\varphi\|_{1,\bar{\alpha}, \overline{\Omega}}+\|f_{\rm sh}\|_{1,\bar{\alpha},[\xi^{P_3},\xi^{P_2}]} \,:\, \;
\parbox{14em}{$\varphi$ is an admissible solution \\ corresponding to $\btheta\in\Theta \cap [0,\theta_{\ast}]^2$ } \;\;
\Big\}
\end{equation*}
is bounded, where $\eta = f_{\rm{sh}}(\xi)$ is the expression of $\Gamma_{\rm shock}$ as a graph,
given by Lemma~\ref{Prop:GammaShockExpres}.
For each admissible solution, the corresponding pseudo-subsonic region $\Omega$ is a bounded domain
enclosed by $\Gamma^5_{\rm sonic}$, $\Gamma_{\rm shock}$, $\Gamma^6_{\rm sonic}$,
and $\Gamma_{\rm sym}$.
These four curves intersect only at $P_i$ for $i = 1,2,3,4$.
Note that $\Gamma^5_{\rm sonic}, O_5, P_1$, and $P_2$ depend continuously
on $\theta_1\in[0,\theta^{\rm d}]$, while $\Gamma^6_{\rm sonic}, O_6, P_3$, and $P_4$
depend continuously on $\theta_2\in[0,\theta^{\rm d}]$.
Combining the above results with \eqref{Eq:ResultInStep3} and Lemma~\ref{Lem:SeqResult},
we have the following compactness result;
the details of the proof can be found in~\cite[Proposition~11.6.1]{ChenFeldman-RM2018}.

\begin{lemma}\label{Lem:Compactness4AdmiSoluSet}
Fix $\gamma\geq1,$ $v_2\in(v_{\min},0),$ and $\theta_{\ast}\in(0,\theta^{\rm d})$.
Let $\big\{\btheta^{(j)}\big\}_{j\in\mathbb{N}} \subseteq \Theta \cap [0,\theta_{\ast}]^2$
be a sequence of parameters satisfying that
$\lim_{j\to\infty} \btheta^{(j)} = \btheta^{(\infty)}$
for some $\btheta^{(\infty)} \in [0,\theta_{\ast}]^2$.
For each $j\in\mathbb{N},$ let $\varphi^{(j)}$ be an admissible solution corresponding to
parameters $\btheta^{(j)},$
with pseudo-subsonic region $\Omega^{(j)}$
and curved transonic shock $\Gamma^{(j)}_{\rm shock}$.
Then there exists a subsequence $\big\{\varphi^{(j_k)}\big\}_{k\in\mathbb{N}}$ such that
the following properties hold{\rm :}
\begin{enumerate}[{\rm (i)}]
\item
\label{Lem:Compactness4AdmiSoluSet-item1}
$\big\{\varphi^{(j_k)}\big\}_{k\in\mathbb{N}}$ converges uniformly on any compact subset of $\overline{\mathbb{R}^2_+}$ to a function
$\varphi^{(\infty)}\in C^{0,1}_{\rm loc}\big(\overline{\mathbb{R}^2_+}\big),$
and the limit function $\varphi^{(\infty)}$ is an admissible solution corresponding
to $\btheta^{(\infty)}${\rm ;}
\item
\label{Lem:Compactness4AdmiSoluSet-item2}
$\Omega^{(j_k)} \to \Omega^{(\infty)}$ in the Hausdorff metric{\rm ;}
\item
\label{Lem:Compactness4AdmiSoluSet-item3}
If $\bm{\xi}^{(j_k)} \in \overline{\Omega^{(j_k)}}$ for each \(k \in \mathbb{N},\) and $\big\{ \bm{\xi}^{(j_k)}\big\}_{k\in\mathbb{N}}$
converges to $\bm{\xi}^{(\infty)}\in\overline{\Omega^{(\infty)}},$ then
\begin{equation*}
\varphi^{(j_k)}(\bm{\xi}^{(j_k)}) \to \varphi^{(\infty)}(\bm{\xi}^{(\infty)})\,,\quad
D\varphi^{(j_k)}(\bm{\xi}^{(j_k)}) \to D\varphi^{(\infty)}(\bm{\xi}^{(\infty)})
\qquad\,\,\mbox{as $k\to\infty$}\,,
\end{equation*}	
 where, in the case that $\bm{\xi}^{(j_k)}\in\Gamma^{(j_k)}_{\rm shock},$
 the derivative of $\varphi^{(j_k)}$ at $\bm{\xi}^{(j_k)}$ is defined as the limit of
 the derivative from the interior of $\Omega^{(j_k)}${\rm :}
\begin{equation*}
\qquad D\varphi^{(j_k)}(\bm{\xi}^{(j_k)}) \defeq \lim_{ \bm{\xi}\to\bm{\xi}^{(j_k)},\,\bm{\xi}\in\Omega^{(j_k)}} D\varphi^{(j_k)}(\bm{\xi})\,,
\end{equation*}
and, similarly, for $\bm{\xi}^{(\infty)}\in \Gamma^{(\infty)}_{\rm shock},$
\begin{equation*}
\qquad D\varphi^{(\infty)}(\bm{\xi}^{(\infty)}) \defeq \lim_{ \bm{\xi}\to\bm{\xi}^{(\infty)},\,\bm{\xi}\in\Omega^{(\infty)}} D\varphi^{(\infty)}(\bm{\xi})\,.
\end{equation*}
\end{enumerate}
\end{lemma}

\section{Iteration Method and Existence of Admissible Solutions}
\label{Sec:IterationMethod-Existence}
In this section, we carefully construct the iteration set and the iteration map based
on the uniform estimates obtained in Section~\ref{Sec:UniformEstiAdmiSolu}.
Then we apply the Leray-Schauder degree theorem to show the existence of admissible solutions.
The construction of the iteration set follows closely the process
described in~\cite[Chapter 4]{BCF-2019} with the main difference
that both sonic arcs \(\Gamma^5_{\rm sonic}\) and \(\Gamma^6_{\rm sonic}\) may
degenerate into single points.
In this section, we always fix
\(\gamma \geq 1\) and \(v_2 \in (v_{\min},0)\).

\medskip
\subsection{Mapping between the elliptic and standard domains}
\subsubsection{Mapping into the standard domain}
Define the standard domain to be the rectangle \(\qiter \defeq (-1,1)\times(0,1)\).
Before constructing a map between $\Omega$ and $\qiter$,
we introduce some useful notation and basic geometric properties.
Let \(\varphi_2\) be given by~\eqref{PseudoVelocity123},
and let \((\varphi_5,\varphi_6)\) be given by~\eqref{eq:def-weak-state-5-6}.
For any $\delta_0>0$, we define
$S_{2j}^{\delta_0}
\defeq \{ \bm{\xi} \in \mathbb{R}^2 : (\varphi_2 - \varphi_j)(\bm{\xi}) = -\delta_0 \}$ and
$q_j^{\delta_0} \defeq {\rm dist}(O_j,S_{2j}^{\delta_0})$ for \(j = 5,6\), and also write
\begin{align*}
    u_5^{\delta_0} \defeq u_5 + q_5^{\delta_0} \sin{\theta_{25}} \geq 0\,, \qquad
    u_6^{\delta_0} \defeq u_6 - q_6^{\delta_0} \sin{\theta_{26}} \leq 0 \,.
\end{align*}
For constant $\hat{\delta}>0$ from~\eqref{Eq:DistO5P01minusO5S25}, we set $\tilde{\delta}_0 \defeq\frac{\hat{\delta}}{4}$,
which depends only on $(\gamma,v_2)$.
Then, for $\hat{c}_j$ given by~\eqref{Eq:Sec3-Def4hatcj}, a direct calculation gives that,
for any \(\btheta \in \cl{\Theta}\),
\begin{equation*}\label{Eq:S2jdelta0-2pts}
    \hat{c}_j-q_j^{\delta_0}\geq\hat{\delta}-\frac{\delta_0}{\ell(\rho_j)}\geq2\tilde{\delta}_0\,,
\end{equation*}
after choosing $\delta_0>0$ small enough, depending only on $(\gamma,v_2)$, and
using {\rm Lemma~\ref{lem:properties-of-state-5-and-6}}.
Therefore, \(S_{2j}^{\delta_0}\) and \(\partial B_{\hat{c}_j}(O_j)\) intersect
at two distinct points for any $\btheta\in\cl{\Theta}$.

\begin{definition}[Extended domain \(Q^{\btheta}\)] \label{def:Q-theta1-theta2}
Let $\delta_0>0$ be chosen as above.
Define \(P_2'\) and \(P_3'\) by
\begin{equation*}
\{P'_2\} \defeq S_{25}^{\delta_0} \cap \partial B_{\hat{c}_5}(O_5)\cap\{\xi\geq u_5^{\delta_0}\}\,,
\qquad
\{P'_3\} \defeq S_{26}^{\delta_0} \cap \partial B_{\hat{c}_6}(O_6)\cap\{\xi\leq u_6^{\delta_0}\}\,.
\end{equation*}
Denote by \(\Gamma_{\rm sonic}^{5, \delta_0}\) the arc of \(\partial B_{\hat{c}_5}(O_5)\cap\{\xi\geq u_5^{\delta_0}\}\)
with endpoints \(\{P_1, P'_2\},\)
and denote by \(\Gamma_{\rm sonic}^{6, \delta_0}\) the arc of \(\partial B_{\hat{c}_6}(O_6)\cap\{\xi\leq u_6^{\delta_0}\}\) with endpoints \(\{P_4, P'_3\}\).
Define the extended domain \(\qtheta\) as the open bounded region enclosed by
\(\Gamma_{\rm sonic}^{5,\delta_0}, \Gamma_{\rm sonic}^{6,\delta_0}, S_{25}^{\delta_0}, S_{26}^{\delta_0},\)
and \(\Gamma_{\rm sym};\) see {\rm Fig.~\ref{Fig04Sec4RS}}.
\end{definition}

\begin{figure}[b]
\centering
\includegraphics[width=11cm]{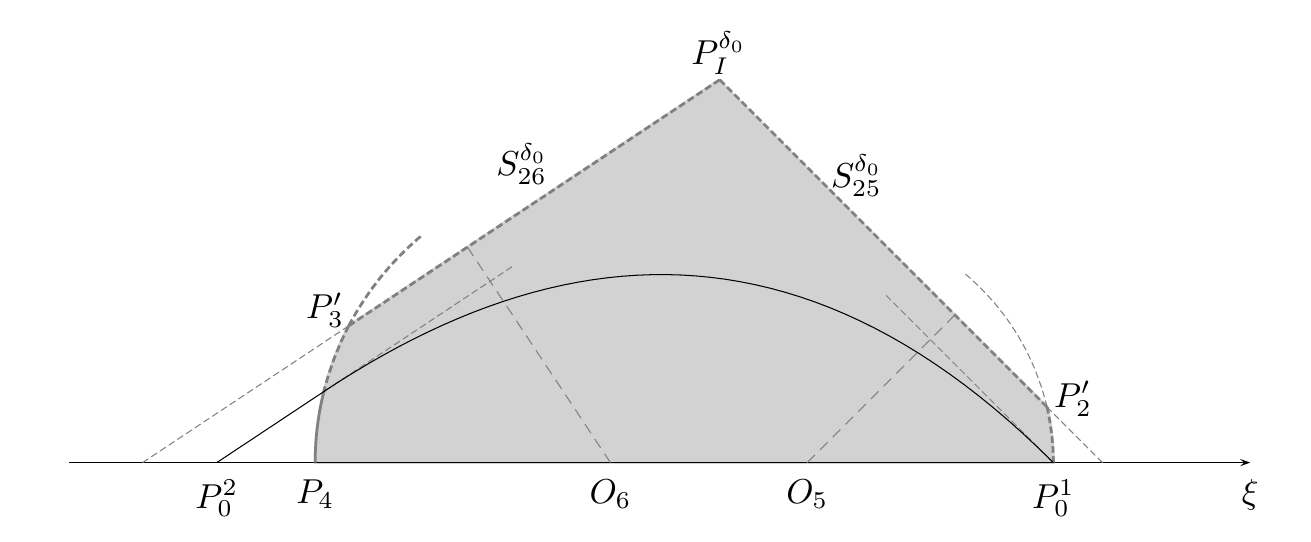}
\vspace{-0.5em}
\caption{The extended domain \(\qtheta\) in grey for the case: \(\btheta \in [\theta^{\rm s},\theta^{\rm d}) \times (0,\theta^{\rm s})\).} \label{Fig04Sec4RS}
\end{figure}

For any $\bm{\xi}= (\xi,\eta) \in (\mathbb{R} \setminus [u_6^{\delta_0},u_5^{\delta_0}]) \times [0,\infty)$, define the $(x,y)$--coordinates by $(x,y) \defeq \rcal (\bm{\xi})$ if and only if
\begin{equation}\label{Eq:DefPolarCoordiSec4}
\bm{\xi} = \begin{cases}
    O_6+({c}_6-x)\left(\cos(\pi-y),\,\sin(\pi-y)\right) &\quad  \text{if}\; \xi< u_6^{\delta_0}\,,\\
    O_5+({c}_5-x)\left(\cos y,\,\sin y\right) &\quad  \text{if}\; \xi> u_5^{\delta_0}\,,
\end{cases}
\end{equation}
with $y \in (0,\pi)\,$. For any constant \(\epsilon > 0\), we define two sets \(\dfive_\epsilon, \dsix_{\epsilon} \subseteq Q^{\btheta}\) by
\begin{equation}\label{eq:domain-D5n6epsilon}
    \dfive_\epsilon \defeq
    \big\{\bm{\xi}\in\qtheta \,:\, \xi > u_5^{\delta_0}\big\} \setminus \cl{B_{\hat{c}_5 - \epsilon}(O_5)}\,,
    \qquad
     \dsix_\epsilon \defeq
      \big\{\bm{\xi}\in\qtheta \,:\, \xi < u_6^{\delta_0}\big\} \setminus \cl{B_{\hat{c}_6 - \epsilon}(O_6)}\,.
\end{equation}
Let $\omega_0\in(0,\frac{\pi}{2})$ be the solution of
$(q^{\delta_0} + \tilde{\delta}_0)\cos\omega_0=q^{\delta_0}$, where \(q^{\delta_0} \defeq \max\{q_5^{\delta_0},q_6^{\delta_0}\}\).
Then
\begin{equation*}
\tilde{\omega}_0
\defeq
\frac12 \inf_{\btheta\in\cl{\Theta}}\big\{\theta_{25}-\tfrac{\pi}{2},\,\tfrac{\pi}{2}-\theta_{26},\,\omega_0\big\}
> 0
\end{equation*}
by \eqref{Eq:Sec3-Def4hatcj}--\eqref{Eq:DistO5P01minusO5S25} and Lemma~\ref{lem:properties-of-state-5-and-6}.
A direct computation shows that, for any \(\epsilon \in (0,\tilde{\delta}_0)\),
\begin{equation}\label{Eq:D5n6epsilonSetBdd}
\begin{aligned}
& \dfive_{\epsilon} \subseteq
\big\{ \bm{\xi} \,:\, x_{P_2} < x < x_{P_2} + \epsilon, \, 0 < y< \theta_{25} - \tfrac\pi2 - \tilde{\omega}_0 \big\}
 \cap \big\{ \xi > u_5^{\delta_0} + \tilde{\delta}_0 \sin{\tilde{\omega}_0} \big\} \,,\\
& \dsix_{\epsilon} \subseteq \big\{ \bm{\xi} \,:\, x_{P_3} < x < x_{P_3} + \epsilon, \,
 0 < y < \tfrac{\pi}{2} - \theta_{26} - \tilde{\omega}_0 \big\}
 \cap \big\{ \xi < u_6^{\delta_0} - \tilde{\delta}_0 \sin{\tilde{\omega}_0} \big\} \,.
\end{aligned}
\end{equation}
In particular, we fix a constant \(k > 4\) sufficiently large, depending only on \((\gamma, v_2)\),
such that $\max\{\hat{c}_5,\hat{c}_6\}\leq \frac{k\tilde{\delta}_0}{4} \sin{\tilde{\omega}_0}$
for any \(\btheta \in \cl{\Theta}\) and
\begin{equation}\label{Eq:D5n6epsSetBdd4k}
\dfive_{\frac{3\hat{c}_5}{k}}
\subseteq \big\{ \xi > u_5^{\delta_0} + \frac{3\hat{c}_5}{k} \big\}\,,
\qquad
\dsix_{\frac{3\hat{c}_6}{k}}  \subseteq \big\{ \xi < u_6^{\delta_0} - \frac{3\hat{c}_6}{k} \big\}\,.
\end{equation}

For each \(\btheta \in \cl{\Theta}\), we define an invertible map
\(\gtheta : \cl{\qtheta} \to [-1,1] \times [0,\infty) \),
which flattens the extended sonic arcs \(\Gamma^{5,\delta_0}_{\rm sonic}\)
and \(\Gamma^{6,\delta_0}_{\rm sonic}\).
For each admissible solution \(\varphi\) corresponding to parameters \(\btheta \in \cl{\Theta}\),
we further construct a map \(G_{2,\gsh}\) from \(\gtheta(\Omega)\) to the standard domain \(\qiter\),
where \(\Omega\) is the pseudo-subsonic domain associated with \(\varphi\).
Finally, we present the \textit{a priori} estimates in the standard domain for admissible solutions
under
map \(G_{2,\gsh} \circ \gtheta\).

\smallskip
\noindent
\textbf{Construction of \(\gtheta\).}
For $r_0\in\mathbb{R}$ and $r_{\rm d}>0$, we fix a basic cut-off function
$\tilde{\zeta}( \blank; r_0, r_{\rm d})\in C^4(\mathbb{R})$
such that
$\tilde{\zeta}'( \blank; r_0, r_{\rm d})\in[0,\frac{2}{r_{\rm d}}]$ and
\begin{equation}\label{Eq:Basic-Cut-off-zeta}
    \tilde{\zeta}(r;\,r_0,r_{\rm d})=
    \begin{cases}
        \, 0  &\quad \text{for} \;r \leq r_0\,,\\
        \, 1  &\quad \text{for} \;r  \geq r_0+r_{\rm d}\,.
    \end{cases}
\end{equation}
For $j=5,6$, let the cut-off functions \(( \zeta_j,\chi_j)\) be defined by
\begin{equation*}\label{Eq:Def-Chi-Zeta-j}
\begin{aligned}
& \zeta_j(r) \defeq \tilde{\zeta}(r;\,(1-\frac{3}{k})\hat{c}_j, \frac{\hat{c}_j}{k})\,,
\qquad
\chi_j(\xi) \defeq \tilde{\zeta}((-1)^{j-1}\xi;\,(-1)^{j-1}u_j^{\delta_0},\frac{2\hat{c}_j}{k})\,.
\end{aligned}
\end{equation*}
Then we define
\(F_1 : \cl{\qtheta} \to \mathbb{R}^2\)
by $F_1 (\xi , \eta ) \defeq ( h_1(\xi,\eta), \eta)$, where
\begin{equation*}
     h_1(\bm{\xi}) \defeq \zeta_{\rm d}(\bm{\xi}) u_r +  \big(1-\zeta_{\rm d}(\bm{\xi})\big) \xi\,,
\end{equation*}
and $\zeta_{\rm d}$ and $u_r$ are given by
\begin{align*}
&
\zeta_{\rm d}(\bm{\xi})
\defeq \chi_5(\xi) \zeta_5(|\bm{\xi}-O_5|) + \chi_6(\xi) \zeta_6(|\bm{\xi}-O_6|)\,,\\
&
u_r \defeq (u_5+|\bm{\xi}-O_5|) \chi_5(\xi) + (u_6-|\bm{\xi}-O_6|) \chi_6(\xi)\,.
\end{align*}

For each \(\btheta \in \cl{\Theta}\),
we write $s_5 \defeq u_5 + \hat{c}_5 > 0$ and $s_6 \defeq u_6 - \hat{c}_6 < 0$,
which depend continuously on $\btheta \in \cl{\Theta}$.
Let $\mathcal{N}_{\varepsilon}(\Gamma)$ be defined by~\eqref{Eq:Def-N-eps}.
By construction of
map \(F_1\), we observe that
\begin{equation*}
h_1(\bm{\xi})=u_j+(-1)^{j-1}{\rm dist}(\bm{\xi},O_j)
\qquad\, \text{for any
$\bm{\xi}\in\qtheta\cap\mathcal{N}_{\frac{2\hat{c}_j}{k}}(\Gamma^{j,\delta_0}_{\rm sonic})$}\,,
\end{equation*}
and $\cl{F_1(\qtheta)} \subseteq [s_6, s_5 ] \times [0 , \infty)$.
Let the linear map $L_{\btheta}:[s_6,s_5]\to[-1,1]$ be given by
\begin{align}\label{Eq:LinearMap-Ltheta}
L_{\btheta}(s') \defeq d_{L} (s' - s_6) - 1
\qquad\,\text{with}\; d_{L}\defeq\frac{2}{s_5 - s_6}\,,
\end{align}
which maps
interval \([s_6,s_5]\) to the standard interval \([-1,1]\).
For
\(k > 4\) given by~\eqref{Eq:D5n6epsSetBdd4k},
define the cut-off functions $\tilde{\chi}_j$
by
\begin{equation*}
\tilde{\chi}_j(s') \defeq
\tilde{\zeta}((-1)^{j-1}s';\,(-1)^{j-1}u_j+(1-\frac{1}{k})\hat{c}_j,\frac{\hat{c}_j}{2k}) \qquad\, \mbox{for $j=5,6$}\,.
\end{equation*}
We define a function \(h_2 : \cl{F_1(\qtheta)} \to [0,\infty)\) by
\begin{align*}
    h_2(s',t'') \defeq \tilde{\chi}_5 \arcsin{\big( \frac{t''}{s' - u_5} \big)}
        + \tilde{\chi}_6 \arcsin{\big( \frac{t''}{u_6 - s'} \big)}
        + (1 - \tilde{\chi}_5-\tilde{\chi}_6)\, t''\,,
\end{align*}
where \(\tilde{\chi}_5\) and \(\tilde{\chi}_6\) are evaluated at \(s'\).
For maps \(F_1\) and \(h_2\) constructed above, we see that
\begin{align*}
    h_2\circ F_1(\bm{\xi}) = y \qquad\,
    \text{for any $\bm{\xi}\in\qtheta\cap \big(\mathcal{N}_{\frac{\hat{c}_5}{2k}}(\Gamma^{5,\delta_0}_{\rm sonic}) \cup \mathcal{N}_{\frac{\hat{c}_6}{2k}}(\Gamma^{6,\delta_0}_{\rm sonic})\big)$}\,,
\end{align*}
for the \((x,y)\)--coordinates given by~\eqref{Eq:DefPolarCoordiSec4}.

Finally, we define
map \(\gtheta : \cl{\qtheta} \to [-1,1] \times [0,\infty)\) by
\begin{align}\label{Eq:Def-gtheta}
    \gtheta(\bm{\xi}) \defeq (L_{\btheta}\circ h_1(\bm{\xi}), \, h_2 \circ F_1(\bm{\xi}))
    \qquad\, \mbox{for any $\bm{\xi}\in\overline{\qtheta}$}\,,
\end{align}
 and write the coordinates \((s,t') \defeq \gtheta(\bm{\xi})\).

It can be verified directly that $\gtheta(\Gamma_{\rm sym}) = \{(s, 0) \,:\, -1 < s < 1 \}$.
Moreover,
map $\mathcal{G}^{\btheta}_1:\overline{\qtheta}\to[-1,1]\times[0,\infty)$ constructed above satisfies the following properties by using~\eqref{nStrictDirectMonotPropty}, \eqref{Eq:Basic-Cut-off-zeta}, and Lemma~\ref{Lem:Compactness4AdmiSoluSet};
the proof is similar to~\cite[Lemmas 4.3 and 4.5]{BCF-2019}.

\begin{lemma}\label{Lem:Propt4gtheta}
There exists a constant \(\mu_{\mathcal{G}_1} \in (0,1)\) depending only on \((\gamma, v_2)\) such that,
for each \(\btheta \in \cl{\Theta},\)
map \(\gtheta\) defined by~\eqref{Eq:Def-gtheta} satisfies the following{\rm:}
\begin{enumerate}[{\rm (i)}]
\item \(\norm{\gtheta}_{C^4(\cl{\qtheta})} + \norm{(\gtheta)^{-1}}_{C^4(\cl{\gtheta(\qtheta)})}
\leq \mu_{\mathcal{G}_1}^{-1},\)
and \(|{\rm det }(D\gtheta) | \geq \mu_{\mathcal{G}_1} \) in \(\cl{\qtheta}\){\rm ;}

\item \label{lem:properties-of-map-Gtheta:phi56-monotonicity}
For \(j = 5,6,\) and for any $(s,t') \in \cl{\gtheta(\qtheta)}${\rm,} \;
$\partial_{t'}\big( (\varphi_2 - \varphi_j) \circ (\gtheta)^{-1}\big) (s, t') \leq  - \mu_{\mathcal{G}_1}${\rm;}

\item \label{Lem:Propt4gtheta-item2}
For any \(\theta_{\ast} \in (0, \theta^{\rm d}),\) there exists a constant \(\mu_\ast > 0\)
depending only on \((\gamma, v_2,\theta_{\ast})\) such that any admissible solution
\(\varphi\) corresponding to
\(\btheta \in \Theta\cap\{ \theta_1,\theta_2 \leq \theta_{\ast} \}\) satisfies
\begin{flalign*}
 &&\partial_{t'} \big( (\varphi_2 - \varphi) \circ (\gtheta)^{-1}\big)
 (s,t') \leq - \mu_\ast < 0
 \qquad \text{for all $(s,t') \in \cl{\gtheta(\Omega)}$}\,.
    &&
    \end{flalign*}
\end{enumerate}
\end{lemma}

Using Lemma~\ref{Lem:Propt4gtheta}\eqref{lem:properties-of-map-Gtheta:phi56-monotonicity},
there exists a unique function \(f_{\btheta} \in C^{0,1}([-1,1])\) such that
\begin{equation*}
\gtheta\big(\{\bm{\xi}\in\mathbb{R}^2_+ \,:\, \varphi_2=\max\{\varphi_5,\varphi_6\}-\delta_0\}\big) =
\{(s,f_{\btheta}(s)) \,:\, -1 < s < 1\}\,,
\end{equation*}
which follows from the implicit function theorem, and
\begin{equation} \begin{aligned}
\label{eq:def-f-theta}
&\gtheta(\qtheta) = \big\{(s,t')\in\mathbb{R}^2_+ \,:\,  -1 < s < 1 , \, 0 < t' < f_{\btheta}(s) \big\}\,.
\end{aligned} \end{equation}

Using \eqref{Eq:ResultInStep3}, Definition~\ref{Def:AdmisSolus}\eqref{item2-Def:AdmisSolus},
Lemma~\ref{Lem:LocalEstimates4Interior}\eqref{Lem:LE4Int-2}, Corollary~\ref{Corol:Esti-AwaySonic5n6},
Propositions~\ref{Prop:Esti4theta1a}--\ref{Prop:Esti4theta1c} and \ref{Prop:Esti4theta1d},
and Lemma~\ref{Lem:Propt4gtheta},
we introduce a function \(\gsh : [-1,1] \to [0,\infty) \) representing \(\Gamma_{\rm shock}\)
in the \((s,t')\)--coordinates
and state its important properties; the proof is similar to~\cite[Proposition 4.6]{BCF-2019}.

\begin{proposition}
\label{prop:properties-of-gshock}
For each admissible solution \(\varphi\) corresponding to \(\btheta\in\Theta,\)
there exists a unique function \(\gsh : [-1,1] \to [0,\infty)\) satisfying the following properties{\rm:}
\begin{enumerate}[{\rm (i)}]
\item \label{prop:properties-of-gshock:1n2}
In the \((s,t')\)--coordinates{\rm:}
    \begin{equation*}\begin{aligned}
        &\gtheta(\Omega) = \big\{ (s,t') \,:\, -1 < s < 1, \, 0 < t' < \gsh(s)\big\}\,, \\
        &\gtheta(\Gamma_{\rm shock}) = \big\{ (s,\gsh(s)) \,:\, -1 < s < 1,\, t'=\gsh(s)\big\}\,,
    \end{aligned} \end{equation*}
    and, for any interval $I_1\Subset(-1,1),$
$\norm{\gsh}_{C^3(\overline{I_1})} \leq C_{_{I_1}}$
for some constant \(C_{I_1} > 0 \) depending only on \((\gamma, v_2, I_1)\).

    \item \label{prop:properties-of-gshock:3}
    Let \(\qtheta_0\) be the bounded region enclosed by
    \(\Gamma^5_{\rm sonic}, \Gamma^6_{\rm sonic}, S_{25}, S_{26},\)
    and \(\Gamma_{\rm sym},\) which satisfies $\Omega \subseteq \qtheta_0 \subseteq \qtheta$.
    For \(\dfive_{\epsilon}\) and \(\dsix_{\epsilon}\) given by~\eqref{eq:domain-D5n6epsilon},
    there exist unique functions \(\gfive\) and \(\gsix\) with
    \begin{equation*}
    \gtheta(\qtheta_0 \cap \mathcal{D}^j_\epsilon)
    = \big\{ (s,t')\in\qiter \,:\, 1+(-1)^j s<d_{L}\epsilon, \, 0<t'<\mathfrak{g}_{j}(s)\big\}
    \qquad\,\text{for $j=5,6$}\,.
    \end{equation*}
  Let \(\epsilon^\ast_0 > 0\) be the smallest
  constant
  \(\varepsilon_0\) in {\rm Step~\textbf{1}} of the proofs of {\rm Propositions~\ref{Prop:Esti4theta1a}}
  and~{\rm\ref{Prop:Esti4theta1c}}.
  For $d_L>0$ given by~\eqref{Eq:LinearMap-Ltheta}{\rm,} set $\hat{\epsilon}^*_0\defeq d_L\epsilon^*_0$.
  There exists a constant \(C> 0\) depending only on \((\gamma, v_2)\) such that
  \begin{align*}
  \norm{\gfive}_{C^3([1 - \hat{\epsilon}^*_0,1])} + \norm{\gsix}_{C^3([-1, -1 + \hat{\epsilon}^*_0])} \leq C\,.
  \end{align*}

  \item \label{prop:properties-of-gshock:4}
  For each \(\theta_{\ast} \in (0, \theta^{\rm d}),\) there exist \(\bar{\alpha} \in (0,1)\)
  and \(C_{\theta_{\ast}}>0\) depending only on \((\gamma, v_2,\theta_{\ast})\) such that,
  for any admissible solution corresponding to
  \(\btheta\in\Theta \cap \{ \theta_1, \theta_2 \leq \theta_{\ast}\},\)
  \begin{align*}
  &\norm{\gsh}^{(-1-\bar{\alpha}),\{-1\}}_{2, \bar{\alpha}, (-1,-1+\hat{\epsilon}_0^\ast)}
  + \norm{\gsh}^{(-1-\bar{\alpha}),\{1\}}_{2, \bar{\alpha}, (1-\hat{\epsilon}_0^\ast, 1)} \leq C_{\theta_{\ast}}\,, \\
  & \frac{{\rm d}^m}{{\rm d}s^m}(\gsh - \mathfrak{g}_{j})((-1)^{j-1}) = 0\qquad
    \text{for $m=0,1,$ and $j=5,6$}\,.
    \end{align*}
where
\(\norm{\cdot}^{(\sigma),\{x_0\}}_{2, \alpha, I}\) is given by {\rm Definition~\ref{Def:WeightedHolder-Sec3}\eqref{Def:WeightedHolder-Sec3-item2}}.
Note that the properties above are equivalent to
    \begin{align*}
        \norm{\gsh - \gfive}^{(1+\bar{\alpha}),({\rm par})}_{2,\bar{\alpha},(1 - \hat{\epsilon}_0^\ast,1)} +
        \norm{\gsh - \gsix}^{(1+\bar{\alpha}),({\rm par})}_{2,\bar{\alpha},(- 1, -1 + \hat{\epsilon}_0^\ast)}
        \leq C'_{\theta_{\ast}}
    \end{align*}
for some constant \(C'_{\theta_{\ast}} > 0\) depending only on \((\gamma, v_2, \theta_{\ast}),\)
where
\(\norm{\cdot}^{(1+\alpha),({\rm par})}_{2,\alpha,I}\) is given
by {\rm Definition~\ref{Def:ParabolicNorm01}\eqref{DefPNorm03}} with replacement of \(x\) by \(1 - |s|\).

\smallskip
\item \label{prop:properties-of-gshock:5}
    For each \(\theta_{\ast}\in (0, \theta^{\rm d}),\) there exists a constant \(\hat{k} >1\)
    depending only on \((\gamma, v_2,\theta_{\ast})\) such that, for any admissible solution \(\varphi\)
    corresponding to \(\btheta \in \Theta \cap \{ \theta_1,\theta_2 \leq \theta_{\ast} \}\)
    and any \(s \in [-1,1],\)
\begin{equation*}
\begin{split}
    &\min \{ \gsh(-1) + \frac{s+1}{\hat{k}} , \, \gsh(1) + \frac{1 - s}{\hat{k}}, \, \frac{1}{\hat{k}} \}
    \leq \gsh(s) \\
    &\hspace{4em}
    \leq \min \{ \gsh(-1) + \hat{k}(s+1), \,
    \gsh(1) + \hat{k}(1-s), \, f_{\btheta}(s) -\frac{1}{\hat{k}}\}\,.
\end{split}
    \end{equation*}
\end{enumerate}
\end{proposition}

\begin{remark}
By {\rm Propositions~\ref{Prop:Esti4theta1a}}--{\rm\ref{Prop:Esti4theta1b},} for each $\alpha\in(0,1),$
there exist $\hat{\epsilon}_3>0$ and $C_{\alpha}>0$ depending only on $(\gamma, v_2,\alpha)$ such that,
for any admissible solution corresponding to $\btheta \in \Theta,$
\begin{equation*}
\norm{\gsh-\mathfrak{g}_{j}}^{(\rm{par})}_{2,\alpha,I_{j}} \leq C_{\alpha} \qquad \text{whenever $\theta_{j-4} \leq \theta^{\rm s}$}\,,
\end{equation*}
for $j = 5, 6,$
where $I_j\defeq\{s\in(-1,1) \,:\, 1+(-1)^{j}s<\hat{\epsilon}_3\},$
and
\(\norm{\cdot}^{({\rm par})}_{2, {\alpha},I_j}\) is given
by {\rm Definition~\ref{Def:ParabolicNorm01}\eqref{DefPNorm03}} with replacement of \(x\) by \(1 - |s|\).

By {\rm Proposition~\ref{Prop:Esti4theta1c},} for each $\alpha\in(0,1),$ there exist constants $\hat{\epsilon}_4>0$
and $C'_{\alpha}>0$ depending only on $(\gamma, v_2,\alpha)$ such that, for any admissible solution corresponding to $\btheta \in \Theta,$
\begin{equation*}\begin{aligned}
\norm{\gsh-\mathfrak{g}_{j}}_{2,\alpha,\overline{I'_j}} \leq C'_{\alpha}\,,
\qquad\;
\frac{{\rm d}^m}{{\rm d}s^m}(\gsh-\mathfrak{g}_{j})((-1)^{j-1})=0 \quad\, \text{for $m=0,1,2$}\,,
\end{aligned}\end{equation*}
whenever $\theta_{j-4}\in[\theta^{\rm s},\theta^{\rm s}+\sigma_3]$ for $j=5, 6,$
where $I'_j\defeq\{s\in(-1,1)\,:\,1+(-1)^{j}s<\hat{\epsilon}_4\}$.
\end{remark}

\smallskip
\noindent \textbf{Construction of $u^{(\varphi,\btheta)}$.}
For each \(\btheta \in \cl{\Theta}\setminus \{\bm{0}\}\),
it can be verified directly
that \(S_{25}\)
and \(S_{26}\)
given by Definition~\ref{def:state5andstate6}
intersect at the unique point \(P_I \defeq (\xi^{P_I}, \eta^{P_I})\) with
\begin{align} \label{eq:intersection-of-S25-S26}
   \xi^{P_I} \defeq \frac{u_5 \xi^{P_0^1} - u_6 \xi^{P_0^2}}{u_5 - u_6}\,, \qquad
    \eta^{P_I} \defeq \frac{u_5 u_6}{u_5 - u_6} \frac{\xi^{P_0^1} - \xi^{P_0^2}}{v_2}\,.
\end{align}
Similarly, \(S_{25}^{\delta_0}\) and \(S_{26}^{\delta_0}\) intersect at the unique point \(P_I^{\delta_0} \defeq (\xi^{P_I} , \eta^{P_I^{\delta_0}})\) with
    $\eta^{P_I^{\delta_0}} \defeq \eta^{P_I} - \tfrac{\delta_0}{v_2} > \eta^{P_I}$.
In Lemma~\ref{lem:intersection-point-P_I}, we show the limits:
\begin{align*}
    \lim_{ \btheta \to \bm{0},\,\btheta \in \cl{\Theta}\setminus\{\bm{0}\}} (\xi^{P_I}, \eta^{P_I})
    = ( 0 , \eta_0 )\,,
\end{align*}
where \(\eta_0\) is given by~\eqref{Sol-NormalReflec0-RH}.
Then we define \(P_I|_{\btheta = \bm{0}} \defeq (0, \eta_0)\) such that \(P_I\) is \(C(\cl{\Theta})\).
By~\eqref{eq:intersection-of-S25-S26} and Lemma~\ref{lem:monotonicityOfTheta5}\eqref{item2-lem:monotonicityOfTheta5},
\(\xi^{P_I} \geq 0 \) for any \(\btheta \in \cl{\Theta} \cap \{\theta_2 = 0\}\).
Therefore, there exists a constant \(\sigma_{\rm sep} >0\) small enough, depending only on \((\gamma, v_2)\), such that
\begin{align*}
    \xi^{P_3} \leq -d_{\rm sep} < - \tfrac{d_{\rm sep}}{2} \leq \xi^{P_I} \qquad
    \text{for any $\btheta \in \cl{\Theta} \cap \{ 0 \leq \theta_2  \leq \sigma_{\rm sep} \}$}\,,
\end{align*}
where \(d_{\rm sep} > 0\) is given by~\eqref{Def4dsep}.
Also, by Lemma~\ref{lem:monotonicityOfTheta5}\eqref{item2-lem:monotonicityOfTheta5},
there exists a small constant \(\delta_{\rm sep} > 0\) depending only on \((\gamma, v_2,\sigma_{\rm sep})\) such that
\begin{align*}
    \eta^{P_3} \leq \eta_0 - \delta_{\rm sep} < \eta_0 \leq \min\{ a_{25} , a_{26} \}
    \leq \eta^{P_I} \qquad\,
    \text{for any $\btheta \in \cl{\Theta} \cap \{ \sigma_{\rm sep} \leq \theta_2 \leq  \theta^{\rm d} \}$}\,.
\end{align*}
From the above two inequalities, using \(P_3, P_I \in S_{26}\) and  Lemma~\ref{lem:properties-of-state-5-and-6}, we obtain the uniform positive lower bound:
\begin{align*}
    \inf_{\btheta \in \cl{\Theta}} ( \xi^{P_I} - \xi^{P_3} ) > 0 \,.
\end{align*}

\begin{definition} \label{def:varphi-theta-star}
Let $\tilde{\zeta} (r;r_0,r_{\rm d})$ be given by~\eqref{Eq:Basic-Cut-off-zeta}{\rm,}
and let \(\xi^{P_I}\) be given by~\eqref{eq:intersection-of-S25-S26}.
For \(\btheta \in \cl{\Theta},\) define
\begin{equation*}
\chi_{\btheta}^*(\xi)\defeq \tilde{\zeta} ( {-{}\xi};-\xi^{P_I},\,\frac{\xi^{P_I}-\xi^{P_3}}{5C_{\rm bg}})
\end{equation*}
for some constant $C_{\rm bg}>1$.
The smooth background function $\vphithetastar (\bm{\xi})$ is defined by
\begin{align*}
\vphithetastar (\bm{\xi}) \defeq \varphi_5(\bm{\xi}) \, (1 - \chi_{\btheta}^\ast(\xi)) + \varphi_6(\bm{\xi})\, \chi_{\btheta}^\ast(\xi)\,.
\end{align*}
For constant \( \delta_0 > 0\) from {\rm Definition~\ref{def:Q-theta1-theta2},}
constant $C_{\rm bg}$ is fixed sufficiently large, depending only on \((\gamma,v_2),\)
such that $\max\{\varphi_5,\varphi_6\} -\vphithetastar\leq\frac{\delta_0}{2}$.
\end{definition}

From the construction above,
it can be checked that $\vphithetastar \leq \max\{\varphi_5,\varphi_6\}$ in $\mathbb{R}^2$
and there exists a constant \(\bar{k} > 1\) sufficiently large,
depending only on \((\gamma, v_2)\), such that, for $j=5,6$,
\begin{align} \label{eq:def-of-k-bar}
    \vphithetastar = \varphi_\btheta = \varphi_j \, \qquad \text{in $\mathcal{D}^j_{\hat{c}_j/\bar{k}}$}\,,
\end{align}
where $\hat{c}_j$ is given by~\eqref{Eq:Sec3-Def4hatcj}.
Then
\begin{equation}\label{Eq:Sec4-Propty4Frakg2}
\big\{ \bm{\xi} \in\mathbb{R}^2_+ \,:\,  \xi^{P_3} < \xi < \xi^{P_2},
\,(\varphi_2 - \vphithetastar)(\bm{\xi}) = 0 \big\}
\subseteq \qtheta\,.
\end{equation}

Using Lemma~\ref{Lem:Propt4gtheta}\eqref{lem:properties-of-map-Gtheta:phi56-monotonicity}
and Definition~\ref{def:varphi-theta-star},
we can directly obtain that, for any $\btheta\in\overline{\Theta}$,
there exists a constant \(\bar{\mu} > 0\) depending only on \((\gamma, v_2)\) such that
\begin{equation}\label{Eq:LemMonotone-phi2ast}
 \partial_{t'} \big(  ( \varphi_2 - \vphithetastar ) \circ (\gtheta)^{-1} \big) (s,t')
 \leq - \bar{\mu}
 \qquad\, \text{for any $(s,t') \in \cl{\gtheta(\qtheta)}$}\,.
\end{equation}

For any admissible solution \(\varphi\) corresponding to parameters \(\btheta \in \Theta\),
let \(\gsh: [-1,1] \to [0,\infty) \) be the unique function given by Proposition~\ref{prop:properties-of-gshock}.
Then
\(G_{2,\gsh} : \gtheta(\qtheta) \to \mathbb{R}^2\) is defined by
\begin{align} \label{eq:map-G2}
    (s,t)=G_{2,\gsh}(s,t')\defeq (s , \frac{t'}{\gsh(s)})\qquad\, \text{for any $(s,t')\in \gtheta(\qtheta)$}\,.
\end{align}
By Proposition~\ref{prop:properties-of-gshock}\eqref{prop:properties-of-gshock:5}, map \(G_{2,\gsh}\)
is well defined and invertible with $G_{2,\gsh}^{-1}(s,t) = (s,t\gsh(s))$.
Then
$G_{2,\gsh} \circ \gtheta (\Omega) = \qiter=(-1,1)\times(0,1)$
by Proposition~\ref{prop:properties-of-gshock}\eqref{prop:properties-of-gshock:1n2}.
Therefore, \(u^{(\varphi,\btheta)}\), defined as
\begin{align} \label{eq:phi-to-u}
     u^{(\varphi,\btheta)}(s,t) \defeq (\varphi - \vphithetastar) \circ (G_{2,\gsh} \circ \gtheta)^{-1} (s,t)\,,
\end{align}
is well defined for all $(s,t)\in \qiter$.

\smallskip
\noindent
\textbf{A priori estimates for \(u^{(\varphi,\btheta)}\).}
For \(\epsilon^\ast_0 > 0\) from Proposition~\ref{prop:properties-of-gshock}\eqref{prop:properties-of-gshock:3},
define a constant $\epsilon'_0>0$ depending only on $(\gamma,v_2)$ as
\begin{align} \label{eq:def-epsilon-prime}
    \epsilon'_0\defeq \min_{\btheta \in \cl{\Theta}}d_{L}\epsilon^\ast_0 \,\, \leq \,\,
    \hat{\epsilon}^*_0\,,
\end{align}
where $d_L>0$ is given by~\eqref{Eq:LinearMap-Ltheta}.
For $r\in(0,1)$ and for $j=5,6$, we define
\begin{equation}\label{Eq:def-q-set}
\begin{aligned}
 \qint_{r}\defeq\big\{(s,t)\in\qiter \,:\, |s|<1-r\big\} \,, \qquad
 \mathcal{Q}^j_{r} \defeq \big\{(s,t)\in\qiter \,:\,  1+(-1)^{j}\,s<r\big\}\,.
\end{aligned}\end{equation}

\begin{definition}[Weighted H{\"o}lder norms in the standard domain] \label{def:weighted-holder-norm-qiter}
For \(\alpha \in (0,1)\) and any \(u \in C^2(\qiter),\) introduce the norm{\rm :}
\begin{align*}
        \norm{u}^{(\ast)}_{2 ,\alpha, \qiter} \defeq
        \norm{u}_{2,\alpha,\cl{\qint_{\epsilon'_0/4}}} +
        \sum_{j = 5,6} \Big(
        \norm{u}^{(1+\alpha),({\rm par})}_{2, \alpha , \mathcal{Q}^{j}_{\epsilon'_0}}
        + \norm{u}^{(1+\alpha),({\rm subs})}_{1, \alpha , \mathcal{Q}^{j}_{\epsilon'_0}}
        \Big) \, ,
        \vspace{-8pt}
    \end{align*}
    where \(\norm{\cdot}_{2,\alpha,U}\) is the standard H{\"o}lder norm, and the weighted H{\"o}lder
    norms \( \norm{\cdot}^{(\sigma),({\rm subs})}_{m,\alpha,U}\) and \(\norm{\cdot}^{(\sigma),({\rm par})}_{m,\alpha,U},\)
    for an open set \(U \subseteq \qiter,\) $\sigma>0${\rm ,} and $m\in\mathbb{N},$ are given as follows{\rm :}
\begin{enumerate}[{\rm (i)}]
\item
For any $\bm{s}=(s,t), \, \tilde{\bm{s}}=(\tilde{s},\tilde{t})\in\qiter,$ define
\begin{align*}
&\delta^{({\rm subs})}_{\alpha}(\bm{s},\tilde{\bm{s}})\defeq
\big((s-\tilde{s})^2+(\max\{1-|s|,1-|\tilde{s}|\})^2(t-\tilde{t})^2\big)^{\frac{\alpha}{2}}\,.
\end{align*}
The H{\"o}lder norms with subsonic scaling are given by
\begin{equation*}\begin{aligned}
& \|u\|^{(\sigma),({\rm subs})}_{m,0,U}\defeq \sum_{0\leq k+l\leq m}\sup_{\bm{s}\in U}\big((1-|s|)^{k-\sigma}|\partial^{k}_s\partial^{l}_t u(\bm{s})|\big)\,,\\
& [u]^{(\sigma),({\rm subs})}_{m,\alpha,U} \defeq
\sum_{k+l = m}
\sup\limits_{\substack{\bm{s},\tilde{\bm{s}}\in U\\
\bm{s}\neq\tilde{\bm{s}}}}\Big(\min\big\{(1-|s|)^{\alpha+k-\sigma},(1-|\tilde{s}|)^{\alpha+k-\sigma}\big\}
\frac{|\partial^{k}_s\partial^{l}_t u(\bm{s})-\partial^{k}_s\partial^{l}_t u(\tilde{\bm{s}})|}{\delta^{({\rm subs})}_{\alpha}(\bm{s},\tilde{\bm{s}})}\Big)\,,\\
& \|u\|^{(\sigma),({\rm subs})}_{m,\alpha,U}=\|u\|^{(\sigma),({\rm subs})}_{m,0,U}+[u]^{(\sigma),({\rm subs})}_{m,\alpha,U}\,.
\end{aligned}\end{equation*}

\item
For any $\bm{s}=(s,t), \tilde{\bm{s}}=(\tilde{s},\tilde{t})\in\qiter,$ define
\begin{equation*}
\delta^{({\rm par})}_{\alpha}(\bm{s},\tilde{\bm{s}})
\defeq \big((s-\tilde{s})^2+ \max\{1-|s|,1-|\tilde{s}|\} (t-\tilde{t})^2\big)^{\frac{\alpha}{2}}\,.
\end{equation*}
The H{\"o}lder norms with parabolic scaling are given by
\begin{equation*}\begin{aligned}
& \|u\|^{(\sigma),({\rm par})}_{m,0,U}\defeq \sum_{0\leq k+l\leq m}\sup_{\bm{s}\in U}\big((1-|s|)^{k+\frac{l}{2}-\sigma}|\partial^{k}_s\partial^{l}_t u(\bm{s})|\big)\,,\\
& [u]^{(\sigma),({\rm par})}_{m,\alpha,U} \defeq
\sum_{k+l = m}
\sup\limits_{\substack{\bm{s},\tilde{\bm{s}}\in U\\
\bm{s}\neq\tilde{\bm{s}}}}\Big(\min\big\{(1-|s|)^{\alpha+k+\frac{l}{2}-\sigma},(1-|\tilde{s}|)^{\alpha+k+\frac{l}{2}-\sigma}\big\}
\frac{|\partial^{k}_s\partial^{l}_t u(\bm{s})-\partial^{k}_s\partial^{l}_t u(\tilde{\bm{s}})|}{\delta^{({\rm par})}_{\alpha}(\bm{s},\tilde{\bm{s}})}\Big)\,,\\
& \|u\|^{(\sigma),({\rm par})}_{m,\alpha,U}=\|u\|^{(\sigma),({\rm par})}_{m,0,U}+[u]^{(\sigma),({\rm par})}_{m,\alpha,U}\,.
\end{aligned}\end{equation*}
\end{enumerate}
Denote by \(C^{2, \alpha}_{(\ast)}(\qiter)\) the set of
all functions in \(C^2(\qiter)\) whose \(\norm{\cdot}^{(\ast)}_{2,\alpha,\qiter}\)--norms are finite.
\end{definition}

Note that \(C^{2,\alpha}_{(\ast)}(\qiter)\) is compactly embedded into \(C^{2,\tilde{\alpha}}_{(\ast)}(\qiter)\)
whenever \(0\leq \tilde{\alpha}<\alpha < 1\), which follows from the properties
given in~\cite[Lemma~4.6.3 and Corollary~17.2.7]{ChenFeldman-RM2018}.

 Using Corollary~\ref{Corol:Esti-AwaySonic5n6} and Propositions~\ref{Prop:Esti4theta1a}--\ref{Prop:Esti4theta1c} and~\ref{Prop:Esti4theta1d},
we have the following estimate for admissible solutions $\varphi$ after mapping the pseudo-subsonic domain $\Omega$
to the standard domain $\qiter$; the proof is similar to~\cite[Proposition 4.12]{BCF-2019}.

\begin{proposition}
\label{prop:apriori-estimates-on-u-prop4.12}
For each \(\theta_{\ast} \in (0,\theta^{\rm d}),\)
there exist constants \(M>0\) and \(\bar{\alpha} \in (0, \frac13]\) depending only on \((\gamma, v_2,\theta_{\ast})\)
such that, for any admissible solution \(\varphi\) corresponding to \(\btheta\in\Theta\cap\{\theta_1,\theta_2
\leq \theta_{\ast}\},\) \(u = u^{(\varphi,\btheta)} : \qiter \to \mathbb{R}\),
defined by~\eqref{eq:phi-to-u}, satisfies \(u \in C^{2,\bar{\alpha}}_{(\ast)}(\qiter)\) and
\begin{flalign*}
&&
    \norm{u}^{(\ast)}_{2 ,\bar{\alpha}, \qiter} = \norm{u}_{2,\bar{\alpha},\qint_{\epsilon'_0/4}}
    + \sum_{j=5,6} \Big(
    \norm{u}^{(1+\bar{\alpha}),({\rm par})}_{2,\bar{\alpha},\mathcal{Q}^j_{\epsilon'_0}}
    + \norm{u}^{(1+\bar{\alpha}),({\rm subs})}_{1,\bar{\alpha},\mathcal{Q}^j_{\epsilon'_0}} \Big)
    \leq M\,.
    &&
\end{flalign*}
\end{proposition}

\subsubsection{Mapping to approximate admissible solutions}
For each \(\btheta \in \cl{\Theta}\), let the extended domain \(\qtheta\) be given by Definition~\ref{def:Q-theta1-theta2},
and let \(\vphithetastar\) be given by Definition~\ref{def:varphi-theta-star}.
For each \(s^\ast \in (-1,1)\), define
$\qtheta(s^\ast) \defeq \qtheta \cap (\gtheta)^{-1} ( \{ s = s^\ast \} )$.
Then, for \(i=1,2\),
\begin{align*}
    \inf_{\qtheta( (-1)^{i-1}) } (\varphi_2 - \vphithetastar) < 0 \leq \sup_{\qtheta( (-1)^{i-1})} (\varphi_2 - \vphithetastar)\,,
\end{align*}
where the inequality on the right above becomes a strict inequality when \(\theta_i \in [0,\theta^{\rm s})\)
and an equality when \(\theta_i \in [\theta^{\rm s},\theta^{\rm d}]\).
Moreover, by the choice of $C_{\rm bg}>0$ in Definition~\ref{def:varphi-theta-star},
\begin{equation*}\label{Eq:Propts4VarphiThetaStar}
\sup_{\qtheta(s)} (\varphi_2-\vphithetastar) - \inf_{\qtheta(s)} (\varphi_2-\vphithetastar)  \geq \frac{\delta_0}{2} >0
\qquad\, \text{for any $s\in(-1,1)$}\,.
\end{equation*}

\begin{definition}
\label{def:u-to-phi}
Fix \(\alpha \in (0,1), \,\theta_{\ast} \in (0,\theta^{\rm d}), \)
and \(\btheta\in\Theta \cap \{ \theta_1,\theta_2 \leq \theta_{\ast}\}\).
Let \(u \in C^{1,\alpha}(\cl{\qiter})\) be a function satisfying that
\begin{align} \label{eq:approx-admissible-condition-u}
    \inf_{\qtheta(s)} (\varphi_2 - \vphithetastar) < u(s,1) < \sup_{\qtheta(s)} (\varphi_2 - \vphithetastar)
    \qquad\, \text{for any $s\in(-1,1)$}\,.
\end{align}
\begin{enumerate}[{\rm (i)}]
    \item \label{item1-def:u-to-phi}
    For each \(s \in (-1,1),\) let
    \(\bar{t}' > 0\) be the unique solution
    of $ (\varphi_2 - \vphithetastar) \circ  (\gtheta)^{-1}  (s,\bar{t}') = u(s,1),$
    which exists by~\eqref{Eq:LemMonotone-phi2ast} and~\eqref{eq:approx-admissible-condition-u}.
    Define \(\gsh^{(u,\btheta)}:(-1,1) \to (0,\infty)\) by
     $\,\gsh^{(u,\btheta)}(s) \defeq \bar{t}'$.

    \item \label{item2-def:u-to-phi}
    For \(\gtheta\) and \(G_{2,\gsh^{(u,\btheta)}}\) given by~\eqref{Eq:Def-gtheta} and~\eqref{eq:map-G2} respectively,
    define
    \(\Futheta:\qiter \to \qtheta\) by
    $\,\Futheta \defeq (\gtheta)^{-1} \circ G_{2,\gsh^{(u,\btheta)}}^{-1}$.

    \item \label{item3-def:u-to-phi}
    Define sets \(
    \Omega(u,\btheta)\defeq \Futheta(\qiter)\) and
    \(\Gamma_{\rm shock}(u,\btheta)\defeq \Futheta( (-1,1) \times \{1\})\).
    Finally, define function \(\varphi^{(u,\btheta)}(\bm{\xi}) : \Omega(u,\btheta) \to \mathbb{R}\) by
     $\,\varphi^{(u,\btheta)}(\bm{\xi}) \defeq (u \circ \Futheta^{-1})(\bm{\xi}) + \vphithetastar(\bm{\xi})$
     for all $\bm{\xi} \in \Omega(u,\btheta)$.
\end{enumerate}
\end{definition}

We call $\varphi^{(u,\btheta)}$ above an approximate admissible solution
if it satisfies the same boundary conditions on $\Gamma_{\rm sonic}^5\cup\Gamma_{\rm sonic}^6$
as the admissible solutions in Definition~\ref{Def:AdmisSolus}.
For \(\alpha \in (0,1)\) and \(\theta_{\ast} \in (0, \theta^{\rm d})\),
define the space of such approximate admissible solutions:
\begin{align} \label{eq:def-G-alpha-theta}
\mathfrak{G}^{\theta_{\ast}}_{\alpha}
\defeq \bigg\{
(u,\btheta) \in C^{1,\alpha}(\cl{\qiter}) \times [0,\theta_{\ast}]^2 \,:\, \,
\parbox{11em}{$(u,\btheta)$ satisfies~\eqref{eq:approx-admissible-condition-u} and\\ $(u,Du)(\pm1,\blank) = (0,\mathbf{0})$} \,\,
    \bigg\}.
\end{align}
By construction, for any \((u,\btheta) \in \mathfrak{G}^{\theta_{\ast}}_{\alpha}\),
map \(\Futheta\) satisfies
that \(P_1 = \Futheta(1,0)\), \(P_2 = \Futheta(1,1)\), \(P_3 = \Futheta(-1,1)\), and \(P_4 = \Futheta(-1,0)\)
for points \(P_i\), \(i = 1,2,3,4\), given by Definition~\ref{Def:GammaSonicAndPts}.
Furthermore, $\varphi^{(u,\btheta)}=\varphi_2$ on $\Gamma_{\rm shock}(u,\btheta)$.
Other properties of set \(\mathfrak{G}^{\theta_{\ast}}_{\alpha}\) are given below;
the proof is similar to~\cite[Lemmas 12.2.7 and 17.2.13]{ChenFeldman-RM2018}.

\begin{lemma}\label{lem:properties-of-G-alpha-theta}
Fix \(\alpha \in (0,1)\) and \(\theta_{\ast} \in (0,\theta^{\rm d})\).
For each \((u,\btheta) \in \mathfrak{G}^{\theta_{\ast}}_{\alpha},\) \(\gsh^{(u,\btheta)} \in C^{1,\alpha}([-1,1])\)
and the following properties hold{\rm :}
\begin{enumerate}[{\rm (i)}]
\item \label{lem:properties-of-G-alpha-theta-2}
$\Omega(u,\btheta) \cup \Gamma_{\rm shock}(u,\btheta) \subseteq \qtheta \subseteq \mathbb{R}^2_{+}$.
 Moreover, \(\Gamma_{\rm shock}(u,\btheta)\) is a \(C^{1,\alpha}\)--curve up to
 its endpoints \(P_2\) and \(P_3,\)
 and is tangential to \(S_{25}\) and \(S_{26}\) at \(P_2\) and \(P_3\) respectively.
 If \({f}_{j,0}, j=5,6,\) are given as in {\rm Step~\textbf{1}} of the proofs
 of {\rm Propositions~\ref{Prop:Esti4theta1a}} and~{\rm\ref{Prop:Esti4theta1c},} then
\begin{equation*}
    \big( \gsh^{(u,\btheta)}\,, \frac{ {\rm d} }{{\rm d}s} \gsh^{(u,\btheta)} \big) \big|_{s=(-1)^{j-1}}
    = \big( {f}_{j,0}(0)\,,  \frac{(-1)^{j}}{d_{L}} {f}'_{j,0} (0)\big) \qquad\mbox{for $j=5,6$}\,,
\end{equation*}
 where $d_{L}$ is given by~\eqref{Eq:LinearMap-Ltheta}.

\smallskip
\item \label{lem:properties-of-G-alpha-theta-3}
    Let \(\mathcal{N}_{\varepsilon}(\Gamma)\) denote the open \(\varepsilon\)-neighbourhood of a set \(\Gamma,\) as given by~\eqref{Eq:Def-N-eps}.
    Let \(\delta_0\) be from {\rm Definition~\ref{def:Q-theta1-theta2},} and
    let the coordinates \((x,y) = \rcal(\bm{\xi})\) be given by~\eqref{Eq:DefPolarCoordiSec4}.
    Then there exists a constant \(\epsilon_0 > 0\) depending only on \((\gamma, v_2)\) such that,
    for $j=5,6,$ and any \(\epsilon \in(0,\epsilon_0),\)
    \begin{equation*}\begin{aligned}
        &\hat{\Omega}^j_\epsilon \defeq \mathcal{N}_{\epsilon_0} (\Gamma^{j,\delta_0}_{\rm sonic})
        \cap \big\{\bm{\xi}= \rcal^{-1}(x,y) \,:\,  x_{P_{j-3}} < x < x_{P_{j-3}} + \epsilon \big\}
        \cap \Omega(u,\btheta) \\
        &\hspace{1.33em}= \big\{\bm{\xi}= \rcal^{-1}(x,y) \,:\,  x_{P_{j-3}} < x < x_{P_{j-3}}+\epsilon, \, 0 < y < \hat{f}_{j,{\rm sh}}(x)\big\}\,,\\
        & \Gamma_{\rm shock}(u,\btheta) \cap \partial \hat{\Omega}^j_\epsilon   = \big\{\bm{\xi}
          = \rcal^{-1}(x, \hat{f}_{j,{\rm sh}}(x)) \,:\, x_{P_{j-3}} < x < x_{P_{j-3}}+\epsilon \big\}\,,\\
        &\Gamma^j_{\rm sonic} \cap \partial \hat{\Omega}^j_\epsilon = \big\{\bm{\xi}
          =\rcal^{-1}(x_{P_{j-3}},y) \,:\, 0 < y < \hat{f}_{j,{\rm sh}}(x_{P_{j-3}}) \big\}= \Gamma^j_{\rm sonic}\,,\\
        & \Gamma_{\rm sym} \cap \partial \hat{\Omega}^j_\epsilon   = \big\{\bm{\xi}= \rcal^{-1}(x,0) \,:\, x_{P_{j-3}} < x < x_{P_{j-3}} + \epsilon \big\}\,,
    \end{aligned}\end{equation*}
    where $\hat{f}_{j,{\rm sh}}(x) \defeq \gsh^{(u,\btheta)} \circ L_\btheta (u_j +(-1)^{j-1} (c_j-x))$
    and \(L_\btheta\) is given by~\eqref{Eq:LinearMap-Ltheta}.

\item \label{lem:properties-of-G-alpha-theta-4}
Let \((u,\btheta), (\tilde{u},\tilde{\btheta}) \in \mathfrak{G}_\alpha^{\theta_{\ast}}\)
satisfy \(\norm{(u,\tilde{u})}_{C^{1,\alpha}(\cl{\qiter})} < M\) for some constant \(M >0\).
Then there exists a constant \(C>0\) depending only on \((\gamma, v_2,\theta_{\ast},M,\alpha)\)
such that
    \begin{align*}
        &\norm{\gsh^{(u,\btheta)}}_{C^{1,\alpha}([-1,1])} + \norm{\Futheta}_{C^{1,\alpha}(\cl{\qiter})} \leq C\,, \\
        &\norm{ \gsh^{(u,\btheta)} - \gsh^{(\tilde{u},\tilde{\btheta})} }_{C^{1,\alpha}([-1,1])}
         +\norm{ \Futheta - \mathfrak{F}_{(\tilde{u},\tilde{\btheta})} }_{C^{1,\alpha}(\cl{\qiter})}\\
        &\,\,\,+\norm{ \varphi^{(u,\btheta)} \circ \Futheta
         - \varphi^{(\tilde{u},\tilde{\btheta})} \circ \mathfrak{F}_{(\tilde{u},\tilde{\btheta})} }_{C^{1,\alpha}(\cl{\qiter})}
        +\norm{  \vphithetastar \circ \Futheta -  \varphi_{\tilde{\btheta}}^\ast \circ \mathfrak{F}_{(\tilde{u},\tilde{\btheta})} }_{C^{1,\alpha}(\cl{\qiter})}\\
         &\,\,\, \leq C \big( \norm{u - \tilde{u}}_{C^{1,\alpha}(\cl{\qiter})} + \abs{\btheta - \tilde{\btheta}} \big)\,.
    \end{align*}

\item \label{lem:properties-of-G-alpha-theta-6}
 Let \(\epsilon_0 > 0\) be the constant from~\eqref{lem:properties-of-G-alpha-theta-3}{\rm,}
 and let \(\hat{\epsilon}_0\defeq d_L\epsilon_0\) with $d_L>0$ given by~\eqref{Eq:LinearMap-Ltheta}.
 Assume that, for constants \(\alpha\in(0,1), \, \sigma \in (1,2],\) and \(M>0,\)
    \begin{align} \label{eq:properties-of-G-alpha-theta-bounded-assumption}
        \norm{u}_{2,\alpha,\qiter\cap\{\abs{s} < 1 - \frac{\hat{\epsilon}_0}{10}\}}
        + \norm{u}^{(\sigma),({\rm par})}_{2,\alpha,\qiter\cap\{\abs{s}> 1 - \hat{\epsilon}_0\}}
        \leq M\,.
    \end{align}
 Then there exists a constant \(C>0\) depending only on \((\gamma, v_2,\theta_{\ast},\alpha,\sigma)\) such that
    \begin{align*}
        \norm{\gsh^{(u,\btheta)}}_{2,\alpha,[-1+\frac{\hat{\epsilon}_0}{10}, 1 - \frac{\hat{\epsilon}_0}{10}]}
        + \norm{\gsh^{(u,\btheta)} - \gfive}^{(\sigma),({\rm par})}_{2, \alpha, (1, 1 - \hat{\epsilon}_0) }
        + \norm{\gsh^{(u,\btheta)} - \gsix}^{(\sigma),({\rm par})}_{2, \alpha, (-1, -1 + \hat{\epsilon}_0) }
        \leq C M\,,
    \end{align*}
    where $\mathfrak{g}_{5}$ and $\mathfrak{g}_{6}$ are defined in {\rm Proposition~\ref{prop:properties-of-gshock}\eqref{prop:properties-of-gshock:3}}.

    Furthermore, for
    $I_{\rm so} \defeq (-1,-1+\hat{\epsilon}_0) \cup (1-\hat{\epsilon}_0,1),$
    define \(\mathfrak{F}_{(0,\btheta)} : I_{\rm so} \times (0,\infty) \to \mathbb{R}^2 \) by
   \begin{align*}
       \mathfrak{F}_{(0,\btheta)}(s,t') \defeq \begin{cases}
       \, (G_{2,\gfive} \circ \gtheta)^{-1}(s,t') &\quad \text{if} \;\; s\in
       I_{\rm so} \cap \{ s > 0\}\,, \\[1mm]
       \, (G_{2,\gsix} \circ \gtheta)^{-1}(s,t') &\quad  \text{if} \;\; s \in
       I_{\rm so} \cap \{s < 0\}\,.
       \end{cases}
   \end{align*}
    Then there exists a constant \(C_0 > 0\)
    depending only on \((\gamma, v_2,\theta_{\ast})\)
    such that
    \begin{align*}
        \norm{ \mathfrak{F}_{(0,\btheta)} }_{C^3( \overline{I_{\rm so}}\times[0,1]) }
        +\norm{\mathfrak{F}_{(u,\btheta)}}_{2,\alpha,\qiter\cap\{\abs{s} < 1 - \frac{\hat{\epsilon}_0}{10}\}} +
        \norm{\mathfrak{F}_{(u,\btheta)} - \mathfrak{F}_{(0,\btheta)}}^{(\sigma),({\rm par})}_{2,\alpha,I_{\rm so}\times(0,1)} \leq C_0\,.
    \end{align*}

    \item \label{lem:properties-of-G-alpha-theta-7}
    Let \(f_\btheta\) be given by~\eqref{eq:def-f-theta}.
    For constants \(M>0\) and \(\delta_{\rm sh}>0,\) assume that \((u,\btheta)\in\mathfrak{G}_{\alpha}^{\theta_{\ast}}\)
    satisfies~\eqref{eq:properties-of-G-alpha-theta-bounded-assumption} and, for any \(s\in[-1,1],\)
    \begin{align*}
        &\min\big\{ \gsh^{(u,\btheta)}(-1) + \frac{s+1}{M},\,\gsh^{(u,\btheta)}(1) + \frac{1-s}{M}, \,\delta_{\rm sh}\big\}
        \leq \gsh^{(u,\btheta)}(s) \\
        &\hspace{4em}
        \leq \min\big\{ \gsh^{(u,\btheta)}(-1) + M(s+1), \,\gsh^{(u,\btheta)}(1) + M(1-s), f_\btheta(s) - \frac{1}{M}\big\} \,.
    \end{align*}
    Then, for any \(\epsilon \in (0,\frac{1}{4}\min{\{s_5, -s_6\}}),\)  there exists a constant \(C_\epsilon > 0\)
    depending only on \((\gamma, v_2, \theta_{\ast}, \alpha, \delta_{\rm sh}, \epsilon, M )\) such that
    \begin{align*}
\norm{(\Futheta^{-1},\varphi - \vphithetastar)}_{2,\alpha,\Omega(u,\btheta)\setminus(\cl{\hat{\Omega}^5_\epsilon \cup \hat{\Omega}^6_\epsilon})} \leq C_\epsilon\,.
    \end{align*}

\item \label{lem:properties-of-G-alpha-theta-8}
Let
\((u,\btheta), (\tilde{u},\tilde{\btheta}) \in \mathfrak{G}_\alpha^{\theta_{\ast}}\) satisfy \(\norm{(u,\tilde{u})}_{C^{1,\alpha}(\cl{\qiter})} < M\) for some constant \(M >0\).
For any open set \(K \Subset \qiter\) with $\delta={\rm dist}(K,\overline{\qiter}\cap\{|s|=1\}),$
    there exists a constant \(C_\delta > 0\) depending only
    on \((\gamma, v_2, \theta_{\ast}, \alpha, \sigma, \delta,M)\) such that
\begin{equation*}\begin{aligned}
        &\norm{ \Futheta - \mathfrak{F}_{(\tilde{u},\tilde{\btheta})}}_{C^{2,\alpha}(\cl{K})}
        \leq C_\delta \big( \norm{(u - \tilde{u})(\cdot,1)}_{C^{2,\alpha}([-1 + \delta, 1 - \delta])} + \abs{\btheta - \tilde{\btheta}}\big)\,, \\
        &\norm{(\varphi^{(u,\btheta)} \circ \Futheta - \varphi^{(\tilde{u},\tilde{\btheta})}
        \circ \mathfrak{F}_{(\tilde{u},\tilde{\btheta})},\,\psi^{(u,\btheta)} \circ \Futheta - \psi^{(\tilde{u},\tilde{\btheta})}
        \circ \mathfrak{F}_{(\tilde{u},\tilde{\btheta})}}_{C^{2,\alpha}(\cl{K})}\\
       &\,\,\, \leq C_\delta\big( \norm{u - \tilde{u}}_{C^{2,\alpha}(\cl{K})} + \abs{\btheta - \tilde{\btheta}}\big)\,,
\end{aligned}
\end{equation*}
    where \(\psi^{(u,\btheta)}\) is given by \(\psi^{(u,\btheta)} \defeq \varphi^{(u,\btheta)} - \vphithetastar\) for each \((u,\btheta) \in \mathfrak{G}_\alpha^{\theta_{\ast}}\).
\end{enumerate}
\end{lemma}

\smallskip
\subsection{Definition and properties of the iteration set}\label{Sec4:Def-IterSet}

For any \(\theta_\ast\in(0,\theta^{\rm d})\), we now define the iteration set
\(\mathcal{K} \subseteq C^{1,\alpha}(\cl{\qiter}) \times [0,\theta_\ast]^2\)
for some \(\alpha \in (0,1)\), as a set of functions satisfying the properties similar to the admissible solutions.
In the definition of \(\mathcal{K}\), we construct a nonlinear boundary value problem in a fixed domain
and obtain the well-posedness and \textit{a priori} estimates for the solutions.

\subsubsection{Definition of the iteration set} \label{Sec4-3:DefIterationSet}
Before giving the definition of the iteration set,
we collect some useful estimates for admissible solutions.

\begin{lemma}\label{Lem:UsefulEst4AdmisSolu}
Fix $\theta_{\ast}\in(0,\theta^{{\rm d}})$.
For any admissible solution $\varphi$ corresponding
to $\btheta\in\overline{\Theta}\cap\{\theta_1,\theta_2\leq\theta_\ast\}${\rm ,}
denote $\psi \defeq (\varphi - \max\{\varphi_5,\varphi_6\})\circ \rcal^{-1}$.
Then there exist constants $C_1, \,\bar{\epsilon} > 0$ depending only on $(\gamma,v_2)$,
and $C^{\ast}_2>0$ depending on $(\gamma,v_2,\theta_{\ast})$ such that,
for any \(\epsilon \in (0,\bar{\epsilon}],\)
\begin{equation*}\begin{aligned}
	& \norm{\varphi-\varphi_5}_{C^{0,1}(\overline{\Omega})}+
    \norm{\varphi-\varphi_6}_{C^{0,1}(\overline{\Omega})}
    \leq C_1\,,\\
	& |D_{(x,y)} \psi (x,y)|\leq C^\ast_2 x
    \qquad \text{for}\; (x,y) \in \rcal \big( \overline{\Omega}\cap(\dfive_\epsilon \cup \dsix_\epsilon) \big)\,.
\end{aligned}\end{equation*}
\end{lemma}

\begin{proof}
The first result follows directly from \eqref{eq:def-weak-state-5-6},
{\rm Proposition~ \ref{Prop:localtheorystate5}},  and {\rm Lemma~\ref{UniformBound-Lem}}.

Without loss of generality, we assume that $\theta_*\in(\theta^{\rm s},\theta^{\rm d})$.
Then, for sufficiently small $\sigma>0$ and any $\btheta \in [0,\theta^{\rm s}+\sigma] \times [0,\theta_\ast]$,
\begin{equation*}
   \abs{D_{(x,y)} \psi (x,y)} \leq C_2 x
    \qquad\, \text{for $(x,y) \in \rcal \big(\overline{\Omega}\cap\dfive_\epsilon\big)$}\,,
\end{equation*}
whenever \(\epsilon \in (0,\bar{\epsilon}]\), where \(\bar{\epsilon} > 0\)
is the minimum of the constants, \(\varepsilon > 0\),
in {\rm Propositions~\ref{Prop:Esti4theta1a}--\ref{Prop:Esti4theta1c}}.

From the relations: $\ell(\rho_5) \cos\theta_{25}=v_2$ and \(M_2^{(\theta_1)} = \abs{(-\xi^{P_0^1},v_2)}\),
and {\rm Lemmas~\ref{lem:monotonicityOfTheta5}\eqref{item1-lem:monotonicityOfTheta5}}
and~{\rm\ref{lem:MonotonicityOfMach2}},
\begin{equation}\label{Eq:xP01MonotoneWrtBeta}
\frac{{\rm d}x_{P_0^1}}{{\rm d}\theta_1}
= \frac{{\rm d}}{{\rm d}\theta_1} \big(c_5 - (\xi^{P_0^1} - u_5) \big)
= \frac{{\rm d}}{{\rm d}\theta_1} \big(c_5 -\xi^{P_0^1} + v_2\tan\theta_{25}\big) > 0\,.
\end{equation}
Then there exists $C'_2>0$ depending only on $(\gamma,v_2,\theta_{\ast})$ such that,
for any $\btheta \in [\theta^{\rm s}+\frac{\sigma}{2},\theta_\ast] \times [0,\theta_\ast]$,
\begin{equation*}
    |D_{(x,y)}\psi (x,y)|\leq C'_2x^{\bar{\alpha}}
    \leq C'_2\big(x_{P_0^1}|_{\theta_1=\theta^{\rm s}+\frac{\sigma}{2}}\big)^{\bar{\alpha}-1}x
    \qquad\, \text{for $(x,y) \in \rcal\big( \overline{\Omega}\cap\dfive_\epsilon\big)$}\,,
\end{equation*}
after possibly reducing \(\bar{\epsilon}\) further so that \(\bar{\epsilon} < r\), for \(r > 0\)
from Proposition~\ref{Prop:Esti4theta1d} and  $\bar{\alpha}\in(0,\frac{1}{3}]$
from {\rm Proposition~\ref{prop:apriori-estimates-on-u-prop4.12}}.
From the symmetry of the four-shock interaction problem, we can repeat the above discussion for $\psi(x,y)$ in
$\rcal(\overline{\Omega}\cap\dsix_\epsilon)$.
\end{proof}

Define \(u^{({\rm norm})} \in C^3(\cl{\qiter})\) by~\eqref{eq:phi-to-u} with \(\btheta =\bm{0}\)
and \(\varphi = \varphi_0\) from {\rm \S\ref{SubSec-202NormalShock}}.
    Note that \(\vphithetastar \equiv \varphi_0\) in \(\qtheta\) by {\rm Definition~\ref{def:varphi-theta-star}}
     because \(\varphi_5 = \varphi_6 = \varphi_0\) when \(\btheta = \bm{0}\), which yields that
    \begin{align*}
        u^{({\rm norm})} \equiv 0 \qquad \text{in $\qiter$}\,.
    \end{align*}

Let \(\bar{\alpha} \in (0,\frac13]\) be the constant from Proposition~\ref{prop:apriori-estimates-on-u-prop4.12}, \(\epsilon_0 > 0\) from Lemma~\ref{lem:properties-of-G-alpha-theta}\eqref{lem:properties-of-G-alpha-theta-3}, and
\(\bar{\epsilon} >0\) from  Lemma~\ref{Lem:UsefulEst4AdmisSolu}.
For constants $\alpha \in (0,\frac{\bar{\alpha}}{2}]$, $\delta_1$, $\delta_2$, $\delta_3$,
$\epsilon\in (0,\tfrac{1}{2}\min\{\epsilon_0,\bar{\epsilon}\})$, and \(N_1 > 1\) to be specified later,
we now define the iteration set
\(\mathcal{K}\subseteq C^{2,\alpha}_{(\ast)}(\qiter) \times [0,\theta_\ast ]^2\),
where space \(C^{2,\alpha}_{(\ast)}(\qiter)\) is given by Definition~\ref{def:weighted-holder-norm-qiter}.
For simplicity, in the following,
we use $\big(\gsh, \varphi, \Omega, \Gamma_{\rm shock} \big)$
to denote $\big(\gsh^{(u,\btheta)}, \varphi^{(u,\btheta)}, \Omega(u,\btheta),  \Gamma_{\rm shock}(u,\btheta)\big)$
that are given by {\rm Definition~\ref{def:u-to-phi}}.

\begin{definition}[Iteration set] \label{def:iteration-set}
Fix \(\theta_\ast \in (0,\theta^{\rm d})\).
The iteration set \(\mathcal{K} \subseteq C^{2,\alpha}_{(\ast)}(\qiter) \times [0,\theta_\ast ]^2\)
is the set of all \((u,\btheta)\) satisfying the following properties{\rm :}
\begin{enumerate}[{\rm (i)}]
    \item \label{def:iteration-set-1}
    \(\norm{u - u^{({\rm norm})}}^{(\ast)}_{2, \alpha,\qiter} < \mathscr{K}_1(\max\{\theta_1,\theta_2\})\)
    for \(\mathscr{K}_1 \in C^{0,1}(\mathbb{R})\) given by
    \begin{align*}
        \mathscr{K}_1(\theta) = \begin{cases}
        \delta_1  &\quad \text{if} \;\; \theta \leq \frac{\delta_1}{N_1}\,, \\
        N_0 &\quad \text{if} \;\; \theta \geq \frac{2\delta_1}{N_1}\,,\\
        \text{linear}
        &\quad \text{otherwise}\,,
        \end{cases}
    \end{align*}
    with \(N_0 \defeq \max\{10M,1\}\) for
    \(M=M(\gamma,v_2,\theta_*)>0\) from {\rm Proposition~\ref{prop:apriori-estimates-on-u-prop4.12}}.

    \smallskip
    \item \label{def:iteration-set-2}
    \((u,\btheta) \in \mathfrak{G}_\alpha^{\theta_\ast}\){\rm ,} where
    \(\mathfrak{G}_\alpha^{\theta_\ast}\) is defined by~\eqref{eq:def-G-alpha-theta}.

    \smallskip
    \item  \label{def:iteration-set-3}
    \(\Gamma_{\rm shock}\) satisfies ${\rm dist}(\Gamma_{\rm shock}, B_1(O_2)) > N_2^{-1},$ and
    \(\gsh\) satisfies \(\gsh(\pm 1) \geq 0\) and
    \begin{equation*}\begin{aligned}
        &\min_{i=1,2}\big\{ \gsh((-1)^{i-1})  + \frac{1+(-1)^{i}s}{N_3}, \,\frac{1}{N_3}\big\}
        \leq \gsh(s) \\
        &\hspace{4em}
        \leq \min_{i=1,2} \big\{ \gsh((-1)^{i-1}) + N_3 (1+(-1)^{i}s), \,f_\btheta(s)- \frac{1}{N_3}\big\}
        \qquad\mbox{for all \(s \in (-1,1)\)},
    \end{aligned}\end{equation*}
with \(N_2 \defeq 2C_{\rm sh}\) for \(C_{\rm sh}=C_{\rm sh}(\gamma,v_2)>0\) from
{\rm Proposition~\ref{prop:lowerBoundBetweenShockAndSonicCircle},} \(N_3 \defeq 2 \hat{k}\)
for \(\hat{k}=\hat{k}(\gamma, v_2,\theta_\ast)>0\) from {\rm Proposition~\ref{prop:properties-of-gshock}\eqref{prop:properties-of-gshock:5},}
and \(f_\btheta\) given by~\eqref{eq:def-f-theta}.

\smallskip
\item \label{def:iteration-set-4}
Let coordinates \((x,y) = \rcal(\bm{\xi})\) be given by~\eqref{Eq:DefPolarCoordiSec4} and,
for \(r > 0,\) let \(\dfive_r\) and \(\dsix_r\) be given by~\eqref{eq:domain-D5n6epsilon}.
Then \(\varphi\) and \(\psi \defeq (\varphi - \max{\{\varphi_5,\varphi_6\}})\circ\rcal^{-1}\) satisfy
    \begin{equation}\label{eq:iteration-set-4-1n2}
    \begin{aligned}
    & \varphi - \max{\{\varphi_5,\varphi_6\}} > \mathscr{K}_2(\max\{\theta_1,\theta_2\})
    \qquad &&\text{in $\cl{\Omega} \setminus (\dfive_{\epsilon/10} \cup \dsix_{\epsilon/10})$}\,, \\
    & \partial_{\bm{e}_{S_{2j}}} (\varphi_2 - \varphi) <  -\mathscr{K}_2(\max\{\theta_1,\theta_2\})
       \quad && \text{in $\cl{\Omega} \setminus \mathcal{D}^{j}_{\epsilon/10}$ for $j=5,6$}\,,
\end{aligned}
\end{equation}
and
\begin{equation}\label{eq:iteration-set-4-3to7}
\begin{aligned}
    & \abs{ \partial_x \psi (x,y) } < \mathscr{K}_3(\theta_{j-4}) x \quad &&
    \text{in $\rcal \big( \cl{\Omega} \cap (\mathcal{D}^{j}_{\epsilon_0} \setminus \mathcal{D}^{j}_{\epsilon/10}) \big)$ for $j=5,6$}\,,\\
    & \abs{\partial_y \psi (x,y) } < N_4 x \quad
     &&\text{in $\rcal \big( \cl{\Omega} \cap \big( (\dfive_{\epsilon_0} \setminus \dfive_{\epsilon/10} )
        \cup (\dsix_{\epsilon_0} \setminus \dsix_{\epsilon/10} ) \big) \big)$}\,,\\
    &\abs{ D_{(x,y)} \psi} < N_4 \epsilon \quad
       &&\text{in $\rcal \big( \cl{\Omega} \cap (\cl{\dfive_\epsilon} \cup \cl{\dsix_\epsilon}) \big)$}\,,\\
    & \min\left\{ \partial_{\bm{\nu}} (\varphi_2 - \varphi), \, \partial_{\bm{\nu}} \varphi \right\} > \mu_1  \quad
      && \text{on $\cl{\Gamma_{\rm shock}}$}\,,\\
   & \norm{\varphi - \varphi_j}_{C^{0,1}(\cl{\Omega})} < N_5 && \text{for $j = 5,6$}\,,
\end{aligned}
\end{equation}
for $\bm{e}_{S_{25}}$ and $\bm{e}_{S_{26}}$ given by~\eqref{eq:tangent-vectors-e25-e26}{\rm,}
and the unit normal vector \(\bm{\nu}\) on \(\Gamma_{\rm shock}\) towards the interior of \(\Omega\).
In the above conditions, functions \(\mathscr{K}_2,\mathscr{K}_3 \in C^{0,1}(\mathbb{R})\) are defined by
\begin{align*}
        \mathscr{K}_2(\theta) \defeq \delta_2 \min\big\{ \theta - \frac{\delta_1}{N_1^2},\,\frac{\delta_1}{N_1^2}\big\}\,,
        \qquad
        \mathscr{K}_3( \theta) \defeq \begin{cases}
        \frac{2- \mu_0}{1+\gamma} \quad &\quad \text{if} \;\; 0 \leq \theta \leq \theta^{\rm s} + \frac{ \sigma_2}{2}\,,\\
        N_4 \quad &\quad \text{if} \;\;  \theta \geq \theta^{\rm s} + \sigma_2\,,\\
        \text{linear } \quad &\quad \text{otherwise}\,,
        \end{cases}
    \end{align*}
    for constants \(\mu_1, \, \sigma_2, \, \mu_0, \, \epsilon_0, \, N_4,\) and \(N_5\) chosen as follows{\rm :}

    \smallskip
    \begin{itemize}
    \item Let \(\delta^{\prime}=\delta^{\prime}(\gamma,v_2)>0\) be from {\rm Lemma~\ref{Lem:LocalEstimates4Interior}}.
     Choose \(\mu_1 \defeq \frac{ \delta^{\prime}}{2}\).
    \item Let \(\sigma_2=\sigma_2(\gamma,v_2)>0\) and \(\delta=\delta(\gamma,v_2) > 0\) be from {\rm Step~\textbf{1}}
    of the proof
    of {\rm Proposition~\ref{Prop:Esti4theta1c}}. Choose \(\mu_0 \defeq \tfrac{\delta}{2}\).

    \item Let \(\epsilon_0=\epsilon_0(\gamma,v_2) > 0\) be from {\rm Lemma~\ref{lem:properties-of-G-alpha-theta}\eqref{lem:properties-of-G-alpha-theta-3}}.

    \item Let $C_1=C_1(\gamma,v_2)>0$ and $C_2^\ast = C_2^\ast (\gamma,v_2,\theta_*)>0$ be from {\rm Lemma~\ref{Lem:UsefulEst4AdmisSolu}}.
    Choose $N_4 \defeq 10C^\ast_2$ and $N_5\defeq 10C_1$.
    \end{itemize}

\smallskip
    \item \label{def:iteration-set-6}
    The density function $\rho(\abs{D\varphi},\varphi)$ defined by \eqref{Eq:DefRhopzSec302}
    satisfies
    \begin{align*}
        \frac{\rho^\ast(\gamma)}{2} < \rho(\abs{D\varphi},\varphi) < 2C_{\rm ub} \qquad
        \text{in $\cl{\Omega} \setminus (\mathcal{D}^5_{\epsilon/10} \cup \mathcal{D}^6_{\epsilon/10})$}
    \end{align*}
    for constants \(\rho^\ast(\gamma)>0\) and \(C_{\rm ub}=C_{\rm ub}(\gamma,v_2) > 0\) given by~\eqref{UB-03} in {\rm Lemma~\ref{UniformBound-Lem}}.

\item  \label{def:iteration-set-5}
    Let the sound speed \(c(\abs{D\varphi},\varphi)\) as defined in \eqref{Def4SonicSpeedInSec3}.
Function \(\varphi\) satisfies
    \begin{align*}
        \frac{\abs{D\varphi(\bm{\xi})}^2}{c^2(\abs{D\varphi(\bm{\xi})},\varphi(\bm{\xi}))}
        <
        1 - \tilde{\mu} \, {\rm dist}^\flat (\bm{\xi}, \Gamma^5_{\rm sonic} \cup \Gamma^6_{\rm sonic} )
        \qquad
\text{for $\bm{\xi} \in \cl{\Omega} \setminus \big(\mathcal{D}^5_{\epsilon/10} \cup \mathcal{D}^6_{\epsilon/10}\big)$}\,,
\end{align*}
where \(\tilde{\mu} \defeq \frac{\mu_{\rm el}}{2}\) for \(\mu_{\rm el} =\frac{\mu}{C_{\flat}}> 0\) with $C_{\flat}=C_{\flat}(\gamma,v_2)>0$
from~\eqref{Eq:FuncgEquivDistb}{\rm,} $\mu=\mu(\gamma,v_2)>0$ is from {\rm Proposition~\ref{prop:ellipticDegeneracyNearSonicBoundary},}
and  ${\rm dist}^{\flat}(\cdot, \cdot)$ is defined in \eqref{eq:dist-flat}.

\smallskip
 \item \label{def:iteration-set-7}
    The iteration boundary value problem
    \begin{align} \label{eq:iteration-BVP}
        \begin{cases}
        \, \mathcal{N}_{(u,\btheta)} (\hat{\phi}) = A_{11} \hat{\phi}_{\xi\xi} + 2 A_{12} \hat{\phi}_{\xi\eta}
        + A_{22} \hat{\phi}_{\eta\eta} = 0 \quad &\quad \text{in $\Omega$}\,, \\
        \, \mathcal{M}_{(u,\btheta)}(D\hat{\phi},\hat{\phi}, \bm{\xi}) = 0 \quad
          &\quad \text{on $\Gamma_{\rm shock}$}\,, \\
        \, \hat{\phi} = \max\{ \varphi_5, \varphi_6\} + \frac12\abs{\bm{\xi}}^2 \quad
          &\quad \text{on $\Gamma^5_{\rm sonic} \cup \Gamma^6_{\rm sonic}$}\,,\\
        \, \partial_\eta \hat{\phi} = 0 \quad &\quad \text{on $\Gamma_{\rm wedge}$}\,,
        \end{cases}
    \end{align}
    has a unique solution \(\hat{\phi} \in C^2(\Omega) \cap C^1(\cl{\Omega}),\)
    where the nonlinear operators \(\mathcal{N}_{(u,\btheta)}\) and \(\mathcal{M}_{(u,\btheta)}\) are determined later.
    Moreover, solution $\hat{\phi}${\rm ,} under the transformation{\rm :}
    \begin{align} \label{eq:def-u-hat}
        \hat{u} \defeq (\hat{\phi} - \frac12\abs{\bm{\xi}}^2 - \vphithetastar) \circ \Futheta \qquad \text{in} \; \qiter\,,
    \end{align}
    satisfies the estimate{\rm :}
    \begin{align}
    \label{eq:iteration-u-uhat-estimate}
        \norm{\hat{u} - u}^{(\ast)}_{2,\alpha,\qiter} < \delta_3\,.
    \end{align}
    \end{enumerate}
\end{definition}
\begin{definition}[Extended iteration set]
Define the extended iteration set \(\mathcal{K}^{\rm ext}\) as
\begin{align*}
\mathcal{K}^{\rm ext}
\defeq \big\{ (u,\btheta) \in  C^{2,\alpha}_{(\ast)}(\qiter) \times [0,\theta_\ast]^2 \,:\, (u,\btheta)
\text{ \rm satisfies {\rm Definition~\ref{def:iteration-set}\eqref{def:iteration-set-1}--\eqref{def:iteration-set-6}}} \big\} \,.
\end{align*}
Write \(\cl{\mathcal{K}}\) and \(\cl{\mathcal{K}^{\rm ext}}\) for the closures
of \(\mathcal{K}\) and \(\mathcal{K}^{\rm ext}\)
in  \(C^{2,\alpha}_{(\ast)}(\qiter) \times [0,\theta_\ast]^2\) respectively.
\end{definition}

To complete Definition~\ref{def:iteration-set}, it remains to construct a suitable iteration boundary value
problem~\eqref{eq:iteration-BVP} through the definition of the nonlinear differential
operators \(\mathcal{N}_{(u,\btheta)}\) and \(\mathcal{M}_{(u,\btheta)}\) for each \((u,\btheta) \in \cl{\mathcal{K}^{\rm ext}}\).
The construction of these operators is similar to~\cite[\S4.4]{BCF-2019}.

In the following, we write
\(\Gamma_{\rm sonic} \defeq \Gamma_{\rm sonic}^5 \cup \Gamma_{\rm sonic}^6\)
and \(\mathcal{D}_r \defeq \mathcal{D}_r^5 \cup \mathcal{D}_r^6\) for any \(r>0\)
with \(\mathcal{D}^5_r\) and \(\mathcal{D}^6_r\) given by~\eqref{eq:domain-D5n6epsilon}, respectively.
Moreover, for any \((u,\btheta) \in \cl{\mathcal{K}^{\rm ext}}\), we define
\begin{align*}
    \phi \defeq \varphi + \frac12\abs{\bm{\xi}}^2\,,
    \qquad
    \psi \defeq (\varphi - \vphithetastar) \circ \rcal^{-1}
    \equiv
    (\phi - \phi_\btheta^\ast) \circ \rcal^{-1} \,,
\end{align*}
where \(\varphi = \varphi^{(u,\btheta)}\) is given by Definition~\ref{def:u-to-phi}, \(\vphithetastar\) is given by Definition~\ref{def:varphi-theta-star}, \(\phi_\btheta^\ast \defeq \vphithetastar + \frac12\abs{\bm{\xi}}^2\),
and coordinates \((x,y) = \rcal(\bm{\xi})\) are given by~\eqref{Eq:DefPolarCoordiSec4}.

\medskip
\noindent \textbf{Construction of \(\mathcal{N}_{(u,\btheta)}\) in~\eqref{eq:iteration-BVP}.}
For \((u,\btheta) \in \cl{\mathcal{K}^{\rm ext}}\),
the nonlinear operator \(\mathcal{N}_{(u,\btheta)}\) in~\eqref{eq:iteration-BVP} is defined by
\begin{equation} \label{eq:nonlinear-operator-N}
    \mathcal{N}_{(u,\btheta)}(\hat{\phi}) \defeq \sum_{i,j=1}^{2} A_{ij}(D\hat{\phi},\bm{\xi}) \partial_{i} \partial_{j} \hat{\phi}
    \qquad\, \mbox{for any \(\hat{\phi}\in C^2(\Omega)\)}\,,
\end{equation}
with \((\partial_1,\partial_2) \defeq (\partial_\xi, \partial_\eta)\)
and coefficients \(A_{ij}(\bm{p},\bm{\xi})\), \(i,j = 1,2\), determined below such that
equation \(\mathcal{N}_{(u,\btheta)}(\hat{\phi}) = 0\)
is strictly elliptic in \(\cl{\Omega} \setminus \Gamma_{\rm sonic} \) and,
if \(\phi\) solves~\eqref{eq:iteration-BVP}, then
\(\mathcal{N}_{(u,\btheta)}(\phi)=0\) coincides with~\eqref{Eq4phi}.
Coefficients \(A_{ij}(\bm{p},\bm{\xi})\) are constructed
similarly to~\cite[\S4.4.1]{BCF-2019}, the main difference being that both sonic boundaries
may now degenerate into single points.
We briefly describe the construction in the following four steps
and highlight where the differences arise.

\smallskip
\textbf{1. Away from \(\Gamma_{\rm sonic}\).}
Let \(\epsilon_0 > 0 \) be the constant from Lemma~\ref{lem:properties-of-G-alpha-theta}\eqref{lem:properties-of-G-alpha-theta-3},
and let \(\epsilon_{\rm eq} \in (0,\frac{ \epsilon_0}{2})\) be chosen later.
Away from \(\Gamma_{\rm sonic}\), coefficients \(A_{ij}(\bm{p},\bm{\xi})\) are independent of \(\bm{p}\).
Indeed, for \(i, j = 1, 2\), and \(\bm{\xi}\in \Omega \setminus \mathcal{D}_{\epsilon_{\rm eq}/10}\), we define \(A^{(1)}_{ij}(\bm{\xi})\) by
\begin{align} \label{eq:coefficient-A^(1)_ij}
    A^{(1)}_{ij} (\bm{\xi}) \defeq A^{\rm pot}_{ij} (D\phi(\bm{\xi}),\phi(\bm{\xi}),\bm{\xi})\,,
\end{align}
where
$A^{\rm pot}_{ij} (\bm{p} , z , \bm{\xi}) \defeq \delta_{ij} c^2 - (p_i - \bm{\xi}\cdot \bm{e}_i) (p_j - \bm{\xi}\cdot \bm{e}_j)$
for \(c = c( \abs{\bm{p} - \bm{\xi}}, z - \frac12 \abs{\bm{\xi}}^2 )\)
given in Definition~\ref{def:iteration-set}\eqref{def:iteration-set-5}, where $\bm{e}_1 \defeq (1,0)$ and $\bm{e}_2 \defeq (0,1)$.

\smallskip
\textbf{2. Supersonic near \(\Gamma_{\rm sonic}\).}
For constant \(\mu_0 > 0\) from Definition~\ref{def:iteration-set}(\ref{def:iteration-set-4}),
fix an odd function of a single variable \(\zeta_1 \in C^3(\mathbb{R})\) such that,
for all \(s \in \mathbb{R}\),
\begin{align*}
 \zeta_1(s) = \begin{cases}
        s &\quad \text{if} \; \abs{s} \leq \frac{2 - \frac{\mu_0}{5}}{1 + \gamma}\,,\\
        \frac{(2 - \frac{\mu_0}{10}) \sgn (s)}{1 + \gamma} &\quad \text{if} \; \abs{s} > \frac{2}{1+ \gamma}\,,
    \end{cases}
    \quad
    0 \leq \zeta_1'(s) \leq 10\,,
    \quad
    -\frac{20(1+\gamma)}{\mu_0} \leq \zeta_1''(\abs{s}) \leq 0\,.
\end{align*}
For any \(\hat{\phi} \in C^2(\Omega)\), set
\begin{align*}
    \hat{\psi} \defeq \big( \hat{\phi} - ( \vphithetastar + \frac12 \abs{\bm{\xi}}^2 ) \big) \circ \rcal^{-1}
    \equiv \big( \hat{\phi} - \phi_\btheta^\ast \big) \circ \rcal^{-1}
    \qquad \text{in $\rcal ( \Omega \cap \mathcal{D}_{2\epsilon_{\rm eq}}) $}\,,
\end{align*}
and define the nonlinear operator \(\mathcal{N}^{\rm polar}_{(u,\btheta)} (\hat{\phi})\)
in \(\Omega \cap \mathcal{D}_{2 \epsilon_{\rm eq}}\) by
\begin{align*}
\begin{split}
\mathcal{N}^{\rm polar}_{(u,\btheta)} (\hat{\phi})
\defeq
\Big(
\big(2x - (\gamma+1)x \, \zeta_1(\frac{\hat{\psi}_x}{x}) +\co^{\rm mod}_1\big) \hat{\psi}_{xx}
    &+
    \co^{\rm mod}_2 \hat{\psi}_{xy}
    + \big(\frac{1}{c_\btheta}  + \co^{\rm mod}_3\big) \hat{\psi}_{yy} \\
    &\quad
    -{\,} \big(1 + \co^{\rm mod}_4\big) \hat{\psi}_x
    +
    \co^{\rm mod}_5 \hat{\psi}_y \Big) \circ \rcal\,,
\end{split}
\end{align*}
with the modified remainder terms \(\co_j^{\rm mod}\), \(j = 1, \cdots ,5\), given by
\begin{align*}
    \co^{{\rm mod}}_j \defeq
    \co_j \big( x^{3/4} \zeta_1 \big( \frac{\hat{\psi}_x}{x^{3/4}} \big),
        N_4 (\gamma+1) x \,\zeta_1 \big( \frac{\hat{\psi}_y}{N_4 (\gamma+1) x} \big),
        \psi (x,y),
        x,
        c_\btheta
        \big)\,,
\end{align*}
where
\(\co_j (p_x,p_y,z,x,c)\) are given by~\eqref{Eq:Def-co-k},
constant \(N_4 > 0\) is from Definition~\ref{def:iteration-set}(\ref{def:iteration-set-4}),
and \(c_\btheta \defeq c_5 (1 - \chi_\btheta^\ast) + c_6 \chi_\btheta^\ast\) with
function \(\chi_\btheta^\ast\) from Definition~\ref{def:varphi-theta-star}.

For
\((\bm{p},\bm{\xi}) \in \mathbb{R}^2 \times ( \Omega \cap  \mathcal{D}_{2\epsilon_{\rm eq}} )\),
coefficients \(A^{(2)}_{ij}(\bm{p},\bm{\xi}),\, i,j=1,2,\) are defined to be the second-order coefficients
of operator \(c_\btheta \mathcal{N}^{\rm polar}_{(u,\btheta)}(\hat{\phi})\) in the \(\bm{\xi}\)--coordinates,
that is,
\begin{align*}
    c_\btheta \mathcal{N}^{\rm polar}_{(u,\btheta)} (\hat{\phi})
    \eqdef \sum_{i,j = 1}^{2} A^{(2)}_{ij} (D_{\bm{\xi}}\hat{\phi},\bm{\xi}) \partial_{i} \partial_{j} \hat{\phi}
    \qquad\,\,
    \text{in $\Omega \cap \mathcal{D}_{2 \epsilon_{\rm eq}}$}\,.
\end{align*}
In particular, there are no lower order terms in the above expression,
since they vanish identically in \(\Omega \cap \mathcal{D}_{2 \epsilon_{\rm eq}}\),
which can be verified by direct computation.

\smallskip
\textbf{3. Subsonic cut-off near \(\Gamma_{\rm sonic}\).}
For any \(\sigma \in (0,1)\), there exists a constant \(C_\sigma>0\)
depending only on \((\gamma, v_2,\theta_\ast,\sigma)\) such that,
for any \((u,\btheta)\in \cl{\mathcal{K}^{\rm ext}}\), there is a function
\(v_\sigma^{(u,\btheta)}\in C^4(\cl{\Omega})\) with
\begin{enumerate}[\quad(a)]
\item
\(\norm{v_\sigma^{(u,\btheta)} - \phi}_{C^1(\cl{\Omega})}
\leq \sigma^2\) and \(\norm{v_\sigma^{(u,\btheta)}}_{C^4(\cl{\Omega})} \leq C_\sigma\);

\item
For any sequence \(\{(u_k,\btheta_k)\} \subset \cl{\mathcal{K}^{\rm ext}}\)
converging to \((u,\btheta)\) in \(C^{1,\alpha}(\cl{\qiter}) \times [0,\theta_\ast]^2\),
\begin{align*}
    v_\sigma^{(u_k,\btheta_k)} \circ \mathfrak{F}_{(u_k,\btheta_k)} \to v_\sigma^{(u,\btheta)} \circ \Futheta
    \qquad
    \text{in $C^{1,\alpha}(\cl{\qiter})$}\,.
\end{align*}
\end{enumerate}
The proof of this statement can be found in~\cite[Lemma~4.26]{BCF-2019}.

For any \(\sigma > 0\), fix a family of cut-off functions \(\varsigma_\sigma \in C^\infty(\mathbb{R})\) satisfying
\begin{align} \label{eq:varsigma-cut-off}
    \varsigma_\sigma(t) \defeq \varsigma_1 \big(\frac{t}{\sigma} \big) \qquad\,\,\, \text{with} \,\,\,
    \varsigma_1 (t) =
    \begin{cases}
        1 &\,\,\text{for} \; t < 1\,, \\
        0 &\,\, \text{for} \; t > 2\,,
    \end{cases}
    \quad\mbox{and}\,\,\,    0 \leq \varsigma_1 \leq 1 \,\,\, \text{on $\mathbb{R}$}\,.
\end{align}
For a constant \(\sigma_{\rm cf} \in (0,1)\) to be specified later,
we define the subsonic coefficients:
\begin{align*}
    A^{\rm subs}_{ij} (\bm{p},\bm{\xi}) \defeq
    \varsigma_{\sigma_{\rm cf}}
    A^{\rm pot}_{ij}( \bm{p}, \phi(\bm{\xi}) , \bm{\xi} )
    +
    (1 - \varsigma_{\sigma_{\rm cf}} )
    A^{\rm pot}_{ij} ( Dv_{\sigma_{\rm cf}}^{(u,\btheta)}(\bm{\xi}) , \phi(\bm{\xi}), \bm{\xi} )
\qquad\mbox{for \(i, j = 1, 2\)}\,,
\end{align*}
with \(\varsigma_{\sigma_{\rm cf}} = \varsigma_{\sigma_{\rm cf}} (\abs{\bm{p} - Dv_{\sigma_{\rm cf}}^{(u,\btheta)}(\bm{\xi})})\).
In particular, from now on,  we choose \(\sigma_{\rm cf} \defeq \sqrt{\delta_1}\).

Fix a cut-off function \(\chi_{\rm eq} \in C^\infty(\mathbb{R})\) satisfying
\begin{align*}
    \chi_{\rm eq}(\theta) =
    \begin{cases}
    1  &\,\,\,\text{for} \; \theta \leq \theta^{\rm s} + \frac{\sigma_3}{4}\,, \\
    0  &\,\,\, \text{for} \; \theta \geq \theta^{\rm s} + \frac{\sigma_3}{2}\,,
    \end{cases}
    \qquad \chi'_{\rm eq}(\theta) \leq 0 \,\,\,\, \text{on $\mathbb{R}$}\,.
\end{align*}
Then, for
\((\bm{p},\bm{\xi}) \in \mathbb{R}^2 \times ( \Omega \cap  \mathcal{D}_{2\epsilon_{\rm eq}} )\),
coefficients \(A^{(3)}_{ij}(\bm{p},\bm{\xi}),\,i,j=1,2,\) are defined by
\begin{align*}
A^{(3)}_{ij}(\bm{p},\bm{\xi}) \defeq
    \chi_{\rm eq}(\theta_{l-4}) A_{ij}^{(2)}(\bm{p},\bm{\xi})
    + (1 - \chi_{\rm eq}(\theta_{l-4})) A_{ij}^{{\rm subs}}(\bm{p},\bm{\xi})
    \qquad \text{for $\bm{\xi} \in \Omega \cap \mathcal{D}^{l}_{2 \epsilon_{\rm eq}},\, l=5,6$}\,.
\end{align*}

\smallskip
\textbf{4. Coefficients \(A_{ij}(\bm{p},\bm{\xi})\). }
We now complete the definition of coefficients \(A_{ij}(\bm{p},\bm{\xi})\)
for the nonlinear differential operator \(\mathcal{N}_{(u,\btheta)}\) given by~\eqref{eq:nonlinear-operator-N}.
Let \(\epsilon_0\) be from Lemma~\ref{lem:properties-of-G-alpha-theta}\eqref{lem:properties-of-G-alpha-theta-3}.
For each \(\epsilon \in (0, \tfrac{ \epsilon_0}{2})\), fix a family of cut-off functions
\(\zeta_2^{(\epsilon,\btheta)} \in C^4(\cl{Q^{\btheta}})\) with the following properties:
\begin{enumerate}[\quad(a)]
    \item
    Function
    \(\zeta_2^{(\epsilon, \, \mathclap{\cdot} \,)} (\cdot): (\bm{\xi}, \btheta) \mapsto \zeta_2^{(\epsilon,\btheta)}(\bm{\xi}) \)
    is continuous on \(\cup_{\btheta \in [0, \theta_\ast]^2} Q^\btheta \times \{ \btheta \}\);

    \item
     There exists a constant \(C_\epsilon > 0\)  depending only on \((\gamma, v_2,\epsilon)\)
     such that \(\norm{\zeta_2^{(\epsilon,\btheta)}}_{C^4(\cl{Q^{\btheta}})} \leq C_\epsilon\);

     \item
     $\,
    \zeta_2^{(\epsilon,\btheta)}(\bm{\xi}) = \begin{cases}
    1 &\quad \text{for} \; \bm{\xi} \in \Omega \setminus \mathcal{D}_{\epsilon}\,, \\
    0 &\quad \text{for} \; \bm{\xi} \in \Omega \cap \mathcal{D}_{\epsilon/2} \,.
    \end{cases}
$
\end{enumerate}
Such a family of cut-off functions is constructed in~\cite[Definition~4.28]{BCF-2019}.

Coefficients \(A_{ij}(\bm{p},\bm{\xi}), \, i,j=1,2,\) are defined by
\begin{align} \label{eq:def-of-Aij}
    A_{ij}(\bm{p},\bm{\xi}) \defeq \zeta_2^{(\epsilon_{\rm eq},\btheta)}(\bm{\xi}) A_{ij}^{(1)}(\bm{\xi})
    + (1 - \zeta_2^{(\epsilon_{\rm eq},\btheta)}(\bm{\xi}) ) A_{ij}^{(3)}(\bm{p},\bm{\xi})
    \qquad\mbox{for \((\bm{p},\bm{\xi}) \in \mathbb{R}^2 \times \Omega\)}\,.
\end{align}
We have the following lemma concerning the properties of coefficients \(A_{ij}\); see~\cite[Lemma~4.30]{BCF-2019}.

\begin{lemma}\label{Lem:Property4ModifyEq}
There exist positive constants
\(\epsilon^{(1)}, \delta_1^{(1)}, \epsilon_{\rm eq} \in (0, \frac{\epsilon_0}{2}),
\lambda_0 \in (0,1),
N_{\rm eq} \geq 1,\)
and \(C > 0\)
with \((\epsilon^{(1)}, \delta_1^{(1)}, \lambda_0)\) depending only on \((\gamma, v_2),\)
\((\epsilon_{\rm eq}, N_{\rm eq})\) depending only on \((\gamma, v_2, \theta_{\ast}),\)
and \(C > 0\) depending only on \((\gamma, v_2, \theta_\ast, \alpha)\)
such that,
whenever \((\epsilon, \delta_1)\) from {\rm Definition~\ref{def:iteration-set}}
are chosen from \((0,\epsilon^{(1)}]\times (0,\delta_1^{(1)}],\)
for each \((u,\btheta)\in\cl{\mathcal{K}^{\rm ext}},\) coefficients \(A_{ij}(\bm{p},\bm{\xi})\)
of \(\mathcal{N}_{(u,\btheta)}\) given by~\eqref{eq:def-of-Aij} satisfy the following properties{\rm :}
\begin{enumerate}[{\rm (i)}]
 \item \label{item1-Lem:Property4ModifyEq}
 \(A_{12}(\bm{p},\bm{\xi}) = A_{21}(\bm{p},\bm{\xi})\)
 for all \((\bm{p},\bm{\xi}) \in \mathbb{R}^2 \times \Omega\){\rm ,}
 and \(A_{ij}(\bm{p},\bm{\xi}) = A_{ij}^{(1)}(\bm{\xi})\)
 if \(\bm{\xi}\in \Omega \setminus \mathcal{D}_{\epsilon_{\rm eq}}\).
 Furthermore, \(D_{\bm{p}}^m A_{ij} \in C^{1,\alpha} (\mathbb{R}^2 \times (\cl{\Omega} \setminus \Gamma_{\rm sonic}))\)
 for \(m = 0, 1, 2,\)
 \begin{equation*}
    \lambda_0 \, {\rm dist}(\bm{\xi}, \,\Gamma_{\rm sonic} ) \abs{\bm{\kappa}}^2 \leq
        \sum_{i,j = 1}^2 A_{ij}(\bm{p},\bm{\xi}) \kappa_i \kappa_j \leq \lambda_0^{-1} \abs{\bm{\kappa}}^2
    \qquad\,\mbox{for any \(\bm{\kappa}=(\kappa_1, \kappa_2) \in \mathbb{R}^2\)}\,,\\
\end{equation*}
for any $(\bm{p},\bm{\xi}) \in \mathbb{R}^2\times \Omega,$ and
\begin{equation*}
\begin{aligned}
   &\norm{A_{ij}}_{L^\infty(\mathbb{R}^2 \times \Omega)} + \norm{A_{ij}(\bm{p}, \cdot)}_{C^{3/4}(\cl{\Omega})}
   + \norm{ D_{\bm{p}} A_{ij} (\bm{p}, \cdot)}_{L^{\infty}(\Omega)} \leq N_{\rm eq}\qquad
   \mbox{for any $\bm{p}\in\mathbb{R}^2$}\,, \\
    &\norm{A_{ij}}_{C^{1,\alpha}(\mathbb{R}^2\times\,\cl{\Omega \setminus \mathcal{D}_{\epsilon_{\rm eq}}}\,)} +
        s^5\norm{D_{\bm{p}}^m A_{ij}}_{C^{1,\alpha}(\mathbb{R}^2 \times (\cl{\Omega} \setminus  \mathcal{N}_s (\Gamma_{\rm sonic}) ) )}
         \leq C \qquad \text{for each $s \in (0, \frac{\epsilon_0}{2})$}\,.
 \end{aligned}
 \end{equation*}

 \item \label{item2-Lem:Property4ModifyEq}
 \(A_{ij} (\bm{p}, \bm{\xi} ) = A_{ij}^{(3)} (\bm{p},\bm{\xi})\)
  for any
 \((\bm{p},\bm{\xi}) \in \mathbb{R}^2 \times (\Omega \cap \mathcal{D}_{\epsilon_{\rm eq}/2}^l),\, l=5,6\)\,.
  Furthermore, if \(\theta_{l-4} \leq \theta^{\rm s} + \frac{\sigma_3}{4},\)
 then \(A_{ij}(\bm{p},\bm{\xi}) = A^{(2)}_{ij}(\bm{p},\bm{\xi})\){\rm;}
 whereas, if \(\theta_{l-4} \in [\theta^{\rm s} + \delta, \theta_\ast]\)
 for some \(\delta \in (0,\frac{\sigma_3}{2}),\)
 then
 \begin{equation*}\begin{aligned}
  &\quad\lambda_0 ( {\rm dist}(\bm{\xi}, \,\Gamma_{\rm sonic}) + \delta ) \abs{\bm{\kappa}}^2
    \leq \sum_{i,j=1}^2 A_{ij}(\bm{p}, \bm{\xi}) \kappa_i \kappa_j
     \leq \lambda_0^{-1} \abs{\bm{\kappa}}^2
     \qquad\mbox{for any $\bm{\xi}\in\Omega\cap\mathcal{D}^{l}_{\epsilon_{\rm eq}/2}$}\,, \\
 &\quad\sup_{\bm{p}\in\mathbb{R}^2} \norm{D_{\bm{p}}^m A_{ij}(\bm{p}, \cdot)}_{C^{1,\alpha}(\cl{\Omega \cap\mathcal{D}^{l}_{\epsilon_{\rm eq}/2}})}
 \leq C \qquad \text{for $m = 0, 1, 2$}\,.
 \end{aligned}\end{equation*}

\item \label{item3-Lem:Property4ModifyEq}
Suppose that \(\epsilon\) from {\rm Definition~\ref{def:iteration-set}}
satisfies \(\epsilon \in (0, \frac{\epsilon_{\rm eq}}{2})\).
Then the equation{\rm :} \(\mathcal{N}_{(u,\btheta)}(\phi) = 0\)
coincides with~\eqref{Eq4phi} in \(\Omega \setminus \mathcal{D}_{\epsilon/10}\).
In addition, for \(l = 5, 6,\) if \(x_{P_0^{l-4}} \geq \frac{\epsilon}{10}\)
or \(\theta_{l-4} \geq \theta^{\rm s} + \frac{\sigma_3}{2},\)
then the equation{\rm :} \(\mathcal{N}_{(u,\btheta)}(\phi) = 0\) coincides with~\eqref{Eq4phi} in \(\Omega \setminus \mathcal{D}^{11-l}_{\epsilon/10}\).
\end{enumerate}
\end{lemma}

\medskip
\noindent \textbf{Construction of \(\mathcal{M}_{(u,\btheta)}\) in~\eqref{eq:iteration-BVP}.}
We turn our attention to the definition of the nonlinear linear operator
\(\mathcal{M}_{(u,\btheta)}(\bm{p},z,\bm{\xi})\) in~\eqref{eq:iteration-BVP}.
The construction is presented in the following four steps that
follow closely~\cite[\S4.4.2]{BCF-2019}.

\smallskip
\textbf{1.}
For \(g^{\rm sh}\) given by~\eqref{Def4FreeBdryFunc}, set
\begin{align*}
    \mathcal{M}_0(\bm{p}, z, \bm{\xi}) \defeq g^{\rm sh}( \bm{p} - \bm{\xi}, z - \frac12\abs{\bm{\xi}}^2, \bm{\xi} )\,,
\qquad
    \mathcal{M}_1(\bm{p}, z, \xi) \defeq \mathcal{M}_0 (\bm{p}, z, (\xi, \frac{z}{v_2}) )\,.
\end{align*}
For \(j=5,6\), write \(\phi_j \defeq \varphi_j + \frac12\abs{\bm{\xi}}^2\),
 which is independent of the \(\eta\)--coordinate since \(O_j \in \{\eta = 0\}\).
Then \((\xi,\frac{1}{v_2} \phi_j(\bm{\xi})) \in \{\varphi_j = \varphi_2\}\)
for any \(\bm{\xi}  \in \mathbb{R}^2\), and \(\mathcal{M}_1\) is homogeneous in the sense that
\begin{align*}
    \mathcal{M}_1 ( D \phi_j (\bm{\xi}), \phi_j (\bm{\xi}), \xi ) = 0 \qquad
    \text{for all $\bm{\xi} = (\xi,\eta) \in \mathbb{R}^2$ and $j=5,6$}\,.
\end{align*}
Furthermore, observe that \(\bm{\xi} = (\xi, \eta) \in \Gamma_{\rm shock}\)
if and only if \(\bm{\xi} \in \cl{\Omega}\) with \(\eta = \frac{1}{v_2}\phi(\bm{\xi})\), and hence
\begin{align} \label{eq:M0-and-M1-relation-gamma-shock}
    \begin{cases}
    \mathcal{M}_1 ( D \phi(\bm{\xi}), \phi(\bm{\xi}), \xi )
    = \mathcal{M}_0 ( D\phi(\bm{\xi}), \phi(\bm{\xi}), \bm{\xi} ) \, , \\[1mm]
    D_{\bm{p}} \mathcal{M}_1 ( D \phi(\bm{\xi}), \phi(\bm{\xi}), \xi) = D_{\bm{p}} \mathcal{M}_0 ( D\phi(\bm{\xi}), \phi(\bm{\xi}), \bm{\xi} ) \, ,
    \end{cases}
    \quad \text{for} \; \bm{\xi} \in \Gamma_{\rm shock}\,.
\end{align}

\smallskip
\textbf{2.}
For the cut-off function \(\varsigma_\sigma\) given by~\eqref{eq:varsigma-cut-off}
and a constant \(\sigma_{\rm bc} > 0\) to be determined later, we set
\begin{align} \label{eq:def-BC-M}
\begin{split}
\mathcal{M}(\bm{p}, z,\bm{\xi})
    \defeq &\,\big( \varsigma_{\sigma_{\rm bc}}^{(5)} \varsigma_{\sigma_{\rm bc}}^{(6)}
      + ( 1 - \varsigma_{\sigma_{\rm bc}}^{(5)}) \varsigma_{\sigma_{\rm bc}}^{(6)}
       + \varsigma_{\sigma_{\rm bc}}^{(5)} ( 1 - \varsigma_{\sigma_{\rm bc}}^{(6)})\big)\mathcal{M}_1 (\bm{p}, z, \xi)\\
    &\;+ ( 1- \varsigma_{\sigma_{\rm bc}}^{(5)})(1 - \varsigma_{\sigma_{\rm bc}}^{(6)})
    \mathcal{M}_0 (\bm{p}, z, \bm{\xi})\,,
\end{split}
\end{align}
where \(\varsigma_{\sigma_{\rm bc}}^{(j)}
\defeq \varsigma_{\sigma_{\rm bc}}( \abs{ (\bm{p},z) - (D\phi_j, \phi_j(\bm{\xi}))  }  )\) for \(j = 5, 6 \).
There exist constants \(\sigma_{\rm bc}, \delta_{\rm bc}, \bar{\epsilon}_{\rm bc} > 0\)
depending only on \((\gamma, v_2)\) such that,
if \(\epsilon\) from Definition~\ref{def:iteration-set}
satisfies \(\epsilon \in (0, \bar{\epsilon}_{\rm bc})\), then
\(\mathcal{M}(\bm{p},z,\bm{\xi})\) satisfies that, for all \(\bm{\xi} \in \Gamma_{\rm shock}\),
\begin{align*}
    \delta_{\rm bc} \leq D_{\bm{p}} \mathcal{M} (D\phi(\bm{\xi}),\phi(\bm{\xi}),\bm{\xi}) \cdot \bm{\nu}_{\rm sh}(\bm{\xi})
    \leq \delta_{\rm bc}^{-1}\,,\qquad
    D_z \mathcal{M} (D\phi(\bm{\xi}),\phi(\bm{\xi}),\bm{\xi}) \leq - \delta_{\rm bc}\,,
\end{align*}
where \(\bm{\nu}_{\rm sh}\) is the unit normal vector to
\(\Gamma_{\rm shock}\) towards the interior of \(\Omega\).
These results follow from the calculations in~\cite[Lemma~4.32]{BCF-2019}.
Indeed, the first follows from a direct calculation
by using~\eqref{eq:M0-and-M1-relation-gamma-shock}, \cite[Eq.~A.18]{BCF-2019},
the fact that \(\mathcal{M}_0\) is \(C^1\),
and the properties
in Definition~\ref{def:iteration-set}(\ref{def:iteration-set-5})--(\ref{def:iteration-set-6}).
For the second, direct computation gives
\begin{align*}
    D_z \mathcal{M}_0 (D \phi(\bm{\xi}), \phi(\bm{\xi}), \bm{\xi} )
    &= - \rho^{2-\gamma} D \varphi \cdot \bm{\nu}_{\rm sh} (\bm{\xi}) < 0\qquad \text{for $\bm{\xi} \in \Gamma_{\rm shock}$}\,,
\end{align*}
whilst, for \(j = 5,6\), and any \(\bm{\xi} = (\xi,\eta) \in Q^\btheta\), we have
\begin{align*}
    D_z \mathcal{M}_1 ( D\phi_j(\bm{\xi}), \phi_j(\bm{\xi}), \xi)
    &= -\rho_j^{2-\gamma}
    \, {\rm dist}(O_j, S_{2j}) + (-1)^j \frac{\rho_j - 1}{v_2} \cos{\theta_{2j}} < 0\,,
\end{align*}
after which an argument similar to~\cite[Lemma~4.32]{BCF-2019} applies.
The constant, \(\sigma_{\rm bc} > 0\), is now fixed.

\smallskip
\textbf{3.}
For \(j = 5,6\), set
\begin{align} \label{eq:def-M-j}
    \mathcal{M}^{(j)} (\bm{q} , z , \bm{\xi} ) \defeq \mathcal{M} ( \bm{q} + D \phi_j , z + \phi_j(\bm{\xi}) , \bm{\xi} )\,.
\end{align}
Then, for the coordinates  \((x,y) = \bm{x} = \rcal(\bm{\xi})\) given by~\eqref{Eq:DefPolarCoordiSec4}, define
\begin{align} \label{eq:def-M-hat-j}
\hat{\mathcal{M}}^{(j)} ( q_x, q_y, z, \bm{x} )
 \defeq \mathcal{M}^{(j)} \big( (-1)^{j}( q_x \cos{y} + \frac{q_y \sin{y}}{c_j - x}), -q_x \sin{y} + \frac{q_y \cos{y}}{c_j - x} , z , \rcal^{-1}(\bm{x}) \big)\,.
\end{align}

\smallskip
\textbf{4.}
Finally, we extend the definition of \(\mathcal{M}\) in~\eqref{eq:def-BC-M}
to all \((\bm{p},z, \bm{\xi}) \in \mathbb{R}^2\times \mathbb{R} \times \cl{\Omega}\).
For each \((u,\btheta) \in \cl{\mathcal{K}^{\rm ext}}\) and a constant \(\sigma >0\), let \(v^{(u,\btheta)}_{\sigma} \in C^4(\cl{\Omega})\)
be given as in Step~{\bf3} in the construction of \(\mathcal{N}_{(u,\btheta)}\)
above.
We define a linear operator \(\mathcal{L}_\sigma^{(u,\btheta)} (\bm{p},z,\bm{\xi})\) by
\begin{align*}
    \begin{split}
    \mathcal{L}_\sigma^{(u,\btheta)} (\bm{p},z,\bm{\xi})
    \defeq & \,\mathcal{M} (
    Dv_\sigma^{(u,\btheta)} (\bm{\xi}) ,
    v_\sigma^{(u,\btheta)}(\bm{\xi}) ,
    \bm{\xi}
    )
    +
    D_{\bm{p}} \mathcal{M} (
    Dv_\sigma^{(u,\btheta)} (\bm{\xi}) ,
    v_\sigma^{(u,\btheta)}(\bm{\xi}) ,
    \bm{\xi}
    ) \cdot \bm{p} \\
    &\;\; +
    D_z \mathcal{M} (
    Dv_\sigma^{(u,\btheta)} (\bm{\xi}) ,
    v_\sigma^{(u,\btheta)}(\bm{\xi}) ,
    \bm{\xi}
    ) z\,,
    \end{split}
\end{align*}
The operator, \(\mathcal{M}_{(u,\btheta)}(\bm{p},z,\bm{\xi})\), in~\eqref{eq:iteration-BVP}
is then defined by
\begin{align} \label{eq:def-of-M-u-theta}
    \mathcal{M}_{(u,\btheta)}(\bm{p},z,\bm{\xi})
    \defeq \varsigma_\sigma \mathcal{M}(\bm{p},z,\bm{\xi}) + (1 - \varsigma_\sigma)
    \mathcal{L}^{(u,\btheta)}_\sigma (\bm{p} - D v^{(u,\btheta)}_\sigma(\bm{\xi}) , z - v^{(u,\btheta)}_\sigma (\bm{\xi}), \bm{\xi} )
   \end{align}
for \(\varsigma_\sigma \defeq \varsigma_\sigma ( \abs{ (\bm{p},z) - (D v^{(u,\btheta)}_\sigma(\bm{\xi}), v^{(u,\btheta)}_\sigma (\bm{\xi}) ) } )\)
given by~\eqref{eq:varsigma-cut-off}.
We have the following lemma concerning the properties of the nonlinear operator \(\mathcal{M}_{(u,\btheta)} (\bm{p},z,\bm{\xi}) \);
see~\cite[Lemma~4.34]{BCF-2019}.

\begin{lemma}
\label{Lem:Property4ModifyBdryCondi}
Let \(\bar{\epsilon}_{\rm bc}\) be the constant from {\rm Step \textbf{2}} above.
Then there exist positive constants \(\epsilon_{\rm bc} \in (0,\bar{\epsilon}_{\rm bc}),
\delta_1^{(2)}, N^{(1)}_1, \delta_{\rm bc}, C, C_{\theta_\ast},\)
and \(\epsilon_{\mathcal{M}}\in (0, \epsilon_{\rm bc}]\)
with \((\epsilon_{\rm bc},\delta^{(2)}_1, N^{(1)}_1, \delta_{\rm bc}, C)\)
depending only on \((\gamma, v_2),\)
\(\epsilon_\mathcal{M}\) depending on \((\gamma, v_2,\theta_\ast),\)
and \(C_{\theta_\ast}\) depending on \((\gamma, v_2, \theta_\ast, \alpha) \) such that,
whenever parameters \((\epsilon, \delta_1, N_1)\)
from {\rm Definition~\ref{def:iteration-set}} are chosen from
\((0, \bar{\epsilon}_{\rm bc}] \times (0, \delta_1^{(2)}] \times [N_1^{(1)}, \infty),\)
the following statements hold{\rm :}
For each \((u,\btheta) \in \cl{\mathcal{K}^{\rm ext}},\)
operator \(\mathcal{M}_{(u,\btheta)} : \mathbb{R}^2 \times \mathbb{R} \times \cl{\Omega} \to \mathbb{R}\)
given by~\eqref{eq:def-of-M-u-theta} with \(\sigma =\sigma_{\rm cf} \equiv \sqrt{\delta_1}\) satisfies
\begin{enumerate}[{\rm (i)}]
\item \label{item1-Lem:Property4ModifyBdryCondi}
$\mathcal{M}_{(u,\btheta)}\in C^3(\mathbb{R}^2\times\mathbb{R}\times\overline{\Omega})$ and,
for all $(\bm{p},z,\bm{\xi})\in\mathbb{R}^2\times\mathbb{R}\times\overline{\Omega}$ and $\bm{\xi}_{\rm sh}\in\overline{\Gamma_{\rm shock}},$
\begin{align*}
&\norm{ ( \mathcal{M}_{(u,\btheta)}(\bm{0}, 0, \cdot),
 D^m_{(\bm{p},z)} \mathcal{M}_{(u,\btheta)}(\bm{p},z,\cdot) ) }_{C^3(\cl{\Omega})}
 \leq C_{\theta_\ast} \qquad \text{for $m = 1, 2, 3\,$}, \\
        &\mathcal{M}_{(u,\btheta)}(\bm{p}, z  ,\bm{\xi}) = \mathcal{M}(\bm{p},z,\bm{\xi})
        \qquad \text{for $\left|(\bm{p},z)-(D\phi(\bm{\xi}),\phi(\bm{\xi}))\right| < \frac12 {\sqrt{\delta_1}}$}\,,\\
        &|D_{(\bm{p}, z)}\mathcal{M}_{(u,\btheta)}(\bm{p},z,\bm{\xi})-D_{(\bm{p}, z)}
        \mathcal{M}(D\phi(\bm{\xi}),\phi(\bm{\xi}),\bm{\xi})|\leq C\sqrt{\delta_1}\,,\\
        &\delta_{\rm bc}\leq D_{\bm{p}}\mathcal{M}_{(u,\btheta)}(\bm{p},z,\bm{\xi}_{\rm sh})\cdot\bm{\nu}_{\rm sh} \leq \frac{1}{\delta_{\rm bc}},
        \qquad D_z\mathcal{M}_{(u,\btheta)}(\bm{p},z,\bm{\xi}_{\rm sh})\leq-\delta_{\rm bc}\,,
    \end{align*}
    where $\mathcal{M}(\bm{p},z,\bm{\xi})$ is defined by~\eqref{eq:def-BC-M}{\rm,}
    and $\bm{\nu}_{\rm sh}$ is the unit normal vector to $\Gamma_{\rm shock}$ pointing into $\Omega$.

    \smallskip
    \item \label{item2-Lem:Property4ModifyBdryCondi}
    Denote \(\mathcal{B}_{\sigma,\Gamma_{\rm shock}}^{(u,\btheta)}(\bm{p},z,\bm{\xi})
    \defeq \mathcal{L}_{\sigma}^{(u,\btheta)}(\bm{p}-Dv_\sigma^{(u,\btheta)}(\bm{\xi}),z-v_\sigma^{(u,\btheta)}(\bm{\xi}),\bm{\xi})\) and set
    \[\mathcal{B}_{\sigma,\Gamma_{\rm shock}}^{(u,\btheta)}(\bm{p},z,\bm{\xi}) \eqdef
    (b_1^{\rm sh}(\bm{\xi}),b_2^{\rm sh}(\bm{\xi}),b_0^{\rm sh}(\bm{\xi}),h^{\rm sh}(\bm{\xi}))\cdot(p_1,p_2,z,1)\,.\]
    Then
    $\norm{(b_i^{\rm sh},h^{\rm sh})}_{C^{3}(\cl{\Gamma_{\rm shock}})} \leq C_{\theta_*}$
    for $i=0,1,2$,
    and, for all \((\bm{p},z,\bm{\xi})\in\mathbb{R}^2\times\mathbb{R}\times\overline{\Omega},\)
    \begin{equation*}\begin{aligned}
    &\|\big(\mathcal{M}_{(u,\btheta)}-\mathcal{B}_{\sigma_{\rm cf},\Gamma_{\rm shock}}\big)(\bm{p},z,\bm{\xi})\|\leq C\sqrt{\delta_1}\,
    \tabs{
    (\bm{p},z) - \big(Dv_{\sigma_{\rm cf}}^{(u,\btheta)}(\bm{\xi}),v_{\sigma_{\rm cf}}^{(u,\btheta)}(\bm{\xi})\big)
    } \,,\\
    &\left|D_{(\bm{p},z)}\big(\mathcal{M}_{(u,\btheta)}-\mathcal{B}_{{\sigma_{\rm cf}},\Gamma_{\rm shock}}\big)(\bm{p},z,\bm{\xi})\right|\leq C\sqrt{\delta_1}\,.
    \end{aligned} \end{equation*}

    \item \label{item3-Lem:Property4ModifyBdryCondi}
    $\mathcal{M}_{(u,\btheta)}$ is homogeneous in the sense that, for \(j = 5, 6,\)
    \begin{align*}
      \qquad  \mathcal{M}_{(u,\btheta)} ( D \phi_j (\bm{\xi}), \phi_j (\bm{\xi}), \bm{\xi} ) = 0
        \qquad \text{for all} \;
        \begin{cases}
            \, \bm{\xi} \in \Gamma_{\rm shock} \cap \mathcal{D}_{\epsilon_{\mathcal{M}}}
               &\,\, \text{for any $\btheta \in [0,\theta_\ast]^2$}\,,\\
            \, \bm{\xi} \in \Gamma_{\rm shock}
              &\,\, \text{when $\max\{\theta_1,\theta_2\} \in [0, \frac{\delta_1}{N_1}]$}\,.
        \end{cases}
    \end{align*}

    \item \label{item4-Lem:Property4ModifyBdryCondi}
    For coordinates $(x,y) = \rcal(\bm{\xi})$ given by~\eqref{Eq:DefPolarCoordiSec4}{\rm,}
    \(j = 5, 6,\) and \(\bm{\xi} \in \cl{\Gamma_{\rm shock} \cap \mathcal{D}^j_{\epsilon_{\rm bc}}},\)
    define \(\hat{\mathcal{M}}_{(u,\btheta)}^{(j)}(q_x,q_y,z,x,y)\) by~\eqref{eq:def-M-hat-j} with \(\mathcal{M}\)
    replaced by \(\mathcal{M}_{(u,\btheta)}\) in~\eqref{eq:def-M-j}.
    Then, for \(j = 5,6,\) \(\hat{\mathcal{M}}_{(u,\btheta)}^{(j)}\) satisfies the following properties
    when \(\Gamma_{\rm shock} \cap \mathcal{D}^j_{\epsilon_{\rm bc}}\) is non-empty{\rm:}
    \begin{align*}
    &\norm{
    \hat{\mathcal{M}}_{(u,\btheta)}^{(j)}
    }_{ C^3( \mathbb{R}^2 \times \mathbb{R} \times \rcal(\cl{ \Gamma_{\rm shock} \cap \mathcal{D}^j_{ \epsilon_{\rm bc}}}))}
    \leq C_{\theta_\ast} \,,\\
    &\hat{\mathcal{M}}^{(j)}_{(u,\btheta)} (\bm{q},z,x,y)
       = \hat{\mathcal{M}}^{(j)}(\bm{q},z,x,y) \qquad\mbox{in
       $\rcal(\Gamma_{\rm shock} \cap \mathcal{D}^j_{\epsilon_{\rm bc}})\,$
       for any \(\abs{(\bm{q},z)} \leq \frac{\delta_{\rm bc}}{C} \,\)}{\rm ,}\\
    &    \partial_{a} \hat{\mathcal{M}}^{(j)}_{(u,\btheta)} (\bm{q},z,x,y) \leq - \delta_{\rm bc}
        \qquad \text{in $\rcal(\Gamma_{\rm shock} \cap \mathcal{D}^j_{\epsilon_{\mathcal{M}}})\,$
        for any \((\bm{q},z) \in \mathbb{R}^2 \times \mathbb{R},\)}
    \end{align*}
    with $\partial_a = \partial_{q_x}, \,\partial_{q_y}$ or $\partial_{z},$
    provided that \(\Gamma_{\rm shock} \cap \mathcal{D}^j_{\epsilon_{\mathcal{M}}}\) is non-empty.

    \smallskip
    \item \label{item5-Lem:Property4ModifyBdryCondi}
    \(\mathcal{M}_{(u,\btheta)}(D\phi(\bm{\xi}),\phi(\bm{\xi}),\bm{\xi}) = 0 \) on \(\Gamma_{\rm shock}\) if and only if \(\varphi = \phi - \frac12 \abs{\bm{\xi}}^2\) satisfies the Rankine--Hugoniot jump condition~\eqref{Def4FreeBdryFunc} on \(\Gamma_{\rm shock} = \{\varphi = \varphi_2\}\).
\end{enumerate}
\end{lemma}

\subsubsection{Well-posedness of the boundary value problem~\eqref{eq:iteration-BVP}}
Since the complete definition of the iteration boundary value problem~\eqref{eq:iteration-BVP} has been given,
we now consider its well-posedness.

\begin{lemma}
\label{Lem:Wellpose4iterBVP}
Fix \(\gamma \geq 1, \, v_2 \in (v_{\min},0)\), and \(\theta_\ast \in (0, \theta^{\rm d})\).
Let constant $\sigma_2$ be from {\rm Step~\textbf{1}} of the proof
of {\rm Proposition~\ref{Prop:Esti4theta1c}}.
Let $\bar{\alpha}\in(0,\frac13]$ and $\epsilon_0>0$ be from {\rm Proposition~\ref{prop:apriori-estimates-on-u-prop4.12}}
and {\rm Lemma~\ref{lem:properties-of-G-alpha-theta}\eqref{lem:properties-of-G-alpha-theta-3}} respectively,
and let $\alpha\in(0,\frac{\bar{\alpha}}{2}]$ be from {\rm Definition~\ref{def:iteration-set}}.
Then there exist constants $\epsilon^{\rm (w)}\in(0,\epsilon_0],$
$\delta_1^{\rm (w)}\in(0,1),$ $N_1^{\rm (w)}\geq1,$
and $\alpha_1^\ast\in(0,\bar{\alpha}]$ depending only on $(\gamma, v_2,\theta_{\ast})$ such that,
whenever parameters $(\epsilon,\delta_1,N_1)$ in {\rm Definition~\ref{def:iteration-set}} are chosen
with $\epsilon\in(0,\epsilon^{\rm (w)}],$ $\delta_1\in(0,\delta_1^{\rm (w)}],$ and $N_1\geq N_1^{\rm (w)},$
the following statements hold{\rm :}

\smallskip
{\rm Case (i).}
If $\max\{\theta_1,\theta_2\} \leq \theta^{\rm s}+\sigma_2,$
then the boundary value problem~\eqref{eq:iteration-BVP} associated
with $(u,\btheta)\in\overline{\mathcal{K}^{\rm ext}}\cap\{\max\{\theta_1,\theta_2\} \leq \theta^{\rm s}+\sigma_2\}$
has a unique solution
$\hat{\phi}\in C^2(\Omega)\cap C^1(\overline{\Omega}\setminus({\Gamma^5_{\rm sonic}}\cup{\Gamma^6_{\rm sonic}}))\cap C^0(\overline{\Omega})$.
Moreover, there exists a constant $C>0$ depending only on $(\gamma, v_2,\theta_{\ast},\alpha)$ such that solution $\hat{\phi}$ satisfies
\begin{equation}\label{Eq:Sec4WP-iterBVP-onBdry}
\norm{\hat{\phi}}_{L^{\infty}(\Omega)}\leq C\,,\qquad\,\,
|\hat{\phi}(\bm{\xi})-\phi_{\btheta} (\bm{\xi})|\leq C \, {\rm dist}(\bm{\xi}, \, \Gamma^5_{\rm sonic}\cup\Gamma^6_{\rm sonic}) \,\,\,\,\, \text{in $\Omega$}\,,
\end{equation}
for $\phi_{\btheta} \defeq\max\{\varphi_5,\varphi_6\}+\frac12|\bm{\xi}|^2$.
Furthermore, for each $d\in(0,\epsilon_0),$ there exists a constant $C_d>0$ depending only on $(\gamma, v_2,\theta_{\ast},\alpha,d)$ such that
\begin{equation}\label{Eq:Sec4WP-iterBVP-interior}
\norm{\hat{\phi}}_{2,\alpha_1^{\ast},\Omega\setminus\mathcal{D}_d} \leq C_d\,.
\end{equation}

{\rm Case (ii).}
For each $\delta\in(0,\frac{\sigma_2}{2}),$ whenever $\max\{\theta_1,\theta_2\} \in [\theta^{\rm s}+\delta,\theta_{\ast}],$
the boundary value problem~\eqref{eq:iteration-BVP} associated
with $(u,\btheta)\in\overline{\mathcal{K}^{\rm ext}}\cap\{\max{\{\theta_1,\theta_2\}} \geq \theta^{\rm s}+\delta\}$
has a unique solution $\hat{\phi}\in C^2(\Omega)\cap C^1(\overline{\Omega}\setminus({\Gamma^5_{\rm sonic}}\cup{\Gamma^6_{\rm sonic}})) \cap C^0(\overline{\Omega})$.
Moreover, there exists a constant $C'>0$ depending only on $(\gamma, v_2,\theta_{\ast},\alpha,\delta)$ such that solution $\hat{\phi}$
satisfies~\eqref{Eq:Sec4WP-iterBVP-onBdry}{\rm,} while there exists a constant $C'_d>0$ depending only
on $(\gamma, v_2,\theta_{\ast},\alpha,\delta,d)$ such that~\eqref{Eq:Sec4WP-iterBVP-interior} holds.
\end{lemma}

\begin{proof}
For {\rm Case (i)}, we fix
$(u,\btheta)\in\overline{\mathcal{K}^{\rm ext}}\cap\{\max\{\theta_1,\theta_2\}\leq \theta^{\rm s}+\sigma_2\}$
and fix a constant $h>0$.
Using $\gtheta$ defined by~\eqref{Eq:Def-gtheta}, we rewrite the boundary value problem~\eqref{eq:iteration-BVP} associated with $(u,\btheta)$
as a new boundary value problem for the unknown function:
\begin{equation*}
v(\bm{x}) \defeq \hat{\phi} \circ (L_h\circ\gtheta)^{-1}(\bm{x})\qquad\,
\text{for any $\bm{x} \in
L_h\circ\gtheta(\Omega(u,\btheta))$}\,,
\end{equation*}
where $\bm{x} \defeq L_{h}(s,t')=(\frac{s+1}{2}h,t')$ and $(s,t')=\gtheta(\bm{\xi})$ for any $\bm{\xi}\in\Omega(u,\btheta)$.
We write
\begin{equation*}
\Gamma_0\defeq L_h\circ\gtheta(\Gamma^6_{\rm sonic}),\,\,\,
\Gamma_1\defeq L_h\circ\gtheta(\Gamma_{\rm shock}(u,\btheta)),\,\,\,
\Gamma_2\defeq L_h\circ\gtheta(\Gamma^5_{\rm sonic}),\,\,\,
\Gamma_3\defeq L_h\circ\gtheta(\Gamma_{\rm sym})\,.
\end{equation*}
Moreover, we take
\begin{equation*}
g_{\rm so}(\bm{x}) \defeq (1-\varsigma_{h}) \phi_5 \circ (L_h\circ\gtheta)^{-1}(\bm{x})+\varsigma_{h} \phi_6 \circ (L_h\circ\gtheta)^{-1}(\bm{x})\,,
\end{equation*}
where $\varsigma_{h}\defeq\varsigma_1(\frac12+\frac{2}{h}x_1)$, and the cut-off function
$\varsigma_1(\cdot)\in C^{\infty}(\mathbb{R})$ is given by~\eqref{eq:varsigma-cut-off}.
Using Lemmas~\ref{Lem:Propt4gtheta} and~\ref{Lem:Property4ModifyEq}--\ref{Lem:Property4ModifyBdryCondi},
we can choose suitable constants $\epsilon^{(\rm w)}\in(0,\epsilon]$, $\delta_1^{(\rm w)}>0$,
and $N_1^{(\rm w)}\geq1$ such that the new boundary value problem for $v$ satisfies the conditions in Proposition~\ref{Prop:AppendixB01} when
$\epsilon\in(0,\epsilon^{(\rm w)}]$, $\delta_1\in(0,\delta_1^{(\rm w)}]$,
and $N_1 \geq N_1^{(\rm w)}$.
Therefore, Proposition~\ref{Prop:AppendixB01} gives  the existence and uniqueness
of a solution
$\hat{\phi}\in C(\overline{\Omega})\cap C^{2,\alpha_1}(\overline{\Omega}\setminus({\Gamma^5_{\rm sonic}}
\cup{\Gamma^6_{\rm sonic}}))$ of the boundary value problem~\eqref{eq:iteration-BVP}
satisfying~\eqref{Eq:Sec4WP-iterBVP-onBdry}--\eqref{Eq:Sec4WP-iterBVP-interior}.

For Case (ii), fix any $\delta\in(0,\frac{\sigma_2}{2})$.
In the subcase: $\min\{\theta_1,\theta_2\}<\theta^{\rm s}+\delta \leq \max\{\theta_1,\theta_2\}$
(without loss of generality, we consider the case: $\theta_1<\theta^{\rm s}+\delta \leq \theta_2$),
one can follow the same analysis as above to check all the conditions for Proposition~\ref{Prop:AppendixB02}, from which we obtain the existence of a unique $\hat{\phi}$ satisfying
\eqref{Eq:Sec4WP-iterBVP-onBdry}--\eqref{Eq:Sec4WP-iterBVP-interior}
for any $(u,\btheta)\in\overline{\mathcal{K}^{\rm ext}}\cap\{\theta_1<\theta^{\rm s}+\delta \leq \theta_2\}$.
In the other subcase: $\min\{\theta_1,\theta_2\}\in [\theta^{\rm s}+\delta,\theta_{*}]$,
we can apply Proposition~\ref{Prop:AppendixB03} to prove the same results.
\end{proof}

For each $(u,\btheta)\in\overline{\mathcal{K}^{\rm ext}}$, the corresponding pseudo-subsonic
region $\Omega=\Omega(u,\btheta)$ depends continuously on $(u,\btheta)$.
We rewrite~\eqref{eq:iteration-BVP} as a boundary value problem for
\begin{equation}\label{Eq:phiHat2uHhat}
\hat{u}(s,t)\defeq (\hat{\phi}-\frac12|\bm{\xi}|^2-\varphi_{\btheta}^{\ast})\circ\mathfrak{F}_{(u,\btheta)}(s,t)
\qquad\, \text{in $\qiter$}\,,
\end{equation}
where $\varphi_{\btheta}^{\ast}$ is given by Definition~\ref{def:varphi-theta-star},
and $\mathfrak{F}=\mathfrak{F}_{(u,\btheta)}$ is given by Definition~\ref{def:u-to-phi}\eqref{item2-def:u-to-phi}.
We obtain
\begin{equation}\label{Eq:iterBVP4uHat}
\begin{cases}
\begin{aligned}
\,& \sum_{i,j=1}^2\mathcal{A}_{ij}^{(u,\btheta)}(D_{(s,t)}\hat{u},s,t)\partial_{i} \partial_{j}\hat{u}
+\sum_{i=1}^2\mathcal{A}_{i}^{(u,\btheta)}(D_{(s,t)}\hat{u},s,t)\partial_{i}\hat{u} = f^{(u,\btheta)}
&& \quad\text{in $\qiter$}\,,\\
& \mathscr{M}_{(u,\btheta)}(D_{(s,t)}\hat{u},\hat{u},s) = 0
&& \quad\text{on $\partial_{\rm sh}\qiter$}\,,\\[1mm]
& \hat{u} = 0
&& \quad\text{on $\partial_{\rm so}\qiter$}\,,\\[-1mm]
& \mathscr{B}^{\rm (w)}_{(u,\btheta)}(D_{(s,t)}\hat{u},s)\defeq \sum_{i=1}^{2}b^{\rm (w)}_{i}(s)\partial_{i}\hat{u} = 0
&& \quad\text{on $\partial_{\rm w}\qiter$}\,,
\end{aligned}\end{cases}
\end{equation}
where $(\partial_1,\partial_2) \defeq (\partial_s,\partial_t)$
and
\begin{equation*}
\partial_{\rm sh}\qiter\defeq (-1,1)\times\{1\}\,,\quad
\partial_{\rm so} \qiter \defeq \{-1,1\}\times(0,1)\,,\quad  \partial_{\rm w}\qiter \defeq (-1,1)\times\{0\}\,.
\end{equation*}

From Lemmas~\ref{lem:properties-of-G-alpha-theta} and~\ref{Lem:Property4ModifyEq}--\ref{Lem:Wellpose4iterBVP},
we can obtain the following results; see~\cite[Lemma~4.36]{BCF-2019}.

\begin{lemma}\label{Lem:UniformlyConvergeSeq}
For every $(u,\btheta)\in\overline{\mathcal{K}^{\rm ext}},$
let $\mathcal{A}_{ij}^{(u,\btheta)}, \mathcal{A}_{i}^{(u,\btheta)}, f^{(u,\btheta)},  \mathscr{M}_{(u,\btheta)}, \mathscr{B}^{\rm (w)}_{(u,\btheta)},$
and $b^{\rm (w)}_{i}$ be given as in~\eqref{Eq:iterBVP4uHat}.
Then the following properties hold{\rm :}
\begin{enumerate}[{\rm (i)}]
\item $\mathcal{A}_{ij}^{(u,\btheta)}, \mathcal{A}_{i}^{(u,\btheta)}\in C(\mathbb{R}^2\times \qiter),$ $f^{(u,\btheta)}\in C(\qiter),$
$\mathscr{M}_{(u,\btheta)}\in C(\mathbb{R}^2\times\mathbb{R}\times\partial_{\rm sh}\qiter),$
and $\mathscr{B}^{\rm (w)}_{(u,\btheta)}\in C(\mathbb{R}^2\times\mathbb{R}\times\partial_{\rm w}\qiter)${\rm ;}
	
\item If a sequence $\{(u_k,\btheta_k)\}_{k\in\mathbb{N}}\subseteq\overline{\mathcal{K}^{\rm ext}}$
converges to $(u,\btheta)\in\overline{\mathcal{K}^{\rm ext}}$
in
$C^{2,\alpha}_{(*)}(\qiter)\times[0,\btheta_*]^2$,
then the following sequences converge uniformly{\rm :}
	\begin{itemize}
		\item $(\mathcal{A}_{ij}^{(u_k,\btheta_k)}, \mathcal{A}_{i}^{(u_k,\btheta_k)})\to(\mathcal{A}_{ij}^{(u,\btheta)}, \mathcal{A}_{i}^{(u,\btheta)})$
on compact subsets of $\mathbb{R}^2\times \qiter${\rm ,}
		\item $f^{(u_k,\btheta_k)}\to f^{(u,\btheta)}$ on compact subsets of $\qiter${\rm ,}
		\item $\mathscr{M}_{(u_k,\btheta_k)}\to \mathscr{M}_{(u,\btheta)}$ on compact subsets
           of $\mathbb{R}^2\times\mathbb{R}\times\partial_{\rm sh}\qiter${\rm ,}
		\item $\mathscr{B}^{\rm (w)}_{(u_k,\btheta_k)}\to \mathscr{B}^{\rm (w)}_{(u,\btheta)}$ on compact subsets of $\mathbb{R}^2\times\mathbb{R}\times\partial_{\rm w}\qiter$.
	\end{itemize}
\end{enumerate}
\end{lemma}

\begin{corollary}\label{corol:ConvergeInHolder4uHat}
Let $\bar{\alpha}\in(0,1)$ be from {\rm Proposition~\ref{prop:apriori-estimates-on-u-prop4.12},}
and let
$\alpha_1^*\in(0,\bar{\alpha}],$
$\epsilon^{({\rm w})}, \delta_1^{({\rm w})}$, and $N_1^{({\rm w})}$
be from {\rm Lemma~\ref{Lem:Wellpose4iterBVP}}.
Let
$\epsilon,\delta_1,$ and $N_1$ from {\rm Definition~\ref{def:iteration-set}} satisfy
 $\epsilon\in(0,\epsilon^{({\rm w})}],$ $\delta_1\in(0,\delta_1^{({\rm w})}],$ and $N_1\geq N_1^{({\rm w})}$.
\begin{enumerate}[{\rm (i)}]
\item For each $(u,\btheta)\in\overline{\mathcal{K}^{\rm ext}},$ $\hat{\phi}$ solves
the boundary value problem~\eqref{eq:iteration-BVP} if and only if $\hat{u}$ given
by~\eqref{Eq:phiHat2uHhat} solves the boundary value problem~\eqref{Eq:iterBVP4uHat}.
Thus,~\eqref{Eq:iterBVP4uHat} has a unique solution
$\hat{u}\in C^2(\qiter)\cap C^1(\overline{\qiter}\setminus\overline{\partial_{\rm so}\qiter})\cap C(\overline{\qiter})$.
	Furthermore, there exists a constant $C\geq1$ depending on $(\gamma, v_2,\theta_*,\alpha)$ such that
	\begin{equation*}
	    |\hat{u}(s,t)|\leq C(1-|s|) \qquad\text{in $\qiter$} \, .
  \end{equation*}
  For each $\hat{d}\in(0,\frac12),$ there exists $C_{\hat{d}}>0$ depending
  on $(\gamma, v_2,\theta_*,\hat{d},\alpha)$ such that
  \begin{equation*}
		\|\hat{u}\|_{2,\alpha_1^*,\qiter\cap\{1-|s|>\hat{d}\}}
	   \leq C_{\hat{d}} \, .
	\end{equation*}
	
\item Let $\{(u_k,\btheta_k)\}_{k\in\mathbb{N}} \subseteq \overline{\mathcal{K}^{\rm ext}}$ converge
to $(u,\btheta)\in\overline{\mathcal{K}^{\rm ext}}$
in $C^1(\overline{\qiter})\times[0,\theta_*]^2,$ and let $\hat{u}_k$ be the solution of
the boundary value problem~\eqref{Eq:iterBVP4uHat} associated with $(u_k,\btheta_k)$.
Then there exists a unique solution
$\hat{u}\in C^{2}(\qiter)\cap C^1(\overline{\qiter}\setminus\overline{\partial_{\rm so}\qiter})\cap C(\overline{\qiter})$
of the boundary value problem~\eqref{Eq:iterBVP4uHat}
associated with $(u,\btheta)$.
Moreover, $\hat{u}_k$ converges to $\hat{u}$ uniformly in $\overline{\qiter}$ and,
for any $\alpha'\in[0,\alpha_1^*)${\rm ,}
	\begin{itemize}
		\item $\hat{u}_k\to \hat{u}$ in $C^{1,\alpha'}(K)$ for any compact subset
$K\subseteq\overline{\qiter}\setminus\overline{\partial_{\rm so}\qiter}${\rm ,}
		\item $\hat{u}_k\to \hat{u}$ in $C^{2,\alpha'}(K)$ for any compact subset $K\subseteq\qiter$.

	\end{itemize}

    \smallskip
    \item If $(u,\btheta)\in\overline{\mathcal{K}},$ then $(u,\btheta)$ satisfies
    {\rm Definition~\ref{def:iteration-set}\eqref{def:iteration-set-7}} with nonstrict inequality in~\eqref{eq:iteration-u-uhat-estimate}.
\end{enumerate}
\end{corollary}

\begin{remark} \label{rem:BCF-remark-4.38}
For a constant $M>0,$ define a set $\mathcal{K}^{E}_{M}$ by
\begin{equation*}
\mathcal{K}^{E}_{M}\defeq\left\{(u,\btheta)\in C^{2,\alpha}_{(*)}(\qiter)\times[0,\btheta_*]^2 \,:\,\,
\parbox{16.3em}{\rm $(u,\btheta)$ satisfies
Definition~\ref{def:iteration-set}\eqref{def:iteration-set-2}--\eqref{def:iteration-set-6}\, \\
\hspace{2mm} and
$\|u\|^{(*)}_{2,\alpha,\qiter}\leq M$ $\quad$} \right\}.
\end{equation*}
Let $\overline{\mathcal{K}^{E}_{M}}$ be the closure of $\mathcal{K}^{E}_{M}$ in $C^{2,\alpha}_{(*)}(\qiter)\times[0,\btheta_*]^2$.
Then {\rm Lemma~\ref{Lem:UniformlyConvergeSeq}} and {\rm Corollary~\ref{corol:ConvergeInHolder4uHat}}
still hold when $\overline{\mathcal{K}^{\rm ext}}$ is replaced by $\overline{\mathcal{K}^{E}_{M}}$ for some $M>0$.
\end{remark}

\subsubsection{Properties of the iteration set}
From the well-posedness of the iteration boundary value problem~\eqref{eq:iteration-BVP},
we can obtain the following {\it a priori} estimates on the iteration set.
Let \(\bar{\alpha} \in (0,\frac13]\) be the constant from {\rm Proposition~\ref{prop:apriori-estimates-on-u-prop4.12}},
and let \((\epsilon^{\rm (w)},\delta_1^{\rm (w)},N_1^{\rm (w)})\) be the constants from Lemma~\ref{Lem:Wellpose4iterBVP}.

\begin{lemma}[{\it A priori} estimates]\label{lem:BCF-cor4.45}
There exist positive constants \(\alpha^{\rm (ap)} \in (0,\frac{\bar{\alpha}}{2}),\) \(\epsilon^{\rm (ap)} \in (0 , \epsilon^{\rm (w)}],\)
\(\delta_1^{\rm (ap)} \in (0 , \delta_1^{\rm (w)}],\) \(N_1^{(\rm adm)} \geq N_1^{\rm (w)},\)
and \(\delta_3^{(\rm ap)}\) with \((\alpha^{\rm (ap)},\epsilon^{\rm (ap)}, \delta_1^{\rm (ap)})\)
depending only on \((\gamma, v_2,\theta_\ast),\) \(N_1^{\rm (adm)}\) depending only
on \((\gamma, v_2,\theta_\ast,\delta_1),\)
and \(\delta_3^{\rm (ap)}\) depending only on \((\gamma, v_2,\theta_\ast, \delta_1, \delta_2, N_1)\)
such that, whenever parameters \((\alpha, \epsilon, \delta_1,\delta_3, N_1)\)
in {\rm Definition~\ref{def:iteration-set}} are chosen according to
\begin{itemize}
\item
\((\alpha, \epsilon, \delta_1) \in (0, \alpha^{\rm (ap)} ] \times
    (0, \epsilon^{\rm (ap)} ] \times (0, \delta_1^{\rm (ap)}] \,,\)
\item
    \( N_1 \in [ N_1^{\rm (adm)}(\delta_1) , \infty)\,, \)
\item
    \( \delta_3 \in ( 0 , \delta_3^{\rm (ap)}(\delta_1, \delta_2, N_1)] \,,\)
\end{itemize}
with parameter \(\delta_2 > 0\) to be determined later in {\rm \S\ref{sec:iteration-mapping-existence}},
the following statements hold{\rm :}
\begin{enumerate}[{\rm (i)}]
\item \label{lem:BCF-cor4.45-item-i}
For any admissible solution \(\varphi\) corresponding to
parameter \(\btheta \in \Theta \cap \{\theta_1,\theta_2 \leq \theta_\ast\},\)
function \(u = u^{(\varphi, \btheta)}\)
given by~\eqref{eq:phi-to-u} satisfies that \((u,\btheta) \in \mathcal{K}\).

\item \label{lem:BCF-cor4.45-item-ii}
There exists a constant \(C > 0\) depending only on \((\gamma, v_2,\theta_\ast)\) such that,
for each \((u,\btheta) \in \cl{\mathcal{K}^{\rm ext}},\)
the unique solution
\(\hat{\phi}\in C^2(\Omega) \cap C^1(\cl{\Omega}\setminus \Gamma_{\rm sonic} )\cap C^0(\cl{\Omega})\)
of the boundary value problem~\eqref{eq:iteration-BVP} associated with \((u,\btheta)\) satisfies
\begin{align} \label{eq:uniform-estimate-u-hat}
    \norm{\hat{u}}^{(\ast)}_{2,2\alpha,\qiter} \leq C\,,
\end{align}
for \(\hat{u} : \cl{\qiter} \to \mathbb{R}\) given by~\eqref{eq:def-u-hat}{\rm,}
whenever
\((u,\btheta) \in \cl{\mathcal{K}^{\rm ext}}\) satisfies
\begin{align} \label{eq:delta-sharp}
     \norm{u^\sharp - u}_{C^1(\cl{\qiter})} + \abs{\btheta^\sharp - \btheta} \leq \delta^\sharp
\end{align}
for some \((u^\sharp,\btheta^\sharp) \in \cl{\mathcal{K}}\)
and some sufficiently small positive constant \(\delta^\sharp\)
depending only on \((\gamma, v_2,\theta_\ast,\delta_2,\delta_3,u^\sharp, \btheta^\sharp)\).
\end{enumerate}
\end{lemma}

\noindent
\textit{Proof.}
We describe the proof in four steps,
which follow~\cite[Corollary~4.40, Lemma~4.42, Lemma~4.44, Corollary~4.45]{BCF-2019} closely.

\smallskip
\textbf{1.}
The proof of statement~\eqref{lem:BCF-cor4.45-item-i} above follows similarly to~\cite[Corollary~4.40]{BCF-2019}.
Indeed, let \(u = u^{(\varphi,\btheta)}\) be given by~\eqref{eq:phi-to-u} for any admissible
solution \(\varphi\) corresponding to \(\btheta \in \Theta \cap [0,\theta_*]^2\).
Then \((u,\btheta)\) automatically satisfies properties~(\ref{def:iteration-set-2})--(\ref{def:iteration-set-6})
of Definition~\ref{def:iteration-set} through the choices of constants
\(\{N_i\}_{i=2}^5, \sigma_2,  \mu_0, \mu_1, \tilde{\mu}, \rho^\ast(\gamma)\), and \(C_{\rm ub}\)
in Definition~\ref{def:iteration-set}\eqref{def:iteration-set-3}--\eqref{def:iteration-set-6}.

Property~(\ref{def:iteration-set-1}) of Definition~\ref{def:iteration-set} follows from the choice of \(N_0\)
in Definition~\ref{def:iteration-set}\eqref{def:iteration-set-1} and the choice of \(N_1\) below, in
which we use an argument
similar to~\cite[Lemma~4.39]{BCF-2019}.
Indeed, by Lemma~\ref{Lem:UniformlyConvergeSeq}, Corollary~\ref{corol:ConvergeInHolder4uHat},
and Remark~\ref{rem:BCF-remark-4.38}, for any \(\alpha \in (0, \frac{\bar{\alpha}}{2}]\), we have the following continuity property of admissible solutions:
For any sequence $\{\btheta_{k}\}_{k\in\mathbb{N}}\subseteq\Theta$
with $\btheta_{k}\to\mathbf{0}$ as $k\to\infty,$
let $\varphi^{(\btheta_k)}$ be any admissible solution corresponding to $\btheta_k,$ and let $u^{(\btheta_k)}$
be given by \eqref{eq:phi-to-u} with \(\varphi = \varphi^{(\btheta_k)}\) and \(\btheta = \btheta_k\).
Then there exists a subsequence of $\{u^{(\btheta_k)}\}_{k\in\mathbb{N}}$
converging in $C^{2,\alpha}_{(\ast)}(\qiter)$ to $u^{({\rm norm})} = 0$.
Using this continuity property, for any \(\delta_1 \in (0,\delta_1^{\rm(w)}]\),
there exists \(N_1^{\rm(adm)} \geq N_1^{\rm (w)}\) depending on \((\gamma, v_2, \theta_\ast, \delta_1)\) such that
$\norm{u -u^{\rm(norm)}}^{(\ast)}_{2,\alpha,\qiter} <  \tfrac{\delta_1}{2}$
whenever $\max\{\theta_1,\theta_2\} \in (0,\tfrac{2\delta_1}{N_1^{\rm(adm)}}]$.

Therefore, we have shown that \((u,\btheta) \in \mathcal{K}^{\rm ext}\).
Finally, property~\eqref{def:iteration-set-7} of Definition~\ref{def:iteration-set}
now follows from Corollary~\ref{corol:ConvergeInHolder4uHat} because \((u,\btheta) \in \mathcal{K}^{\rm ext}\).
We conclude that \((u,\btheta) \in \mathcal{K}\).

\smallskip
\textbf{2.} It remains to prove statement~\eqref{lem:BCF-cor4.45-item-ii}.
Fix any \((u^\sharp,\btheta^\sharp) \in \cl{\mathcal{K}}\).
We claim that there exist constants \(\epsilon^{\rm (lb)} \in (0, \epsilon^{\rm (w)})\)
depending only on \((\gamma, v_2,\theta_\ast)\),
\(\delta_3^{\rm (ap)} > 0\) depending only on \((\gamma, v_2,\theta_\ast,\delta_1,\delta_2,N_1)\),
and \(\delta^\sharp\) depending only on \((\gamma, v_2,\theta_\ast,\delta_2,\delta_3,u^\sharp, \btheta^\sharp)\)
such that
$\hat{\phi} - ( \vphithetastar + \frac12 \abs{\bm{\xi}}^2 ) > 0$
in $\Omega$,
whenever \((u,\btheta)\in \cl{\mathcal{K}^{\rm ext}}\) satisfies~\eqref{eq:delta-sharp}.
The proof of this claim follows from that for ~\cite[Lemma~4.42]{BCF-2019}, which is split into two cases:
(i) \(\max\{\theta_1,\theta_2\} \in [\frac{2\delta_1}{N_1^2}, \theta_\ast]\),
and (ii) \(\max\{\theta_1,\theta_2\} \in [0,\frac{2\delta_1}{N_1^2}]\).
In particular, we can choose \(\delta_3^{\rm (ap)} = \frac{\delta_1 \delta_2}{2 N_1^2}\)
and \(\epsilon^{\rm (lb)} = \bar{k}^{-1} \inf_{\btheta \in \cl{\Theta}}\{\hat{c}_5,\hat{c}_6\}\)
for constant \(\bar{k} > 1\) that has been fixed after
Definition~\ref{def:varphi-theta-star} such that~\eqref{eq:def-of-k-bar} holds.

\smallskip
\textbf{3.}
We obtain the {\it a priori} estimates for \(\hat{\phi}\) near \(\Gamma_{\rm sonic}\).
Let coordinates \((x,y) = \rcal(\bm{\xi})\) be given by~\eqref{Eq:DefPolarCoordiSec4}.
For \(j = 5, 6\), we have the following estimates (with estimates (b)--(c) below only
valid when  \(\theta_\ast \in (\theta^{\rm s},\theta^{\rm d})\)):
\begin{enumerate}[{\rm (a)}]
    \item \label{item1-Lem:AprioriIterBVP}
    Whenever \(\theta_{j-4} \in [0,\theta^{\rm s})\), for each \(\alpha' \in (0,1)\),
    there exist positive constants \(\epsilon_{\rm p} \in (0,\epsilon_0]\) and \(C_{\alpha'}\)
    depending only on \((\gamma, v_2,\theta_\ast,\alpha')\) with
    \begin{align*}
\norm{(\hat{\phi} - \phi_j)\circ \rcal^{-1} }_{2,\alpha',\rcal( \Omega \cap \mathcal{D}^j_{\epsilon_{\rm p}}) }^{(2),(\rm par)}
 \leq C_{\alpha'}\,;
    \end{align*}

    \item \label{item2-Lem:AprioriIterBVP}
    There exists a constant \(\delta_{\rm p} \in (0,\theta_\ast - \theta^{\rm s})\)
    depending only on \((\gamma, v_2,\theta_\ast)\) such that,
    whenever \(\theta_{j-4} \in [\theta^{\rm s}, \theta^{\rm s} + \delta_{\rm p}]\),
    for each \(\alpha' \in (0,1)\), there exist positive constants \(\epsilon_{\rm p} \in (0,\epsilon_0]\)
    depending only on \((\gamma, v_2,\theta_\ast)\)
    and \(C_{\alpha'}\) depending only on \((\gamma, v_2,\theta_\ast,\alpha')\) so that
    \begin{align*}
        \norm{\hat{\phi} - \phi_j}_{C^{2,\alpha'}(\Omega \cap \mathcal{D}^j_{\epsilon_{\rm p}})} \leq C_{\alpha'}\,,
        \qquad\,\,
        D^m (\hat{\phi} - \phi_j)(P_0^{j-4}) = 0 \quad \text{for $m = 0, 1, 2$}\,;
    \end{align*}

    \item \label{item3-Lem:AprioriIterBVP}
    For \(\delta_{\rm p}>0\) as above, there exist constants $\hat{\alpha}\in(0,\frac13]$ and \(C>0\)
    depending only on \((\gamma, v_2,\theta_\ast)\) such that,
    whenever \(\theta_{j-4} \in [\theta^{\rm s} + \frac12 \delta_{\rm p} , \theta_\ast ]\),
    \begin{align*}
        \norm{\hat{\phi} - \phi_j}_{2,\hat{\alpha}, \Omega \cap \mathcal{D}^j_{\epsilon_0}}^{(-1 - \hat{\alpha}),\{P_0^{j-4}\}} \leq C\,,
        \qquad\,\,
        D^m (\hat{\phi} - \phi_j)(P_0^{j-4}) = 0 \quad \text{for $m = 0, 1$}\,.
    \end{align*}
\end{enumerate}
Using Step~\textbf{2}, the proof of these {\it a priori} estimates follows from that for~\cite[Lemma~4.44]{BCF-2019},
whereby it is necessary to further reduce constants \(\epsilon\) and \(\delta_1\)
depending only on \((\gamma,v_2,\theta_{\ast})\),
and we define $\alpha^{\rm (ap)}\defeq\frac12\min\{\alpha_1^*,\hat{\alpha}\}$ with $\alpha_1^*$ given in Lemma~\ref{Lem:Wellpose4iterBVP}.
In particular, the estimate in (i) is similar to that
for Propositions~\ref{Prop:Esti4theta1a}--\ref{Prop:Esti4theta1b},
with the free boundary replaced by a fixed boundary.
The estimate in (ii) is similar to that for Proposition~\ref{Prop:Esti4theta1c},
with the free boundary replaced again by a fixed boundary.
Finally, the estimate in (iii) is similar to that for Proposition~\ref{Prop:Esti4theta1d}.

\smallskip
\textbf{4.}
Following the method from~\cite[Proposition~4.12]{BCF-2019}, we are able to combine the estimates
from Step~\textbf{3} with the interior estimates from Lemma~\ref{Lem:Wellpose4iterBVP}
to obtain~\eqref{eq:uniform-estimate-u-hat}.
\qed

\smallskip
Let parameters \((\alpha, \epsilon, \delta_1, \delta_3, N_1)\) in Definition~\ref{def:iteration-set} be chosen as in Lemma~\ref{lem:BCF-cor4.45}.
It is simple to verify that \(\mathcal{K}^{\rm ext} \subseteq C^{2,\alpha}_{(\ast)}(\qiter) \times [0,\theta_\ast]^2\)
is relatively open.
Indeed, the argument is similar to that for~\cite[Lemmas~12.8.1 and 17.5.1]{ChenFeldman-RM2018}
and~\cite[Lemma~4.41]{BCF-2019}.
Furthermore, following the proofs of~\cite[Propositions~12.8.2 and~17.5.3]{ChenFeldman-RM2018}
and~\cite[Lemma~4.46]{BCF-2019},
and applying the {\it a priori} estimates from Lemma~\ref{lem:BCF-cor4.45},
we conclude that the iteration set
\(\mathcal{K} \subseteq C^{2,\alpha}_{(\ast)}(\qiter) \times [0,\theta_\ast]^2\) is relatively open.

\begin{proposition}[Openness of the iteration set]
\label{Prop:openness-of-iteration-set}
For the choice of parameters \((\alpha,\epsilon,\delta_1,\delta_3,N_1)\) given in {\rm Lemma~\ref{lem:BCF-cor4.45},}
the iteration set \(\mathcal{K}\) defined by {\rm Definition~\ref{def:iteration-set}} is a relatively open subset
of \(C^{2,\alpha}_{(\ast)}(\qiter) \times [0,\theta_\ast]^2\).
\end{proposition}

\subsection{Definition of the iteration map}\label{sec:iteration-mapping-existence}
Fix $\theta_*\in(0,\theta^{\rm d})$.
For the iteration set $\mathcal{K}$ given by Definition~\ref{def:iteration-set},
let parameters $(\alpha,\epsilon,\delta_1,\delta_3,N_1)$ be chosen as in Lemma~\ref{lem:BCF-cor4.45}.
Define
\begin{equation*}
\mathcal{K}(\btheta) \defeq \big\{ u\in C^{2,\alpha}_{(*)}(\qiter)  \,:\,  (u,\btheta)\in\mathcal{K} \big\}
\qquad \text{for each} \; \btheta\in[0,\theta_*]^2 \,,
\end{equation*}
and similarly for \(\cl{\mathcal{K}}(\btheta)\).
We now define an iteration map $\mathcal{I} : \overline{\mathcal{K}} \rightarrow C^{2,\alpha}_{(*)}(\qiter)$
satisfying the following properties:
\begin{enumerate}[\quad(a)]
\item
For each $\btheta\in[0,\theta_*]^2$, there exists $u\in\mathcal{K}(\btheta)$
satisfying $\mathcal{I}(u,\btheta)=u$;

\item
If $u$ satisfies $\mathcal{I}(u,\btheta)=u$, then $\varphi^{(u,\btheta)}$ defined
in Definition~\ref{def:u-to-phi}\eqref{item3-def:u-to-phi} is an admissible solution
corresponding to $\btheta$.
\end{enumerate}

Let $\hat{\phi}=\hat{\varphi}^{(u,\btheta)}+\frac{1}{2}|\bm{\xi}|^2\in C^2(\Omega)\cap C^1(\overline{\Omega})$
solve the iteration boundary value problem~\eqref{eq:iteration-BVP} associated with $(u,\btheta)$ in Definition~\ref{def:iteration-set}.
Accordingly, from the definition of  $\hat{u}:\overline{\qiter}\to\mathbb{R}$ given by~\eqref{eq:def-u-hat},
we have
\begin{equation*}
\hat{\varphi}^{(u,\btheta)}(\bm{\xi})=\hat{u}\circ(\Futheta)^{-1}(\bm{\xi}) + \varphi^*_{\btheta}(\bm{\xi})\,,
\end{equation*}
where $\varphi^*_{\btheta}(\bm{\xi})$ is given by Definition~\ref{def:varphi-theta-star}.
We also define functions $(w, w_2, \hat{w})$ by
\begin{equation}\label{Eq:Sec5-Def-3w-funcs}
(w,w_2,\hat{w})(s,t') \defeq
(\varphi-\varphi^*_\btheta,\varphi_2-\varphi^*_\btheta,
\hat{\varphi}^{(u,\btheta)}-\varphi^*_\btheta)\circ (\gtheta)^{-1}(s,t')\,.
\end{equation}

From~\eqref{Eq:LemMonotone-phi2ast} and the implicit function theorem,
for any $\btheta\in\overline{\Theta}$, there exists a unique function
$\mathfrak{g}_2:[-1,1]\to \overline{\mathbb{R}_+}$ satisfying $w_2(s,\mathfrak{g}_2(s))=0$ on $[-1,1]$.
Then, for $\qtheta$ defined by {\rm Definition~\ref{def:Q-theta1-theta2}},
\begin{equation*}
\left\{(s,t') \,:\,  -1 < s < 1, \, t'=\mathfrak{g}_2(s)\right\} \subseteq \gtheta(\qtheta)\,.
\end{equation*}
It follows from~\eqref{Eq:Sec4-Propty4Frakg2}, Lemma~\ref{Lem:Propt4gtheta},
and Definition~\ref{def:varphi-theta-star} that $\|\mathfrak{g}_2\|_{C^3([-1,1])}\leq C$
for some constant $C>0$ depending only on $(\gamma, v_2)$.

For any $g \in C^{0,1}([-1,1])$ satisfying $g(s)>0$ for all $s\in(-1,1)$, introduce the following sets:
$R_{\infty}\defeq(-1,1)\times\mathbb{R}_+$ and
\begin{equation}\label{Eq:DefRgDomain}
\begin{aligned}
& R_g \defeq \big\{(s,t')\in\mathbb{R}^2_+ \,:\, -1<s<1, \, 0<t'<g(s) \big\}\,,\\
& \Sigma_g\defeq \big\{(s,t')\in\mathbb{R}^2_+ \,:\, -1<s<1, \, t'=g(s)\big\}\,.
\end{aligned}
\end{equation}
For each $(u,\btheta)\in\overline{\mathcal{K}}$,
$\gsh=\gsh^{(u,\btheta)}:[-1,1]\to\overline{\mathbb{R}_+}$ given by Definition~\ref{def:u-to-phi}
is a Lipschitz function on $[-1,1]$
and is positive on $(-1,1)$.
Then we can define sets $R_{\gsh}$ and $\Sigma_{\gsh}$ with $g=\gsh$ in~\eqref{Eq:DefRgDomain}.
Note that $w$ and $\hat{w}$ are defined on $R_{\gsh}$, while $w_{2}$ is defined on $R_{\infty}$.

The \textit{regularized distance} function
$\delta_{\gsh}\in C^{\infty}(\overline{R_{\infty}}\setminus\overline{R_{\gsh}})$ is given by Lemma~\ref{Lem:RegularizedDist-13-9-1}.
Let $C_{\rm rd}>0$ be the constant from Lemma~\ref{Lem:RegularizedDist-13-9-1},
which depends only on ${\rm Lip}[\gsh]$.
We define
\begin{equation}\label{Eq:Def-RD-star}
\delta^*_{\gsh}(s,t')\defeq C_{\rm rd}^{-1}\delta_{\gsh}(s,t')
\qquad\text{for any $(s,t')\in\overline{R_{\infty}}\setminus\overline{R_{\gsh}}$}\,.
\end{equation}
Then, for all $(s,t')\in R_{(1+\kappa)\gsh}\setminus\overline{R_{\gsh}}$ and all $\lambda\in[1,2]$,
\begin{equation}\label{Eq:GeometryRela-RD-star}
(s,t'-\lambda\delta^*_{\gsh}(s,t')) \in \{s\}\times \big[ \, \frac{\gsh(s)}{3}, \, \gsh(s)-(t'-\gsh(s)) \big] \Subset R_{\gsh}\,,
\end{equation}
where constant $\kappa\in(0,\frac{1}{3}]$ depends only on ${\rm Lip}[\gsh]$.

Let $\left\{(u_k,\btheta_k)\right\}_{k\in\mathbb{N}} \subseteq \overline{\mathcal{K}^{\rm ext}}$
converge to $(u,\btheta)$ in $C^{2,\alpha}_{(*)}(\qiter)\times[0,\theta_*]^2$.
Note that
\begin{equation*}
\|\gsh^{(u,\btheta)}\|^{(-1-\alpha),\{-1,1\}}_{2,\alpha,(-1,1)} \leq C N_0\,,
\end{equation*}
where $N_0>0$ is from Definition~\ref{def:iteration-set}\eqref{def:iteration-set-1},
and $C>0$ depends only on $(\gamma, v_2,\alpha)$. Moreover, for $\gfive$ and $\gsix$ defined in
{\rm Proposition~\ref{prop:properties-of-gshock}\eqref{prop:properties-of-gshock:3}},
\begin{equation*}
 \frac{{\rm d}^m}{{\rm d} s^m}(\gsh^{(u,\btheta)}-\gfive)(1)=\frac{{\rm d}^m}{{\rm d} s^m}(\gsh^{(u,\btheta)}-\gsix)(-1)=0
 \qquad \text{for $m=0,1$}\,.
\end{equation*}
Using Lemma~\ref{lem:properties-of-G-alpha-theta}\eqref{lem:properties-of-G-alpha-theta-4}
and~\eqref{lem:properties-of-G-alpha-theta-7}--\eqref{lem:properties-of-G-alpha-theta-8},
we see that $\gsh^{(u_k,\btheta_k)}$ converges to $\gsh^{(u,\btheta)}$ in $C^{0,1}([-1,1])$ and
\begin{equation*}
\lim\limits_{k\to\infty} \|\delta_{\gsh^{(u_k,\btheta_k)}}-\delta_{\gsh^{(u,\btheta)}}\|_{C^m(K)}=0
\end{equation*}
for any compact set $K\subseteq \overline{R_{\infty}}\setminus\overline{R_{\gsh}}$ and  $m\in\mathbb{N}\cup\{0\}$,
which follows by Lemma~\ref{Lem:RegularizedDist-13-9-1}.

From~\cite[Lemma~13.9.2]{ChenFeldman-RM2018}, there exists a function $\Psi\in C^{\infty}_{\rm c}(\mathbb{R})$
satisfying
\begin{equation*}
\supp \Psi\subseteq[1,2]\,, \qquad\,
\int^{\infty}_{-\infty}\lambda^{m} \Psi(\lambda) \,{\rm d}\lambda = 1-\sgn (m) \quad \text{for $m=0,1,2$}\,.
\end{equation*}

\begin{definition}[Extension operator]
\label{Def:ExtenMap4gsh}
For each $(u,\btheta)\in\overline{\mathcal{K}^{\rm ext}},$
write $\gsh=\mathfrak{g}^{(u,\btheta)}_{\rm sh},$ and let $\delta^*_{\gsh}$ be defined by~\eqref{Eq:Def-RD-star} satisfying~\eqref{Eq:GeometryRela-RD-star}
with constants $C_{\rm rd}>0$ and $\kappa\in(0,\frac{1}{3}]$ depending only on $(\gamma, v_2,\theta_*)$.
For any $v\in C^{0}(\overline{R_{\gsh}})\cap C^2(R_{\gsh} \cup \Sigma_{\gsh}),$
define the extension of $v(s,t')$ by
\begin{equation*}
\mathcal{E}_{\gsh}(v)(s,t') \defeq \begin{cases}\begin{aligned}
\,& v(s,t')\quad && \text{for $(s,t')\in\overline{R_{\gsh}}$}\,,\\
& \int_{-\infty}^{\infty}v(s,t'-\lambda \delta^*_{\gsh}(s,t'))\Psi(\lambda) \, {\rm d}\lambda \quad &&
\text{for $(s,t')\in R_{(1+\kappa)\gsh}\setminus\overline{R_{\gsh}}$}\,.
\end{aligned}\end{cases}
\end{equation*}
\end{definition}

For simplicity, we use the following notation:
For any $a,b\in[-1,1]$, denote
\begin{equation*}
R^{(u,\btheta)}_{\gsh}[a,b] \defeq
\big\{(s,t')\in\mathbb{R}^2_+\,:\, a<s<b\,, 0<t'<\gsh^{(u,\btheta)}(s)\big\}\,.
\end{equation*}
We often use the notation  $R_{\gsh}$ to represent $R^{(u,\btheta)}_{\gsh}$.

Combining the above with Lemma~\ref{lem:properties-of-G-alpha-theta}\eqref{lem:properties-of-G-alpha-theta-4}
and the property that the uniform bound of ${\rm Lip}[\gsh]$ depends only on $(\gamma, v_2,\theta_*)$,
we obtain the following proposition, for which the details of the proof can be found
in~\cite[Lemma~13.9.6, and Theorems~13.9.5 and~13.9.8]{ChenFeldman-RM2018} for each case, respectively.

\begin{proposition}[Properties of the extension operator $\mathcal{E}$]
\label{Prop:Propty4ExtenOper}
Fix $\alpha\in(0,1)$.
For each $(u,\btheta)\in\overline{\mathcal{K}^{\rm ext}},$
let the extension operator
$\mathcal{E}_{\gsh}: C^2(R_{\gsh}\cup\Sigma_{\gsh}) \to C^2(R_{(1+\kappa)\gsh})$
be given by {\rm Definition~\ref{Def:ExtenMap4gsh}}.
We introduce the notation for the following three different cases{\rm :}
\begin{enumerate}[{\rm (a)}]
\item \label{item-a-Prop:Propty4ExtenOper} Fix any $(b_1, b_2)$ with $-1<b_1 < b_2<1,$
and $C=C_{\rm int}>0$ depending only on $(\gamma, v_2,\theta_*,\alpha)$.
For any $\alpha'\in(0,\alpha),$ denote the function spaces{\rm :}
\begin{equation*}
X\defeq C^{2,\alpha}({R^{(u,\btheta)}_{\gsh}[b_1,b_2]}),
\quad Y\defeq C^{2,\alpha}(R^{(u,\btheta)}_{(1+\kappa)\gsh}[b_1,b_2]),
\quad
Y^-\defeq C^{2,\alpha'}(R^{(u,\btheta)}_{(1+\frac{\kappa}{2})\gsh}[b_1,b_2]) \, ;
\end{equation*}

\item \label{item-b-Prop:Propty4ExtenOper}
Fix $\sigma>0$ and $\epsilon\in(0,\frac14]$.
Fix $(b_1,b_2)=
(-1,-1+\epsilon)$ or
$(1-\epsilon,1),$
and $C=C_{\rm par}>0$ depending only on $(\gamma, v_2,\theta_*,\alpha,\sigma)$.
For any $\alpha'\in(0,\alpha)$ and $\sigma'\in(0,\sigma),$ denote the function spaces{\rm :}
\begin{equation*}
\begin{aligned}
&X\defeq C_{(\sigma),({\rm par})}^{2,\alpha}({R^{(u,\btheta)}_{\gsh}[b_1,b_2]})\,,
\qquad
Y\defeq C_{(\sigma),({\rm par})}^{2,\alpha}({R^{(u,\btheta)}_{(1+\kappa)\gsh}[b_1,b_2]})\,,\\
&Y^-\defeq C_{(\sigma'),({\rm par})}^{2,\alpha'}({R^{(u,\btheta)}_{(1+\frac{\kappa}{2})\gsh}[b_1,b_2]})\,;
\end{aligned}
\end{equation*}

\item \label{item-c-Prop:Propty4ExtenOper}
Fix $(b_1, b_2)$ with either $(b_1, b_2)=(b^{(1)}_1,b^{(1)}_2)=(\frac{1}{2},1)$
or $(b_1, b_2)=(b^{(2)}_1,b^{(2)}_2)=(-1,-\frac{1}{2}),$ and $C=C_{\rm sub}>0$
depending only on $(\gamma, v_2,\theta_*,\alpha)$.
For any $\alpha'\in(0,\alpha),$ denote the function spaces{\rm :}
\begin{equation*}\begin{aligned}
& X\defeq C^{2,\alpha}_{(-1-\alpha),\{s=(-1)^{i-1}\}}({R^{(u,\btheta)}_{\gsh}[b^{(i)}_1,b^{(i)}_2]})\,,
\quad
Y\defeq C^{2,\alpha}_{(-1-\alpha),\{s=(-1)^{i-1}\}}({R^{(u,\btheta)}_{(1+\kappa)\gsh}[b^{(i)}_1,b^{(i)}_2]})\,, \hspace{-1.6em} \\
& Y^-\defeq C^{2,\alpha'}_{(-1-\alpha'),\{s=(-1)^{i-1}\}}({R^{(u,\btheta)}_{(1+\frac{\kappa}{2})\gsh}[b^{(i)}_1,b^{(i)}_2]})\,.
\end{aligned}\end{equation*}
\end{enumerate}

\noindent Then the extension operator $\mathcal{E}_{\gsh}$ satisfies the following{\rm :}
\begin{enumerate}[{\rm (i)}]
\item \label{Prop:P4EO-1}
There exists $C>0$ for each case above such that
$\|\mathcal{E}_{\gsh}(v)\|_{Y} \leq C \|v\|_{X}$\,.
Furthermore, for {\rm Case}~${\rm (c)},$ if $(v,Dv)=(0,\mathbf{0})$
on $\overline{R^{(u,\btheta)}_{\gsh}}\cap\{x_1=(-1)^{i-1}\}$ for $i=1,2,$ then
	\begin{equation*}
		(\mathcal{E}_{\gsh}(v),D\mathcal{E}_{\gsh}(v))=(0,\mathbf{0})
		\qquad \text{on $\overline{R^{(u,\btheta)}_{(1+\kappa)\gsh}}\cap\{x_1=(-1)^{i-1}\}$}\,.
	\end{equation*}
	
\item \label{Prop:P4EO-2}
For each case, $\mathcal{E}_{\gsh}: X \rightarrow Y$ is linear and continuous.

\item \label{Prop:P4EO-3}
Suppose that $\{(u_k,\btheta_k)\}_{k\in\mathbb{N}} \subseteq \overline{\mathcal{K}^{\rm ext}}$
converges to $(u,\btheta)$ in $C^{2,\tilde{\alpha}}_{(*)}(\qiter)\times[0,\theta_*]^2$
for some $\tilde{\alpha}\in(0,1)$.
Write $X_{k}$ for the function space $X$ with $(u,\btheta)$ replaced by $(u_k,\btheta_k)$.
If sequence $\{v_k\}_{k\in\mathbb{N}}\subseteq X_{k}$ satisfies
		$\,\,\|v_k\|_{X_{k}} \leq M$ for all $k\in\mathbb{N}$
for some constant $M>0$ and $\{v_k\}_{k\in\mathbb{N}}$ converges uniformly
to $v$ on any compact set $K\Subset R^{(u,\btheta)}_{\gsh}$ for some $v \in X,$ then
	\begin{equation*}
		\mathcal{E}_{\gsh^{(u_k,\btheta_k)}}(v_k) \rightarrow
		\mathcal{E}_{\gsh^{(u,\btheta)}}(v)
		\qquad\text{in $Y^-$}\,,
	\end{equation*}
where $\mathcal{E}_{\gsh^{(u_k,\btheta_k)}}(v_k)$ is well defined on $\overline{R^{(u,\btheta)}_{(1+\frac{\kappa}{2})\gsh}[b_1,b_2]}$ for large enough $k$.
\end{enumerate}
\end{proposition}

The proof of the following result is similar to~\cite[Lemma 5.5]{BCF-2019}.

\begin{lemma}\label{Lem:hat-g-sh-exist}
Let parameters $(\alpha,\epsilon,\delta_1,\delta_3,N_1)$ in {\rm Definition~\ref{def:iteration-set}}
be chosen as in {\rm Lemma~\ref{lem:BCF-cor4.45},}
with $\delta_2>0$ to be specified later.
Then there exists a constant $\delta_3^{({\rm imp})}>0$ depending only on $(\gamma, v_2,\theta_*,\delta_2)$
such that, if $\delta_3\in(0,\delta^{({\rm imp})}_3),$
then, for each $(u,\btheta)\in\overline{\mathcal{K}},$
there is a unique function $\hat{\mathfrak{g}}_{\rm sh}:[-1,1]\to\overline{\mathbb{R}_+}$ such that
\begin{equation}\label{Eq:Given-hat-g-sh}
\big( w_2-\mathcal{E}_{\gsh^{(u,\btheta)}}(\hat{w}) \big) (s,\hat{\mathfrak{g}}_{\rm sh}(s))=0
\qquad\, \text{for all $s\in[-1,1]$}\,.
\end{equation}
Furthermore, there exists a constant $C>0$ depending only on $(\gamma, v_2,\theta_*)$ such that
$\hat{\mathfrak{g}}_{\rm sh}$ satisfies
\begin{flalign}
&&
\left\{\,
\begin{aligned}
&\norm{\hat{\mathfrak{g}}_{\rm sh}-\mathfrak{g}_{2}}_{2,2\alpha,(-1,1)}^{(-1-2\alpha),\{-1,1\}}\leq C\,, \qquad
\norm{\hat{\mathfrak{g}}_{\rm sh}-\gsh^{(u,\btheta)}}_{2,\frac{\alpha}{2},(-1,1)} \leq C\delta_3\,,\\
& \frac{{\rm d}^m }{{\rm d}s^m}(\hat{\mathfrak{g}}_{\rm sh}-\mathfrak{g}_2)(\pm1)=
\frac{{\rm d}^m }{{\rm d}s^m}(\hat{\mathfrak{g}}_{\rm sh}-\gsh^{(u,\btheta)})(\pm1)=0 \qquad\,
\text{for $m=0,1$}\,.
\end{aligned}
\right.
&&
\end{flalign}
\end{lemma}

\begin{definition}[Iteration map]
\label{Def:IterMap-I-u-btheta}
Let parameters $(\alpha,\epsilon,\delta_1,\delta_3,N_1)$ in {\rm Definition~\ref{def:iteration-set}}
be chosen as in {\rm Lemma~\ref{lem:BCF-cor4.45},}
with $\delta_2>0$ to be specified later.
Furthermore, let \(\delta_3 \in (0, \delta_3^{({\rm imp})})\) for constant $\delta_3^{({\rm imp})}>0$ from {\rm Lemma~\ref{Lem:hat-g-sh-exist}}.
For each $(u,\btheta)\in\overline{\mathcal{K}},$ let $\tilde{u}:\overline{\qiter}\to\mathbb{R}$ be given by
\begin{equation*}
\tilde{u}\defeq\mathcal{E}_{\gsh^{(u,\btheta)}}(\hat{w})\circ(G_{2,\hat{\mathfrak{g}}_{\rm sh}})^{-1}\,,
\end{equation*}
where $G_{2,\hat{\mathfrak{g}}_{\rm sh}}$ is given by~\eqref{eq:map-G2} with $\hat{\mathfrak{g}}_{\rm sh}:[-1,1]\to\mathbb{R}_+$ from~\eqref{Eq:Given-hat-g-sh}.
Then define the iteration map $\mathcal{I}:\overline{\mathcal{K}}\to C^{2,\alpha}_{(*)}(\qiter)$ by
\begin{equation*}
\mathcal{I}(u,\btheta)=\tilde{u}\,.
\end{equation*}
\end{definition}

Note that map $\mathcal{I}:\overline{\mathcal{K}}\to C^{2,\alpha}_{(*)}(\qiter)$ defined as above
is reasonable, since
\begin{equation}\label{Eq:UniformBdd-tilde-u}
\|\mathcal{I}(u,\btheta)\|^{(*)}_{2,2\alpha,\qiter}\leq C
\qquad\text{for all $(u,\btheta)\in\mathcal{\mathcal{K}}$}
\end{equation}
for $C>0$ depending only on $(\gamma, v_2,\theta_*)$,
which follows from Propositions~\ref{Prop:openness-of-iteration-set} and \ref{Prop:Propty4ExtenOper},
and Lemma~\ref{Lem:hat-g-sh-exist}.

\subsection{Proof of Theorem 2.1: Existence of admissible solutions} \label{subsec:proof-existence}
Fix $\theta_*\in(0,\theta^{\rm d})$. For the iteration map $\mathcal{I}$ defined in
Definition~\ref{Def:IterMap-I-u-btheta}, we now prove that,
for any fixed point $u\in\overline{\mathcal{K}}(\btheta)$ of \(\mathcal{I}(\cdot,\btheta)\) for
some $\btheta\in\Theta \cap [0,\theta_*]^2$,
then $\varphi=\varphi^{(u,\btheta)}$ defined in Definition~\ref{def:u-to-phi}\eqref{item3-def:u-to-phi}
is an admissible solution corresponding to $\btheta$ in the sense of Definition~\ref{Def:AdmisSolus}.

\begin{proposition}[Fixed points of map \(\idot\)]
\label{Prop:FixedPt-eqiv-AdmisSolu}
Let parameters $(\alpha,\epsilon,\delta_1,\delta_3,N_1)$ in {\rm Definition~\ref{def:iteration-set}}
be chosen as in {\rm Definition~\ref{Def:IterMap-I-u-btheta}}.
Then parameters $(\epsilon,\delta_1)$ can be chosen even smaller, depending only on $(\gamma, v_2,\theta_*),$
such that, for each $\btheta\in\Theta \cap [0,\theta_*]^2,$
$u\in\mathcal{K}(\btheta)$ is a fixed point of
$\mathcal{I}(\cdot,\btheta):\overline{\mathcal{K}}(\btheta)\to C^{2,\alpha}_{(*)}(\qiter)$
if and only if $\varphi=\varphi^{(u,\btheta)}$ defined in
{\rm Definition~\ref{def:u-to-phi}\eqref{item3-def:u-to-phi}} is an admissible solution
corresponding to $\btheta$ in the sense of {\rm Definition~\ref{Def:AdmisSolus}}
after extending $\varphi$ to $\mathbb{R}^2_+$ via~\eqref{eq:extension-of-varphi}.
\end{proposition}

\begin{proof}
Since the proof is similar to that for~\cite[Proposition 5.8]{BCF-2019},
we omit the details of the proof here and sketch only the main ideas in the following four steps.

\smallskip
\textbf{1.}
For $\btheta\in\Theta \cap [0,\theta_*]^2$, let $\varphi$ be an admissible solution corresponding to $\btheta$,
and let $u\defeq u^{(\varphi,\btheta)}$ be given by~\eqref{eq:phi-to-u},
 which satisfies \(u \in \mathcal{K}(\btheta)\) by Lemma~\ref{lem:BCF-cor4.45}\eqref{lem:BCF-cor4.45-item-i}.
It follows directly from Definition~\ref{def:u-to-phi}\eqref{item1-def:u-to-phi}
and the boundary condition on $\Gamma_{\rm shock}$ that $\gsh^{(u,\btheta)}=\gsh$.
It is clear that $\hat{\phi}\defeq\phi=\varphi+\frac12|\bm{\xi}|^2$ solves the iteration boundary value problem~\eqref{eq:iteration-BVP},
which indicates that $\hat{w}=w$ in~\eqref{Eq:Sec5-Def-3w-funcs}.
Then $\hat{\mathfrak{g}}_{\rm sh}=\gsh$ by solving~\eqref{Eq:Given-hat-g-sh}. We conclude that
\begin{equation*}
\mathcal{I}(u,\btheta)=\mathcal{E}_{\gsh^{(u,\btheta)}}(\hat{w})\circ(G_{2,\hat{\mathfrak{g}}_{\rm sh}})^{-1}
=\mathcal{E}_{\gsh}({w})\circ(G_{2,\gsh})^{-1}=u\,,
\end{equation*}
so \(u \in \mathcal{K}(\btheta)\) is a fixed point of \(\mathcal{I}(\cdot,\btheta)\).

\smallskip
\textbf{2.}
On the other hand, for any fixed point $u$ of
map $\mathcal{I}(\cdot,\btheta):\overline{\mathcal{K}}(\btheta)\to C^{2,\alpha}_{(*)}(\qiter)$, let $\varphi^{(u,\btheta)}$
be given by Definition~\ref{def:u-to-phi}\eqref{item3-def:u-to-phi}.
We now show that $\varphi^{(u,\btheta)}$
is an admissible solution in the sense of Definition~\ref{Def:AdmisSolus} corresponding to $\btheta$.
Using Definition~\ref{def:u-to-phi}\eqref{item1-def:u-to-phi}, \eqref{Eq:Given-hat-g-sh},
and Definition~\ref{Def:IterMap-I-u-btheta}, we have
\begin{equation*}
w_2(s,\gsh^{(u,\btheta)})=u(s,1)=\mathcal{E}_{\gsh^{(u,\btheta)}}(\hat{w})(s,\hat{\mathfrak{g}}_{\rm sh})=w_2(s,\hat{\mathfrak{g}}_{\rm sh})\,,
\end{equation*}
from which we obtain that $\hat{\mathfrak{g}}_{\rm sh}=\gsh^{(u,\btheta)}$, $\hat{u}=u$,
and $\hat{\varphi}^{(u,\btheta)}=\varphi^{(u,\btheta)}$.

Define $v\defeq\hat{\varphi}^{(u,\btheta)}-\varphi_2$.
Then $v$ solves the following strictly elliptic equation:
\begin{equation*}
\mathcal{L}_{(u,\btheta)}(v)=\sum_{i,j=1}^{2}A_{ij}(D\hat{\phi}(\bm{\xi}),\bm{\xi})\partial_{i} \partial_{j} v = 0\,.
\end{equation*}
From the boundary conditions of the iteration boundary value problem~\eqref{eq:iteration-BVP} and $\hat{u}=u$ above,
we see that $v\leq0$ on $\Gamma_{\rm shock}\cup\Gamma_{\rm sonic}$
and  $\partial_{\eta}v=-v_2>0$ on $\Gamma_{\rm sym}$.
From the maximum principle and
Hopf's lemma, we obtain that $\hat{\varphi}^{(u,\btheta)}\leq\varphi_2$ on $\overline{\Omega}$.

When
$\max\{\theta_1,\theta_2\}\in[\frac{2\delta_1}{N_1^2},\theta_*]$, we see that
$\hat{\varphi}^{(u,\btheta)}\geq\max\{\varphi_5,\varphi_6\}$ in $\Omega$,
which follows directly from~\eqref{eq:iteration-set-4-1n2} in {\rm Definition~\ref{def:iteration-set}}
and Step~\textbf{1} of Lemma~\ref{lem:BCF-cor4.45}.

When
$\max\{\theta_1,\theta_2\}\in(0,\frac{2\delta_1}{N_1^2}]$, we can show the same result
by applying the maximum principle as in the proof of~\cite[Lemma 4.42]{BCF-2019}.
Then inequality~\eqref{EntropyIneqMutant} in Definition~\ref{Def:AdmisSolus}\eqref{item4-Def:AdmisSolus} is proved.

Similarly, for inequality~\eqref{DirecDerivMonotone-S25-S26} in
Definition~\ref{Def:AdmisSolus}\eqref{item4-Def:AdmisSolus}, we also split into two cases.
Denote $w\defeq\partial_{\bm{e}_{S_{25}}}(\varphi_2-{\varphi}^{(u,\btheta)})$.
We introduce $(X_1,X_2)$--coordinates by $\bm{\xi} \eqdef X_1\bm{e}_{S_{25}}+X_2\bm{e}^{\perp}_{S_{25}}$.
From the equation for $v$, one can derive the following strictly elliptic equation for $w=-\partial_{X_1}v$:
\begin{equation*}
\hat{\mathcal{L}}_{(u,\btheta)}(w)\defeq
\sum_{i,j=1}^{2}\hat{A}_{ij}\partial_{X_iX_j}w+\sum_{i=1}^{2}\hat{A}_{i}\partial_{X_i}w=0
\qquad\,\, \text{in $\Omega$}\,.
\end{equation*}

When $\max\{\theta_1,\theta_2\}\in[\frac{2\delta_1}{N_1^2},\theta_*]$,
it follows from~\eqref{eq:iteration-set-4-1n2} in {\rm Definition~\ref{def:iteration-set}} that
$w<0$ in $\overline{\Omega}\setminus\mathcal{D}^5_{\epsilon/10}$.
In domain $\Omega\cap\mathcal{D}^5_{\epsilon}$, one can check
\begin{equation}\label{Eq:BdryCondi4w}
\begin{cases}
\, w=0 \qquad &\quad \text{on $\Gamma^5_{\rm sonic}$}\,,\\
\, \bm{b}_{\rm w}\cdot Dw=0 \qquad &\quad\text{with $\bm{b}_{\rm w}\cdot\bm{\nu}>0$ on $\Gamma_{\rm sym}$}\,,\\
\, \bm{b}_{\rm sh}\cdot Dw=0 \qquad &\quad\text{with $\bm{b}_{\rm sh}\cdot\bm{\nu}>0$ on $\Gamma_{\rm shock}$}\,,
\end{cases}
\end{equation}
where $\bm{\nu}$ is the unit normal vector on $\Gamma_{\rm sym}\cup\Gamma_{\rm shock}$ pointing to $\Omega$.
Note that we have used~\cite[Lemma 13.4.5]{ChenFeldman-RM2018} to derive the boundary condition
on $\Gamma_{\rm shock}$.
By the maximum principle and Hopf's lemma, we obtain that $w\leq0$ in ${\Omega}\cap\mathcal{D}^5_{\epsilon}$
when $\epsilon\in(0,\epsilon_{\rm sh}]$, for some sufficiently small $\epsilon_{\rm sh}>0$ depending only on $(\gamma,v_2,\theta_*)$.

When $\max\{\theta_1,\theta_2\}\in(0,\frac{2\delta_1}{N_1^2}]$,
using property \eqref{def:iteration-set-1} of {\rm Definition~\ref{def:iteration-set}},
one can follow the same procedure as the first case to derive that
$w\leq0$ in $\Omega$ when $\delta_1\in(0,\delta_{\rm sh}]$, for some $\delta_{\rm sh}>0$
depending only on $(\gamma,v_2,\theta_*)$.
Therefore, $\hat{\varphi}^{(u,\btheta)}$ satisfies~\eqref{DirecDerivMonotone-S25-S26}
when $\epsilon\in(0,\epsilon_{\rm sh}]$ and $\delta_1\in(0,\delta_{\rm sh}]$,
since the argument above also works for $\partial_{\bm{e}_{S_{26}}}(\varphi_2-{\varphi}^{(u,\btheta)})$.

\smallskip
\textbf{3.} Now we check that \(\mathcal{N}_{(u,\btheta)}(\phi) = 0\)
in {\rm Definition~\ref{def:iteration-set}\eqref{def:iteration-set-7}}
coincides with~\eqref{Eq4phi} in \(\Omega\).
From Lemma~\ref{Lem:Property4ModifyEq}\eqref{item3-Lem:Property4ModifyEq},
this property holds in \(\Omega \setminus \mathcal{D}_{\epsilon/10}\),
so it remains to show that this property holds in \(\Omega \cap \mathcal{D}_{\epsilon/2}\).
We only consider domain $\Omega\cap \mathcal{D}^5_{\epsilon/2}$,
since the proof is the same for domain $\Omega\cap \mathcal{D}^6_{\epsilon/2}$
due to the symmetry between $\theta_1$ and $\theta_2$ in the problem.

For any \(\bm{\xi}\in \Omega \cap \mathcal{D}_{\epsilon}\), let coordinates $(x,y) = \rcal(\bm{\xi})$ be given by~\eqref{Eq:DefPolarCoordiSec4}.
Using the monotone property~\eqref{Eq:xP01MonotoneWrtBeta}, we choose a constant
$\epsilon'_{\rm sh}\in(0,\epsilon_{\rm sh}]$ sufficiently small, depending only on $(\gamma,v_2,\theta_*)$,
such that
$\theta_1\leq\theta^{\rm s}+\min\{\frac{\sigma_3}{2},\delta_{\rm p}\}$ whenever $x_{P_0^1}<\frac{\epsilon'_{\rm sh}}{10}$,
for $\delta_{\rm p}>0$ the constant from Step~\textbf{3}\eqref{item2-Lem:AprioriIterBVP}
in the proof of Lemma~\ref{lem:BCF-cor4.45}.

Choose \(\epsilon \in ( 0 ,\epsilon'_{\rm sh}]\), and suppose $x_{P_0^1}<\frac{\epsilon}{10}$.
Denote $\tilde{w}\defeq Ax-\partial_x\psi(x,y)$ for $A\defeq\frac{1}{1+\gamma}(2-\frac{\mu_0}{5})$
and $\psi= (\varphi^{(u,\btheta)}-\varphi_5) \circ \rcal^{-1}$.
It is clear that $\tilde{w}=0$ on $\rcal(\Gamma_{\rm sonic}^5)$,
and $\tilde{w}_y=0$ on $\rcal(\Gamma_{\rm sym}\cap\partial\mathcal{D}_{\epsilon/2}^5)$.
From Lemma~\ref{Lem:Property4ModifyBdryCondi}\eqref{item1-Lem:Property4ModifyBdryCondi}
and~\eqref{item3-Lem:Property4ModifyBdryCondi}--\eqref{item4-Lem:Property4ModifyBdryCondi},
the boundary condition: $\mathcal{M}_{(u,\btheta)}(D\hat{\phi},\hat{\phi},\bm{\xi})=0$
in~\eqref{eq:iteration-BVP} can be written as
\begin{equation*}
(b_1,b_2,b_0)\cdot(D\hat{\phi},\hat{\phi})=0 \qquad
\text{on $\Gamma_{\rm shock}\cap\mathcal{D}_{\epsilon}^5$}\,,
\end{equation*}
for some $b_0,b_1,b_2\in[-\frac{1}{\delta_{\rm bc}},-\delta_{\rm bc}]$.
Using the {\it a priori} estimates~\eqref{eq:uniform-estimate-u-hat} in the case:
$\theta_1\leq\theta^{\rm s}+\delta_{\rm p}$,
we see that $|\psi_x|\leq \frac{1}{\delta^2_{\rm bc}}(|\psi_y|+|\psi|)\leq C x^{\frac{3}{2}}$
for any $(x,y)\in \rcal(\Gamma_{\rm shock}\cap\mathcal{D}_{\epsilon}^5)$.
Reducing $\epsilon'_{\rm sh}>0$ further, depending only on \((\gamma,v_2)\), we obtain that $\tilde{w}>0$ on
$\rcal(\Gamma_{\rm shock}\cap\mathcal{D}_{\epsilon}^5)$ for any $\epsilon\in(0,\epsilon'_{\rm sh}]$.

By Lemma~\ref{Lem:Property4ModifyEq}\eqref{item2-Lem:Property4ModifyEq} and the fact that $\hat{\phi}=\phi$,
we have
\begin{equation*}
\sum_{i,j=1}^2A_{ij}(D\hat{\phi},\bm{\xi})\partial_{i} \partial_{j}\hat{\phi}
=\sum_{i,j=1}^2A^{(3)}_{ij}(D\hat{\phi},\bm{\xi})\partial_{i} \partial_{j}\hat{\phi}
=  c_\btheta \mathcal{N}^{\rm polar}_{(u,\btheta)} (\hat{\phi})\,,
\end{equation*}
from where we can derive the following equation for $\tilde{w}$:
\begin{equation*}
a_{11}\tilde{w}_{xx}+2a_{12}\tilde{w}_{xy}+a_{22}\tilde{w}_{yy}+a_1\tilde{w}_x+a_2\tilde{w}_y=-A((\gamma+1)A-1)+E(x,y) \qquad
\text{in $\rcal(\Omega\cap\mathcal{D}_{\epsilon/2}^5)$}\,.
\end{equation*}
From expression~\eqref{Eq:Def-co-k} and the {\it a priori} estimates~\eqref{eq:uniform-estimate-u-hat}
in Lemma~\ref{lem:BCF-cor4.45}, reducing $\epsilon'_{\rm sh}>0$ further if necessary,
we obtain directly that $E(x,y)<A((\gamma+1)A-1)$
in $\rcal(\Omega\cap\mathcal{D}_{\epsilon/2}^5)$ for any $\epsilon\in(0,\epsilon'_{\rm sh}]$.
Therefore, the property: $\tilde{w}=Ax-\partial_x\psi(x,y)\geq0$ holds
in $\rcal(\Omega\cap \mathcal{D}^5_{\epsilon/2})$, which follows from
Lemma~\ref{Lem:Property4ModifyEq}\eqref{item1-Lem:Property4ModifyEq}, the maximum principle, and Hopf's lemma.

Now we show the inequality: $\partial_x\psi(x,y)\geq-Ax$.
From Step~\textbf{2}, we have
\begin{equation*}
\big( \partial_{e_{S_{25}}}(\varphi_2-\varphi) \big) \circ \rcal^{-1} = \psi_x\cos(\theta_{25}-y)-\frac{\sin(\theta_{25}-y)}{c_5-x}\psi_y \leq0
\qquad \text{in $\rcal(\Omega\cap\mathcal{D}^5_{\epsilon/2})$}\,.
\end{equation*}
This implies that
\begin{equation*}
\psi_x\geq-\frac{\tan(\frac{\pi}{2}-\tilde{\omega}_0)}{c_5-x}\psi_{y}\geq -Cx^{\frac32}\qquad
\text{in $\rcal(\Omega\cap\mathcal{D}^5_{\epsilon/2})$}\,,
\end{equation*}
by using property~\eqref{Eq:D5n6epsilonSetBdd} and
the {\it a priori} estimates~\eqref{eq:uniform-estimate-u-hat} in Lemma~\ref{lem:BCF-cor4.45}.
Then $\partial_x\psi(x,y)\geq-Ax$ after further reducing $\epsilon'_{\rm sh}>0$.
Thus, $|\psi_x|\leq Ax$ is proved, which indicates that
\(\mathcal{N}_{(u,\btheta)}(\phi) = 0\) in {\rm Definition~\ref{def:iteration-set}\eqref{def:iteration-set-7}}
coincides with~\eqref{Eq4phi} in \(\Omega\cap\mathcal{D}^5_{\epsilon/2}\).
Moreover, if $x_{P_0^{1}} \geq \frac{\epsilon}{10}$,
the same conclusion holds in
\(\Omega\cap\mathcal{D}^5_{\epsilon/2}\subseteq \Omega \setminus \mathcal{D}^{6}_{\epsilon/10}\)
by Lemma~\ref{Lem:Property4ModifyEq}\eqref{item3-Lem:Property4ModifyEq}.

\smallskip
\textbf{4.} Definition~\ref{Def:AdmisSolus}\eqref{item3-Def:AdmisSolus} follows directly
from  Lemma~\ref{Lem:Property4ModifyEq}\eqref{item1-Lem:Property4ModifyEq} and the fact proved in Step~{\bf 3}.
Let the sound speed $c(|D\varphi|,\varphi)$ be given by \eqref{Def4SonicSpeedInSec3}.
The strict ellipticity implies that
\begin{equation*}
M\defeq\frac{|\partial_{\bm{\nu}}\varphi(\bm{\xi})|}{c(|D\varphi(\bm{\xi})|,\varphi(\bm{\xi}))}
<1 \qquad \text{on $\Gamma_{\rm shock}$}\,.
\end{equation*}
Define $M_{2}\defeq|\partial_{\bm{\nu}}\varphi_2(\bm{\xi})|$.
Similar to~\cite[Eq.~$(2.4.9)$]{BCF-2019}, we have the following equality:
\begin{equation*}
(1+\tfrac{\gamma-1}{2}M^2)M^{-\tfrac{2(\gamma-1)}{\gamma+1}}=(1+\tfrac{\gamma-1}{2}M_2^2)M_2^{-\tfrac{2(\gamma-1)}{\gamma+1}}\,.
\end{equation*}
Then $M_2>1$, and so $|D\varphi_2|>1=c(|D\varphi_2|,\varphi_2)$ holds
on $\Gamma_{\rm shock}$,
\textit{i.e.}, $\Gamma_{\rm shock}\subseteq \mathbb{R}^2_+\setminus\overline{B_{c_2}(O_2)}$.
Finally, for the property: $\Gamma_{\rm shock}\subseteq \{\xi^{P_3}<\xi<\xi^{P_2}\}$,
we follow the same proof as in Lemma~\ref{Prop:GammaShockExpres} by using all the properties above.
Therefore, all parts of Definition~\ref{Def:AdmisSolus} have been verified,
which implies that $\varphi^{(u,\btheta)}$ is an admissible solution
in the sense of {\rm Definition~\ref{Def:AdmisSolus}}.
\end{proof}

\medskip
\noindent \textbf{Proof of Theorem~\ref{Thm:Existence-AdmisSolu}.}
We now complete the proof of Theorem~\ref{Thm:Existence-AdmisSolu} in five steps.

\smallskip
\textbf{1.} We prove that the iteration map $\mathcal{I}:\overline{\mathcal{K}}\to C^{2,\alpha}_{(*)}(\qiter)$
is a compact map in the sense of {\rm Definition~\ref{Def:CompactThm}}.
Let $\{(u_k,\btheta_k)\}_{k\in\mathbb{N}}\subseteq\mathcal{K}$
converge to $(u,\btheta)$ in $C^{2,\alpha}_{(*)}(\qiter)\times[0,\theta_*]^2$.
For each $k\in\mathbb{N}$, write $\Omega^{(k)} \defeq \Omega(u_k,\btheta_k)$
and $\mathfrak{g}_{\rm sh}^{(k)} \defeq \mathfrak{g}_{\rm sh}^{(u_k,\btheta_k)}$,
as given by Definition~\ref{def:u-to-phi},
and $\Gamma_{\rm sonic}^{(k)}\defeq\Gamma^{5}_{\rm sonic}\cup\Gamma^6_{\rm sonic}$
for the sonic arcs
related to parameters $\btheta_k$.
By Lemma~\ref{Lem:Wellpose4iterBVP}, there exists a unique solution
$\hat{\phi}^{(k)}\in C^2(\Omega^{(k)})\cap C^1(\overline{\Omega^{(k)}}\setminus \Gamma_{\rm sonic}^{(k)})
\cap C^0(\overline{\Omega^{(k)}})$ of the iteration boundary value problem~\eqref{eq:iteration-BVP} related to $(u_k,\btheta_k)$.
For any $(s,t')\in R_{\mathfrak{g}^{(k)}_{\rm sh}}$, define
\begin{equation*}
\hat{w}^{(k)}(s,t')\defeq \big(\hat{\phi}^{(k)}+\frac{1}{2}|\bm{\xi}|^2-\varphi^*_{\btheta_k}\big) \circ (\mathcal{G}_1^{\btheta_k})^{-1}(s,t')\,,
\end{equation*}
where $R_{\mathfrak{g}^{(k)}_{\rm sh}}$, $\varphi^*_{\btheta_k}$,
and $\mathcal{G}_1^{\btheta_k}$ are given by~\eqref{Eq:DefRgDomain},
Definition~\ref{def:varphi-theta-star}, and~\eqref{Eq:Def-gtheta}, respectively.
Then
\begin{equation*}
\hat{w}^{(k)}=\hat{u}^{(k)} \circ G_{2,\mathfrak{g}^{(k)}_{\rm sh}}
\end{equation*}
for $\hat{u}^{(k)}:\overline{\qiter}\to\mathbb{R}$ given by~\eqref{eq:def-u-hat}
with $\hat{\phi}=\hat{\phi}^{(k)}$, and
$G_{2,\mathfrak{g}^{(k)}_{\rm sh}}$ given by~\eqref{eq:map-G2}.
Subsequently, we obtain functions $\hat{\mathfrak{g}}^{(k)}_{\rm sh}:[-1,1]\to\overline{\mathbb{R}_+}$ via Lemma~\ref{Lem:hat-g-sh-exist}.

Since $({u}_k,\btheta_k)$ converges to $(u,\btheta)$ in $C^{1,\alpha}(\qiter)\times[0,\theta_*]^2$,
we find that
$\mathfrak{g}_{\rm sh}^{(k)}\to \mathfrak{g}_{\rm sh}$ in $C^{1,\alpha}([-1,1])$
by Lemma~\ref{lem:properties-of-G-alpha-theta}\eqref{lem:properties-of-G-alpha-theta-4}.
Fix any compact subset $K\subseteq R_{\mathfrak{g}_{\rm sh}}$.
From~\eqref{eq:map-G2}, Definition~\ref{def:iteration-set}\eqref{def:iteration-set-3},
and the convergence of $\{\mathfrak{g}_{\rm sh}^{(k)}\}_{k\in\mathbb{N}}$, we see that
$G_{2,\mathfrak{g}^{(k)}_{\rm sh}} \to G_{2,\mathfrak{g}_{\rm sh}}$ in $C^{1,\alpha}(K)$.
Then, by Corollary~\ref{corol:ConvergeInHolder4uHat} and the convergence of $\{\hat{u}^{(k)}\}_{k\in\mathbb{N}}$,
we obtain that
$\hat{w}^{(k)}\to\hat{w}$ in $C^{1,\alpha}(K)$.

By Lemma~\ref{lem:BCF-cor4.45}, Proposition~\ref{Prop:Propty4ExtenOper}\eqref{Prop:P4EO-3},
and the local convergence of $\{\hat{w}^{(k)}\}_{k\in\mathbb{N}}$, for any $b_1,b_2\in[-1,1]$ chosen as in each case of Proposition~\ref{Prop:Propty4ExtenOper}\eqref{item-a-Prop:Propty4ExtenOper}--\eqref{item-c-Prop:Propty4ExtenOper},
 $\{\mathcal{E}_{\mathfrak{g}_{\rm sh}^{(k)}}(\hat{w}^{(k)})\}_{k\in\mathbb{N}}$ converges to $\mathcal{E}_{\mathfrak{g}_{\rm sh}}(\hat{w})$ in the corresponding function space.
Combining the statement above with the convergence of $\{\hat{\mathfrak{g}}^{(k)}_{\rm sh}\}_{k\in\mathbb{N}}$ from Lemma~\ref{Lem:hat-g-sh-exist}, we conclude that
\begin{equation*}
\mathcal{E}_{\mathfrak{g}^{(k)}_{\rm sh}}(\hat{w}^{(k)}) \circ (G_{2,\hat{\mathfrak{g}}^{(k)}_{\rm sh}})^{-1} \to \mathcal{E}_{\mathfrak{g}_{\rm sh}}(\hat{w}) \circ (G_{2,\hat{\mathfrak{g}}_{\rm sh}})^{-1}
\qquad\,\, \text{in $C^{2,\alpha}_{(*)}(\qiter)$}\,.
\end{equation*}
Since $C^{2,2\alpha}_{(*)}(\qiter)$ is compactly embedded into $C^{2,\alpha}_{(*)}(\qiter)$,
we obtain from~\eqref{Eq:UniformBdd-tilde-u} that $\mathcal{I}(U)$ is precompact
in $C^{2,\alpha}_{(*)}(\qiter)$ for any bounded
subset $U\subseteq \overline{\mathcal{K}}$ so
that
the iteration map $\mathcal{I}$ is compact.

\smallskip
\textbf{2.} Next, we show that constants $N_1,\delta_2>0$ in {\rm Definition~\ref{def:iteration-set}}
can be chosen, depending only on $(\gamma, v_2,\theta_*)$,
such that, for any $\btheta\in\Theta\cap[0,\theta_*]^2$, no fixed point of $\mathcal{I}(\cdot,\btheta)$
lies on boundary $\partial\mathcal{K}(\btheta)$,
where $\partial\mathcal{K}(\btheta)$ is considered relative to
$C^{2,\alpha}_{(*)}(\qiter)$.

Suppose that $(u,\btheta)\in\overline{\mathcal{K}}$ satisfies $\mathcal{I}(u,\btheta)=u$.
For $\varphi=\varphi^{(u,\btheta)}$ defined in $\Omega = \Omega(u,\btheta)$
by Definition~\ref{def:u-to-phi}\eqref{item3-def:u-to-phi}, extend \(\varphi\) to \(\mathbb{R}^2_+\)
via~\eqref{eq:extension-of-varphi}.
Then $\varphi$ is an admissible solution corresponding to parameters $\btheta\in\Theta\cap[0,\theta_*]^2$ by Proposition~\ref{Prop:FixedPt-eqiv-AdmisSolu}.

From
\eqref{Eq:FuncgEquivDistb},
Propositions~\ref{prop:lowerBoundBetweenShockAndSonicCircle}, \ref{prop:ellipticDegeneracyNearSonicBoundary},
and~\ref{prop:properties-of-gshock}\eqref{prop:properties-of-gshock:5},
and Lemma~\ref{UniformBound-Lem},
constants $(N_2,\tilde{\mu},,\rho^{*}(\gamma),C_{\rm ub})$ are fixed so that
any admissible solution satisfies the strict inequalities in
Definition~\ref{def:iteration-set}\eqref{def:iteration-set-3} and
\eqref{def:iteration-set-5}--\eqref{def:iteration-set-6}.
Moreover, from Lemma~\ref{lem:BCF-cor4.45}\eqref{lem:BCF-cor4.45-item-i},
$u$ satisfies the strict inequality given in Definition~\ref{def:iteration-set}\eqref{def:iteration-set-1}
when
$N_1\geq N_1^{({\rm adm})}$.

For condition~\eqref{def:iteration-set-4} of Definition~\ref{def:iteration-set},
constants $(\mu_0,\mu_1,N_4,N_5, \sigma_2)$ are fixed such that any admissible solution
corresponding to $\btheta\in\Theta\cap[0,\theta_*]^2$ satisfies
the strict inequalities~\eqref{eq:iteration-set-4-3to7}.
Next, we focus on the strict inequalities~\eqref{eq:iteration-set-4-1n2}.
It is direct to see that the strict inequalities~\eqref{eq:iteration-set-4-1n2}
hold when
$\max\{\theta_1,\theta_2\}<\frac{\delta_1}{N_1^2}$,
since
$\mathscr{K}_2(\max\{\theta_1,\theta_2\})<0$
and inequality~\eqref{EntropyIneqMutant}
in Definition~\ref{Def:AdmisSolus}\eqref{item4-Def:AdmisSolus}
and Lemma~\ref{lem:monotonicityForPhi2-Phi}\eqref{item2-Lem3-2} hold
for any admissible solution corresponding to $\btheta\in\Theta\cap[0,\theta_*]^2$.
When $\max\{\theta_1,\theta_2\}\geq\frac{\delta_1}{N_1^2}$,
we similarly have
\begin{equation*}\begin{aligned}
& \varphi-\max\{\varphi_5,\varphi_6\}>0 \quad &&\text{in} \;\overline{\Omega}\setminus(\mathcal{D}^{5}_{\epsilon/10}\cup\mathcal{D}^{6}_{\epsilon/10})\,,\\
& \partial_{e_{S_{2j}}}(\varphi_2-\varphi)<0 \quad &&\text{in} \;\overline{\Omega}\setminus\mathcal{D}^{j}_{\epsilon/10} \;\text{for} \;j=5,6\,.
\end{aligned}\end{equation*}
By Corollary~\ref{Corol:Esti-AwaySonic5n6}, and Propositions~\ref{Prop:Esti4theta1a}--\ref{Prop:Esti4theta1c} and \ref{Prop:Esti4theta1d},
we obtain that $\varphi\in C^{1,\alpha}(\Omega)$.
Combining this with Lemma~\ref{Lem:Compactness4AdmiSoluSet},
there exist both
$\tau_0>0$ depending only on $(\gamma, v_2,\theta_*)$
and
$\tau_1>0$ depending only on $(\gamma, v_2,\theta_*,\delta_1,N_1)$
such that, for all $\btheta\in\Theta\cap \{\frac{\delta_1}{N_1^2} \leq \max\{\theta_1,\theta_2\}
\leq \theta_* \}$,
\begin{equation*}\begin{aligned}
& \psi=\varphi-\max\{\varphi_5,\varphi_6\}\geq\tau_0 \quad &&\text{in} \;\overline{\Omega}\setminus(\mathcal{D}^{5}_{\epsilon/10}\cup\mathcal{D}^{6}_{\epsilon/10})\,,\\
& \partial_{e_{S_{2j}}}(\varphi_2-\varphi)\leq-\tau_1 \quad &&\text{in} \;\overline{\Omega}\setminus\mathcal{D}^{j}_{\epsilon/10} \;\text{for} \;j=5,6\,.
\end{aligned}\end{equation*}
For any fixed $\delta_1>0$,
we can choose $\delta_2>0$ small enough such that
\begin{equation*}
\mathscr{K}_2(\max\{\theta_1,\theta_2\})\leq\frac{\delta_1\delta_2}{N_1^2}<\min\{\tau_0,\tau_1\}
\qquad\,\, \text{for all
$\btheta\in\Theta\cap \{\frac{\delta_1}{N_1^2} \leq \max\{\theta_1,\theta_2\} \leq \theta_*\}$}\,,
\end{equation*}
which implies that the strict inequalities~\eqref{eq:iteration-set-4-1n2} hold
when
$\max\{\theta_1,\theta_2\}\geq\frac{\delta_1}{N_1^2}$.

Finally, let $(\alpha,\epsilon,\delta_1,\delta_3,N_1)$ be chosen as
in Lemma~\ref{lem:BCF-cor4.45} with $\delta_2$ chosen as above.
Then the strict inequality~\eqref{eq:iteration-u-uhat-estimate} of
Definition~\ref{def:iteration-set}\eqref{def:iteration-set-7} holds
by Proposition~\ref{Prop:openness-of-iteration-set}.
Therefore, $u\in\mathcal{K}(\btheta)$ for any $\btheta\in\Theta\cap[0,\theta_*]^2$,
which completes this step.

\smallskip
\textbf{3.} Now, we can fix constants $ (\alpha,\epsilon,\delta_1,\delta_2,\delta_3,N_1)$ in Definition~\ref{def:iteration-set} as follows:
\begin{itemize}
    \item Choose constants $(\alpha,\epsilon,\delta_1,\delta_3,N_1)$ as in Lemma~\ref{lem:BCF-cor4.45}
    and reduce $(\epsilon,\delta_1)$ as in Proposition~\ref{Prop:FixedPt-eqiv-AdmisSolu},
    depending only on $(\gamma, v_2,\theta_*)$;
    \item Fix $\delta_2$ as in Step~\textbf{2}, depending only on $(\gamma, v_2,\theta_*,\delta_1,N_1)$;
    \item Adjust $\delta_3>0$ as in Definition~\ref{Def:IterMap-I-u-btheta},
    depending only on $(\gamma, v_2,\theta_*,\delta_2)$.
\end{itemize}

\smallskip
\textbf{4.} \textit{Leray-Schauder degree theory.}
When $\btheta=\mathbf{0}$,
$\varphi_5=\varphi_6=\varphi_0$ with $\varphi_0$ defined in \S\ref{SubSec-202NormalShock},
and the iteration boundary value problem
in the $(s,t')$--coordinates is
a homogeneous problem.
From Lemma~\ref{Lem:Wellpose4iterBVP}, the iteration boundary value
problem~\eqref{eq:iteration-BVP} corresponding to any $(u,\mathbf{0}) \in \overline{\mathcal{K}}$
has a unique solution $\hat{\phi} = \frac12|\bm{\xi}|^2 + \varphi_{\mathbf{0}}^{*} = v_2 \eta_0$,
where $\eta_0 > 0$ is defined in \S\ref{SubSec-202NormalShock}.
It follows from Definition~\ref{def:iteration-set} that
the unique fixed point $u \equiv 0$ of $\mathcal{I}(\cdot,\mathbf{0})$ lies in $\mathcal{K}(\mathbf{0})$,
and $\textbf{Ind}(\mathcal{I}(\cdot,\mathbf{0}),\mathcal{K}(\mathbf{0})) = 1$.
From the arguments above,
map $\mathcal{I}(\cdot,\btheta):\overline{\mathcal{K}(\btheta)}\to C^{2,\alpha}_{(*)}(\qiter)$
satisfies $\mathcal{I}(\cdot,\btheta)\in V(\mathcal{K}(\btheta),C^{2,\alpha}_{(*)}(\qiter))$
from Definition~\ref{Def:VertexSet} in Appendix~\ref{Sec:Appendix-B}.
From Theorem~\ref{Thm:HomotopyInvariance4FPI} and for any $t\in[0,1]$, we have
\begin{equation*}
\textbf{Ind}(\mathcal{I}(\cdot,t\btheta),\mathcal{K}(t\btheta))
=\textbf{Ind}(\mathcal{I}(\cdot,\mathbf{0}),\mathcal{K}(\mathbf{0})) \qquad \text{for all $\btheta\in[0,\theta_*]^2$}\,.
\end{equation*}
We obtain the existence of a fixed point of map $\mathcal{I}(\cdot,\btheta)$,
which is equivalent to the existence of an admissible solution corresponding
to $\btheta\in[0,\theta_*]^2$ in the sense of Definition~\ref{Def:AdmisSolus}
by Proposition~\ref{Prop:FixedPt-eqiv-AdmisSolu}.
The arbitrary choice of $\theta_*\in(0,\theta^{\rm d})$ completes
the proof of the existence of admissible solutions for any $\btheta\in[0,\theta^{\rm d})^2$.

\smallskip
\textbf{5.} \textit{Proof of the continuity at $\btheta=\mathbf{0}$.}
Let $\varphi_{\rm norm}$ be the admissible solution for $\btheta=\mathbf{0}$,
which is given by \eqref{eq:normal-reflection-solution},
and let $\Omega^{(\bm{0})}$ be the corresponding pseudo-subsonic domain.
For any sequence $\{\btheta_k\}_{k\in\mathbb{N}} \subseteq \Theta$
with \(\btheta_k \to \bm{0}\) as \(k\to \infty\), let  $\varphi^{(\btheta_k)}$
be an admissible solution corresponding to  $\btheta_k$,
and let $\Omega^{(k)}$ be the corresponding pseudo-subsonic domain.

Let $G\subseteq  \mathbb{R}^2_+$ be any open set satisfying $\overline{G}\Subset \mathbb{R}^2_+$.
Choose any
compact subset $K \Subset G \setminus \partial \Omega^{(\bm{0})}$.
Then, by {\rm Lemma~\ref{Lem:Compactness4AdmiSoluSet}\eqref{Lem:Compactness4AdmiSoluSet-item2}},
passing to a subsequence (still denoted) \(k\),
\(K \cap \Omega^{(\bm{0})} \Subset \Omega^{(k)}\) and \(K\setminus \Omega^{(\bm{0})} \Subset G \setminus \Omega^{(k)}\).
Using Proposition~\ref{prop:properties-of-gshock}\eqref{prop:properties-of-gshock:5},
Proposition~\ref{prop:apriori-estimates-on-u-prop4.12},
Lemma~\ref{lem:properties-of-G-alpha-theta}\eqref{lem:properties-of-G-alpha-theta-7},
and the continuity property of admissible solutions given in Step~\textbf{1} of the proof of Lemma~\ref{lem:BCF-cor4.45}, we have
\begin{equation*}
\norm{\varphi^{(\btheta_{k})} - \varphi_{\rm norm}}_{1,\alpha,K \cap \Omega^{(\bm{0})}} \to 0
\qquad\,\, \text{as $k \to \infty$}\,.
\end{equation*}
Similarly, using~\eqref{eq:normal-reflection-solution},
Proposition~\ref{Prop:localtheorystate5}\eqref{Prop:localtheorystate5:continuity-properties},
and~\eqref{eq:extension-of-varphi},
we have
\begin{equation*}
\norm{\varphi^{(\btheta_k)} - \varphi_{\rm norm}}_{W^{1,1}(K \setminus \Omega^{(\bm{0})})} \to 0
\qquad \text{as $k \to \infty$}\,.
\end{equation*}

Finally, by~\eqref{eq:extension-of-varphi} and the uniform bound in Lemma~\ref{UniformBound-Lem},
there exists a constant \(C >0\) depending only on \((\gamma,v_2,G)\) such that
\(\norm{\varphi^{(\btheta_k)} - \varphi_{\rm norm}}_{0,1,G} \leq C\) for all \(k\in\mathbb{N}\),
which leads to
\begin{equation*}
\|\varphi^{(\btheta_{k})} - \varphi_{\rm norm}\|_{W^{1,1}(G)}
\leq \|\varphi^{(\btheta_{k})} - \varphi_{\rm norm}\|_{W^{1,1}(K)} + 2 C \, {\rm meas}(G \setminus K)\,.
\end{equation*}
This completes the proof of {\rm Theorem~\ref{Thm:Existence-AdmisSolu}}.
\qed

\section{Optimal Regularity of Solutions and Convexity of Free Boundaries}
\label{Sec:OptimalRegularity-Convexity}
With the {\it a priori} estimates obtained in \S\ref{Sec:UniformEstiAdmiSolu}--\S\ref{Sec:IterationMethod-Existence},
we can now give the complete proofs
of Theorems~\ref{RegularityThm}--\ref{Thm3:ConvexTransonicShock}, respectively.

\subsection{Proof of Theorem~\ref{RegularityThm}: Optimal regularity of solutions} \label{subsec:optimal-reg}
This section is devoted to the complete proof of Theorem~\ref{RegularityThm}.
Let $\varphi$ be an admissible solution in the sense of Definition~\ref{Def:AdmisSolus}.

\begin{proof}
Fix \(\btheta \in \Theta\).
By the symmetry of the problem, it suffices to prove the statement near \(\Gamma_{\rm sonic}^5\).

\smallskip
\textbf{1}. \textit{Proof of {\rm Theorem~\ref{RegularityThm}\eqref{ReguThm01}}.}
First, it follows from Lemmas~\ref{LocBddLem} and~\ref{Lem:LocalEstimates4Interior}
that $\Gamma_{\rm{shock}}$ is $C^{\infty}$ in its
relative interior and $\varphi\in C^{\infty}(\overline{\Omega}\setminus \Gamma_{\rm{sonic}}^{5}\cup\Gamma_{\rm{sonic}}^{6})$.
By Definition~\ref{Def:GammaSonicAndPts}, $\Gamma_{\rm{sonic}}^{5}$ is a closed arc of a circle
when $\theta_1\in[0,\theta^{\rm{s}})$, and becomes a point $P_{1}$
when $\theta_{1}\in[\theta^{\rm{s}},\theta^{\rm{d}})$.
If $\theta_1\in[0,\theta^{\rm{s}})$, it follows from Propositions~\ref{Prop:Esti4theta1a}--\ref{Prop:Esti4theta1b}
that $\varphi$ is $C^{1,1}$ up to $\Gamma_{\rm{sonic}}^{5}$,
whilst, if $\theta_{1}\in[\theta^{\rm{s}},\theta^{\rm{d}})$,
then Propositions~\ref{Prop:Esti4theta1c} and~\ref{Prop:Esti4theta1d} imply that
$\varphi$ is $C^{1,\alpha}$ up to $\Gamma_{\rm{sonic}}^{5}=\{P_{1}\}$ for some $\alpha\in(0,1)$.
Furthermore, from Propositions~\ref{Prop:Esti4theta1a}--\ref{Prop:Esti4theta1c}
and \ref{Prop:Esti4theta1d}, we see that $\overline{\Gamma_{\rm shock}\cup S_{25}^{\rm seg}}$ is a $C^{1,\alpha}$--curve,
including at $P_{2}$, for any \(\theta_1 \in [0,\theta^{\rm d})\).
It remains to show that $\overline{\Gamma_{\rm shock}\cup S_{25}^{\rm seg}}$ is a \(C^{2,\alpha}\)--curve when \(\theta_1 \in [0,\theta^{\rm s})\).

For \(\theta_1 \in [0,\theta^{\rm s})\), let coordinates $(x,y) = \rcal(\bm{\xi})$ be
given by~\eqref{PolarCoordiNearO5}.
Let \(\bar{\varepsilon} > 0\),
\(f_{5,0}\), and \(f_{5,{\rm sh}}\) be from Step~\textbf{1} of the proof of Proposition~\ref{Prop:Esti4theta1a}.
We extend
$f_{5,\rm sh}$ to
$(-\bar{\varepsilon},\bar{\varepsilon})$ by
\begin{equation}
\label{f5extend}
f_{5,\rm sh}(x) \defeq {f}_{5,0}(x) \qquad\,
\text{for $x\in(-\bar{\varepsilon},0]$}\,,
\end{equation}
so that
\begin{equation}
 \label{f5extend2}
 (f_{5,\rm sh}-{f}_{5,0})(0)=(f_{5,\rm sh}-{f}_{5,0})^{\prime}(0)=0\,.
\end{equation}
Define $\psi \defeq (\varphi - \varphi_5)\circ \rcal^{-1}$
and $\bar{\psi}_{25} \defeq (\varphi_{2}-\varphi_{5}) \circ \rcal^{-1}$.
Then $\bar{\psi}_{25}(x,{f}_{5,0}(x))=0$, which implies
\begin{equation*}
\label{eqpsi}
\bar{\psi}_{25}(x,f_{5,\rm sh}(x))-\bar{\psi}_{25}(x,{f}_{5,0}(x)) =
\psi (x,f_{5,\rm sh}(x))
\qquad
\text{for $ x \in (0, \bar{\varepsilon})$}\,.
\end{equation*}
Differentiating the above equality with respect to $x$ twice,
we obtain the following expression:
\begin{equation} \label{eq:second-derivative-f5sh-f50}
    (f_{5,{\rm sh}}-{f}_{5,0})^{\prime\prime}(x)=\frac{A_{1}(x)+A_{2}(x)+A_{3}(x)}{\partial_{y}\bar{\psi}_{25}(x,f_{5,\rm sh}(x))}\,,
\end{equation}
where, for \((a_0,a_1,a_2) \defeq (1,2,1)\),
\begin{align*}
    A_{1}(x)&\defeq \sum\limits_{k=0}^{2}a_{k}\big(({f}_{5,0}(x))^{k}\partial_{x}^{2-k}\partial_{y}^{2}\bar{\psi}_{25}(x,{f}_{5,0}(x))
    -(f_{5,{\rm sh}}(x))^{k}\partial_{x}^{2-k}\partial_{y}^{2}\bar{\psi}_{25}(x,f_{5,\rm sh}(x)\big) \,,\\
    A_{2}(x)&\defeq \left(\partial_{y}\bar{\psi}_{25}(x,{f}_{5,0}(x))-\partial_{y}\bar{\psi}_{25}(x,f_{5,\rm sh}(x))\right){f}^{\prime\prime}_{5,0}(x)\,,\\
    A_{3}(x)&\defeq -f_{5,\rm sh}^{\prime\prime}(x)\partial_{y}\psi(x, {f}_{5,\rm sh}(x))-
 \sum\limits_{k=0}^{2}a_{k}(f_{5,\rm sh}^{\prime}(x))^{k}\partial_{x}^{2-k}\partial_{y}^{k}\psi(x,f_{5,\rm sh}(x))\,.
\end{align*}
By~\eqref{f5extend2}, $A_{1}(0)=A_{2}(0)=0$.
We differentiate the boundary condition in the tangential direction along $\Gamma_{\rm shock}$
and apply the arguments in
Step~\textbf{1}
of the proof of Proposition~\ref{Prop:Esti4theta1a} to obtain that there exists
a constant $C>0$ such that
\begin{align*}
    |\psi_{xx}|\leq C |(\psi, \psi_x,\psi_y, \psi_{xy},\psi_{yy})|
\qquad\text{on $\rcal(\Gamma_{\rm shock}\cap  \partial \Omega^5_{\bar{\varepsilon}})$}\,,
\end{align*}
where $\Omega^5_{\bar{\varepsilon}}$ is from Step~\textbf{1} of the proof of Proposition~\ref{Prop:Esti4theta1a}.
From the above estimate and Proposition~\ref{Prop:Esti4theta1a}, we obtain
that
$\psi_{xx}(x,f_{5,\rm sh}(x)) \to 0$ as $x\to 0^+$, which implies that $A_{3}(x) \to 0$ as $x \to 0^{+}$.
Hence, $\partial_{y}\bar{\psi}_{25}(x,f_{5,\rm sh}(x))\neq 0$
on $\rcal(\Gamma_{\rm shock}\cap  \partial \Omega^5_{\bar{\varepsilon}})$.
Then we conclude from~\eqref{eq:second-derivative-f5sh-f50} that
$(f_{5,\rm sh}-{f}_{5,0})^{\prime\prime}(0)=0$.
This implies that the extension of $f_{5,\rm sh}$ given by~\eqref{f5extend} is
in $C^{2}([-\bar{\varepsilon},\bar{\varepsilon}])$.
Furthermore, we conclude from Proposition~\ref{Prop:Esti4theta1a} that
the extension of  $f_{5,\rm sh}$ given by~\eqref{f5extend} is
in $C^{2,\alpha}((-\bar{\varepsilon},\bar{\varepsilon}))$ for any $\alpha\in(0,1)$.
This implies that $\overline{\Gamma_{\rm shock}\cup S_{25}^{\rm seg}}$ is $C^{2,\alpha}$
for any $\alpha\in(0,1)$, including at point $P_2= \rcal^{-1}(0,{f}_{5,\rm sh}(0))$.
This completes the proof of statement~\eqref{ReguThm01}.

\smallskip
\textbf{2}. \textit{Proof of {\rm Theorem~\ref{RegularityThm}\eqref{ReguThm02}--\eqref{ReguThm03}}.}
It suffices to consider the case that $\theta_1 \in [0,\theta^{\rm{s}})$.
We can apply Theorem~\ref{Cite-BCF-2009:Thm3-1} to obtain the regularity of $\psi$
in the neighbourhood of $\rcal(\Gamma_{\rm{sonic}}^{5})$.
Then the admissible solution $\varphi$ satisfies Theorem~\ref{RegularityThm}\eqref{ReguThm02}--\eqref{ReguThm03}.

\smallskip
\textbf{3}. \textit{Proof of {\rm Theorem~\ref{RegularityThm}\eqref{ReguThm04}}.}
By Propositions~\ref{Prop:Esti4theta1a}--\ref{Prop:Esti4theta1b},
$\Gamma_{\rm shock}\cap \partial \Omega^5_{\bar{\varepsilon}}$ can be represented as the graph of $y=f_{5,{\rm sh}}(x)$
for $0\leq x\leq \bar{\varepsilon}$.

Let $\{y_{m}^{(1)}\}_{m\in\mathbb{N}}$ be a sequence satisfying that
$0<y_{m}^{(1)}<{f}_{5,\rm sh}(0)$ for each
$m\in \mathbb{N}$, and $y_{m}^{(1)} \to {f}_{5,\rm sh}(0)$ as ${m\rightarrow \infty}$.
By Theorem~\ref{RegularityThm}\eqref{ReguThm03},
we can choose a sequence $\{x_{m}^{(1)}\}_{m\in\mathbb{N}}$ such that
$(x_{m}^{(1)},y_{m}^{(1)})\in \rcal(\Omega^5_{\bar{\varepsilon}})$,
$0 < x_{m}^{(1)} < \frac{1}{m}$, and
\begin{equation*}
    \tabs{\psi_{xx}(x_{m}^{(1)},y_{m}^{(1)})-\frac{1}{\gamma+1}}\leq \frac{1}{m}\qquad
    \text{for each $m\in \mathbb{N}$}\,.
\end{equation*}
It follows from Step~\textbf{1} of the proof of Proposition~\ref{Prop:Esti4theta1a}
that $0<y_{m}^{(1)}<{f}_{5,\rm sh}(0)<f_{5,\rm sh}(x_{m}^{(1)})$ for each
$m\in \mathbb{N}$ so that
\begin{equation}
\lim\limits_{m\rightarrow \infty}(x_{m}^{(1)},y_{m}^{(1)})=(0,{f}_{5,\rm sh}(0))\,,
\qquad
 \lim\limits_{m\rightarrow\infty}|\psi_{xx}(x_{m}^{(1)},y_{m}^{(1)})|=\frac{1}{\gamma+1}\,.
\end{equation}

On the other hand, there exists a constant $\varepsilon \in (0,\bar{\varepsilon}]$ such that the boundary condition
on $\Gamma_{\rm shock}\cap \partial \Omega^5_{\bar{\varepsilon}}$ can be written as
$\psi_{x}+b_{1}\psi_{y}+b_{0}\psi=0$ with
$(b_0,b_1)=(b_0,b_1)(\psi_{x},\psi_{y},\psi,x,f_{5,\rm sh}(x))$ for any \(x \in (0, \varepsilon)\).
Let $\omega>0$ from Step~\textbf{1} of the proof of Proposition~\ref{Prop:Esti4theta1a}. Then
\begin{equation*}
    \{(x,f_{5,\rm sh}(x)-\tfrac{\omega}{10}x) \,:\, 0<x< \varepsilon\}
    \subseteq \rcal (\Omega^5_{\bar{\varepsilon}}) \,.
\end{equation*}
Denote $F(x)\defeq \psi_{x}(x,f_{5,\rm sh}(x)-\frac{\omega}{10}x)$.
By the geometric relation given above, we have
\begin{align*}
F(x) &= \psi_{x}(x,f_{5,\rm sh}(x))-\tfrac{\omega}{10}x\int_{0}^{1}\psi_{xy}(x,f_{5,\rm sh}(x)-\tfrac{\omega}{10} tx)\,{\rm d}t\\
&= -(b_{1}\psi_{y}+b_{0}\psi)(x,f_{5,\rm sh}(x))
-\tfrac{\omega}{10}x\int_{0}^{1}\psi_{xy}(x,f_{5,\rm sh}(x)-\tfrac{\omega}{10} tx)\,{\rm d}t
\qquad\mbox{for $x \in (0,\varepsilon)$}\,.
\end{align*}
Then Proposition~\ref{Prop:Esti4theta1a} implies that
$F(0)=0$, $F\in C([0,\varepsilon])\cap C^{1}((0,\varepsilon))$,
$\lim\limits_{x\rightarrow 0+}\frac{F(x)}{x}=0$
so that,
by the mean value theorem,
there exists a sequence $\{x_{m}^{(2)}\}_{m\in\mathbb{N}}\subseteq (0,\varepsilon)$ such that
\begin{equation}\label{eqf}
\lim\limits_{m\rightarrow \infty}x_{m}^{(2)}=0\,,\qquad
F^{\prime}(x_{m}^{(2)})=0\,.
\end{equation}
For each $m\in \mathbb{N}$, define $y_{m}^{(2)}\defeq f_{5,\rm sh}(x_{m}^{(2)})-\frac{\omega}{10}x_{m}^{(2)}$,
so that $\{(x_{m}^{(2)},y_{m}^{(2)})\}_{m\in\mathbb{N}} \subseteq \rcal(\Omega^5_{\bar{\varepsilon}})$.
By the definition of $F(\cdot)$ and~\eqref{eqf}, we have
\begin{align}\label{eq6-11}
    \lim\limits_{m\rightarrow \infty}\psi_{xx}(x_{m}^{(2)},y_{m}^{(2)})
    &= \lim\limits_{m\rightarrow \infty}F^{\prime}(x_{m}^{(2)})-\lim\limits_{m\rightarrow \infty}(f_{5,\rm sh}(x_{m}^{(2)})
    -\tfrac{\omega}{10})\psi_{xy}(x_{m}^{(2)},y_{m}^{(2)})\\
    &= -\lim\limits_{m\rightarrow \infty}(f_{5,\rm sh}(x_{m}^{(2)})-\tfrac{\omega}{10})\psi_{xy}(x_{m}^{(2)},y_{m}^{(2)})\,.\nonumber
\end{align}
Since $\lim\limits_{m\rightarrow \infty}(x_{m}^{(2)},y_{m}^{(2)})=(0,f_{5,\rm sh}(0))$,
we combine~\eqref{eq6-11} with Proposition~\ref{Prop:Esti4theta1a}  to obtain
\begin{equation*}
 \lim\limits_{m\rightarrow \infty}\psi_{xx}(x_{m}^{(2)},y_{m}^{(2)})=0\,.
\end{equation*}

We have demonstrated that there are two sequences,
$\{(x_{m}^{(1)},y_{m}^{(1)})\}_{m\in\mathbb{N}}$ and $\{(x_{m}^{(2)},y_{m}^{(2)})\}_{m\in\mathbb{N}}$,
taking values in $\rcal(\Omega^5_{\bar{\varepsilon}})$, such that the limits of both sequences are $(0,f_{5,\rm sh}(0))$, but
\begin{equation*}
\lim\limits_{m\rightarrow \infty}\psi_{xx}(x_{m}^{(1)},y_{m}^{(1)})
\neq \lim\limits_{m\rightarrow \infty}\psi_{xx}(x_{m}^{(2)},y_{m}^{(2)})\,.
\end{equation*}
This completes the proof of Theorem~\ref{RegularityThm}\eqref{ReguThm04}.
\end{proof}

\subsection{Proof of Theorem 2.3: Convexity of free boundaries and transonic shocks}
\label{subsec:convexity}
We now discuss the geometric properties
of the transonic shock as a free boundary.
In Chen-Feldman-Xiang~\cite{ChenFeldmanXiang2020ARMA}, a  general framework was developed under which the self-similar transonic shock, as a free boundary,
is proved to be uniformly convex for the potential flow equation;
see {\rm Framework~(A)} in Appendix~\ref{Sec:Appendix-B-Convexity}.
In this subsection, we apply this framework to prove the uniform convexity
of the transonic shock in the four-shock interaction problem, so it suffices to prove that
the admissible solutions satisfy the conditions in Theorems~\ref{ThC1}--\ref{ThC2}.

\begin{lemma}\label{Convexitylemma}
The following statements hold{\rm:}
\begin{enumerate}[{\rm (i)}]
\item
\label{Convexitylemma-item-1}
Any admissible solution in the sense of {\rm Definition~\ref{Def:AdmisSolus}}
satisfies the conditions of {\rm Theorems~\ref{ThC1}}--{\rm\ref{ThC2}}.

\item
\label{Convexitylemma-item-2}
Any global weak solution of the four-shock interaction problem in the sense of
{\rm Definition~\ref{Def:WeakSoluForPotentialFlowEq}} ---satisfying also all
properties of {\rm Definition~\ref{Def:AdmisSolus}}
except~\eqref{DirecDerivMonotone-S25-S26} and with transonic shock $\Gamma_{\rm{shock}}$ being
a strictly convex graph---satisfies condition~\eqref{DirecDerivMonotone-S25-S26} of {\rm Definition~\ref{Def:AdmisSolus}}.
\end{enumerate}
\end{lemma}

\begin{proof} We divide the proof into six steps. We start with
the proof of assertion~\eqref{Convexitylemma-item-1}.

\smallskip
\textbf{1.} \textit{Claim}: Region $\Omega$ satisfies the conditions in {\rm Framework (A)}.
Indeed, for $\btheta \in [0,\theta^{\rm s})^2$, the required piecewise regularity holds,
since $\Gamma_{\rm{sym}}$ is a straight segment,
while $\Gamma_{\rm{sonic}}^5$ and $\Gamma_{\rm{sonic}}^6$
are the arcs of the sonic circles,
and $\Gamma_{\rm{shock}}$ has the regularity stated in Theorem~\ref{RegularityThm}\eqref{ReguThm01}.
For any incident angles $\btheta \in \Theta$, the fact that all angles of the corners of $\Omega$
are less than $\pi$ can be verified as follows:

If $\theta_1\in[0,\theta^{\rm s})$, since $\overline{S^{\rm seg}_{25}\cup\Gamma_{\rm{shock}}}$ is a $C^{2,\alpha}$--curve for any $\alpha\in(0,1)$, including at points $P_2$,
it follows that $\Gamma_{\rm shock}$ is tangential to $S^{\rm seg}_{25}$ at $P_2$.
Meanwhile, for any $\btheta\in\Theta$, $\partial B_{c_5}(O_5)$ and $S_{25}$
always have two intersection points including $P_2$.
Therefore, the meeting angles at the two corners $\{P_1,P_2\}$ belong
to $[0,\pi)$ when
$\theta_1\in(0,\theta^{\rm s})$.
If $\theta_1\in[\theta^{\rm s},\theta^{\rm d})$, we see that $P_2=P_1=P_0^1$.
From~\eqref{eq:reflected-shock-angle-lower-bound},
the meeting angle at $P_0^1$ is $\pi-\theta_{25}\in(0,\frac{\pi}{2})$, which satisfies
\begin{equation*}
    \pi-\theta_{25}\geq \delta_{5,6}^{(\theta^{\rm s})} > 0\,.
\end{equation*}
Similarly, the meeting angles at the two corners $\{P_3,P_4\}$ also belong to $(0,\pi)$
for any $\theta_{2}\in[0,\theta^{\rm d})$.

\smallskip
\textbf{2.}  The entropy condition \qeqref{Thm-C1-item1} in Theorem~\ref{ThC1} follows directly
from the properties of
Definition~\ref{Def:AdmisSolus}, where state $(0)$ in Theorem~\ref{ThC1} is state $(2)$ in our problem.
From the regularity of $\varphi$ and $\Gamma_{\rm{shock}}$ in Theorem~\ref{RegularityThm}\eqref{ReguThm01},
we see that conditions \qeqref{Thm-C1-item2} and \qeqref{Thm-C1-item4} of Theorem~\ref{ThC1} hold.
Property~\eqref{item3-Def:AdmisSolus} of Definition~\ref{Def:AdmisSolus} implies that condition
\qeqref{Thm-C1-item3} of Theorem~\ref{ThC1} holds.

\smallskip
\textbf{3.}  We use the notation for the endpoints of $\Gamma_{\rm{shock}}$ from {\rm Framework (A)},
that is, $A\defeq P_2$ and $B\defeq P_3$,
where $P_{i+1} \equiv P_0^i$ when \(\theta_i \in [\theta^{\rm s}, \theta^{\rm d})\) for \(i=1,2\).
Then, by the properties of Definition~\ref{Def:AdmisSolus}, we see that
    $\bm{\tau}_{A} = \bm{e}_{S_{25}}$ and
    $\bm{\tau}_{B} = \bm{e}_{S_{26}}$,
and it is clear that $\bm{\tau}_{A}\neq \pm \bm{\tau}_{B}$ for all \(\btheta \neq \bm{0}\).
Accordingly, let ${\rm Con}$
be given by~\eqref{eq:appendixB-cone}.
Combining property~\eqref{DirecDerivMonotone-S25-S26} with the fact that
$\Gamma_{\rm{shock}}$ is the level set: $\{\varphi-\varphi_2=0\}$,
we obtain that
$\{P+{\rm Con}\}\cap \Omega=\varnothing$ for all $P\in \Gamma_{\rm{shock}}$.
Thus, condition \qeqref{Thm-C1-item5} of Theorem~\ref{ThC1} is satisfied.

\smallskip
\textbf{4.}  For condition (C-\ref{Thm-C1-item6}) of Theorem~\ref{ThC1},
we show that Case~(C-\ref{Thm-C1-item6-c}) holds with \(\phi = \varphi - \varphi_2\), $\bm{e}=(0,1) \in {\rm Con}$,
$\Gamma_{1}\defeq \Gamma_{\rm sonic}^6\cup\Gamma_{\rm{sym}}\cup\Gamma_{\rm sonic}^5$,
and $\Gamma_{2}=\varnothing$.
For \(j = 5,6\), write \(w \defeq \partial_{\bm{e}}(\varphi - \varphi_j)\),
which satisfies a strictly elliptic equation in \(\Omega\) by taking the derivative of~\eqref{Eq4phi}
in direction \(\bm{e}\), and using  Definition~\ref{Def:AdmisSolus}\eqref{item3-Def:AdmisSolus}.
By conditions~\eqref{item3n4-FBP}--\eqref{item5-FBP} of Problem~\ref{FBP}, \( w = 0\)
on \(\Gamma_{\rm sym} \cup \Gamma^{j}_{\rm sonic}\), which is a global maximum by Lemma~\ref{lem:monotonicityForPhi2-Phi}\eqref{item3-Lem3-6}.
Therefore, if \(P \in \Gamma_{\rm sym} \cup \Gamma^j_{\rm sonic}\) is a point of local minimum
for \(w\), then \(w = 0\) on \(B_r(P) \cap \Omega\) for some \(r>0\), and so \(w = 0\) in~\(\Omega\) by the strong maximum principle,
which is a contradiction to Lemma~\ref{lem:monotonicityForPhi2-Phi}\eqref{item1-Lem3-2}.
It follows that \(\phi_{\bm{e}} \equiv w - v_2\) cannot attain a local minimum on \(\Gamma_1\).

\smallskip
\textbf{5.} We now show that conditions \qeqref{Thm-C1-item7}--\qeqref{Thm-C1-item10}
are satisfied with $\hat{\Gamma}_0={\Gamma_{\rm{sonic}}^5}\setminus\{P_2\}$
for $\theta_1\in [0,\theta^{\rm s})$
(resp.~$\hat{\Gamma}_{0}=\varnothing$ for $\theta_1\in [\theta^{\rm s},\theta^{\rm d})$),
$\hat{\Gamma}_1=\Gamma_{\rm{sym}}^{0}$, $\hat{\Gamma}_2=\varnothing$,
and $\hat{\Gamma}_3={\Gamma_{\rm{sonic}}^6}\setminus\{P_3\}$ for $\theta_2\in [0,\theta^{\rm s})$
(resp.~$\hat{\Gamma}_{3}=\varnothing$ for $\theta_2\in [\theta^{\rm s},\theta^{\rm d})$).

Indeed, \qeqref{Thm-C1-item7} clearly holds.
Also, \qeqref{Thm-C1-item8} holds since $D\varphi=D\varphi_{5}$ on
$\Gamma_{\rm{sonic}}^5$ for $\theta_1\in [0,\theta^{\rm s})$ so that
$\phi_{\bm{e}}=\partial_{\bm{e}}(\varphi_5-\varphi_2)$ is a constant.
Similarly, $D\varphi=D\varphi_{6}$ on
$\Gamma_{\rm{sonic}}^6$ for $\theta_2\in [0,\theta^{\rm s})$.
Condition \qeqref{Thm-C1-item9} on $\hat{\Gamma}_1=\Gamma_{\rm{sym}}^{0}$ can be checked as follows:
If $\bm{e} \neq (0,\pm1)$,
then, as in Step~{\bf 6} of the proof of~\cite[Lemma~7.8]{ChenFeldmanXiang2020ARMA},
$\phi_{\bm{e}}$ cannot attain its local minima or maxima on
$\Gamma_{\rm{sym}}^{0}$.
In the other case, when $\bm{e}=(0,\pm 1)$,
the slip boundary condition $\partial_{\eta}\phi=0$ on $\Gamma_{\rm{sym}}^{0}$ implies that \(\phi_{\bm{e}}\) is constant
on $\Gamma_{\rm{sym}}^{0}$, which
verifies \qeqref{Thm-C1-item9}.
Case (C-10a) of \qeqref{Thm-C1-item10} clearly holds here since $\hat{\Gamma}_{2}=\varnothing$.
This concludes the proof of assertion~\eqref{Convexitylemma-item-1}.

\smallskip
\textbf{6.} It remains to prove assertion~\eqref{Convexitylemma-item-2}:
Any solution satisfying all the properties of Definition~\ref{Def:AdmisSolus}
except~\eqref{DirecDerivMonotone-S25-S26},
and with transonic shock $\Gamma_{\rm{shock}}$ being a strictly convex graph
in the sense of~\eqref{eqc3}--\eqref{eqc4}, satisfies condition~\eqref{DirecDerivMonotone-S25-S26} of Definition~\ref{Def:AdmisSolus}.

We apply~\eqref{eqc3} in the present case with
$(A,B) = (P_2,P_3)$ and \(\bm{e} = (0,1)\) such that \(\bm{\xi}(S,T) = (-T,S)\).
Then, using the properties of
Definition~\ref{Def:AdmisSolus}(i),
\begin{equation*}
\bm{e}_{S_{25}} = \frac{(-1,f^{\prime}(T_{A}))}{|(-1,f^{\prime}(T_{A}))|}\,,
\qquad\, \bm{e}_{S_{26}} = {-}\frac{(-1,f^{\prime}(T_{B}))}{|(-1,f^{\prime}(T_{B}))|}\,.
\end{equation*}
Also, from the strict concavity of $f$ in the sense of~\eqref{eqc4},
we obtain that $f^{\prime}(T_A)>f^{\prime}(T)>f^{\prime}(T_B)$
and $f(T)<f(T_1)+f^{\prime}(T_1)(T-T_1)$
for all $T,\, T_1\in(T_A,T_B)$.
From this, we see that $\{P+{\rm Con}\}\cap\Omega=\varnothing$ for any
$P\in\Gamma_{\rm{shock}}$.
Since $\varphi\leq\varphi_2$ in $\Omega$ from Definition~\ref{Def:AdmisSolus}
and $\varphi=\varphi_2$ on $\Gamma_{\rm{shock}}$,
we obtain that $\partial_{\bm{e}}\varphi\geq\partial_{\bm{e}}\varphi_2$
for any $\bm{e}\in {\rm Con}$, which implies~\eqref{DirecDerivMonotone-S25-S26}.
\end{proof}

\begin{appendices}
\section{Proof of Lemma~\ref{lem:MonotonicityOfSteadyAngles} and Related Properties of Solutions}
\label{Sec:Appendix-A}
\renewcommand\thesubsection{A.\arabic{subsection}}
\setcounter{equation}{0}
\setcounter{theorem}{0}
\setcounter{definition}{0}
\renewcommand\theequation{A.\arabic{equation}}
\renewcommand{\thedefinition}{A.\arabic{definition}}
\renewcommand{\thelemma}{A.\arabic{lemma}}
\renewcommand{\thetheorem}{A.\arabic{theorem}}
\renewcommand{\thecorollary}{A.\arabic{corollary}}

\subsection{Proof of Lemma~\ref{lem:MonotonicityOfSteadyAngles}: Monotonicity of critical angles for 2-D steady potential flow}
\label{SubSec:A1}
The proof of Lemma~\ref{lem:MonotonicityOfSteadyAngles} is given in four steps.

\smallskip
\textbf{1.}
The shock polar for 2-D steady potential flow connecting a constant supersonic
upstream state \(U_\infty = (\rho_{\infty}, u_{\infty}, 0)\)
to a constant downstream state \(U = (\rho, u , v)\) is given by
\begin{equation}
\begin{aligned}
    u = u_\infty - \frac{2\big( h(\rho)-h(\rho_\infty)\big)}{\rho + \rho_\infty}\frac{\rho}{u_\infty}\,, \quad
    v^2 = u_\infty^2 - u^2 - 2\big(h(\rho)-h(\rho_\infty)\big)\,, \label{Eq:Appendix-A2}
\end{aligned}
\end{equation}
where $h(\rho)\defeq \frac{1}{\gamma-1}(\rho^{\gamma-1}-1)$ for $\gamma>1$, and $h(\rho)=\ln \rho$ for $\gamma=1$;
see {\it e.g.}~\cite[\S 17.2]{Dafermos2016} for the derivation.
The downstream density \(\rho \in(\rho_\infty,\overline{\rho})\)
is a parameter along the shock polar, where the maximal density \(\overline{\rho} > \rho_\infty\)
is determined uniquely by \(v(\overline{\rho})=0\) in~\eqref{Eq:Appendix-A2}, which
satisfies
\begin{align*}
    M_\infty^2 \Big(1 - \big(\frac{\rho_\infty}{\overline{\rho}} \big)^2 \Big)
    = \frac{2\big( h(\overline{\rho}) - h(\rho_\infty) \big)}{c_\infty^2} \,,
\end{align*}
for \(M_\infty \defeq \frac{u_\infty}{c_\infty} > 1\),
the Mach number of the upstream state that is assumed to be supersonic, and $c_\infty$ the sonic speed.
In the following, we denote by \(w \defeq \tan{\theta_{\rm stdy}}\),
where \(\theta_{\rm stdy} \in (0,\frac{\pi}{2})\) is the angle between the velocities of the upstream
and the downstream flows, and by \(\tau \defeq \frac{\rho}{\rho_\infty} \in (1,\overline{\tau})\)
the ratio between the downstream and upstream densities,
with \(\overline{\tau} \defeq \frac{\overline{\rho}}{\rho_\infty}\).
Using the polytropic pressure law with \(\gamma \geq 1\), we have
\begin{align}
    w &= \frac{v }{u }
    =  \frac{ \big( 2 h(\tau) \big)^{1/2}
    \big(M_\infty^2 (1 - \tau^{-2}) - 2 h(\tau) \big)^{1/2} }{ M_\infty^2 (1 + \tau^{-1}) - 2 h(\tau) }\,.
    \label{eq:appendixDeflectionAngleFormula}
\end{align}

From~\cite[Lemma~7.3.2]{ChenFeldman-RM2018}, there exist both a unique detachment
angle \(\theta_{\rm stdy}^{\rm d} \in (0, \frac\pi2)\) and a unique sonic
angle \(\theta_{\rm stdy}^{\rm s} \in (0, \theta_{\rm stdy}^{\rm d})\),
which are characterized by
\begin{align*}
    w^{\rm d} \defeq \tan{\theta_{\rm stdy}^{\rm d}} = \sup{\{ w(\tau)\,:\,\tau \in (1,\overline{\tau})\}}\,,
    \qquad\,\,
    w^{\rm s} \defeq \tan{\theta_{\rm stdy}^{\rm s}} = \frac{v }{u } \quad
    \text{when $u^2 + v^2 = \rho^{\gamma-1}$}\,.
\end{align*}
By differentiation of~\eqref{eq:appendixDeflectionAngleFormula}~with respect
to $\tau\in (1,\overline{\tau})$, it follows that \(w^{\rm d} = w(\tau_{\rm d})\),
where \(\tau_{\rm d} \in (1,\overline{\tau})\) satisfies
\begin{align} \label{eq:appendixDetachment_l}
     M_\infty^2 \big( 2h(\tau_{\rm d}) + (\tau_{\rm d}^2 - 1) h'(\tau_{\rm d}) \big)
     - 2h(\tau_{\rm d}) \big( 2h(\tau_{\rm d}) + \tau_{\rm d} (\tau_{\rm d} + 1) h'(\tau_{\rm d}) \big) = 0 \,.
\end{align}
Similarly, by routine calculation,
it follows that \(w^{\rm s} = w(\tau_{\rm s})\),
where \(\tau_{\rm s} \in (1,\overline{\tau})\) satisfies
\begin{align} \label{eq:AppendixSonicDensityMachRelation}
    (\gamma + 1) h(\tau_{\rm s}) = M_\infty^2 - 1\,.
\end{align}

For the polytropic gas,
it is clear that the detachment angle and sonic angle
$\theta_{\rm stdy}^{\rm d}$ and $\theta_{\rm stdy}^{\rm s}$ depend only on $\gamma$ and $M_\infty$.
We also observe the following useful identities:
\begin{align} \label{eq:appendixPolytropicLawIdentities}
    \tau \, h'(\tau) = (\gamma-1) h(\tau) + 1\,,
    \quad
    \frac{h''(\tau)}{h'(\tau)} = \frac{\gamma-2}{\tau}
    \qquad\; \text{for all $\tau \in (0,\infty)$}\,.
\end{align}

\textbf{2.}
For the detachment angle, we combine~\eqref{eq:appendixDeflectionAngleFormula}
with~\eqref{eq:appendixDetachment_l} to obtain a parametric description of
the detachment angle and the upstream Mach number in terms of \(\tau_\text{d}\):
\begin{equation}\label{Eq:AppendxParamDetchAng}
    w^{\rm d} = \frac{1}{2}
    \frac{ \big( \tau_{\rm d} (\tau_{\rm d}^2-1) h' - 2h \big)^{1/2}
    \big( (\tau_{\rm d}^2-1)h' + 2h \big)^{1/2} }{ \tau_{\rm d}(\tau_{\rm d}+1) h' + h }\,, \qquad
    M_\infty^2 =
    \frac{2h \big( 2h + \tau_{\rm d}(\tau_{\rm d} + 1) h' \big)}{2h + (\tau_{\rm d}^2-1) h'}\,,
\end{equation}
where \(h, h',\) and \(h''\) are evaluated at \(\tau=\tau_{\rm d}\).
A direct calculation shows that the signs of \((w^{\rm d})_{\tau_\text{d}}\)
and \((M_\infty^2)_{\tau_\text{d}}\) are equal to the sign of the quantity:
\begin{align*}
    f_{\rm d}(\tau_\text{d}) \defeq
    2 h^2 \big( 3 + (\tau_{\rm d}+1) \frac{h''}{h'} \big) + (3\tau_{\rm d} -5)(\tau_{\rm d}+1) h h'
    + \tau_{\rm d} (\tau_{\rm d}+1) (\tau_{\rm d}^2-1) (h')^2\,.
\end{align*}
Applying~\eqref{eq:appendixPolytropicLawIdentities} evaluated at $\tau = \tau_{\rm d}$,
along with $\gamma \geq 1$, $\tau_{\rm d} > 1$, and $h > 0$, we obtain
\begin{align*}
    f_{\rm d} (\tau_\text{d}) &=
    \big( 3(\tau_{\rm d}-\tau_{\rm d}^{-1})+ \big( 3(\gamma-1) (\tau_{\rm d}-\tau_{\rm d}^{-1}) + 2(2 - \tau_{\rm d}^{-1}) \big) h \big) h
    + (1 + \tau_{\rm d}^{-1})  \big( \tau_{\rm d}^2 (\tau_{\rm d}^2-1) (h')^2 - 2h \big)
    \\
    &> (1 + \tau_{\rm d}^{-1}) \tilde{f}_{\rm d}(\tau_\text{d})\,,
\end{align*}
where $\tilde{f}_{\rm d}(\tau_\text{d}) \defeq \tau_{\rm d}^2 (\tau_{\rm d}^2-1) (h')^2 - 2h $.
We note that $\tilde{f}_{\rm d}(1) = 0$  and
\begin{equation*}
\tilde{f}_{\rm d}'(\tau) = 2 h'(\tau) \big( (\gamma-1)\tau(\tau^2-1) h'(\tau) + \tau^{\gamma+1}-1 \big) > 0
\qquad\,\, \mbox{for any $\tau>1$}\,.
\end{equation*}
It follows that $\tilde{f}_{\rm d}(\tau_\text{d})>0$ and \(f_{\rm d}(\tau_\text{d})>0\)
so that \((w^{\rm d})_{\tau_{\rm d}}>0\) and \((M_\infty^2)_{\tau_{\rm d}}>0\)
for all \(\tau_\text{d}>1\).
Using the chain rule, we conclude that \(w^{\rm d}\) and
\(\theta^{\rm d}_{\rm stdy}\)
are strictly increasing with respect to \(M_\infty > 1\).

\smallskip
\textbf{3.}
For the sonic angle, we combine~\eqref{eq:appendixDeflectionAngleFormula}
with~\eqref{eq:AppendixSonicDensityMachRelation} to give a parametric description
of the sonic angle and upstream Mach number in terms of~\(\tau_{\rm s}\):
\begin{equation}\label{Eq:AppendxParamSonicAng}
    w^{\rm s} =
    \frac{ \big( 2 h \big)^{1/2} \big(
    (\tau_{\rm s}^2 - 1) \tau_{\rm s}^{\gamma-1} - 2 h
    \big)^{1/2} }{ (\tau_{\rm s} + 1) \tau_{\rm s}^{\gamma-1} + 2h}\,,\qquad
    M_\infty^2 = 1 + (\gamma+1)h\,,
\end{equation}
where $h$ and $h'$ are evaluated at \(\tau = \tau_{\rm s}\).
It is clear that \((M_\infty^2)_{\tau_{\rm s}} > 0\) since \(h' > 0\),
whilst a direct calculation shows that the sign of \((w^{\rm s})_{\tau_{\rm s}}\) is equal to the sign of the quantity
\begin{align*}
    f_{\rm s}(\tau_{\rm s}) &\defeq
    2 h \big(1 + (\gamma-1) h\big) \big(1 + (\gamma+1)h\big)
    + (\tau_{\rm s}+1) \big( (\tau_{\rm s} -1) \tau_{\rm s}^{\gamma-1} - 2h\big) h' \,.
\end{align*}
Applying~\eqref{eq:appendixPolytropicLawIdentities} evaluated at $\tau = \tau_{\rm s}$,
along with $\gamma\geq 1$, $\tau_{\rm s}>1$, and $h'(\tau_{\rm s}) = \tau_{\rm s}^{\gamma-2} > 0$,
we obtain
\begin{align*}
    f_{\rm s}(\tau_{\rm s})
    =  \big( 2(\gamma+1) \tau_{\rm s} h^2
            + \tau_{\rm s} (\tau_{\rm s}^2-1) h' - 2h \big) h'
            >\tilde{f}_{\rm s}(\tau_{\rm s}) h'\,,
\end{align*}
where $ \tilde{f}_{\rm s}(\tau_{\rm s}) \defeq \tau_{\rm s} (\tau_{\rm s}^2-1) h' - 2h $.
Notice that \(\tilde{f}_{\rm s}(1) = 0\) and
\begin{align*}
    \tilde{f}_{\rm s}'(\tau) = (\gamma+1) (\tau^2-1) h'(\tau) > 0
    \qquad\mbox{for any $\tau>1$}.
\end{align*}
Then it follows that $\tilde{f}_{\rm s}(\tau_{\rm s})>0$ and $f_{\rm s}(\tau_{\rm s})>0$
so that \((w^{\rm s})_{\tau_{\rm s}} > 0\).
We conclude via the chain rule that \(w^{\rm s}\) and hence \(\theta_{\rm stdy}^{\rm s}\)
are strictly increasing with respect to \(M_\infty > 1\).

\smallskip
\textbf{4.}
The limiting values stated in the lemma can be checked directly
from~\eqref{Eq:AppendxParamDetchAng}--\eqref{Eq:AppendxParamSonicAng},
now that the monotonicity of \((w^{\rm d},M_\infty^2)\) with respect to \(\tau_{\rm d}\)
in~\eqref{Eq:AppendxParamDetchAng} and the monotonicity of \((w^{\rm s}, M_\infty^2)\)
with respect to \(\tau_{\rm s}\) in~\eqref{Eq:AppendxParamSonicAng} have been verified.
\qed

\subsection{Monotonicity properties with respect to the incident angles}
In Lemma~\ref{lem:MonotonicityOfMach2}, we have shown that the pseudo-Mach number of state \((2)\)
at the reflection point \(P_0^1\) is a strictly decreasing function of \(\theta_1 \in (0,\theta^{\rm cr})\).
Subsequently, in Proposition~\ref{Prop:localtheorystate5}, we have used
Lemmas~\ref{lem:MonotonicityOfMach2} and~\ref{lem:MonotonicityOfSteadyAngles}
to prove the existence of the unique detachment angle \(\theta^{\rm d}\)
and sonic angle \(\theta^{\rm s}\).
In Remark~\ref{remark:local-theory}\eqref{remark:local-theory-item-i}, we have stated the relation
between \(\theta^{\rm d}, \,\theta^{\rm s}\), and \(\theta^{\rm cr}\), which depends on the choice of \((\gamma,v_2)\).
The following lemma provides the proof of this statement.

\begin{lemma}\label{Lem:Defv2d-v2s}
For any \(\gamma > 1,\) there exist constants \(v_2^{\rm d}\) and \(v_2^{\rm s},\)
with \(v_{\min} < v_2^{\rm s} < v_2^{\rm d} < 0,\) uniquely determined by \(\gamma\)
such that
\begin{align}
\label{eq:AppendixCriticalV2-statement-1}
\sgn
\big( \hat{\theta}_{25}(\theta^{\rm cr}; v_2) - \theta_{\rm stdy}^{\rm d}(1,|D\varphi_2(P_{0, {\rm cr}}^1)|)
\big)
&= \sgn (v_2 - v_2^{\rm d})\,,\\
\label{eq:AppendixCriticalV2-statement-2}
\sgn
\big( \hat{\theta}_{25}(\theta^{\rm cr}; v_2) - \theta_{\rm stdy}^{\rm s}(1,|D\varphi_2(P_{0, {\rm cr}}^1)|)
\big)
&= \sgn (v_2 - v_2^{\rm s})\,,
\end{align}
where
$P_{0, {\rm cr}}^1 \defeq P_0^1|_{\theta_1 = \theta^{\rm cr}}${\rm ,} and
we define \(\theta_{\rm stdy}^{\rm d}(1,M_\infty) \equiv \theta_{\rm stdy}^{\rm s}(1,M_\infty) \defeq 0\)
for any \(0 \leq M_\infty \leq 1\).
Furthermore, when \(\gamma = 1,\)
we define \(v_{\min}= v_2^{\rm s} = v_2^{\rm d} = -\infty\).
\end{lemma}

\noindent \textit{Proof.}  We divide the proof into two steps.

\smallskip
\textbf{1.} For \(\gamma >1\) and $ v_2 \in (\vmin, 0)$, a direct calculation by using \eqref{Eq:DefThetaVac},
\eqref{eq:appendixMachNumberState12}, and~Lemma~\ref{lem:MonotonicityOfMach2} gives\\
\begin{equation}\label{eq:AppendixCriticalV2-1}
\begin{aligned}
&M_{2,\min} \defeq \inf_{\theta_1 \in (0,\theta^{\rm cr})} |D\varphi_2(P_0^1)| = |D\varphi_2(P_{0,{\rm cr}}^1)| = \frac{v_2 \vmin}{\sqrt{\vmin^2 - v_2^2}}\,,\\
&\hat{\theta}_{25}(\theta^{\rm cr};v_2) = \arcsin{(\frac{-v_2}{M_{2,\min}})}
= \arccos{(|{\frac{v_2}{\vmin}}|)} > 0\,.
\end{aligned}
\end{equation}
Then \(\hat{\theta}_{25}(\theta^{\rm cr};v_2)\) is strictly increasing with respect to \(v_2 \in (\vmin,0)\).
We also calculate
\begin{equation}
\label{eq:AppendixCriticalV2-2}
   \lim_{v_2 \to \vmin^+} M_{2,\min} = \infty\,,
   \quad
   \lim_{v_2 \to 0^-} M_{2,\min} = 0\,,
   \quad
   \frac{{\rm d} M_{2,\min}^2}{{\rm d}v_2} = \frac{2 \vmin^4 v_2 }{(\vmin^2 - v_2^2)^2} < 0 \,,
\end{equation}
so that \(M_{2,\min}\) is strictly decreasing with respect to \(v_2 \in (\vmin,0)\).
Also observe that
$M_{2,\min} > 1$
if and only if
$v_2 \in (\vmin, v_{\rm mid})$,
for $v_{\rm mid} \defeq \frac{\vmin}{\sqrt{\vmin^2 + 1}}$.
Then, by Lemma~\ref{lem:MonotonicityOfSteadyAngles},
\(\theta_{\rm stdy}^{\rm d}(1,M_{2,\min})\) is strictly decreasing with respect to
\( v_2 \in ( \vmin, v_{\rm mid} ) \).
In particular, by~\eqref{eq:AppendixCriticalV2-1}--\eqref{eq:AppendixCriticalV2-2}
and~\cite[Lemma~7.3.3]{ChenFeldman-RM2018}, we have
\begin{align*}
\begin{alignedat}{3}
    (\hat{\theta}_{25}(\theta^{\rm cr};v_2) ,  \theta_{\rm stdy}^{\rm d}(1,M_{2,\min}) ) &\to ( \arctan{(\abs{\vmin})} , 0 )
    \qquad &&\text{as} \;\;
    v_2 &&\to v_{\rm mid}^-\,,\\
    (\hat{\theta}_{25}(\theta^{\rm cr};v_2),  \theta_{\rm stdy}^{\rm d}(1,M_{2,\min}) ) &\to ( 0, \tfrac\pi2 )
    \quad &&\text{as} \;\;
    v_2 &&\to \vmin^+\,.
\end{alignedat}
\end{align*}
Then we conclude from the monotonicity properties of
$(\hat{\theta}_{25}(\theta^{\rm cr};v_2)\,, \theta_{\rm stdy}^{\rm d}(1,M_{2,\min}))$
with respect to \(v_2 \in ( \vmin, v_{\rm mid} )\) that there exists a unique value
\(v_2^{\rm d} \in ( \vmin, v_{\rm mid} )\), depending only on \(\gamma\),
such that~\eqref{eq:AppendixCriticalV2-statement-1} holds.

The proof of~\eqref{eq:AppendixCriticalV2-statement-2} follows via a similar argument
by using the fact that \(\theta_{\rm stdy}^{\rm s} \in (0, \theta_{\rm stdy}^{\rm d})\) and
\begin{flalign*}
&&
\begin{alignedat}{3}
    \theta_{\rm stdy}^{\rm s}(1,M_{2,\min}) &\to 0 \qquad &&\text{as} \;\;
    v_2 &&\to v_{\rm mid}^-\,,\\
    \theta_{\rm stdy}^{\rm s}(1,M_{2,\min}) &\to \arctan \big( \tfrac{2}{\gamma-1}\big)^{1/2} \qquad &&\text{as} \;\;
    v_2 &&\to \vmin^{+}\,.
\end{alignedat}
&&
\end{flalign*}

\smallskip
\textbf{2.} For \(\gamma = 1\), from~\eqref{Eq:DefThetaVac}, we see that
\(\vmin = -\infty\) and \(\theta^{\rm cr} = \tfrac\pi2\).
Taking the limit: \(\vmin \to -\infty\) in Step~\textbf{1} above, we obtain that
$M_{2,\min} = \abs{v_2}$ and
$\hat{\theta}_{25}(\tfrac\pi2;v_2) = \tfrac\pi2$.
From Lemma~\ref{lem:MonotonicityOfSteadyAngles}, we know that the sonic and detachment angles
for 2-D steady potential flow
satisfy \(0 < \theta^{\rm s}(1,M_\infty) < \theta^{\rm d}(1,M_\infty) < \tfrac\pi2\) for any \(M_\infty > 1\),
which implies that~\eqref{eq:AppendixCriticalV2-statement-1}--\eqref{eq:AppendixCriticalV2-statement-2}
hold with \(v_2^{\rm d} = v_2^{\rm s} = -\infty\). \qed

Lemma~\ref{lem:MonotonicityOfMach2} is also useful to determine the monotonicity of other quantities
related to the four-shock interaction problem.
In the following lemma, we show that the reflected shock angles \(\theta_{25}\) and \(\theta_{26}\)
are monotonic with respect to the incident shock angles \(\theta_1\) and \(\theta_2\) respectively
and, subsequently, that certain intersection points of the reflected shocks are also monotonic functions
of the incident angles.

\begin{lemma} \label{lem:monotonicityOfTheta5}
Fix \(\gamma \geq 1\) and \(v_2 \in (\vmin,0)\).
\begin{enumerate}[{\rm (i)}]
\item \label{item1-lem:monotonicityOfTheta5}
For any incident angle \(\theta_1 \in [0,\theta^{\rm d}),\)
let the reflected shock angle \(\theta_{25} \in (\tfrac{\pi}{2}, \pi]\) be given
by~\eqref{eq:def-reflected-shock-angle-25-26}.
Then \(\theta_{25}\) is a continuous, strictly decreasing function of \(\theta_1,\)
and the map{\rm:} \(\theta_1 \mapsto \theta_{25}\) is \(C^\infty\)--smooth on \(\theta_1 \in [0,\theta^{\rm d})\).

\item \label{item2-lem:monotonicityOfTheta5}
For the right-most intersection \(P_2 = (\xi^{P_2},\eta^{P_2})\)  of \(\partial B_{c_5}(O_5)\) and \(S_{25}\)
as given
in {\rm Definition~\ref{Def:GammaSonicAndPts},}
and for the \(\eta\)--intercept of \(S_{25}\) which is denoted here by \(a_{25}\){\rm:}
\(\eta^{P_2}\) is strictly decreasing with respect to \(\theta_1 \in (0, \theta^{\rm s})\){\rm,}
and \(a_{25}\) is strictly increasing with $\theta_1 \in (0,\theta^{\rm d})$.
\end{enumerate}
\end{lemma}

\begin{proof}
For (i), the continuity and smoothness of the map: \(\theta_1 \mapsto \theta_{25}\)
follows due to Proposition~\ref{Prop:localtheorystate5} and~\eqref{eq:def-reflected-shock-angle-25-26},
whilst the strict monotonicity follows directly from Lemma~\ref{lem:MonotonicityOfMach2} and~\cite[Lemma~2.17]{BCF-2019}.
Indeed, suppose that \(\theta_1,\tilde{\theta}_1 \in [0, \theta^{\rm d})\) are such
that \(\theta_{25}(\theta_1) = \theta_{25}(\tilde{\theta}_1)\).
Then, by~\cite[Lemma~2.17]{BCF-2019}, the uniform state \((5)\) is uniquely determined
by \((v_2, \gamma, \theta_{25})\) so that the location of the reflected shock \(S_{25}\) must
coincide for both \(\theta_1\) and \(\tilde{\theta}_1\).
In particular, for the pseudo-Mach number of state \((2)\) at the reflection point \(P_0^1\),
 as defined by~\eqref{eq:appendixMachNumberState12}, we see that \(M_2^{(\theta_1)} = M_2^{(\tilde{\theta}_1)}\)
so that \(\theta_1 = \tilde{\theta}_1\) by Lemma~\ref{lem:MonotonicityOfMach2}.

We have shown that the map: \(\theta_1 \mapsto \theta_{25}\) is continuous and injective,
and therefore strictly monotonic.
In particular, by property (\ref{Prop:localtheorystate5:continuity-properties}) of
Proposition~\ref{Prop:localtheorystate5} and~\eqref{eq:def-reflected-shock-angle-25-26},
\(\theta_{25} \to \pi^-\) as \(\theta_1 \to 0^+\);
whilst, by property~(\ref{Prop:localtheorystate5:3n4}) of
Proposition~\ref{Prop:localtheorystate5} and~\eqref{eq:def-reflected-shock-angle-25-26},
\(\theta_{25} < \pi\) for any \(\theta_1 \in (0,\theta^{\rm d}]\).
It follows that the map: \(\theta_1 \mapsto \theta_{25}\) is monotonic decreasing.

\smallskip
For (ii), the first result follows directly from part~(i) combined with~\cite[Lemma 2.22]{BCF-2019},
whilst the second result follows directly from part~(i) combined with~\cite[Eqs.~(2.4.14) and (2.4.42)]{BCF-2019}.
\end{proof}

Finally, we give the proof that the intersection point \(P_I\) of the reflected shocks \(S_{25}\) and \(S_{26}\)
can be extended to a \(C(\cl{\Omega})\)--function.

\begin{lemma} \label{lem:intersection-point-P_I}
For \(\btheta \in \cl{\Theta}\setminus\{\bm{0}\},\) denote by \(P_I = (\xi^{P_I},\eta^{P_I})\) the unique
intersection point of the reflected shocks \(S_{25}\) and \(S_{26},\) as given by~\eqref{eq:intersection-of-S25-S26}.
Then
\begin{align} \label{appendix-eq:limit-values-intersection-point}
    \lim_{ \btheta \to \bm{0},\,\btheta \in \cl{\Theta} \setminus \bm{0}} (\xi^{P_I}, \eta^{P_I}) = (0, \eta_0)\,,
\end{align}
where \(\eta_{0}\) is given by~\eqref{Sol-NormalReflec0-RH}.
\end{lemma}

\begin{proof}
By Proposition~\ref{Prop:localtheorystate5} and
Lemma~\ref{lem:monotonicityOfTheta5}\eqref{item1-lem:monotonicityOfTheta5},
limit~\eqref{appendix-eq:limit-values-intersection-point} is equivalent to the limit:
\begin{align*}
 \lim_{\substack{(\theta_{25},\theta_{26}) \to (\pi,0) \\
 ( \theta_{25}, \theta_{26})\in \cl{\Theta_{\rm refl}}\setminus \{(\pi,0)\}}}(\xi^{P_I}, \eta^{P_I}) = (0, \eta_0)\,,
\end{align*}
where $ \Theta_{\rm refl} \defeq \{ (\theta_{25}, \theta_{26}) : \btheta \in \Theta \}$.
The proof of the above limit consists of three steps.

\smallskip
\textbf{1.} We first consider the unilateral normal reflection case:
$\btheta \in \big( \{0\} \times (0,\theta^{\rm d}) \big) \cup \big( (0,\theta^{\rm d}) \times \{ 0\} \big)$.\\
Without loss of generality, we restrict our attention to the case: \(\btheta \in (0,\theta^{\rm d}) \times \{0\}\),
and denote \((v_\infty,\beta) \defeq (-v_2,\pi - \theta_{25})\), where \(\theta_{25}\) depends only on \((v_2,\gamma,\theta_1)\).
We consider the Prandtl-Meyer reflection configuration associated with
parameters \((v_\infty,\beta)\), as defined in~\cite{BCF-2019}, and reflect this configuration in the vertical axis.
By~\cite[Lemma~2.17]{BCF-2019}, the location of the reflected shock \(S_{25}\) is uniquely
determined by \((v_2,\gamma,\theta_{25})\), and thus coincides with the Prandtl-Meyer oblique shock~\(S_{\mathcal{O}}\).
Similarly, the (normal) reflected shock~\(S_{26} = S_0\) coincides with the Prandtl-Meyer normal shock~\(S_{\mathcal{N}}\),
so that the intersection point \(P_I\) of \(S_{25}\) and \(S_{26}\),
as given by~\eqref{eq:intersection-of-S25-S26}, coincides with the intersection point of \(S_{\mathcal{O}}\) and \(S_{\mathcal{N}}\), as given by~\cite[Eq.~(4.1.33)]{BCF-2019}.
That is,
\begin{align*}
    (\xi^{P_I}\,, \eta^{P_I})
    = \big( \frac{a_{25} - \eta_0}{\tan{\theta_{25}}}\,, \eta_0 \big)\,,
\end{align*}
where \(a_{25} \defeq -\xi^{P_0^1} \tan{\theta_{25}}\) and  \(\eta_0\) are the \(\eta\)--intercepts of \(S_{25}\) and the normal reflection \(S_{0}\), respectively, as given in~\S\ref{Sec:FormulateProblemsDefAdmisSolu}.
By Proposition~\ref{Prop:localtheorystate5}\eqref{Prop:localtheorystate5:continuity-properties} and Lemma~\ref{lem:monotonicityOfTheta5}\eqref{item2-lem:monotonicityOfTheta5},
\(a_{25} \to \eta_0\) as \(\theta_{25} \to \pi^-\). Then
we may apply L'H{\^o}pital's rule to give
\begin{align*}
    \lim_{\theta_{25} \to \pi^-} \xi^{P_I} = -\lim_{\theta_{25} \to \pi^-} \frac{{\rm d} a_{25}}{{\rm d} \theta_{25}}\,.
\end{align*}

From~\cite[Eq.~(2.4.14)]{BCF-2019},
\( a_{25} = v_2 - q_2 \sec{\theta_{25}}\),
for
\(q_2 \defeq {\rm dist}(S_{25}, O_2)\).
By taking the derivative with respect to
\(\theta_{25}\),
we have
\begin{align*}
    \frac{{\rm d} a_{25}}{{\rm d} \theta_{25}} = -\sec{\theta_{25}} \frac{{\rm d} q_2}{{\rm d} \theta_{25}}
    - q_2 \sec{\theta_{25}}\tan{\theta_{25}}\,.
\end{align*}
Using~\cite[Lemma~6.1.2]{ChenFeldman-RM2018} and~\eqref{Sol-NormalReflec0-RH},
   $\displaystyle{1 < \lim_{\theta_{25} \to \pi^-} q_2 = \eta_0 - v_2 < \infty}$
so that
\begin{align*}
    \lim_{\theta_{25} \to \pi^-} \frac{{\rm d} a_{25}}{{\rm d} \theta_{25}} = \lim_{\theta_{25} \to \pi^-} \frac{{\rm d} q_2}{{\rm d} \theta_{25}}\,.
\end{align*}

By~\cite[Eq.~(2.4.12)]{BCF-2019}, we have a relation between the pseudo-Mach numbers
\((M_2, M_5) \defeq (\tfrac{q_2}{c_2},  \tfrac{q_5}{ c_5})\)
for \(q_5 \defeq {\rm dist}(S_{25},O_5)\) and the given parameters
\((\gamma, v_2, \theta_{25})\),
expressed as
\begin{align*}
    M_2^{\frac{\gamma-1}{\gamma+1}} ( M_2^{\frac{2}{\gamma+1}} - M_{5}^{\frac{2}{\gamma+1}} ) = v_2 \sec{\theta_{25}}\,.
\end{align*}
Taking the derivative:
\begin{equation*}
\begin{aligned}
    0 &= \lim_{\theta_{25} \to \pi^-} \frac{{\rm d}}{{\rm d} \theta_{25}}
    \Big( M_2^{\frac{\gamma-1}{\gamma+1}} ( M_2^{\frac{2}{\gamma+1}} - M_{5}^{\frac{2}{\gamma+1}} ) \Big)\label{eq:dM_infty-dbeta-calc} \\
    &= \lim_{\theta_{25} \to \pi^-}  \frac{1}{\gamma+1} \Big( (\gamma-1) M_2^{-1} v_2 \sec{\theta_{25}}
    + 2 M_2^{\frac{\gamma-1}{\gamma+1}}\big( M_2^{\frac{1-\gamma}{1+\gamma}} - M_{5}^{\frac{1-\gamma}{1+ \gamma}} \frac{{\rm d} M_{5}}{{\rm d} M_2} \big) \Big)
    \frac{{\rm d} M_2}{{\rm d} \theta_{25}}\,.
    \end{aligned}
\end{equation*}
By~\cite[Eqs.~(2.4.9)--(2.4.10)]{BCF-2019}, we see that
 $\frac{{\rm d} M_{5}}{{\rm d} M_2} < 0$ for all $M_2 \in (0,\infty) \setminus \{1\}$.
Therefore, the bracketed term above
is uniformly positive for
\(\theta_{25}\in(\tfrac{\pi}{2},\pi)\)
so that
\begin{align*}
    \lim_{\theta_{25}\to \pi^-} \frac{{\rm d} q_2}{{\rm d} \theta_{25}}
    = \lim_{\theta_{25}\to \pi^-} \frac{{\rm d} M_2}{{\rm d} \theta_{25}} = 0\,.
\end{align*}

\textbf{2.}
We now consider the general case \(\btheta \in \cl{\Theta}\setminus\{\bm{0}\}\).
By direct calculation, we have
\begin{align*}
    \xi^{P_I} = \frac{a_{25} - a_{26}}{\tan{\theta_{26}} - \tan{\theta_{25}}}\,,
\end{align*}
where \(a_{2j} > 0\) represents the \(\eta\)--intercept of \(S_{2j}\) for \(j = 5, 6\).
Note that \(a_{25}\) depends only on \((v_2,\gamma,\theta_{25})\),
whilst \(a_{26}\) depends only on \((v_2,\gamma,\theta_{26})\).
In particular, from Step~\textbf{1},
\begin{equation} \label{appendix-eq:limits-of-intersection-points}
0 = \lim_{\theta_{25} \to \pi^-} \frac{a_{25} - \eta_0}{\tan{\theta_{25}}} =
\lim_{\theta_{26} \to 0^+} \frac{a_{26} - \eta_0}{\tan{\theta_{26}}}\,.
\end{equation}

Let \(\left\{\btheta^{(k)}\right\}_{k\in\mathbb{N}} \subseteq \cl{\Theta}\setminus\{\bm{0}\}\) be any sequence
of parameters with \(\btheta^{(k)} \to \bm{0}\) as \(k\to\infty\).
By moving to a subsequence, we may assume any of the following cases:
(i) \(\theta_1^{(k)} = 0\) for all $k\in\mathbb{N}$; (ii) \(\theta_2^{(k)} = 0\) for all $k\in\mathbb{N}$;
(iii) both \(\theta_1^{(k)}>0\) and  \(\theta_2^{(k)} > 0\) for all $k\in\mathbb{N}$.
Denote by \(\xi^{P_I}_{(k)}\) the $\xi$--coordinates of the intersection points $P_I$ associated
with parameters \(\btheta^{(k)}\).
In cases~(i) and~(ii), it is clear from the previous step that \(\xi^{P_I}_{(k)} \to 0\) as \(k \to \infty\),
so we focus on case (iii).
Using the triangle inequality, the fact that $\tan\theta_{25}$ and $\tan\theta_{26}$ have opposite signs,
and~\eqref{appendix-eq:limits-of-intersection-points}, we obtain
\begin{align*}
    \abs{\xi^{P_I}_{(k)}}
    =
    \tabs{ \frac{a_{25}^{(k)} - a_{26}^{(k)}}{\tan{\theta_{26}^{(k)} - \tan{\theta_{25}^{(k)}}}}}
    \leq
    \tabs{\frac{a_{25}^{(k)} - \eta_0}{\tan{\theta_{25}^{(k)}}}}
    +
    \tabs{\frac{a_{26}^{(k)}-\eta_0}{\tan{\theta_{26}^{(k)}}}}
    \to 0 \qquad\,\, \text{as $k \to \infty$}\,.
\end{align*}

\smallskip
\textbf{3.}
Finally, we show that \(\eta^{P_I} \to \eta_0\) as \(\btheta \to \bm{0}\)
with \(\btheta \in \cl{\Theta} \setminus \{\bm{0}\}\).
Indeed, it is clear from~\eqref{appendix-eq:limits-of-intersection-points}
that \(a_{25} \to \eta_0\) as \(\theta_1 \to 0^+\), and \(a_{26} \to \eta_0\) as \(\theta_2 \to 0^+\).
Furthermore, it follows from Step~\textbf{1} of the proof of Lemma~\ref{UniformBound-Lem} that
$\min\{a_{25},a_{26}\} \leq \eta^{P_I} \leq \max\{a_{25},a_{26}\}$.
\end{proof}

\section{Some Known Results Needed for the Proofs}
\label{Sec:Appendix-B}
\renewcommand\thesubsection{B.\arabic{subsection}}
\setcounter{subsection}{0}
\setcounter{equation}{0}
\setcounter{theorem}{0}
\renewcommand\theequation{B.\arabic{equation}}
\renewcommand{\theproposition}{B.\arabic{proposition}}
\renewcommand{\thedefinition}{B.\arabic{definition}}
\renewcommand{\thelemma}{B.\arabic{lemma}}
\renewcommand{\thetheorem}{B.\arabic{theorem}}

In this appendix, we present some known results that are used in Sections 3--5
above.

\subsection{Well-posedness of the iteration boundary value problem}
\label{Sec:Appendix-B1}
For a constant $h>0$ and a function $f_{\rm bd}:[0,h]\rightarrow\mathbb{R}_+$,
fix a bounded domain $\Omega\subseteq\mathbb{R}^2$ as
\begin{equation}\label{ChenFeldman-vonN-Eq:4-5-1}
\Omega \defeq \big\{\bm{x}\in\mathbb{R}^2 \,:\, 0<x_1<h, \, 0<x_2<f_{\rm bd}(x_1)\big\}\,,
\end{equation}
where $f_{\rm bd}$ satisfies that, for constants  $t_0\geq0$, $t_h, \, t_1, \,  t_2, \, t_3,\, M\in(0,\infty)$,
and $\alpha\in(0,1)$,
\begin{equation}\label{ChenFeldman-vonN-Eq:4-5-2n15} \begin{aligned}
& f_{\rm bd}\in C^1([0,h])\,, \quad f_{\rm bd}(0)=t_0\,, \quad f_{\rm bd}(h)=t_h\,,\\
& f_{\rm bd}(x_1)\geq\min\{t_0+t_1x_1,\, t_2,\, t_h-t_3(x_1-h)\} \qquad\,\, \text{for all $x_1\in(0,h)$}\,,\\
& \|f_{\rm bd}\|_{2,\alpha,(0,h)}^{(-1-\alpha),\{0,h\}}\leq M\,.
\end{aligned} \end{equation}
The boundary, $\partial\Omega\defeq \cup_{k=0}^{3}(\Gamma_k\cup\{P_{k+1}\})$, with vertices and segments:
\begin{equation}\label{ChenFeldman-vonN-Eq:4-5-5} \begin{aligned}
&  P_1=(h,0)\,, \quad P_2=(h,f_{\rm bd}(h))\,, \quad P_3=(0,f_{\rm bd}(0))\,, \quad P_4=(0,0)\,,\\
& \overline{\Gamma_0}=\partial\Omega\cap\{x_1=0\}\,, \quad \overline{\Gamma_1}=\partial\Omega\cap\{x_2=f_{\rm bd}(x_1)\}\,,\\
& \overline{\Gamma_2}=\partial\Omega\cap\{x_1=h\} \,, \quad
\overline{\Gamma_3}=\partial\Omega\cap\{x_2=0\} \,,
\end{aligned} \end{equation}
and $\Gamma_{k}$, \(k=0,1,2,3,\) are the relative interiors of the segments defined above.

Let $g_{\rm so}\in C^{\infty}(\mathbb{R}^2\setminus\{\frac{h}{3}<x_1<\frac{2h}{3}\})$ be a piecewise smooth function
defined in $\mathbb{R}^2$ such that
\begin{equation}\label{BCF-Eq:C-5-3plus} \begin{aligned}
& \|g_{\rm so}\|_{C^3(\overline{\Omega\setminus\{\frac{h}{3}<x_1<\frac{2h}{3}\}})} \leq C_{\rm so}\,,\\
& g_{\rm so}(\cdot,x_2) \;\;\;\text{is linear on $x_1$ in}\; \{x_1\leq\frac{h}{4}\}\cup\{x_1\geq\frac{3h}{4}\}\,,   \\
& \partial_{x_2}g_{\rm so}=0 \quad \text{on $\Gamma_3$}\,.
\end{aligned} \end{equation}

We consider the following nonlinear problem:
\begin{equation}\label{vonN-Sec4-5-2:NonlinearProblem}
\left\{
\begin{aligned}
\, & \! \sum_{i,j=1}^2\tilde{A}_{ij}(Du,\bm{x})D_{ij}u+\sum_{i=1}^2\tilde{A}_i(Du,\bm{x})D_iu=0 \quad && \text{in}\; \Omega \,,\\
& \tilde{B}(Du,u,\bm{x})=0 \quad &&\text{on} \; \Gamma_1 \, ,\\
& u=g_{\rm so}(\bm{x}) \quad &&\text{on}\; \Gamma_2 \cup \Gamma_0 \,,\\
& \bm{b}^{(\rm w)}(\bm{x}) \cdot Du=0 \quad && \text{on}\; \Gamma_3 \,.
\end{aligned}
\right.
\end{equation}
For constants $\lambda\in(0,1)$, $M<\infty$, $\alpha\in(0,1)$, $\beta\in[\frac12,1)$,
$\delta\in(0,1)$, $\sigma\in(0,1)$, and $\varepsilon\in(0,\frac{h}{10})$,
assume that the following conditions are satisfied:
\begin{enumerate}[{\rm (i)}]
\item \label{vonN-Sec4-5-2:Assump-i}
For any $\bm{x}\in\Omega$ and $\bm{p}, \bm{\mu}\in\mathbb{R}^2$,
$\lambda \, {\rm dist}(\bm{x},\Gamma_{2}\cup\Gamma_0)|\bm{\mu}|^2\leq\sum_{i,j=1}^{2}\tilde{A}_{ij}(\bm{p},\bm{x})\mu_i\mu_j\leq\lambda^{-1}|\bm{\mu}|^2$\,.
Moreover, for any $\bm{x}\in\Omega\setminus\{\frac{\varepsilon}{2}
\leq x_1 \leq h-\frac{\varepsilon}{2}\}$,
\begin{equation*}\label{ChenFeldman-vonN-Eq:4-5-90}
\lambda |\bm{\mu}|^2
\leq\sum_{i,j=1}^{2}\frac{\tilde{A}_{ij}(\bm{p},\bm{x})\mu_i\mu_j}{\left(\min\{x_1,h-x_1,\delta\}\right)^{2-\frac{i+j}{2}}}
\leq\lambda^{-1}|\bm{\mu}|^2\,.
\end{equation*}

\item \label{vonN-Sec4-5-2:Assump-ii}
$(\tilde{A}_{ij},\tilde{A})(\bm{p},\bm{x})$ are independent of $\bm{p}\in\mathbb{R}^2$ on $\Omega\cap\{\varepsilon \leq x_1 \leq h-\varepsilon\}$ with
\begin{equation*}\label{ChenFeldman-vonN-Eq:4-5-91}
\|\tilde{A}_{ij}\|_{L^{\infty}(\Omega\cap\{\varepsilon\leq x_1\leq h-\varepsilon\})}\leq\lambda^{-1}\,,
\qquad \|(\tilde{A}_{ij},\tilde{A})\|_{1,\alpha,\Omega\cap\{\varepsilon\leq x_1\leq h-\varepsilon\}}\leq M \,.
\end{equation*}

\item \label{vonN-Sec4-5-2:Assump-iii}
For any $\bm{p}\in\mathbb{R}^2$,
\begin{equation*}\label{ChenFeldman-vonN-Eq:4-5-92}
\|(\tilde{A}_{ij},\tilde{A}_i)(\bm{p},\cdot)\|_{C^{\beta}(\overline{\Omega\setminus\{2\varepsilon\leq x_1 \leq h- 2\varepsilon \}})}+\|D_{\bm{p}}(\tilde{A}_{ij},\tilde{A}_i)(\bm{p},\cdot)\|_{L^{\infty}(\Omega\setminus\{2\varepsilon\leq x_1 \leq h- 2\varepsilon \})} \leq M \,.
\end{equation*}

\item \label{vonN-Sec4-5-2:Assump-iv}
$(\tilde{A}_{ij},\tilde{A}_i)\in C^{1,\alpha}(\mathbb{R}^2\times(\overline{\Omega}\setminus(\overline{\Gamma}_2\cup\overline{\Gamma_0})))$ and, for any $s\in(0,\frac{h}{4})$,
\begin{equation*}\label{ChenFeldman-vonN-Eq:4-5-93}
\|(\tilde{A}_{ij},\tilde{A}_i)\|_{1,\alpha,\mathbb{R}^2\times(\overline{\Omega}\cap\{s \leq x_1 \leq h-s\})} \leq M\Big(\frac{h}{s}\Big)^{M}.
\end{equation*}

\item \label{vonN-Sec4-5-2:Assump-v}
For each $(\bm{p},\bm{x})\in\mathbb{R}^2\times(\overline{\Omega}\setminus\{\varepsilon < x_1 < h-\varepsilon\})$
and $i,j=1,2$, define
\begin{equation*}\label{BCF-PMRef-AppendixC5-vi}
\hat{\bm{p}}\defeq\bm{p} - Dg_{\rm so}(\bm{x})\,,
\qquad (\tilde{a}_{ij},\tilde{a}_{i})(\hat{\bm{p}},\bm{x})\defeq (\tilde{A}_{ij},\tilde{A}_{i})(\hat{\bm{p}}+Dg_{\rm so}(\bm{x}),\bm{x}) \,.
\end{equation*}
For any $(\bm{p},(x_1,0)) \in \mathbb{R}^2 \times (\Gamma_3 \setminus \{\varepsilon \leq x_1 \leq h-\varepsilon \})$,
\begin{equation*}\label{ChenFeldman-vonN-Eq:4-5-94}
(\tilde{a}_{11},\tilde{a}_{22},\tilde{a}_1)((\hat{p}_1,-\hat{p}_2),(x_1,0))=(\tilde{a}_{11},\tilde{a}_{22},\tilde{a}_1)((\hat{p}_1,\hat{p}_2),(x_1,0)) \, .
\end{equation*}
For any $(\bm{p},\bm{x})\in\mathbb{R}^2\times(\Omega\setminus\{\varepsilon \leq x_1 \leq h-\varepsilon\})$
and $i=1,2$,
\begin{equation*}\label{ChenFeldman-vonN-Eq:4-5-95n96}
 |\tilde{a}_{ii}(\hat{\bm{p}},\bm{x})-\tilde{a}_{ii}(Dg_{\rm so}(\bm{x}_{\rm e}),\bm{x}_{\rm e})|\leq M|\bm{x}-\bm{x}_{\rm e}|^{\beta}\,,\qquad
 (\tilde{A}_{12},\tilde{A}_{21})(\bm{p},\bm{x}_{\rm e})=0\,,
\end{equation*}
either for point $\bm{x}_{\rm e}\defeq (0,x_2)$ or point $\bm{x}_{\rm e}\defeq (h,x_2)$.

\smallskip
\item \label{vonN-Sec4-5-2:Assump-vi}
For any $\bm{p}\in\mathbb{R}^2$ and
$\bm{x}\in\Omega\setminus\{ \frac{\varepsilon}{2} \leq x_1 \leq h-\frac{\varepsilon}{2}\}$,
$\, \tilde{A}_1(\bm{p},\bm{x})\leq-\lambda$\,.

\item \label{vonN-Sec4-5-2:Assump-vii}
The nonlinear boundary condition $\eqref{vonN-Sec4-5-2:NonlinearProblem}_2$ satisfies the following conditions:
{\begin{enumerate}[(\ref{vonN-Sec4-5-2:Assump-vii}--a)]
\item \label{vonN-Sec4-5-2:Assump-viia}
For any $(\bm{p},z,\bm{x})\in\mathbb{R}^2\times\mathbb{R}\times\Gamma_1$,
$D_{\bm{p}}\tilde{B}(\bm{p},z,\bm{x})\cdot\bm{\nu}^{(1)}(\bm{x})\geq\lambda$,
where $\bm{\nu}^{(1)}$ is the unit normal vector on $\Gamma_1$ pointing to $\Omega$.

\item \label{vonN-Sec4-5-2:Assump-viib}
For any $(\bm{p},z)\in\mathbb{R}^2\times\mathbb{R}$ and $k=1,2,3$,
\begin{align*}\label{ChenFeldman-vonN-Eq:4-5-99to102}
&\|\tilde{B}(Dg_{\rm so},g_{\rm so},\cdot)\|_{C^3(\overline{\Omega}\setminus\{\frac{h}{3}\leq x_1\leq \frac{2h}{3}\})}
+ \|D^{k}_{(\bm{p},z)}\tilde{B}(\bm{p},z,\cdot)\|_{C^3(\overline{\Omega})}\leq M\,,\\
&\|D_{\bm{p}}\tilde{B}(\bm{p},z,\cdot)\|_{C^0(\overline{\Omega})}\leq \lambda^{-1}\,,\\
&\begin{aligned}
&D_{z}\tilde{B}(\bm{p},z,\bm{x})\leq-\lambda
\quad && \text{for all}\; \bm{x}\in\Gamma_1\,,\\
&D_{p_1}\tilde{B}(\bm{p},z,\bm{x})\leq-\lambda
\quad && \text{for all}\; \bm{x} \in \Gamma_1\setminus\{\varepsilon\leq x_1\leq h-\varepsilon\}\,.
\end{aligned}
\end{align*}

\item \label{vonN-Sec4-5-2:Assump-viic}
There exist $v\in C^3(\overline{\Gamma_1})$ and a nonhomogeneous linear operator:
\begin{equation*}
L(\bm{p},z,\bm{x})=\bm{b}^{(1)}(\bm{x})\cdot \bm{p}+b^{(1)}_0(\bm{x})z+g_1(\bm{x})
\end{equation*}
defined for $\bm{x}\in\Gamma_1$ and $(\bm{p},z)\in\mathbb{R}^2\times\mathbb{R}$, satisfying
\begin{equation*}\label{ChenFeldman-vonN-Eq:4-5-103}
\|v\|_{C^3(\Omega)}+\|(\bm{b}^{(1)},b^{(1)}_0,g_1)\|_{C^3(\overline{\Gamma_1})} \leq M \,,
\end{equation*}
such that, for any $(\bm{p},z,\bm{x})\in\mathbb{R}^2\times\mathbb{R}\times\Gamma_1$,
\begin{equation*}\label{ChenFeldman-vonN-Eq:4-5-104}
\begin{aligned}
&\big|\tilde{B}(\bm{p},z,\bm{x})-L(\bm{p},z,\bm{x})\big|
  \leq\sigma \big( |\bm{p}-Dv(\bm{x})|+|z-v(\bm{x})| \big)\;,\\
&\big|D_{\bm{p}}\tilde{B}(\bm{p},z,\bm{x})-\bm{b}^{(1)}(\bm{x})\big|
  +\big|D_z\tilde{B}(\bm{p},z,\bm{x})-b_0^{(1)}(\bm{x})\big|\leq\sigma \,.
\end{aligned} \end{equation*}

\item \label{vonN-Sec4-5-2:Assump-viid}
Obliqueness requirements: For the interior unit normal vector $\bm{\nu}^{\rm (w)}$ on $\Gamma_3$ to $\Omega$,
\begin{equation*}\label{ChenFeldman-vonN-Eq:4-5-105to107}
\begin{aligned}
& \bm{b}^{\rm (w)}\cdot\bm{\nu}^{\rm (w)}\geq\lambda\,, \quad b^{\rm (w)}_1\leq0 \qquad \text{on $\Gamma_{3}$}\,,\\
&\bm{b}^{\rm (w)}=(0,1) \;\;\text{on $\Gamma_3\setminus\{\varepsilon \leq x_1 \leq h-\varepsilon\}$}\,,
\qquad \|\bm{b}^{\rm (w)}\|_{C^3(\overline{\Gamma_3})}\le M\,.
\end{aligned} \end{equation*}

\item \label{vonN-Sec4-5-2:Assump-viie}
$\,\,\tilde{B}(\bm{0},0,\cdot)\equiv0\quad$ on $\Gamma_1\setminus\{\varepsilon \leq x_1 \leq h-\varepsilon\}\,$.
\end{enumerate}}
\end{enumerate}

\begin{proposition}[{\cite[Proposition 4.7.2]{ChenFeldman-RM2018}}]
\label{Prop:AppendixB01}
Fix constants $h,t_1,t_2,t_3\in(0,\infty)$ and $t_0,t_h\in[0,\infty)$.
For constants $\lambda>0,$ $M<\infty,$ $\alpha\in(0,1),$ and $C_{\rm so}>0,$
domain $\Omega\subseteq\mathbb{R}^2$ satisfies
conditions~\eqref{ChenFeldman-vonN-Eq:4-5-1}--\eqref{BCF-Eq:C-5-3plus}.
For $\beta\in[\frac12,1),$ $\delta\in(0,1),$ $\sigma\in(0,1),$ and $\varepsilon\in(0,\frac{h}{10}),$
the nonlinear boundary value problem~\eqref{vonN-Sec4-5-2:NonlinearProblem} satisfies
conditions~\eqref{vonN-Sec4-5-2:Assump-i}--\eqref{vonN-Sec4-5-2:Assump-vii}.
Then there exist
$\alpha_1\in(0,\frac12)$ depending only on $\lambda,$
and constants $\delta_0, \sigma\in(0,1)$ depending only on $(\lambda,M,C_{\rm so},\alpha,\beta,\varepsilon)$ such that,
under further requirement $\delta\in[0,\delta_0),$
the boundary value problem~\eqref{vonN-Sec4-5-2:NonlinearProblem} has a unique solution
$u\in C(\overline{\Omega})\cap C^1(\overline{\Omega}\setminus(\overline{\Gamma_2}\cup\overline{\Gamma_0}))
\cap C^2(\Omega)$.
Moreover, $u$ satisfies
\begin{equation}\label{ChenFeldman-vonN-Eq:4-5-119n120p1}
\|u\|_{0,\overline{\Omega}} \leq C\,, \qquad\,
|u(\bm{x})-g_{\rm so}(\bm{x})|\leq C\min\{x_1,h-x_1\} \quad \text{in $\Omega$}\,,
\end{equation}
where $C>0$ depends only on $(\lambda,M,C_{\rm so},\varepsilon)$.
Furthermore, $u \in C(\overline{\Omega}) \cap C^{2,\alpha_1} (\overline{\Omega} \setminus \overline{\Gamma_2}
\cup \overline{\Gamma_0})$ satisfies
\begin{equation}\label{ChenFeldman-vonN-Eq:4-5-138p1}
\|u\|_{2,\alpha_1,\overline{\Omega\cap\{s<x_1<h-s\}}}\leq C_s
\end{equation}
for each $s\in(0,\frac{h}{10}),$ where $C_s>0$ depends only
on $(\lambda,M,C_{\rm so},\alpha,\beta,\varepsilon,s)$.
\end{proposition}

\begin{proposition}[{\cite[Proposition 4.8.7]{ChenFeldman-RM2018}}]
\label{Prop:AppendixB02}
Fix constants $h,t_1,t_2,t_3\in(0,\infty),$ $t_0=0,$ and $t_h\geq0$.
For constants $\lambda>0,$ $M<\infty,$ $\alpha\in(0,1),$ and $C_{\rm so}>0,$
domain $\Omega\subseteq\mathbb{R}^2$ satisfies
conditions~\eqref{ChenFeldman-vonN-Eq:4-5-1}--\eqref{BCF-Eq:C-5-3plus} with changes{\rm :}
$\, P_3=P_4=(0,0)$ and $\overline{\Gamma_0}=\{(0,0)\}$.
For $\beta\in[\frac12,1),$ $\delta\in(0,1),$ $\sigma\in(0,1),$ and $\varepsilon\in(0,\frac{h}{10}),$
the nonlinear boundary value problem~\eqref{vonN-Sec4-5-2:NonlinearProblem} satisfies
conditions~\eqref{vonN-Sec4-5-2:Assump-ii}{\rm,} \eqref{vonN-Sec4-5-2:Assump-iv}{\rm,}
and \eqref{vonN-Sec4-5-2:Assump-vii} above,
and {\rm (\ref{vonN-Sec4-5-2:Assump-i*}*)}{\rm,} {\rm (\ref{vonN-Sec4-5-2:Assump-iii*}*)}{\rm,}
{\rm (\ref{vonN-Sec4-5-2:Assump-v*}*)}{\rm,} and {\rm (\ref{vonN-Sec4-5-2:Assump-vi*}*)} below{\rm:}
\begin{enumerate}[{\rm (i*)}]
\item \label{vonN-Sec4-5-2:Assump-i*}
For any $\bm{x}\in\Omega$ and $\bm{p}, \bm{\kappa}=(\kappa_1\kappa_2)\in\mathbb{R}^2,$
\begin{equation*}\begin{aligned}
& \min\big\{\lambda\,{\rm dist}(\bm{x},\overline{\Gamma_0})+\delta, \lambda\,{\rm dist}(x,\Gamma_2)\big\}|\bm{\kappa}|^2
\leq \sum_{i,j=1}^2\tilde{A}_{ij}(\bm{p},\bm{x})\kappa_i\kappa_j\leq\lambda^{-1}|\bm{\kappa}|^2\,,\\
& \big\|((\tilde{A}_{ij},\tilde{A}_{i})(Dg_{\rm so},\cdot),
  D^{m}_{\bm{p}}(\tilde{A}_{ij},\tilde{A}_{i})(\bm{p},\cdot))\big\|^{(-\alpha),\{P_4\}}_{1,\alpha,\Omega\cap\{x_1<2\varepsilon\}}
  \leq M \qquad\text{for $m=1,2$}\,.
\end{aligned}\end{equation*}
Moreover, for any $\bm{x}\in\Omega\cap\{h-\frac{\varepsilon}{2}<x_1<h\},$
\begin{equation*}\label{ChenFeldman-vonN-Eq:4-5-90p2}
\lambda |\bm{\mu}|^2
\leq\sum_{i,j=1}^{2}\frac{\tilde{A}_{ij}(\bm{p},\bm{x})\mu_i\mu_j}{\left(\min\{h-x_1,\delta\}\right)^{2-\frac{i+j}{2}}}
\leq\lambda^{-1}|\bm{\mu}|^2\,.
\end{equation*}

\addtocounter{enumi}{+1}
\item \label{vonN-Sec4-5-2:Assump-iii*}
For any $\bm{p}\in\mathbb{R}^2,$
\begin{equation*}\label{ChenFeldman-vonN-Eq:4-5-92p2}
\|(\tilde{A}_{ij},\tilde{A}_i)(\bm{p},\cdot)\|_{C^{\beta}(\overline{\Omega\cap\{ h - 2\varepsilon < x_1 < h \}})}
+ \|D_{\bm{p}}(\tilde{A}_{ij},\tilde{A}_i)(\bm{p},\cdot)\|_{L^{\infty}(\Omega\cap\{ h- 2\varepsilon<x_1<h \})} \leq M\,.
\end{equation*}

\addtocounter{enumi}{+1}
\item \label{vonN-Sec4-5-2:Assump-v*}
For each $(\bm{p},\bm{x})\in\mathbb{R}^2\times(\overline{\Omega}\cap\{ h-\varepsilon<x_1<h\})$ and $i,j=1,2,$
define
\begin{equation*}\label{BCF-PMRef-AppendixC5-vip2}
\hat{\bm{p}}\defeq\bm{p} - Dg_{\rm so}(\bm{x})\,,
\qquad (\tilde{a}_{ij},\tilde{a}_{i})(\hat{\bm{p}},\bm{x})\defeq (\tilde{A}_{ij},\tilde{A}_{i})(\hat{\bm{p}}+Dg_{\rm so}(\bm{x}),\bm{x})\,.
\end{equation*}
For any $ (\bm{p},(x_1,0)) \in \mathbb{R}^2 \times (\Gamma_3\cap\{ h -\varepsilon <x_1<h\}) ,$
\begin{equation*}\label{ChenFeldman-vonN-Eq:4-5-94p2}
(\tilde{a}_{11},\tilde{a}_{22},\tilde{a}_1)((\hat{p}_1,-\hat{p}_2),(x_1,0))=(\tilde{a}_{11},\tilde{a}_{22},\tilde{a}_1)((\hat{p}_1,\hat{p}_2),(x_1,0))\,.
\end{equation*}
For any $(\bm{p},\bm{x})\in\mathbb{R}^2\times(\Omega\cap\{h-\varepsilon<x_1<h\})$ and $i=1,2,$
\begin{equation*}\label{ChenFeldman-vonN-Eq:4-5-95n96p2}
\quad|\tilde{a}_{ii}(\hat{\bm{p}},\bm{x})-\tilde{a}_{ii}(Dg_{\rm so}(\bm{x}_{\rm e}),\bm{x}_{\rm e})|\leq M|\bm{x}-\bm{x}_{\rm e}|^{\beta},
\quad
(\tilde{A}_{12},\tilde{A}_{21})(\bm{p},\bm{x}_{\rm e})=0
\qquad\mbox{for $\bm{x}_{\rm e}\defeq (h,x_2)$}\,.
\end{equation*}

\smallskip
\item \label{vonN-Sec4-5-2:Assump-vi*}
For any $\bm{p}\in\mathbb{R}^2$ and
$\bm{x}\in\Omega\cap\{ h-\frac{\varepsilon}{2} < x_1 < h \},$
$\,\, \tilde{A}_1(\bm{p},\bm{x})\leq-\lambda$\,.
\end{enumerate}
Then there exist both $\alpha_1\in(0,\frac12)$ depending only on $(\lambda,\delta)$
and $\sigma\in(0,1)$ depending only on $(\lambda,\delta,M,C_{\rm so},\alpha,\beta,\varepsilon)$ such that,
under the further requirement $\delta\in[0,\delta_0),$
the boundary value problem~\eqref{vonN-Sec4-5-2:NonlinearProblem} has a unique
solution $u\in C(\overline{\Omega})\cap C^1(\overline{\Omega}\setminus(\overline{\Gamma_2}\cup\overline{\Gamma_0}))\cap C^2(\Omega)$.
Moreover, $u$ satisfies
\begin{equation}\label{ChenFeldman-vonN-Eq:4-5-119n120p2}
\|u\|_{0,\overline{\Omega}} \leq C\,, \qquad\,
|u(\bm{x})-g_{\rm so}(\bm{x})|\leq C\min\{x_1,h-x_1\} \quad \text{in $\Omega$}\,,
\end{equation}
where $C>0$ depends only on $(\lambda,\delta,M,C_{\rm so},\varepsilon)$.
Furthermore,
$u \in C(\overline{\Omega}) \cap C^{2,\alpha_1}(\overline{\Omega}\setminus\overline{\Gamma_2}\cup\overline{\Gamma_0})$
satisfies
\begin{equation}\label{ChenFeldman-vonN-Eq:4-5-138p2}
\|u\|_{2,\alpha_1,\overline{\Omega\cap\{s<x_1<h-s\}}}\leq C_s
\end{equation}
for each $s\in(0,\frac{h}{10}),$ where $C_s>0$ depends only
on $(\lambda,\delta,M,C_{\rm so},\alpha,\beta,\varepsilon,s)$.
In addition, $u$ satisfies
\begin{equation*}
\|u\|^{(-1-\alpha_1),\{P_4\}}_{2,\alpha_1,\Omega\cap\{x_1<\frac{h}{4}\}}<\hat{C}\,,
\end{equation*}
for constant $\hat{C}>0$ depending only on $(\lambda,\delta,M,\alpha,\varepsilon)$.
\end{proposition}

\begin{proposition}
\label{Prop:AppendixB03}
Fix constants $h, t_1, t_2, t_3\in(0,\infty)$ and $t_0=t_h=0$. For constants $\lambda>0,$
 $M<\infty,$ $\alpha\in(0,1),$ and $C_{\rm so}>0,$ domain $\Omega\subseteq\mathbb{R}^2$ satisfies
conditions~\eqref{ChenFeldman-vonN-Eq:4-5-1}--\eqref{BCF-Eq:C-5-3plus} with changes{\rm:}
\begin{equation*}
\begin{aligned}
P_3=P_4=(0,0)\,,\quad  \overline{\Gamma_0}=\{(0,0)\}\,; \qquad\,
P_2=P_1=(h,0)\,,\quad  \overline{\Gamma_2}=\{(h,0)\}\,.
\end{aligned}
\end{equation*}
For $\beta\in[\frac12,1),$ $\delta\in(0,1),$ $\sigma\in(0,1),$ and $\varepsilon\in(0,\frac{h}{10}),$
the nonlinear boundary value problem~\eqref{vonN-Sec4-5-2:NonlinearProblem} satisfies
conditions~\eqref{vonN-Sec4-5-2:Assump-ii}{\rm,} \eqref{vonN-Sec4-5-2:Assump-iv}{\rm,}
and \eqref{vonN-Sec4-5-2:Assump-vii} above,  and {\rm (\ref{vonN-Sec4-5-2:Assump-i}*)} below{\rm:}
\begin{enumerate}[{\rm (i*)}]
\item For any $\bm{x}\in\Omega$ and $\bm{p}, \bm{\kappa}=(\kappa_1, \kappa_2)\in\mathbb{R}^2,$
\begin{equation*}
\begin{aligned}
& \min\big\{\lambda\,{\rm dist}(\bm{x},\overline{\Gamma_0})+\delta, \lambda\,{\rm dist}(x,\Gamma_2)+\delta\big\}|
\bm{\kappa}|^2
\leq \sum_{i,j=1}^2\tilde{A}_{ij}(\bm{p},\bm{x})\kappa_i\kappa_j\leq\lambda^{-1}|\bm{\kappa}|^2\,,\\
& \big\|((\tilde{A}_{ij},\tilde{A}_{i})(Dg_{\rm so},\cdot),
   D^{m}_{\bm{p}}(\tilde{A}_{ij},\tilde{A}_{i})(\bm{p},\cdot))\big\|^{(-\alpha),\{P_1,P_4\}}_{1,\alpha,\Omega\setminus\{2\varepsilon<x_1<h-2\varepsilon\}}
   \leq M
\qquad\text{for $m=1,2$}\,.
\end{aligned}\end{equation*}
\end{enumerate}
Then there exist constants $\alpha_1\in(0,\frac12)$ depending only on $(\lambda,\delta),$
and $\sigma\in(0,1)$ depending only on $(\lambda,\delta,M,C_{\rm so},\alpha,\beta,\varepsilon)$
such that, under the further requirement $\delta\in[0,\delta_0),$
the boundary value problem~\eqref{vonN-Sec4-5-2:NonlinearProblem} has a unique solution
$u\in C(\overline{\Omega})\cap C^1(\overline{\Omega}\setminus(\overline{\Gamma_2}
\cup\overline{\Gamma_0}))\cap C^2(\Omega)$.
Moreover, $u$ satisfies
\begin{equation}\label{ChenFeldman-vonN-Eq:4-5-119n120p3}
\|u\|_{0,\overline{\Omega}} \leq C\,, \qquad
|u(\bm{x})-g_{\rm so}(\bm{x})|\leq C\min\{x_1,h-x_1\}
\quad\text{in $\Omega$}\,,
\end{equation}
where
$C>0$ depends only on $(\lambda,\delta,M,C_{\rm so},\varepsilon)$.
Furthermore,
$u \in C(\overline{\Omega}) \cap C^{2,\alpha_1} (\overline{\Omega} \setminus \overline{\Gamma_2}
\cup \overline{\Gamma_0})$ satisfies
\begin{equation}\label{ChenFeldman-vonN-Eq:4-5-138p3}
\|u\|_{2,\alpha_1,\overline{\Omega\cap\{s<x_1<h-s\}}}\leq C_s
\end{equation}
for each $s\in(0,\frac{h}{10}),$ where $C_s>0$ depends only
on $(\lambda,\delta,M,C_{\rm so},\alpha,\beta,\varepsilon,s)$. In addition, $u$ satisfies
\begin{equation*}
\|u\|^{(-1-\alpha_1),\{P_1,P_4\}}_{2,\alpha_1,\Omega\setminus\{\frac{h}{4}\leq x_1
\leq \frac{3h}{4}\}}<\hat{C}\,,
\end{equation*}
for constant $\hat{C}>0$ depending only on $(\lambda,\delta,M,\alpha,\varepsilon)$.
\end{proposition}

\subsection{Regularized distance function}
\begin{lemma}[{\cite[Lemmas 13.9.1]{ChenFeldman-RM2018}}]
\label{Lem:RegularizedDist-13-9-1}
For any $g(\cdot)\in C^{0,1}([-1,1])$ that is positive on $(-1,1),$
there exists a function $\delta_{g}\in C^{\infty}(\overline{R_{\infty}}\setminus\overline{R_{g}}),$
\textit{the regularized distance}, such that
\begin{enumerate}[{\rm (i)}]
\item $|\delta_g(s,t') - {\rm dist}((s,t'),\Sigma_g)| \leq \frac{1}{2}{\rm dist}(\bfx,\Sigma_g)\,\,$
for all $(s,t')\in \mathbb{R}^2\setminus\overline{\Sigma_g}${\rm ;}

\item $|D^m\delta_g(s,t')|\leq C(m) \big({\rm dist}((s,t'),\Sigma_{g})\big)^{1-m}$
for all $(s,t')\in \mathbb{R}^2\setminus\overline{\Sigma_g}$ and $m=1,2,\cdots,$
where $C(m)>0$ depends only on $m${\rm ;}

\item  $\delta_{g}(s,t')\geq C_{\rm rd}(t'-g(s))$ for all $(s,t')\in\overline{R_g}\setminus\overline{\Sigma_g},$
where $C_{\rm rd}>0$ depends only on ${\rm Lip}[g]${\rm ;}

\item Let $g_k, g\in C^{0,1}([-1,1])$ be uniformly bounded Lipschitz functions for all $k\in\mathbb{N},$
with $g_k(s)$ converging to $g(s)$ uniformly on $[-1,1]$.
Then
$\,\, \lim\limits_{k\to\infty} \|\delta_{g_k}(s,t') - \delta_{g}(s,t')\|_{C^{m}(K)}=0$
for any $m=0,1,2,\cdots${\rm ,} and any compact set $K\subseteq \overline{R_{\infty}}\setminus\overline{R_g}$.
\end{enumerate}
\end{lemma}

\subsection{Leray-Schauder degree theorem}
\label{Sec:Appendix-B2}
\begin{definition}[Compact map]
\label{Def:CompactThm}
Let $X$ and $Y$ be Banach spaces.
Fix any open subset $G\subseteq X$.
A map $\mathbf{f}: \overline{G} \to Y$ is called compact if
\begin{enumerate}[{\rm (i)}]
\item $\mathbf{f}$ is continuous{\rm;}
\item $\mathbf{f}(U)$ is precompact in $Y$ for any bounded subset $U\subseteq\overline{G}$.
\end{enumerate}
\end{definition}

\begin{definition}
\label{Def:VertexSet}
Let $X$ be a Banach space, and let $G$ be a bounded open subset of $X$.
Write $V(G,X)$ to denote the set of all maps $\mathbf{f}:\overline{G}\to X$ satisfying
\begin{enumerate}[{\rm (i)}]
\item $\mathbf{f}$ is a compact map in the sense of {\rm Definition~\ref{Def:CompactThm};}
\item $\mathbf{f}$ has no fixed point on the boundary of $\partial G$.
\end{enumerate}
\end{definition}

\begin{definition}\label{Def:CompactlyHomotopic}
Let $X$ be a Banach space, and let $G$ be a bounded open subset of $X$.
Two maps $\mathbf{f}, \mathbf{g}\in V(G,X)$ are called compactly homotopic
on $\partial G$ if there exists a map $\mathbf{H}$ with the following properties{\rm:}
\begin{enumerate}[{\rm (i)}]
\item $\mathbf{H}:\overline{G}\times[0,1] \to X$ is a compact map{\rm ;}
\item $\mathbf{H}(\mathbf{x},\tau)\neq \mathbf{x}$ for all $(\mathbf{x},\tau)\in\partial G\times[0,1]${\rm ;}
\item $\mathbf{H}(\mathbf{x},0)=\mathbf{f}(\mathbf{x})$ and $\mathbf{H}(\mathbf{x},1)=\mathbf{g}(\mathbf{x})$
for all $\mathbf{x}\in\overline{G}$.
\end{enumerate}
\end{definition}

We write $\partial G:\mathbf{f}\cong \mathbf{g}$  if $\mathbf{f}$ and $\mathbf{g}$
are compactly homotopic on $\partial G$ in the sense of
Definition~\ref{Def:CompactlyHomotopic}, and we call $\mathbf{H}$ a compact homotopy.

\begin{theorem}[Leray-Schauder degree theorem]
\label{Thm:LeraySchauderDegree}
Let $X$ be a Banach space, and let $G$ be a bounded open subset of $X$.
Then, for each map $\mathbf{f}\in V(G,X),$ a unique integer $\mathbf{Ind}(\mathbf{f},G)$
can be assigned with the following properties{\rm :}
\begin{enumerate}[{\rm (i)}]
\item If $\mathbf{f}(\cdot)\equiv\mathbf{x}_0$ on $\overline{G}$
for some fixed $\mathbf{x}_0\in G,$ then $\mathbf{Ind}(\mathbf{f},G)=1${\rm ;}
\item If $\mathbf{Ind}(\mathbf{f},G)\neq0,$
then there exists a fixed point $\mathbf{x}\in G$ of map $\mathbf{f},$ that is,
$\mathbf{f}(\mathbf{x})=\mathbf{x}${\rm ;}
\item $\mathbf{Ind}(\mathbf{f},G)=\Sigma_{j=1}^{n}\mathbf{Ind}(\mathbf{f},G_j),$
whenever $\mathbf{f}\in V(G,X) \cap (\cap_{j=1}^{n}V(G_j,X)),$
where $G_i\cap G_j=\varnothing$ for $i\neq j$ and $\overline{G}=\cup_{j=1}^{n}\overline{G_j}${\rm ;}
\item If $\partial G:\mathbf{f}\cong \mathbf{g},$
then $\mathbf{Ind}(\mathbf{f},G)=\mathbf{Ind}(\mathbf{g},G)$.
\end{enumerate}
Such a number $\mathbf{Ind}(\mathbf{g},G)$ is called the fixed point index of $\mathbf{f}$ over $G$.
\end{theorem}

A generalized homotopy invariance of the fixed point index is given by the following theorem.

\begin{theorem}[{\cite[\S13.6, A4*]{Zeidler1986}}]
\label{Thm:HomotopyInvariance4FPI}
Let $X$ be a Banach space, and let $t_2>t_1$.
Let $U\subseteq X\times[t_1,t_2],$ and let $U_t=\{\mathbf{x} \,:\, (\mathbf{x},t)\in U\}$.
Then, for any compact map $\mathbf{h}:\overline{U}\to X$
with $\mathbf{h}(\mathbf{x},t)\neq\mathbf{x}$ on $\partial U,$
\begin{equation*}
\mathbf{Ind}(\mathbf{h}(\cdot,t),U_t)={\rm const}. \qquad \text{for all}\; t\in[t_1,t_2]\,,
\end{equation*}
provided that $U$ is a bounded open set in $X\times[t_1,t_2]$.
\end{theorem}

\subsection{Regularity theorem} \label{Sec:Appendix-B3}
\begin{theorem}[{\cite[Theorem~3.1]{BCF-2009}}]\label{Cite-BCF-2009:Thm3-1}
For constants
$r, \, R>0,$ define $Q_{r,R}^{+}$ by
\begin{equation*}
    Q^{+}_{r,R}\defeq \big\{(x,y)\in\mathbb{R}^2 \,:\, 0<x<r, \, |y|<R\big\}\,.
\end{equation*}
For positive constants $a,\,b,\,M,\,N,$ and $\kappa\in(0,\frac{1}{4}),$
suppose that $\psi\in C(\overline{Q^{+}_{r,R}})\cap C^{2}(Q^{+}_{r,R})$
satisfies $\psi>0,$ $-Mx\leq \psi_{x}\leq \frac{2-\kappa}{a}x,$ and
\begin{equation*}
    (2x-a\psi_{x}+O_{1})\psi_{xx}+O_{2}\psi_{xy}+(b+O_{3})\psi_{yy}-(1+O_{4})\psi_{x}+O_{5}\psi_{y}=0
\qquad\, \mbox{in $Q^{+}_{r,R}$}\,,
\end{equation*}
 and $\psi=0$ on $\partial Q^{+}_{r,R}\cap\{x=0\},$
where $O_{k}(x,y),$ $k=1,\cdots,5,$ are continuously differentiable and
\begin{equation*}
    \frac{|O_{1}(x,y)|}{x^2}+\frac{|D_{(x,y)}O_{1}(x,y)|}{x}+
\sum\limits_{k=2}^{5}\Big(\frac{|O_{k}(x,y)|}{x}+
|D_{(x,y)}O_{k}(x,y)|\Big)\leq N \qquad \mbox{in $Q^{+}_{r,R}$}\,.
\end{equation*}
Then $\psi\in C^{2,\alpha}(\overline{Q^{+}_{{r}/{2},{R}/{2}}})$
for any $\alpha\in(0,1)$ and
\begin{flalign*}
&&
    \psi_{xx}(0,y)=a^{-1}\, ,
    \quad\,\, \psi_{xy}(0,y)=\psi_{yy}(0,y)=0
    \qquad\mbox{for all $y\in(-\frac{R}{2},\frac{R}{2})$}\,.
    &&
\end{flalign*}
\end{theorem}

\subsection{General framework for the convexity of transonic shocks}
\label{Sec:Appendix-B-Convexity}
We state a general framework for the transonic shock as a free boundary, developed in Chen-Feldman-Xiang~\cite{ChenFeldmanXiang2020ARMA}.

Let $\Omega$ be an open, bounded, and connected set,
and $\partial\Omega=\Gamma_{{\rm shock}}\cup \Gamma_{1}\cup \Gamma_{2}$,
where the closed curve segment $\Gamma_{{\rm shock}}$ is a transonic shock that separates
a pseudo-supersonic constant state denoted as state $(2)$ outside $\Omega$
from a pseudo-subsonic (non-constant) state inside $\Omega$, and
$\Gamma_{1}\cup \Gamma_{2}$ is a fixed boundary.

We now present a more structural framework for domain $\Omega$ under consideration.

\medskip
\noindent
{\bf{Framework {\rm (A)}}}.  The structural framework for domain $\Omega$:
\begin{enumerate}[{\rm (i)}]
\item Domain $\Omega$ is bounded, with boundary $\partial\Omega$ that is a continuous closed curve without self-intersections.
Furthermore, $\partial\Omega$ is piecewise $C^{1,\alpha}$ up to the endpoints of each smooth part for some $\alpha\in(0,1)$,
and the number of smooth parts is finite.

\item At each corner point of $\partial\Omega$, angle $\theta$ between the arcs meeting at that point
from the interior of $\Omega$ satisfies $\theta\in (0,\pi)$.

\item  $\partial\Omega=\Gamma_{{\rm shock}}\cup \Gamma_{1}\cup \Gamma_{2}$,
where $\Gamma_{{\rm shock}}$, $\Gamma_{1}$, and $\Gamma_{2}$ are connected and disjoint,
and both $\Gamma_{{\rm shock}}^{0}$ and $\Gamma_{1}\cup \Gamma_{2}$ are non-empty.
Moreover, if $\Gamma_{i}\neq\varnothing$ for some $i\in\{1,2\}$,
then its relative interior is nonempty, {\it i.e.},~$\Gamma_{i}^{0}\neq\varnothing$.

\item $\Gamma_{{\rm shock}}$ includes its endpoints $A$ and $B$ with corresponding unit tangent vectors
$\bm{\tau}_{A}$ and $\bm{\tau}_{B}$ pointing to the interior of $\Gamma_{{\rm shock}}$, respectively.
If $\Gamma_{1}\neq\varnothing$, then $A$ is a common endpoint of  $\Gamma_{{\rm shock}}$
and $\Gamma_{1}$. If $\Gamma_{2}\neq\varnothing$,
then $B$ is a common endpoint of  $\Gamma_{{\rm shock}}$ and $\Gamma_{2}$.
\end{enumerate}

Let $\phi=\varphi-\varphi_{2}$.
Then $\phi$ satisfies the following equation:
\begin{equation}
\label{eqc1}
    (c^2-\varphi_{\xi_1}^2)\phi_{\xi_1\xi_1}-2\varphi_{\xi_1}\varphi_{\xi_2}\phi_{\xi_1\xi_2}+(c^2-\varphi_{\xi_2}^2)\phi_{\xi_2\xi_2}=0 \qquad\text{in $\Omega$}\,,
\end{equation}
with the boundary conditions:
\begin{equation}
\label{eqc2}
\phi=0\,, \quad
\rho(|D\phi+D\varphi_2|^2,\phi+\varphi_2)D(\phi+\varphi_2)\cdot \bm{\nu}=\rho_2 D\varphi_2\cdot \bm{\nu}
\qquad\,\, \text{on $\Gamma_{\rm{shock}}$}\,.
\end{equation}

If $\bm{\tau}_{A}\neq\pm \bm{\tau}_{B}$, define the cone:
\begin{equation} \label{eq:appendixB-cone}
    {\rm Con} \defeq \{r\bm{\tau}_{A}+s\bm{\tau}_{B} \,:\, r,\, s > 0\}\,.
\end{equation}

\begin{theorem}[{\cite[Theorem 2.1]{ChenFeldmanXiang2020ARMA}}]
\label{ThC1}
Assume that domain $\Omega$ satisfies {\rm Framework (A)}.
Assume that $\phi\in C^{1}(\overline{\Omega})\cap C^{2}(\Omega\cup \Gamma_{\rm{shock}}^{0}) \cap C^{3}(\Omega)$
is a solution of~\eqref{eqc1}--\eqref{eqc2}{\rm,} which is not a constant state in $\Omega$.
Moreover, let $\phi$ satisfy the following conditions{\rm :}
\begin{enumerate}[{\rm (C-1)}]
\item \label{Thm-C1-item1} The entropy condition holds
across $\Gamma_{\rm{shock}}${\rm:} $\rho(|D\varphi|^2,\varphi)>\rho_2$
and $\phi_{\bm{\nu}}<0$ along $\Gamma_{\rm{shock}},$ where $\bm{\nu}$ is the unit normal vector on $\Gamma_{\rm{shock}}$ pointing to $\Omega${\rm;}

\item \label{Thm-C1-item2} There exist constants $C_1>0$ and $\alpha_1\in (0,1)$
such that $\|\phi\|_{1+\alpha_1,\overline{\Omega}}\leq C_1${\rm;}

\item \label{Thm-C1-item3}
Equation~\eqref{eqc1} is strictly elliptic in $\Omega\cup \Gamma_{\rm{shock}}^0${\rm:}
$c^2-|D\varphi|^2>0$ in $\Omega\cup \Gamma_{\rm{shock}}^0${\rm;}

\item \label{Thm-C1-item4} $\Gamma_{\rm{shock}}$ is $C^2$ in its relative interior{\rm;}

\item \label{Thm-C1-item5}
$\bm{\tau}_{A}\neq\pm \bm{\tau}_{B},$ and $\{P+{\rm Con}\}\cap \Omega=\varnothing$ for any point $P\in \overline{\Gamma_{\rm{shock}}}${\rm;}

\item \label{Thm-C1-item6}
There exists a vector $\bm{e}\in {\rm Con}$ such that one of the following conditions holds{\rm :}
\begin{enumerate}[{\rm (C-\ref{Thm-C1-item6}a)}]
    \item
    \label{Thm-C1-item6-a}
    $\Gamma_{1}\neq \varnothing,$ and the directional derivative $\phi_{\bm{e}}$ cannot have
    a local maximum point on $\Gamma_{1}^{0}\cup \{A\}$ and a local minimum point on $\Gamma_{2}^{0}${\rm;}
    \item
    \label{Thm-C1-item6-b}
    $\Gamma_{2}\neq \varnothing,$ and  $\phi_{\bm{e}}$ cannot have a local minimum point
    on $\Gamma_{1}^{0}$ and a local maximum point on $\Gamma_{2}^{0}\cup\{B\}${\rm;}
    \item
    \label{Thm-C1-item6-c}
    $\phi_{\bm{e}}$ cannot have a local minimum point on $\Gamma_{1}\cup\Gamma_{2}${\rm;}
\end{enumerate}
where all the local maximum or minimum points are relative to $\overline{\Omega}$.
\end{enumerate}
Then the free boundary $\Gamma_{\rm{shock}}$ is a convex graph in each direction \(\bm{e} \in {\rm Con}\).
That is, there exists a concave function $f\in C^{1,\alpha}(\mathbb{R})$ in the orthogonal coordinate system \(\bm{\xi} \eqdef S \bm{e} + T\bm{e}^\perp\)
such that
\begin{equation}
\label{eqc3}
\begin{aligned}
& \Gamma_{\rm{shock}}=\big\{\bm{\xi}(S,T)\in\mathbb{R}^2 \,:\, S=f(T), \, T_A<T<T_B\big\}\,,\\
& \Omega\cap\big\{\bm{\xi} \,:\, T_A<T<T_B\big\}\subseteq\big\{\bm{\xi} \,:\, S<f(T)\big\}\,.
\end{aligned}
\end{equation}
Moreover, the free boundary $\Gamma_{\rm{shock}}$ is strictly convex in its relative interior in the sense that,
if $P= \bm{\xi} (S,T)\in \Gamma_{\rm{shock}}^0$ and $f^{\prime\prime}(T)=0,$
then there exists an integer $m>1,$ independent of the choice of \(\bm{e} \in {\rm Con}\) such that,
for any $n=2,\cdots,2m-1,$
\begin{equation}
\label{eqc4}
    f^{(n)}(T)=0 \,, \qquad f^{(2m)}(T)<0\,.
\end{equation}
The number of the points at which $f^{\prime\prime}(T)=0$ is at most finite on each compact subset of $\Gamma_{\rm{shock}}^0$.
In particular, the free boundary $\Gamma_{\rm{shock}}$ cannot contain any straight segment.
\end{theorem}

Furthermore, under some additional assumptions, we can show that the shock curve is uniformly convex
in its relative interior in the sense defined in the following theorem:

\begin{theorem}[{\cite[Theorem 2.3]{ChenFeldmanXiang2020ARMA}}]
\label{ThC2}
Let $\Omega$ and $\phi$ be as in {\rm Theorem~\ref{ThC1}}.
Furthermore, assume that, for any unit vector $\bm{e}\in \mathbb{R}^2,$
the boundary part $\Gamma_{1}\cup\Gamma_{2}$ can be further decomposed so that
\begin{enumerate}[{\rm (C-1)}]
\addtocounter{enumi}{+6}
\item \label{Thm-C1-item7}
$\Gamma_{1}\cup\Gamma_{2}=\hat{\Gamma}_{0}\cup\hat{\Gamma}_{1}\cup\hat{\Gamma}_{2}\cup\hat{\Gamma}_{3},$
where some $\hat{\Gamma}_{i}$ may be empty, $\hat{\Gamma}_{i}$ is connected for each $i=0,1,2,3,$
and all curves $\hat{\Gamma}_{i}$ are located along $\partial\Omega$ in the order of their indices,
{\it i.e.},~non-empty sets $\hat{\Gamma}_{j}$ and
$\hat{\Gamma}_{k},$ $k>j,$ have a common endpoint if and only
if either $k=j+1$ or $\hat{\Gamma}_{i}=\varnothing$ for all $i=j+1,\cdots,k-1$.
Also, the non-empty set $\hat{\Gamma}_{i}$ with the smallest $($resp.~largest$)$ index has
the common endpoint $A$ $($resp.~$B$$)$ with $\Gamma_{\rm{shock}}$.
Moreover, if  $\hat{\Gamma}_{i}\neq\varnothing$ for some $i=0,1,2,3,$
then its relative interior is nonempty{\rm :} $\hat{\Gamma}_{i}^{0}\neq\varnothing${\rm;}

\item \label{Thm-C1-item8}
$\phi_{\bm{e}}$ is constant along $\hat{\Gamma}_{0}$ and $\hat{\Gamma}_{3}${\rm;}

\item \label{Thm-C1-item9}
For $i=1,2,$ if $\phi_{\bm{e}}$ attains its local minimum or maximum relative
to $\overline{\Omega}$ on $\hat{\Gamma}_{i}^{0},$ then $\phi_{\bm{e}}$ is constant along $\hat{\Gamma}_{i}${\rm;}

\item \label{Thm-C1-item10}
One of the following two conditions holds{\rm :}
\begin{enumerate}[{\rm (C-\ref{Thm-C1-item10}a)}]
    \item Either $\hat{\Gamma}_{1}=\varnothing$ or $\hat{\Gamma}_{2}=\varnothing${\rm;}
    \item Both $\hat{\Gamma}_{1}$ and  $\hat{\Gamma}_{2}$  are non-empty, and $\hat{\Gamma}_{3}=\varnothing,$
    so that $\hat{\Gamma}_{2}$ has the common endpoint $B$ with  $\Gamma_{\rm{shock}}$.
    At point $B,$ the following conditions hold{\rm :}
    \begin{enumerate}
        \item[{\rm (1*)}] If $\bm{\nu}_{\rm{sh}}(B)\cdot\bm{e}<0,$ then $\phi_{\bm{e}}$ cannot attain
        its local maximum relative to $\overline{\Omega}$ at $B${\rm;}
         \item[{\rm (2*)}] If $\bm{\nu}_{\rm{sh}}(B)\cdot\bm{e}=0,$ then $\phi_{\bm{e}}(B)=\phi_{\bm{e}}(Q^*)$
         for the common endpoint $Q^*$ of $\hat{\Gamma}_{1}$ and  $\hat{\Gamma}_{2}${\rm;}
    \end{enumerate}
    where $\bm{\nu}_{\rm{sh}}(B) \defeq \lim\limits_{\Gamma^{0}_{\rm{shock}}\ni P\rightarrow B}\bm{\nu}(P),$
    which exists since $\Gamma_{\rm{shock}}$ is $C^1$ up to $B$.
\end{enumerate}
\end{enumerate}
Then the shock function $f(T)$ in~\eqref{eqc3} satisfies $f^{\prime\prime}(T)<0$ for all
$T\in(T_A,T_B),$ that is,  $\Gamma_{\rm{shock}}$ is uniformly convex on closed subsets of
its relative interior.
\end{theorem}
\end{appendices}

\medskip
\noindent \textbf{Acknowledgements:}
\addcontentsline{toc}{section}{Acknowledgements}
The research of Gui-Qiang G.~Chen is supported in part
by the UK Engineering and Physical Sciences Research Council Awards
EP/L015811/1, EP/V008854/1, and EP/V051121/1.
The research of Alex Cliffe is supported in part by the UK Engineering and Physical Sciences Research Council Awards EP/N509711/1 and EP/R513295/1.
The research of Feimin Huang is supported in part by the National Key R\&D Program of China No.~2021YFA1000800,
and the National Natural Sciences Foundation of China No.~12288201.
The research of Song Liu is supported in part by the Hong Kong Institute of Advanced Study and the GRF grant CityU~11300420.
The research of Qin Wang is supported in part by the National Natural Sciences Foundation of China No.~12261100.

\end{document}